\documentclass[12pt]{amsart}

\usepackage{fancyhdr} % 引入 fancyhdr 宏包
\fancypagestyle{plain}{
\fancyhead{} % 页眉清空
\fancyfoot{} % 页脚清空
\fancyfoot[C]{\thepage} % 中间显示页数
}
\usepackage{amsaddr}

\usepackage{mathrsfs}  
\usepackage{mathtools}
\usepackage{graphicx} % Required for inserting images
\usepackage{hyperref}

\usepackage{array}
\usepackage{multirow}
\usepackage{makecell}
\usepackage{float}
\usepackage{subcaption}
\usepackage{tikz,tikz-cd}
\usepackage{adjustbox}

\newtheorem{thm}{Theorem}[section]

\newtheorem{prop}{Proposition}[section]
\newtheorem{cor}{Corollary}[section]
\newtheorem{eg}{Example}[section]
\newtheorem{rem}{Remark}[section]

\title{Gauss diagram formulae for Vassiliev invariants from Kauffman polynomial}
\author{Butian Zhang}
\address{Institut de Mathématiques de Toulouse, Université Paul Sabatier}
\email{butian.zhang@math.univ-toulouse.fr}
\begin{document}

\maketitle
\begin{abstract}
A state model for Kauffman polynomial of Dubrovnik-version is given. Based on the state model, the Gauss diagram formulae for Vassiliev invariants are given from the coefficients of Kauffman polynomial following the method of Chmutov and Polyak. Some arrow diagram identities are given to simplify the Gauss diagram formulae of order 3, which give Polyak-Viro and Chmutov-Polyak formulae for the Vassiliev invariant of order 3. The models of Kauffman polynomial and HOMFLY-PT polynomial give different Gauss diagram expressions when specializing to Jones poynomial. 
\end{abstract}
\tableofcontents
%%%%%%%%%%%%%%%%%%%%%%%%%%%%%
%%%%%%%%%%%%%%%%%%%%%%%%%%%%%
\section{Introduction}
Polyak and Viro introduced the tool of arrow diagrams in \cite{Polyak_Viro:1994} to represent Vassiliev invariants (\cite{Vassiliev:1990}, \cite{Birman:1993aa}, \cite{BARNATAN:1995}). An arrow diagram is a diagram with several circles and signed arrows on the circles. The arrows represent the crossings. The foot of the arrow represents the bottom strand at the crossing and the head represents the top strand. An arrow diagram which can be realised by a knot in $\mathbb{R}^3$ is called a Gauss diagram. 
\begin{figure}[H]
\centering
    \begin{subfigure}{0.4\textwidth}
        \centering
        \includegraphics[width=0.3\textwidth]{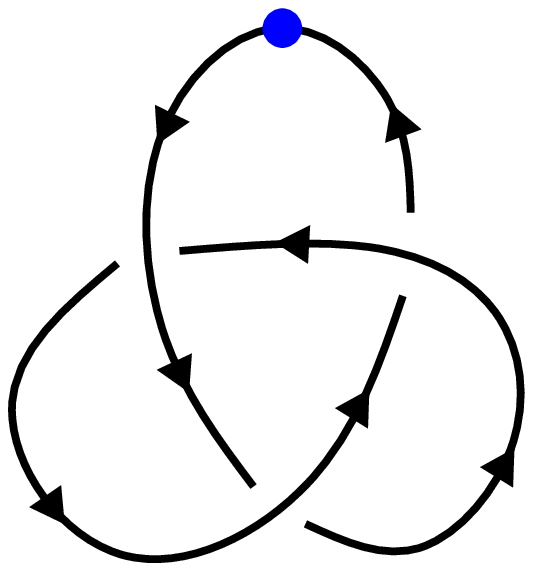}
        \caption{A trefoil}
    \end{subfigure}
    \hfill
    \begin{subfigure}{0.4\textwidth}
        \centering
        \includegraphics[width=0.3\textwidth]{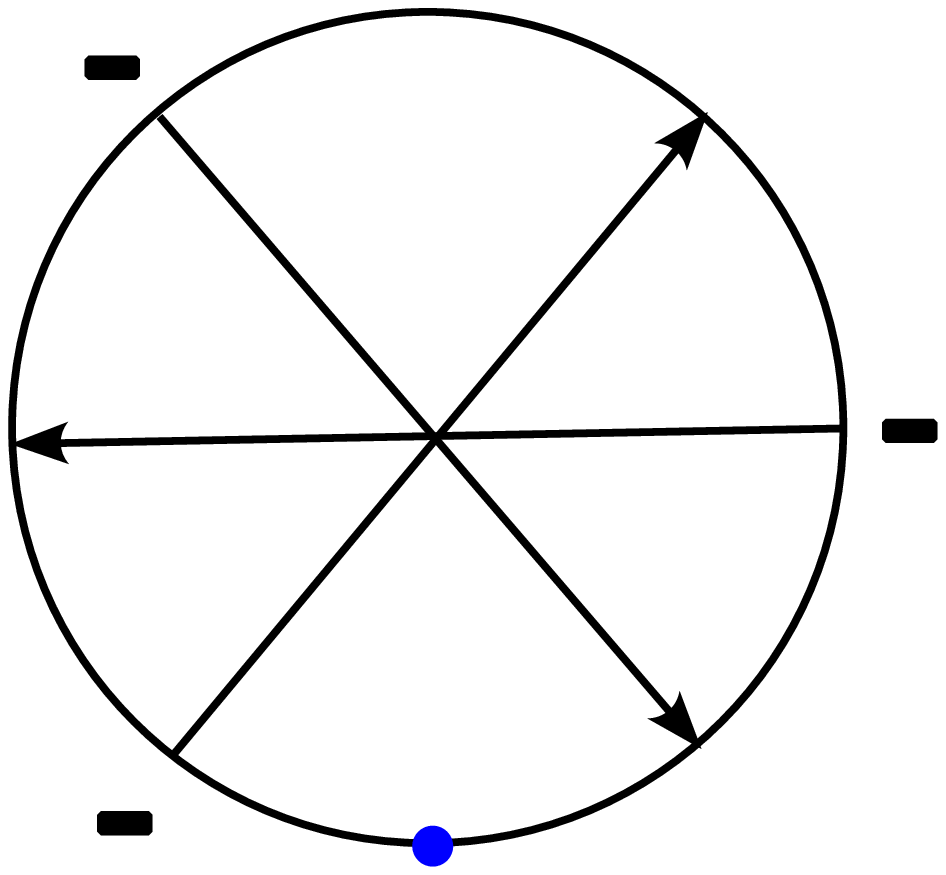}
        \caption{Gauss diagram of the trefoil}
    \end{subfigure}
    \caption{Knot diagram and its Gauss diagram}
    \label{fig.eg_Gauss_diagram}
\end{figure}

It is proved in \cite{Goussarov:1998} that all Vassiliev invariants can be expressed by Gauss diagram formulae (see Section \ref{sec.GDF} for the definition). And after a change of variable all the coefficients of HOMFLY-PT polynomial and Kauffman polynomial are Vassiliev invariants \cite{Birman:1993aa}. Chmutov and Polyak derived the Gauss diagram formulae for the coefficients of HOMFLY-PT polynomial in \cite{Chmutov_Polyak:2009} based on a state model of HOMFLY-PT polynomial \cite{Jaege:1990}. Using the method in \cite{Chmutov_Polyak:2009}, one can derive the Gauss diagram formulae for the coefficients $H_{4,0}, H_{3,1}, H_{2,2}, H_{1,3}$ and $H_{4,0}$ of the terms of degree 4 of HOMFLY-PT polynomial. By checking each term, we have 
\begin{prop}
    $H_{3,1} = H_{1,3} = 0$ and $48H_{0,4} + 12H_{2,2} + 3H_{4,0} + 4H_{0,2} = 0$
    where $H_{k,l}$ is the coefficient of the term $h^kz^l$ in HOMFLY-PT polynomial $H_{K}(e^h,z) = \sum\limits_{k,l}H_{k,l}(K)h^kz^l$. 
\end{prop}

This tells us that the 3 nontrivial Vassiliev invariants of order $4$ are not independent. However, according to Bar-Natan \cite{BARNATAN:1995}, the space of Vassiliev invariants of order 4 has dimension 3. This inspires us to find the other Vassiliev invariants of order 4. We turn to Kauffman polynomial. By the calculation of the Kauffman polynomial (Dubrovnik version) on several knots, we can see that there are 4 nontrivial Vassiliev invariants of order 4 in Kauffman polynomial and any 3 of them are independent. 

The aim of this paper is to derive the Gauss diagram formulae from Kauffman polynomial which is needed for the work of the construction of combinatorial 1-cocycles in a moduli space of knots (see \cite{Fiedler:2023}, \cite{Gros:2022}). For this purpose, we give a state model in Section \ref{sec.state_model}, use the method of Chmutov and Polyak to derive the Gauss diagram formulae in Section \ref{sec.GDF} and simplify the results in order $3$ by using some identities in Section \ref{sec.identities}. And we compare the Gauss diagram formulae from Jones polynomial based on the 2 different state models in Section \ref{sec.GDF_Jones} as suggested by Chmutov. 

The author is grateful to Thomas Fiedler for his encouragement, to Hongmei Li for her companionship, and to Louis Kauffman and Sergei Chmutov for their suggestions. The work is supported by a grant of CSC. 

%%%%%%%%%%%%%%%%%%%%%%%%%%%%%
%%%%%%%%%%%%%%%%%%%%%%%%%%%%%
\section{State model of Kauffman polynomial}\label{sec.state_model}

The 2-variable Kauffman polynomial of Dubrovnik version $D_K(a,z)$ of an unoriented link $K$ (see \cite{Kauffman:2001}) is defined by the following rules:
\begin{enumerate}
    \item\label{regular_isotopy_relation} if $K$ and $K'$ are regular isotopic, $D_K = D_{K'}$
    \item \label{skein_relation} \begin{align*}
        D_{\includegraphics[width = 0.7cm]{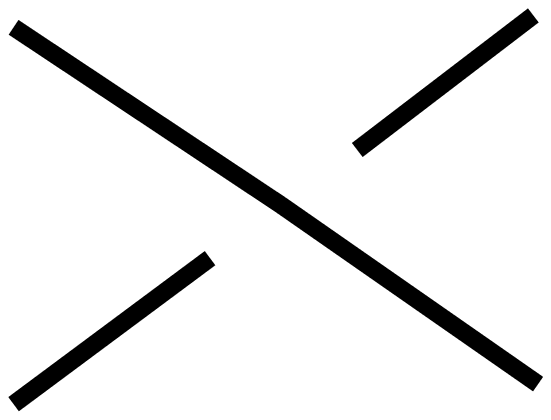}} - D_{\includegraphics[width = 0.7cm]{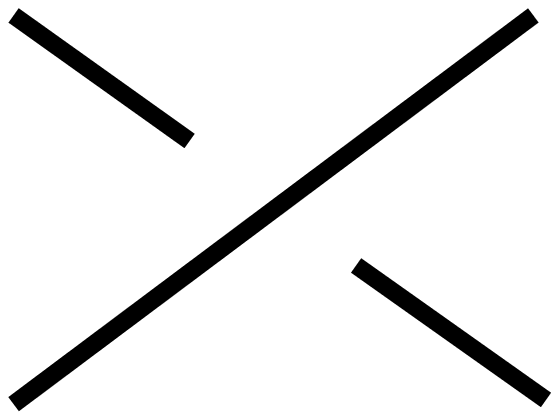}} = z(D_{\includegraphics[width = 0.7cm]{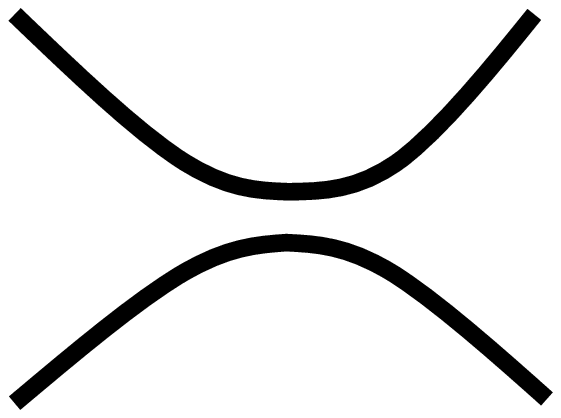}} - D_{\includegraphics[width = 0.7cm]{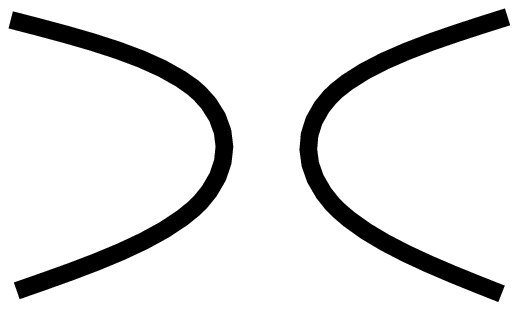}})
        \end{align*}
    \item \label{curl_relation}\begin{align*}
        D_{\includegraphics[width = 0.3cm]{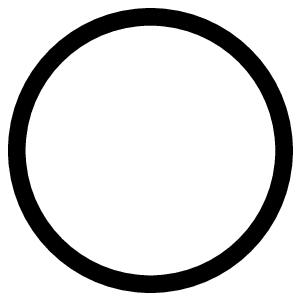}} &= Dd = D\cdot (\dfrac{a-a^{-1}}{z}+1)\\
        D_{\includegraphics[width = 0.5cm]{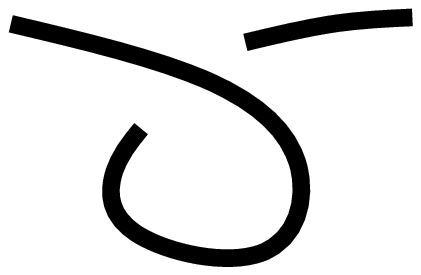}} &= aD\\
        D_{\includegraphics[width = 0.5cm]{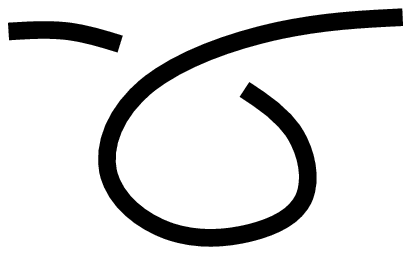}} &= a^{-1}D
    \end{align*}
\end{enumerate}
For an unoriented link $K$, we can first endow an auxiliary orientation on each component in any manner. Then the 
skein relation \ref{skein_relation} can be interpreted as 
$$D_{+} - D_{-} = z(D_0 - D_{\infty})$$
where $+,-$ denote the links with only one single distinct crossing with different signs, $0$ denotes the link doing "smooth" at this crossing and $\infty$ denotes the links doing "singularize". These notations are consistent with the notation of Birman and Lin \cite{Birman:1993aa}. Notice that in the case of $D_{\infty}$, the orientation adopted at first has to be changed on some arcs of the links which will be stated in detail soon. This change of orientation will not impact the result of $D_K$ because it is an unoriented knot invariant and the orientation we choose is just an auxiliary tool. 

Given an oriented link $K$ we define $$DK_K(a,z) = a^{-w(K)}D_K$$ 
where $w(K)$ denotes the writhe of $K$. Then $DK$ is an invariant for oriented links. In the rest of this paper, we shall focus on knots, i.e. the links with only one component and we will not distinguish the projection of a knot into a plane with the knot itself. And for a knot $K$, the relation between the usual Kauffman polynomial $F_K(a,z)$ (see \cite{Kauffman:2001}) and the Dubrovnik version is $$DK_{K}(a,z) = F_{K}(-ia,iz)$$

Based on the definition of Kauffman polynomial, there is a direct algorithm to calculate it. Given an oriented knot $K$, we choose a base point $B$ on it. Going along the orientation of the knot, when we come across a crossing in the first time, if we are on the top strand of this crossing, we go over this crossing and continue. If we are on the bottom strand of this crossing, we use the skein relation \ref{skein_relation} to change the sign of this crossing such that we are on the top strand of this crossing. Each time we use the  skein relation, we will get two more knots to calculate. One is $K_0$ and another is $K_\infty$. The $D_0$ has a consistent orientation while $K_\infty$ not. Hence we shall adopt a new orientation on $D_\infty$ satisfying that the orientation on the path that we have already passed will never change. There are two cases:
\begin{enumerate}
    \item \label{case.one_component} {the crossing is a self-crossing of a single component}
    \item \label{case.two_components} {the crossing is of two components}
\end{enumerate}
In the case \ref{case.one_component}, if we remove a small neighbourhood of this crossing, the component will be cut into two piece. One piece is that we have already passed. Another piece is the one that we have not passed. We remain the orientation on the piece that we have already passed and reverse the orientation on the piece that we have not passed yet. 
In the case \ref{case.two_components}, we remain the orientation on the component that we have passed and reverse the orientation on the component that we have not passed. 
It should be noticed that for a singularized crossing $P$ each time the orientation on an arc is changed, the sign of the crossings (excluding $P$) on this arc are also changed: the sign changes once if it is the crossing of this arc with other parts of the knots; the sign changes twice (i.e. remains) if it is a crossing of the arc with itself. 

When we smoothen a crossing, we may split a single component into two or connect 2 components into one. In the first case, we shall add a new base point on the new component at some place (the arc with the smallest order which is chosen at first). For the second case, we shall remove the base point of the component that we have not passed. 
When we singularize a crossing, we will always get a one-component result: from 1 component to 1 component or from 2 components to a single component. In the second case, we shall also remove the base point on the component that we have not passed. 

Using this algorithm we will finally get several links such that from the base point on each component, we will meet the crossing always on the top strand first. Hence they are isotopic (may not regular) to the unknot. 
Using curl relation and adding all the results according to the skein relation, we will get the $D_K$ and hence $DK_K$.

\begin{figure}[h]
  \centering
  \begin{tikzpicture}[level distance=3cm,
    level 1/.style={sibling distance=5cm},
    level 2/.style={sibling distance=3cm},
    level 3/.style={sibling distance=0.75cm}]
    % Level 1
    \node {\includegraphics[width=2cm]{image/trefoil_left.eps}}
    % Level 2
    child {node[label=below:{ $a^{-1}$}] {\includegraphics[width=2cm]{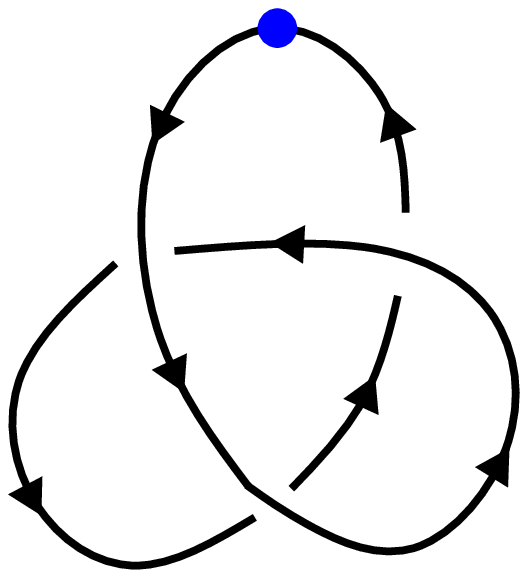}}
    }
    %%%
    child {node[label=above:{ $-z$}] {\includegraphics[width=2cm]{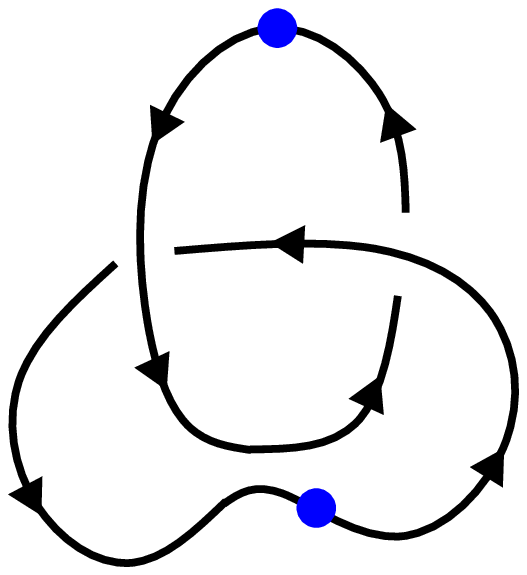}}
      % Level 3
      child {node[label=below:{ $\dfrac{a-a^{-1}}{z}+1$}] {\includegraphics[width=2cm]{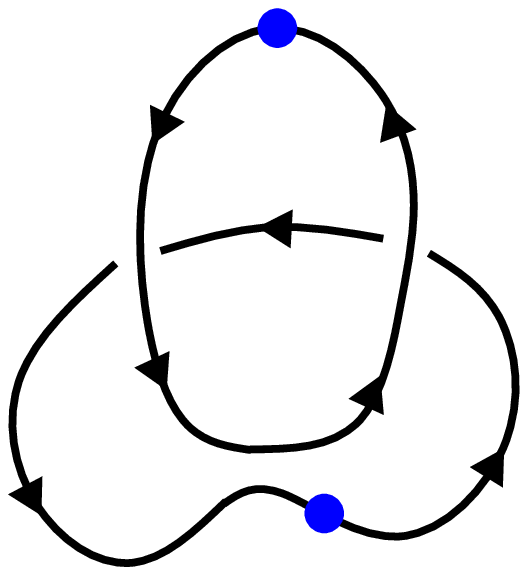}}}
      child {node[label=above:{ $-z$}, label=below:{$a^{-1}$}] {\includegraphics[width=2cm]{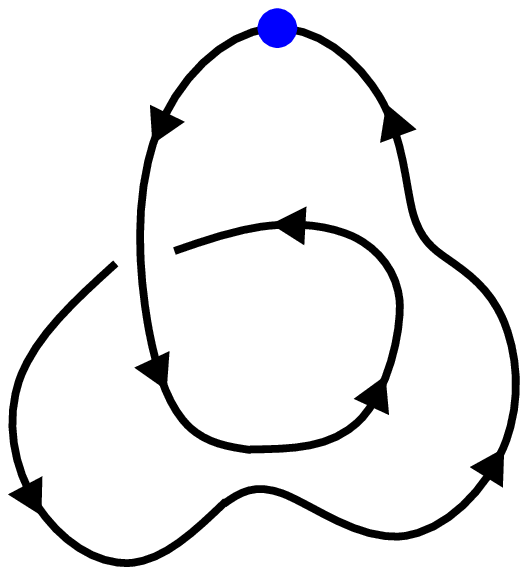}}}
      child {node[label=above:{ $z$},label=below:{$a$}] {\includegraphics[width=2cm]{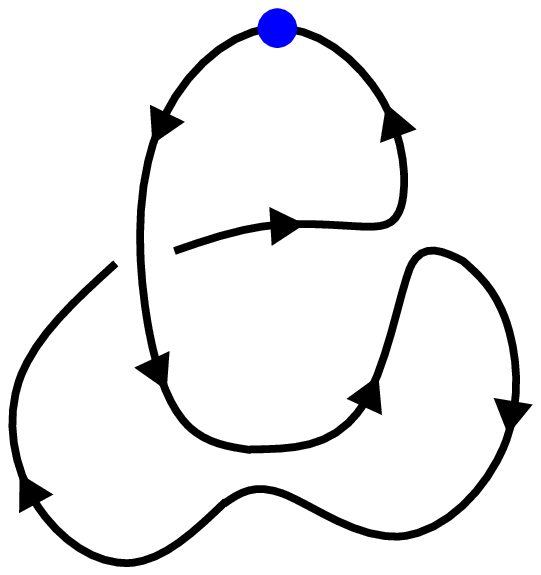}}}
    }
    %%%
    child {node[label=above:{ $z$}, label=below:{ $a^{2}$}] {\includegraphics[width=2cm]{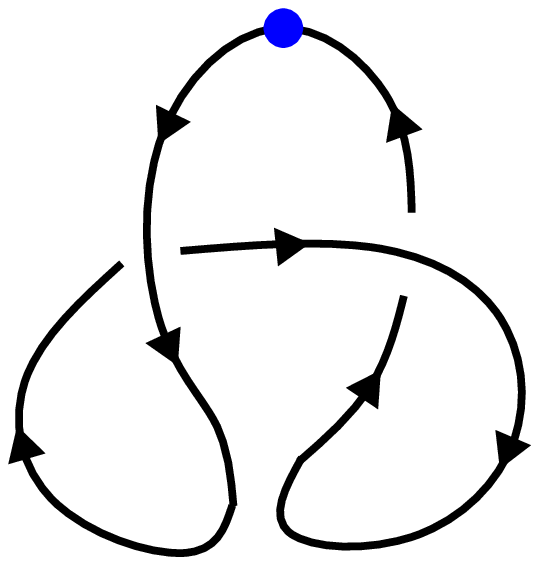}}
    };
  \end{tikzpicture}
  \caption{Calculation on a left trefoil}
  \label{fig.calculation_trefoil_left}
\end{figure}

We represent this algorithm on a left-handed trefoil in Figure \ref{fig.calculation_trefoil_left}. From the figure, we can see that
\begin{align*}
D_{K} &= a^{-1}-z(\frac{a-a^{-1}}{z}+1)+z^2a^{-1}-z^2a+za^2 \\
&= 2a^{-1}-a+a^2z-z+a^{-1}z^2-az^2
\end{align*}
Since $writhe(K) = -3$, we have 
\begin{align*}
DK_{K} &= a^3D_K\\
&= 2a^2-a^4+a^5z-a^3z+a^{2}z^2-a^4z^2
\end{align*}
%%%%%%%

Based on the algorithm above, we can establish a state model. 
Given a Gauss diagram $G$ of a knot $K$, a \textit{state} is the Gauss diagram with arrows labelled $\{ 0, \infty \}$ or non-labeled, equivalently a function $\sigma: A(G) \rightarrow \{\phi, 0, \infty \}$, where $A(G)$ is the set of the arrows of $G$. Given a state $(G,\sigma)$, we give it a process based on the Table \ref{table.process} and we give a weight $w(\sigma)$ according to the process. 

The process is as the following. We go along the circle from the base point. When we meet an arrow at the first time (we may call it the first passage at this arrow), we go through it as the red route indicated in \ref{table.process} and change the orientation on the corresponding arcs of the circle if necessary. If we meet an arrow at the second time, we go along the path indicated by the previous steps. If we have come back to the base point but there are other components unpassed, we set another base point at the arc of unpassed component which has the smallest order number and we go from it again. Repeat these procedure until we have passed all the arc of the circle. Then the process is finished. 

During the whole process, we attach a number $n(\alpha)$ (\textit{change number}) to each arrow $\alpha \in A(G)$. When we meet an arrow $\beta$ labelled $\infty$ at the first time, if only the head or only the foot of $\alpha$ is on the arc that change orientation due to $\beta$, we add $1$ to $n(\alpha)$. If both head and foot of $\alpha$ are on the arc that change orientation due to $\beta$, we add $2$ to $n(\alpha)$. Else we do nothing to $n(\alpha)$. And we ask that if $\alpha$ is an arrow labelled with $0$ or $\infty$, the $n(\alpha)$ stops growth when we meet $\alpha$ at the first time. (If $\alpha$ is labelled $\infty$, the change of orientation due to itself does not contribute to $n(\alpha)$ as well). 

We define the weight of the state $(G,\sigma)$ to be $$w(G,\sigma) = \prod \limits_{\alpha \in A(G)} w(G,\sigma,\alpha)$$ where $w(\sigma,\alpha)$ is given in Table \ref{table.weight} when we meet $\alpha$ at the first time during the process (i.e., at the \textit{first passage}). Notice that the change number $n(\alpha)$ of an unlabelled arrow are not settled until the end of the process. 
\begin{thm}[State model for Kauffman polynomial]
Given a Gauss diagram $G$ of an oriented knot $K$, we have $$DK_{K}(a,z) = \sum \limits_{\sigma \in state(G)} w(G,\sigma)(\dfrac{a-a^{-1}}{z}+1)^{c(\sigma)-1}$$
where $c(\sigma)$ is the number of components of the state $\sigma$ with its process finishing. 
\end{thm}

\begin{table}[htbp]
\centering
\begin{tabular}{| c | c | c | c |}
\hline
first passage& result & first passage & result \\
\hline
\begin{minipage}[b]{0.25\columnwidth}
    \centering
	\raisebox{-.5\height}{\includegraphics[width=\linewidth]{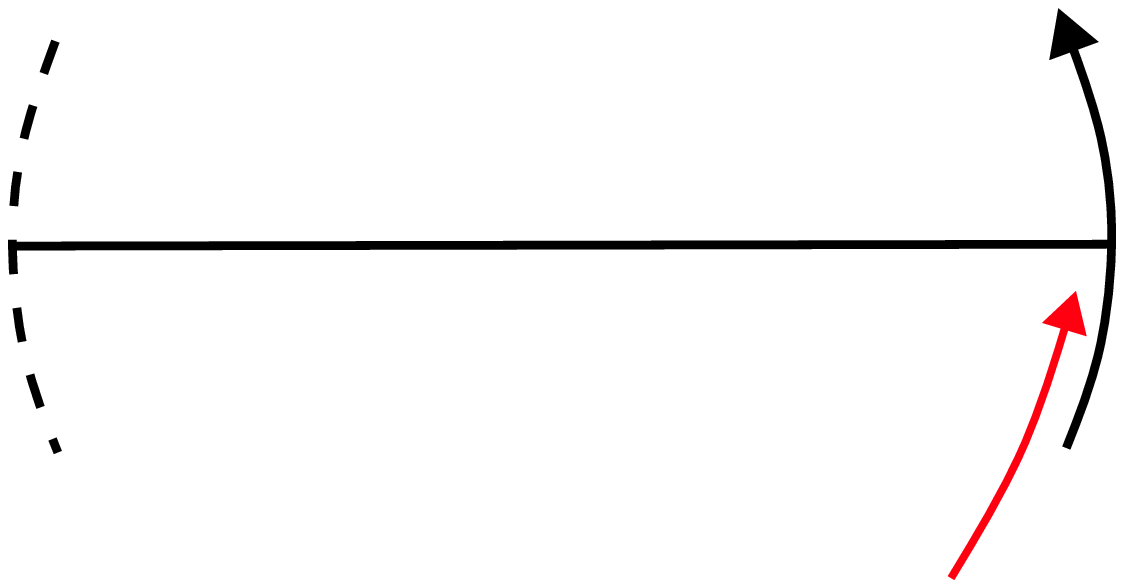}}
\end{minipage}
& \begin{minipage}[b]{0.25\columnwidth}
        \centering
		\raisebox{-.5\height}{\includegraphics[width=\linewidth]{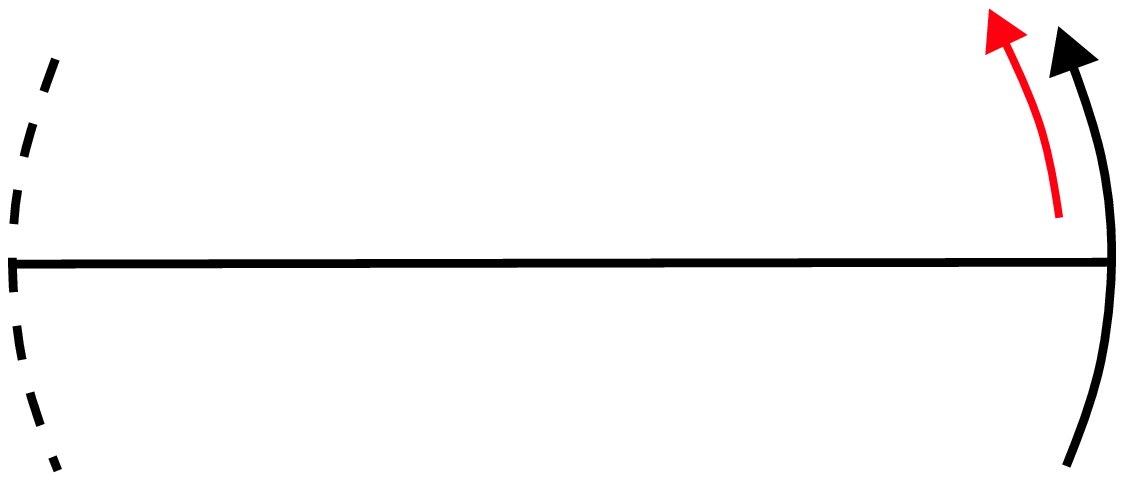}}
	\end{minipage}
&\begin{minipage}[b]{0.25\columnwidth}
    \centering
    \raisebox{-.5\height}{\includegraphics[width=\linewidth]{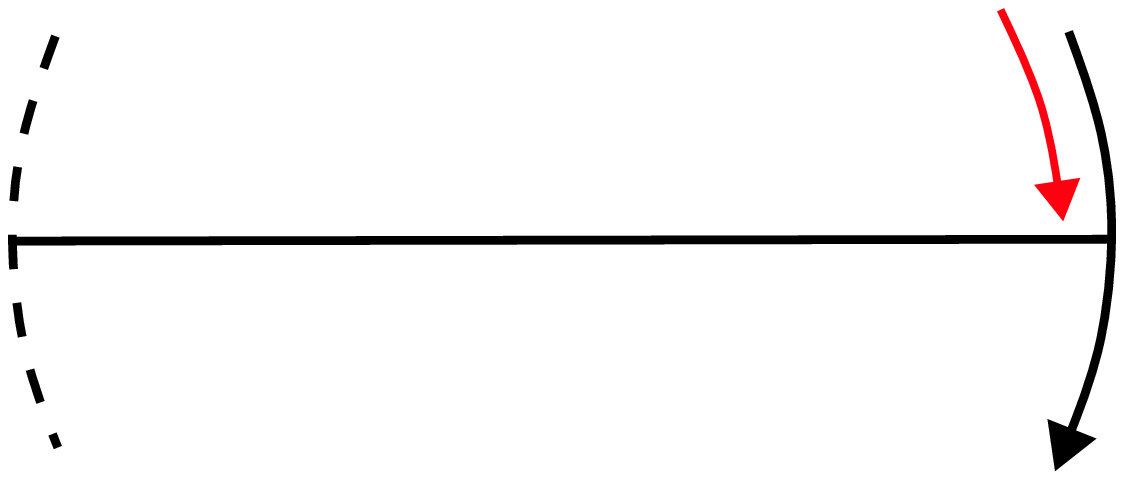}}
    \end{minipage}
& \begin{minipage}[b]{0.25\columnwidth}
		\centering
		\raisebox{-.5\height}{\includegraphics[width=\linewidth]{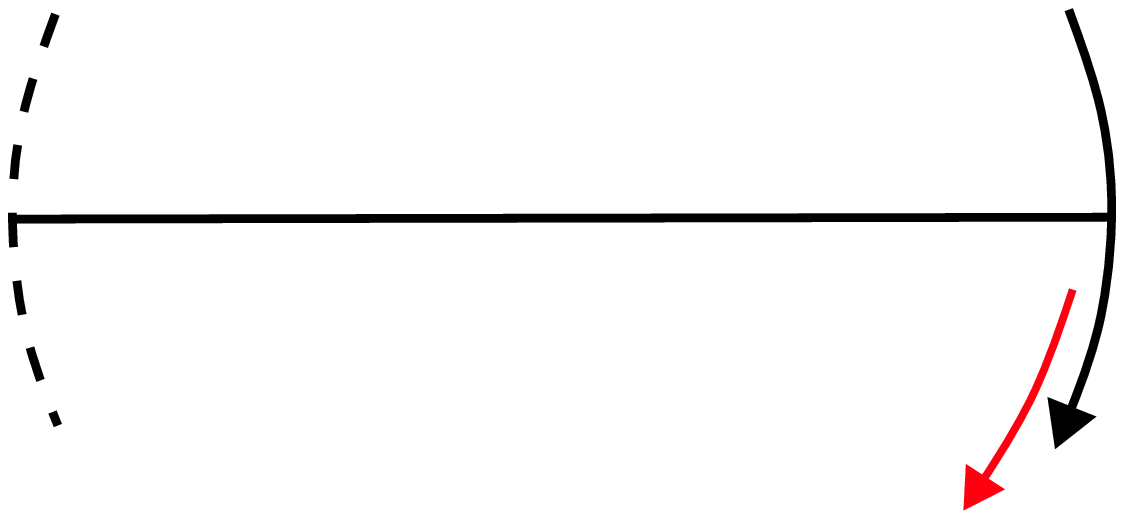}}
	\end{minipage} \\
\hline

\begin{minipage}[b]{0.25\columnwidth}
    \centering
	\raisebox{-.5\height}{\includegraphics[width=\linewidth]{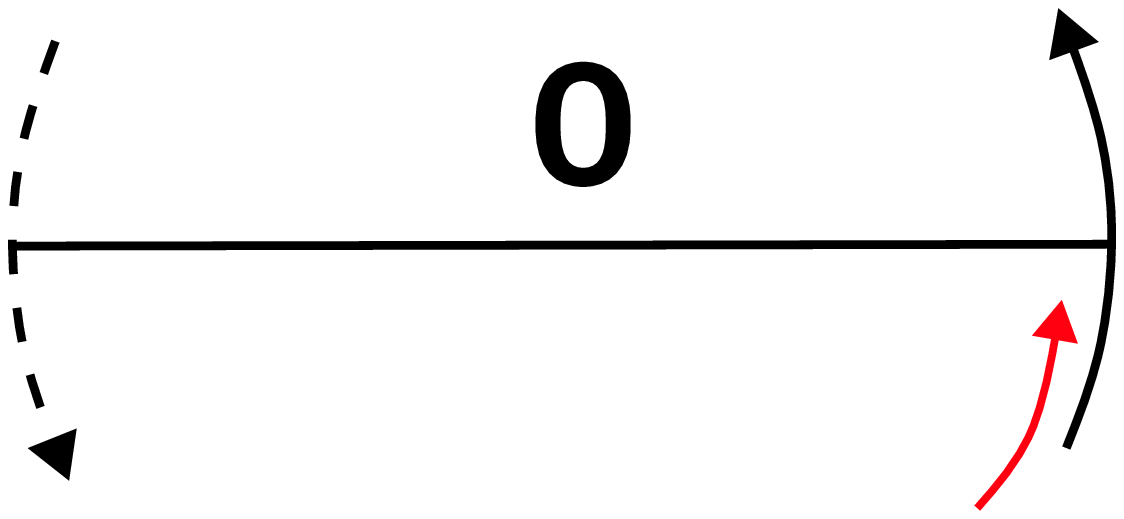}}
\end{minipage}
& \begin{minipage}[b]{0.25\columnwidth}
        \centering
		\raisebox{-.5\height}{\includegraphics[width=\linewidth]{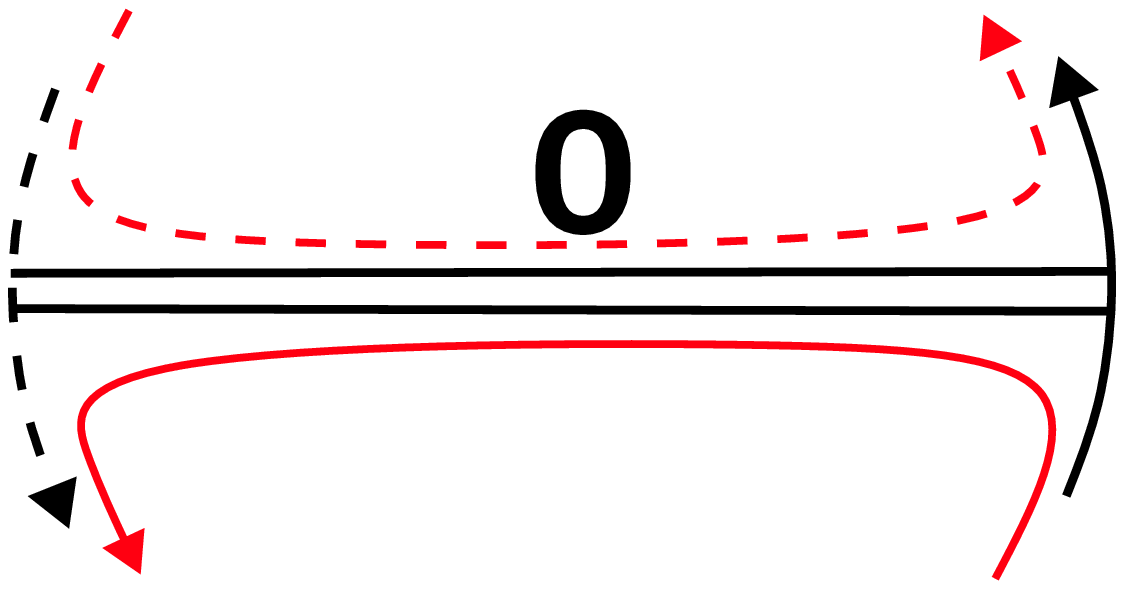}}
	\end{minipage}
&\begin{minipage}[b]{0.25\columnwidth}
    \centering
    \raisebox{-.5\height}{\includegraphics[width=\linewidth]{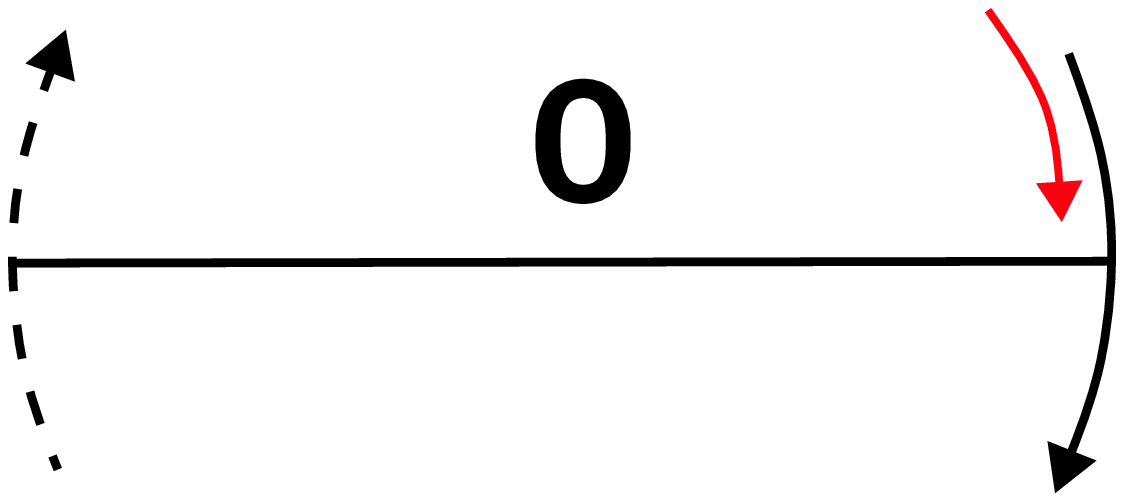}}
    \end{minipage}
& \begin{minipage}[b]{0.25\columnwidth}
		\centering
		\raisebox{-.5\height}{\includegraphics[width=\linewidth]{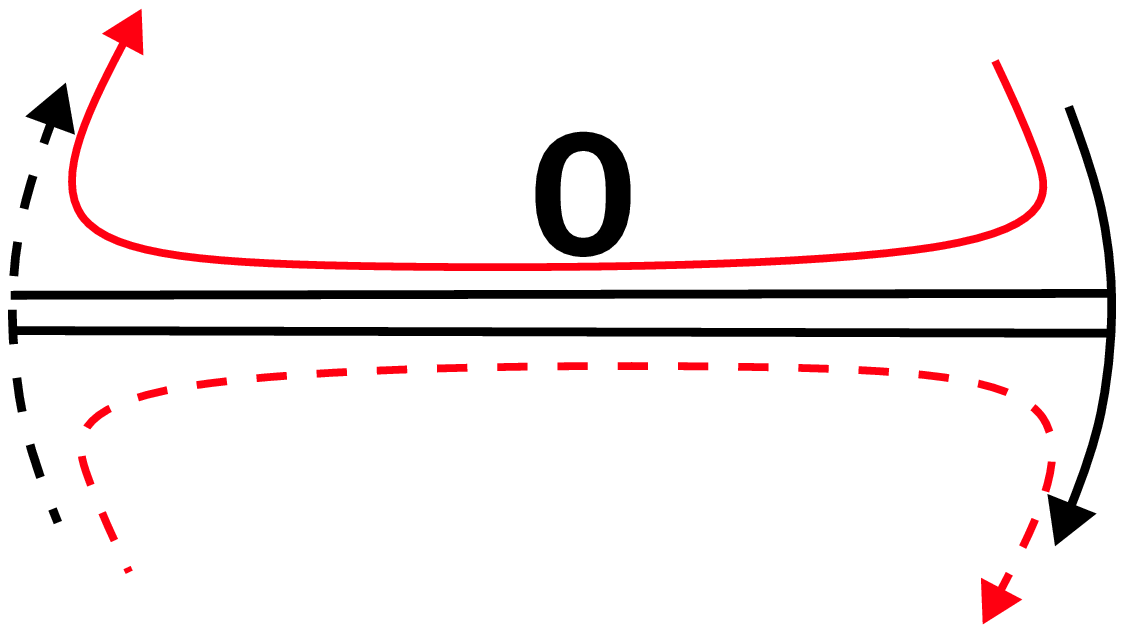}}
	\end{minipage} \\
\hline

\begin{minipage}[b]{0.25\columnwidth}
    \centering
	\raisebox{-.5\height}{\includegraphics[width=\linewidth]{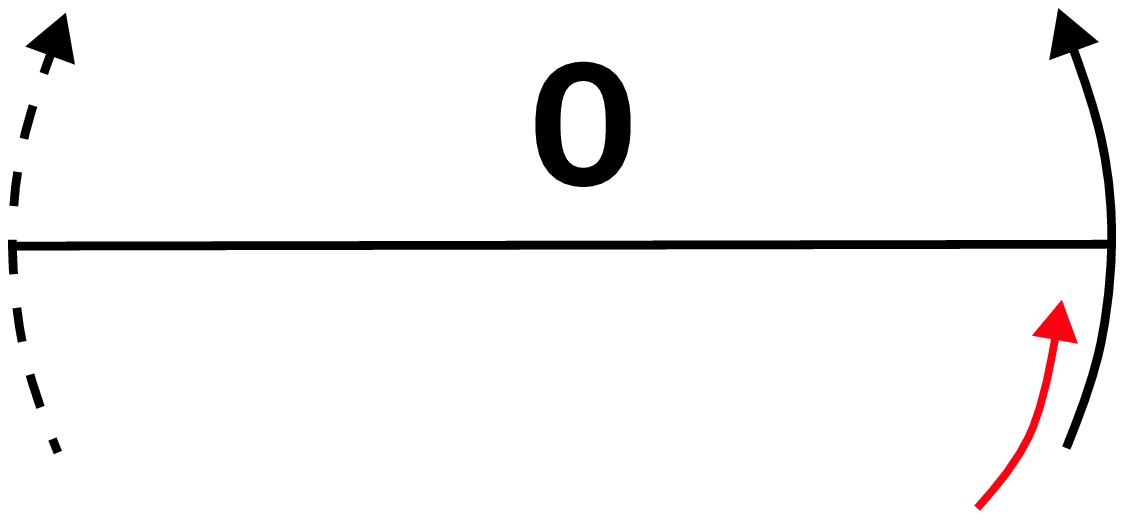}}
\end{minipage}
& \begin{minipage}[b]{0.25\columnwidth}
        \centering
		\raisebox{-.5\height}{\includegraphics[width=\linewidth]{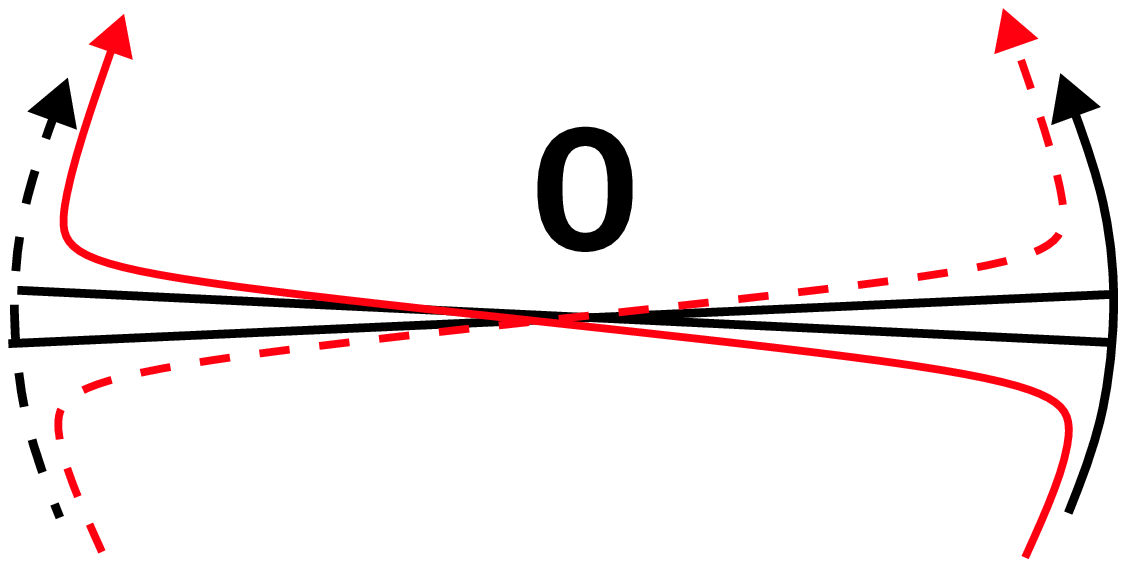}}
	\end{minipage}
&\begin{minipage}[b]{0.25\columnwidth}
    \centering
    \raisebox{-.5\height}{\includegraphics[width=\linewidth]{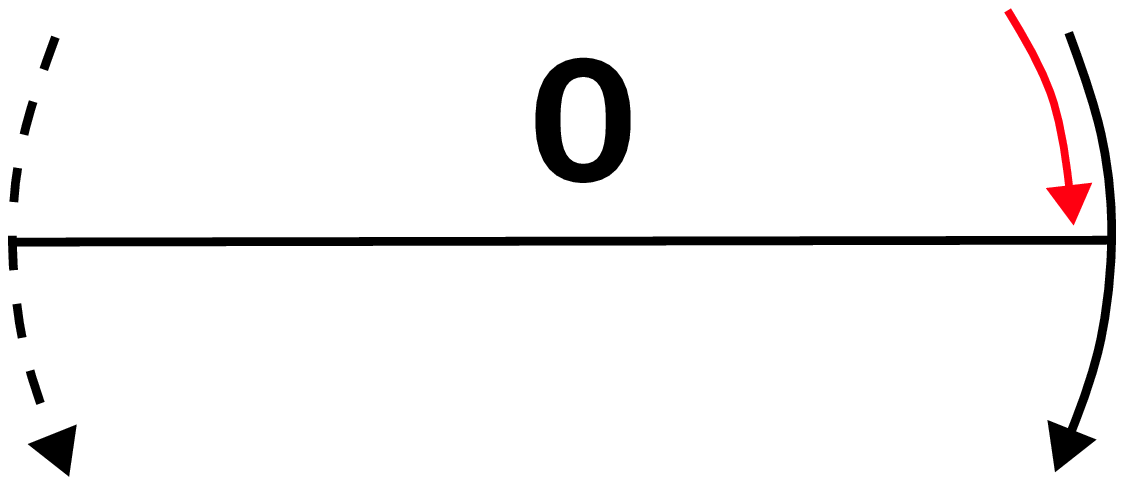}}
    \end{minipage}
& \begin{minipage}[b]{0.25\columnwidth}
		\centering
		\raisebox{-.5\height}{\includegraphics[width=\linewidth]{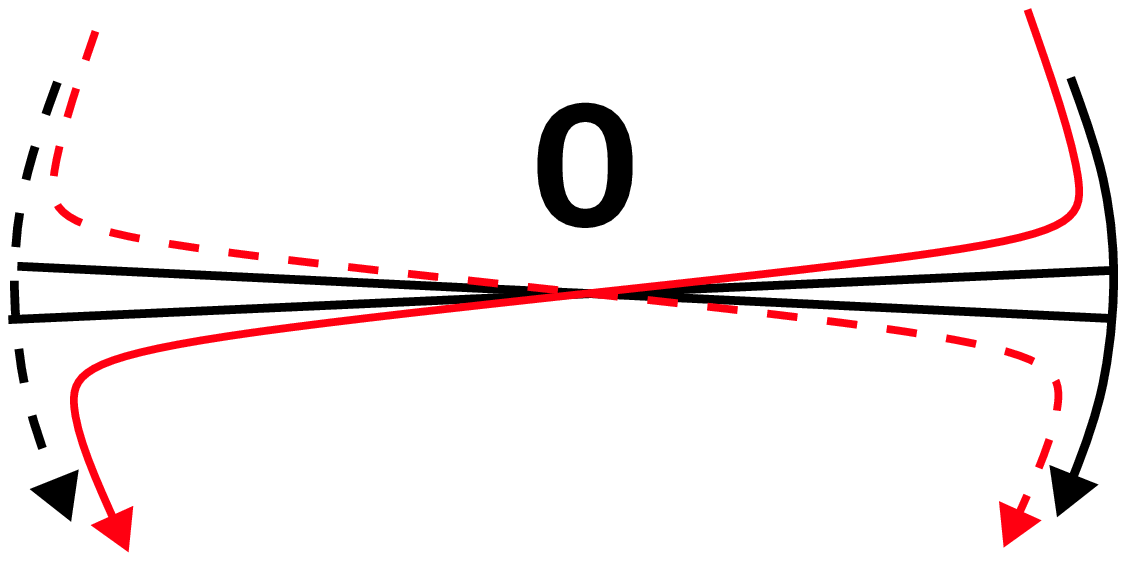}}
	\end{minipage} \\
\hline

\begin{minipage}[b]{0.25\columnwidth}
    \centering
	\raisebox{-.5\height}{\includegraphics[width=\linewidth]{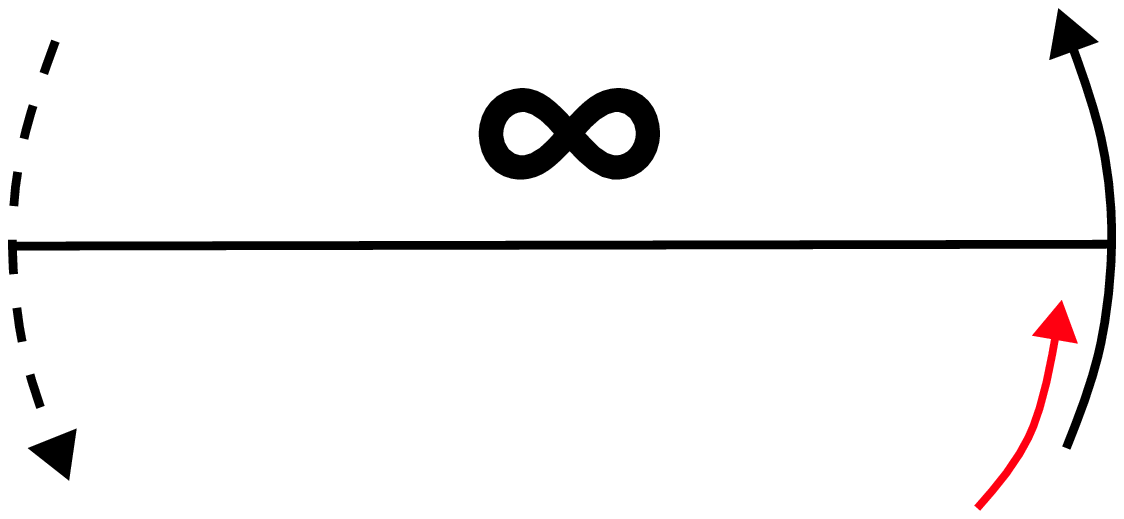}}
\end{minipage}
& \begin{minipage}[b]{0.25\columnwidth}
        \centering
		\raisebox{-.5\height}{\includegraphics[width=\linewidth]{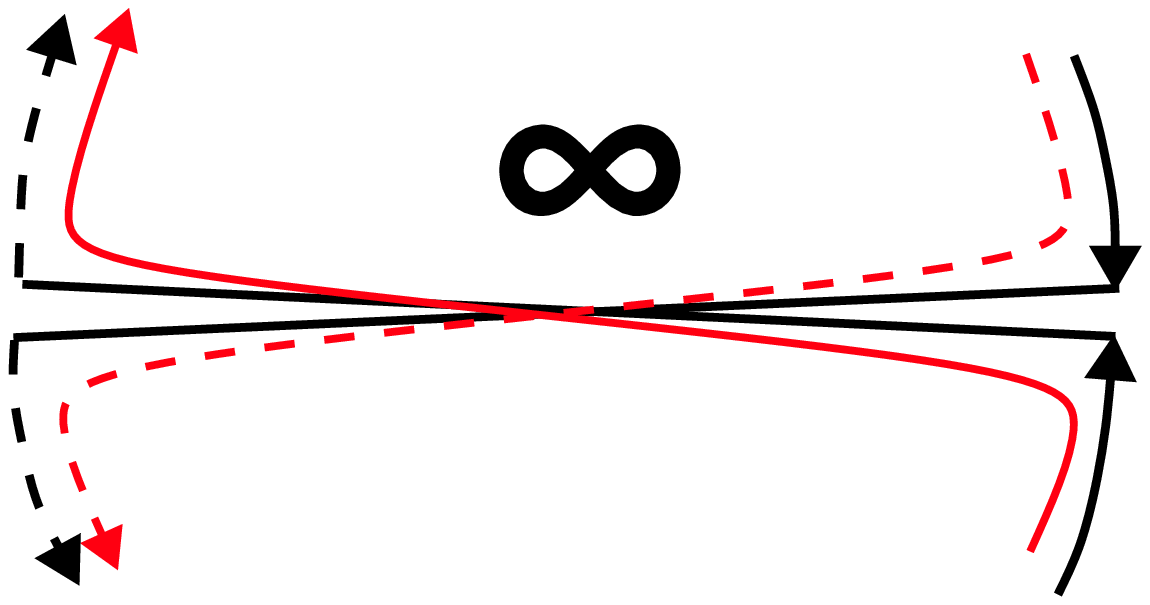}}
	\end{minipage}
&\begin{minipage}[b]{0.25\columnwidth}
    \centering
    \raisebox{-.5\height}{\includegraphics[width=\linewidth]{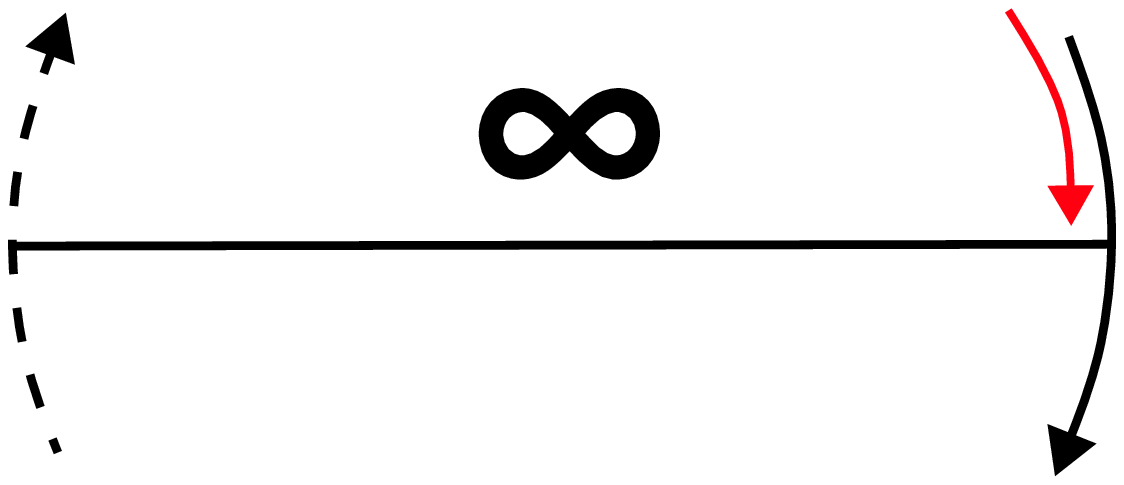}}
    \end{minipage}
& \begin{minipage}[b]{0.25\columnwidth}
		\centering
		\raisebox{-.5\height}{\includegraphics[width=\linewidth]{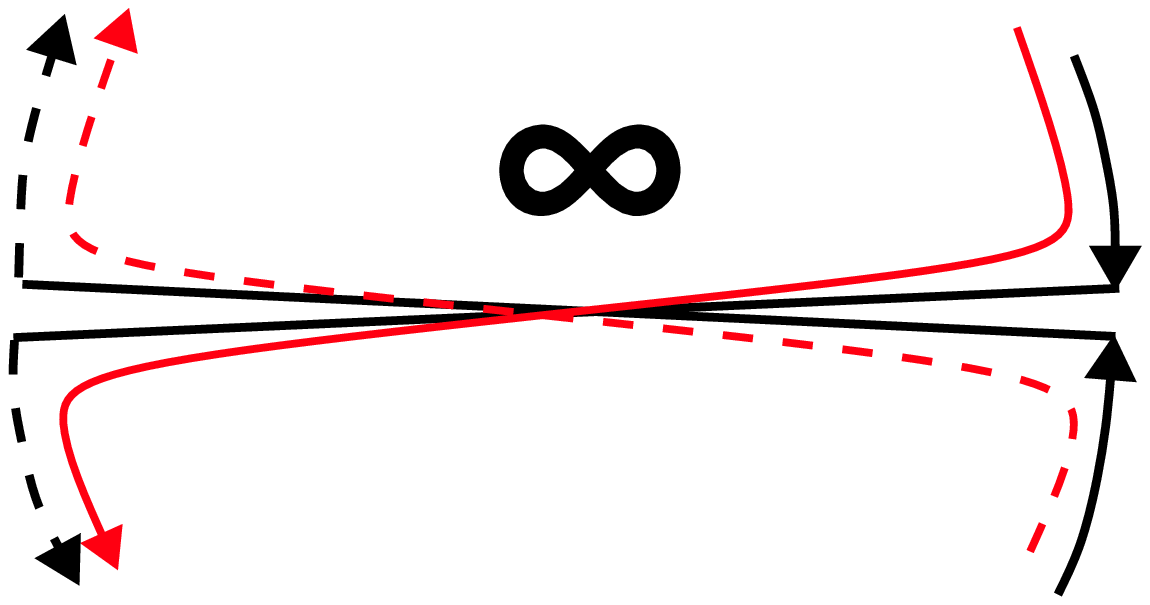}}
	\end{minipage} \\
\hline

\begin{minipage}[b]{0.25\columnwidth}
    \centering
	\raisebox{-.5\height}{\includegraphics[width=\linewidth]{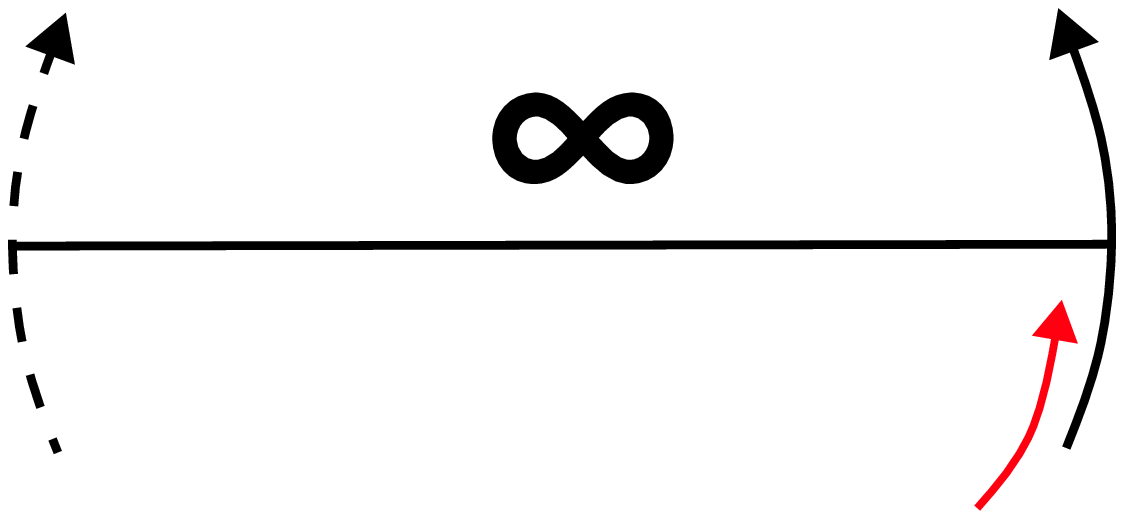}}
\end{minipage}
& \begin{minipage}[b]{0.25\columnwidth}
        \centering
		\raisebox{-.5\height}{\includegraphics[width=\linewidth]{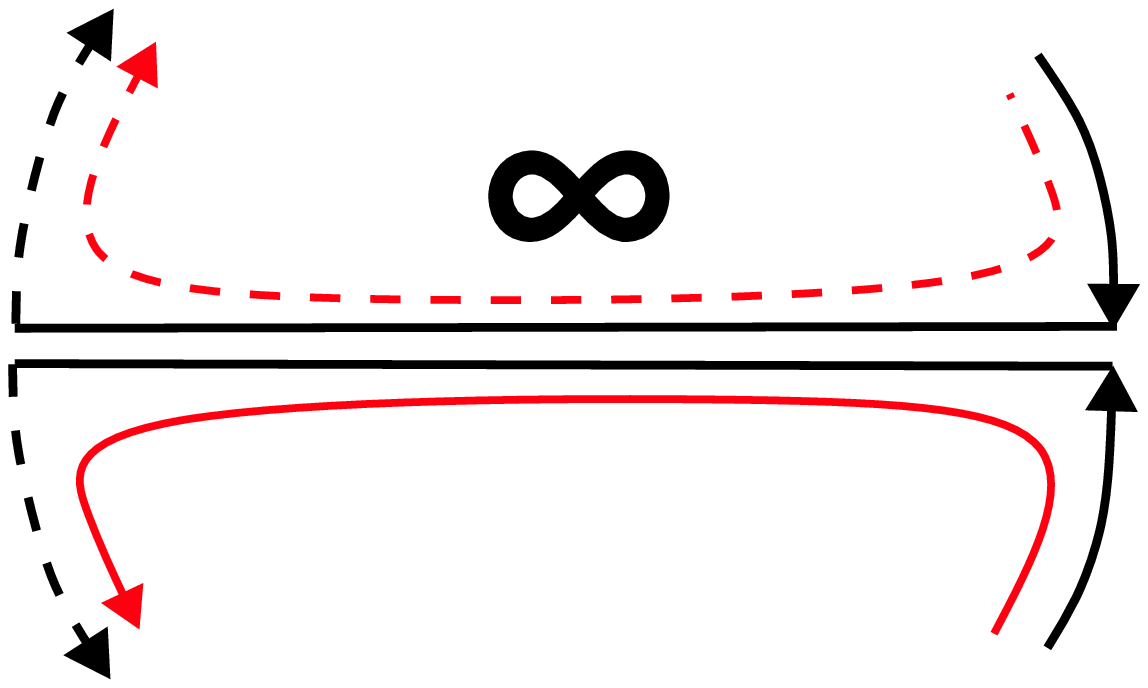}}
	\end{minipage}
&\begin{minipage}[b]{0.25\columnwidth}
    \centering
    \raisebox{-.5\height}{\includegraphics[width=\linewidth]{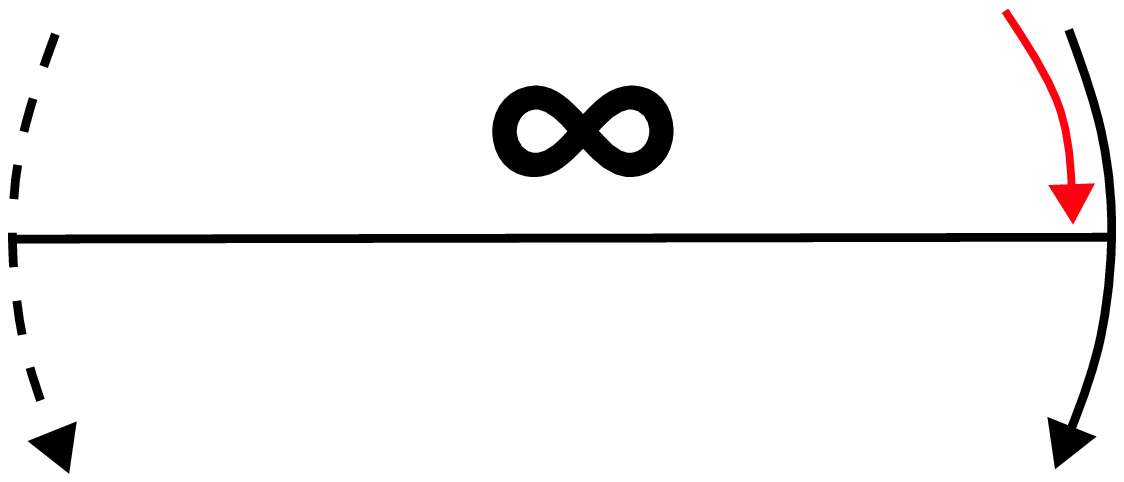}}
    \end{minipage}
& \begin{minipage}[b]{0.25\columnwidth}
		\centering
		\raisebox{-.5\height}{\includegraphics[width=\linewidth]{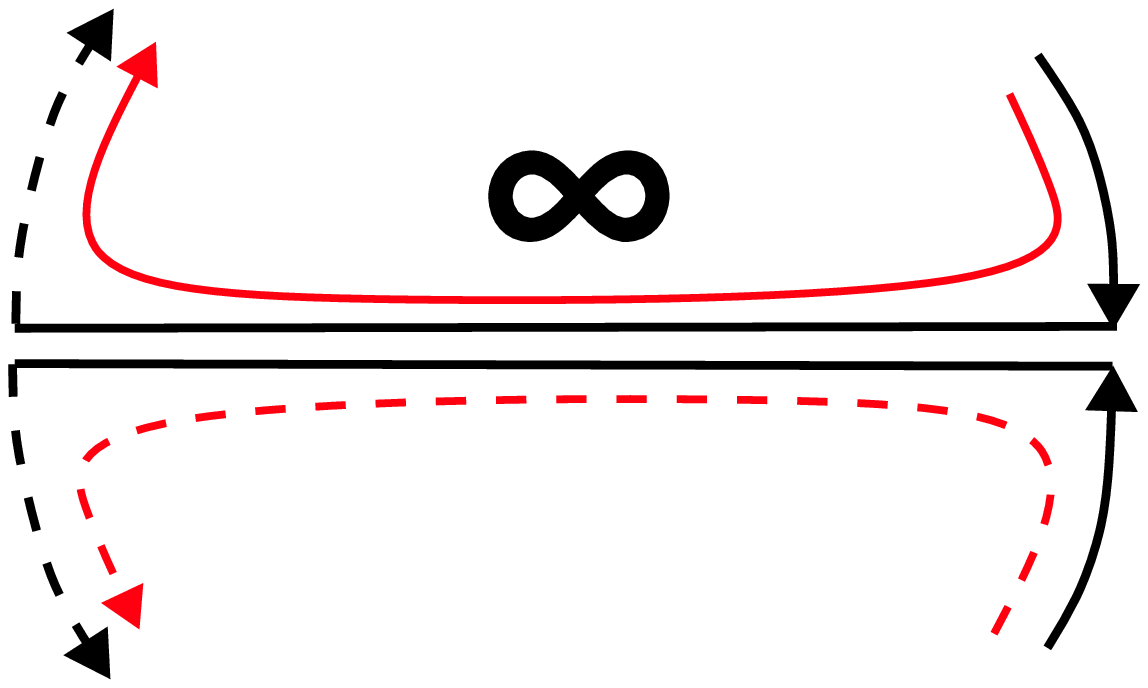}}
	\end{minipage} \\
\hline

\begin{minipage}[b]{0.25\columnwidth}
    \centering
	\raisebox{-.5\height}{\includegraphics[width=\linewidth]{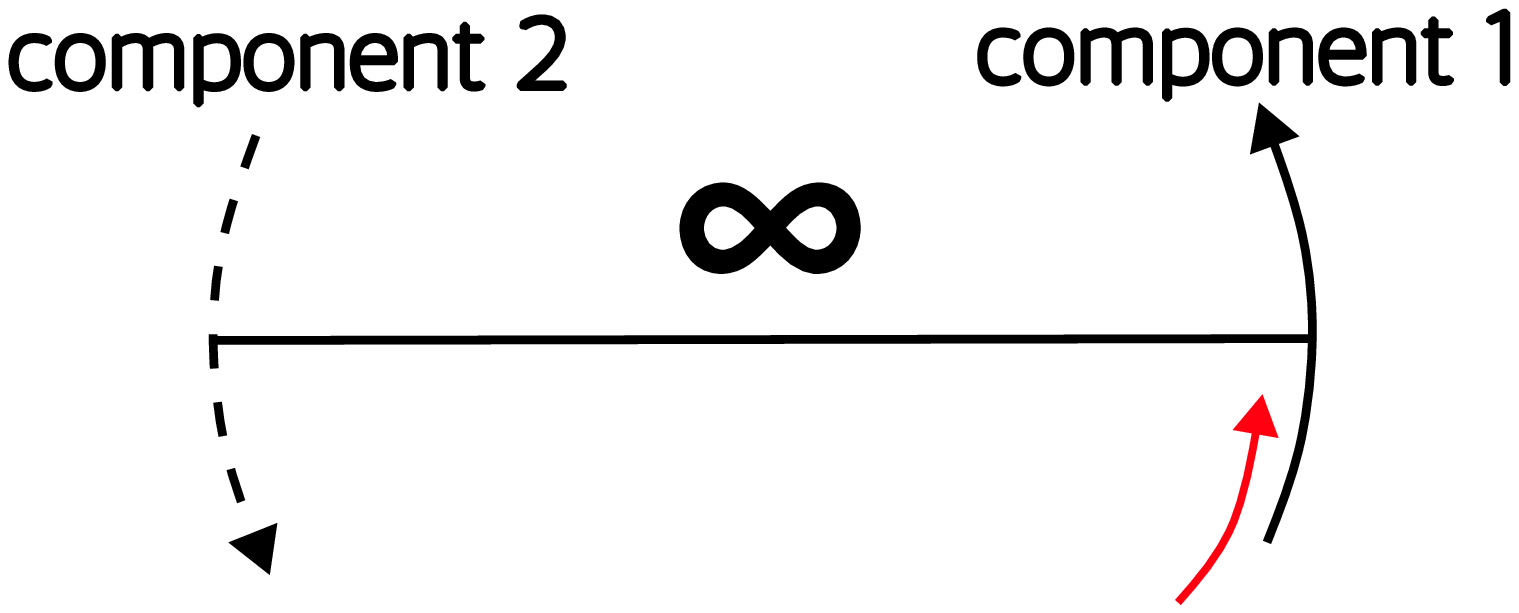}}
\end{minipage}
& \begin{minipage}[b]{0.25\columnwidth}
        \centering
		\raisebox{-.5\height}{\includegraphics[width=\linewidth]{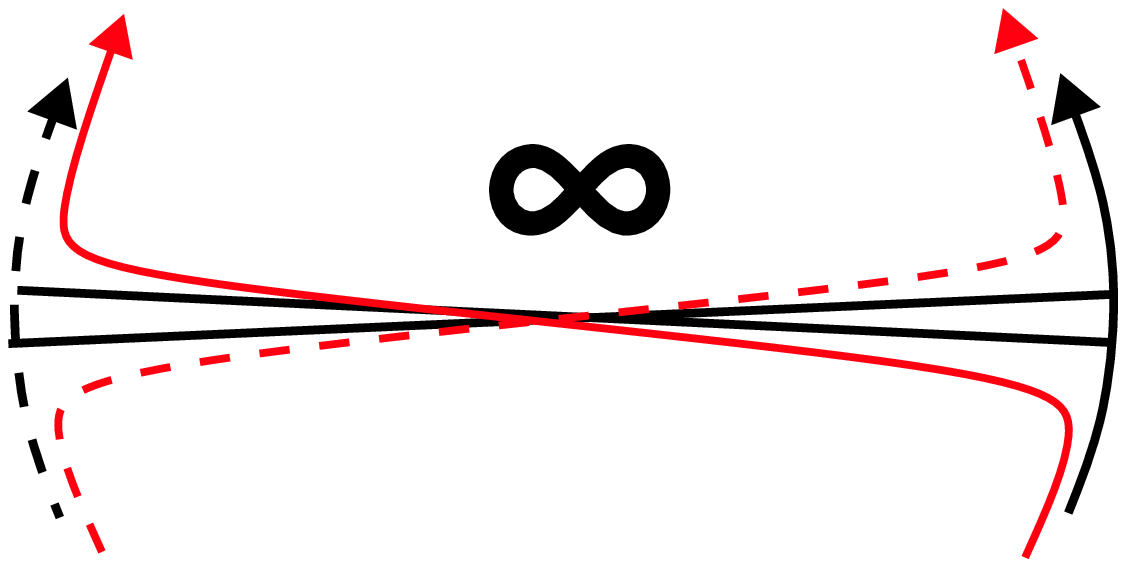}}
	\end{minipage}
&\begin{minipage}[b]{0.25\columnwidth}
    \centering
    \raisebox{-.5\height}{\includegraphics[width=\linewidth]{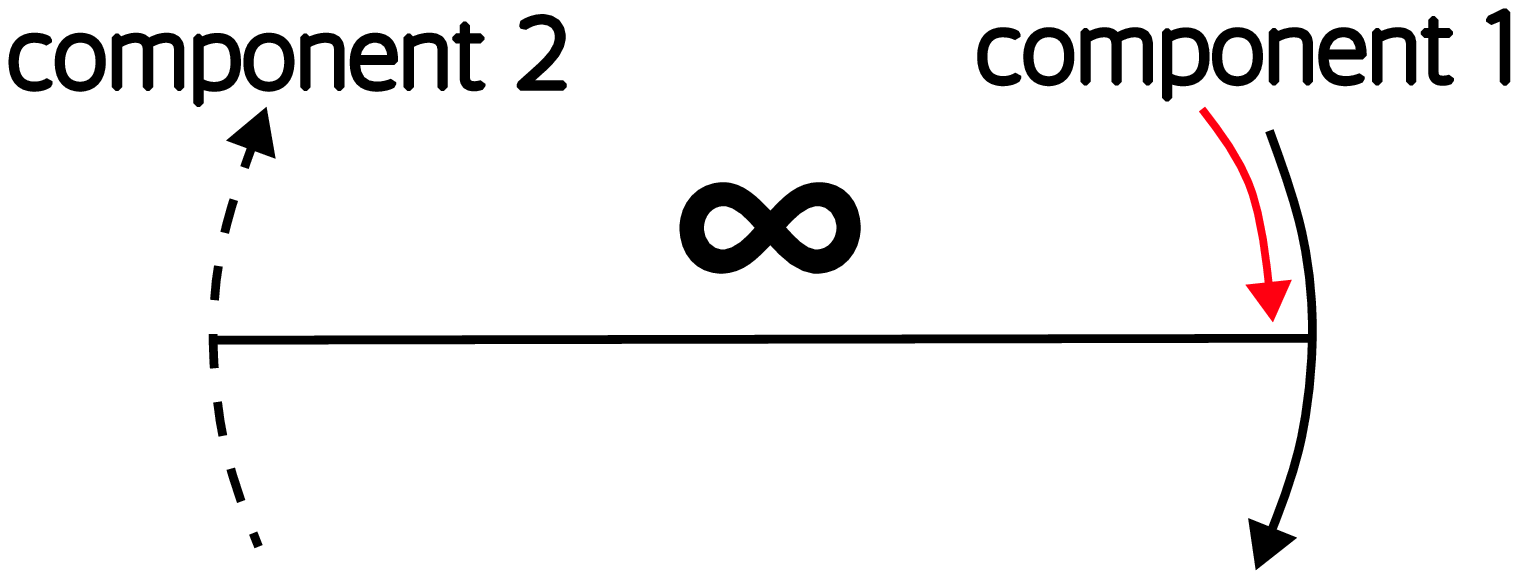}}
    \end{minipage}
& \begin{minipage}[b]{0.25\columnwidth}
		\centering
		\raisebox{-.5\height}{\includegraphics[width=\linewidth]{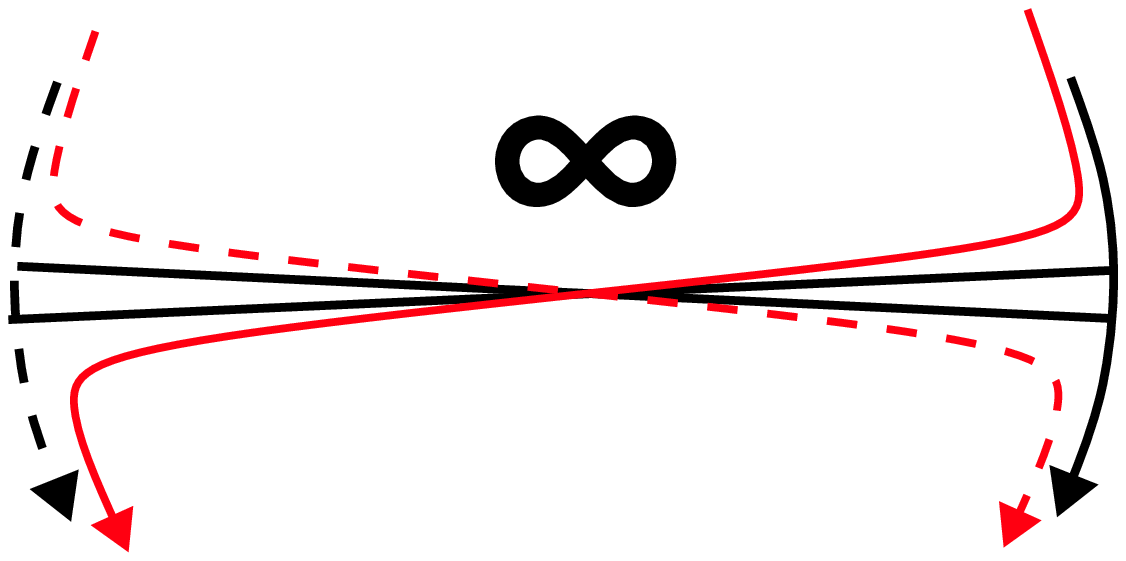}}
	\end{minipage} \\
\hline

\begin{minipage}[b]{0.25\columnwidth}
    \centering
	\raisebox{-.5\height}{\includegraphics[width=\linewidth]{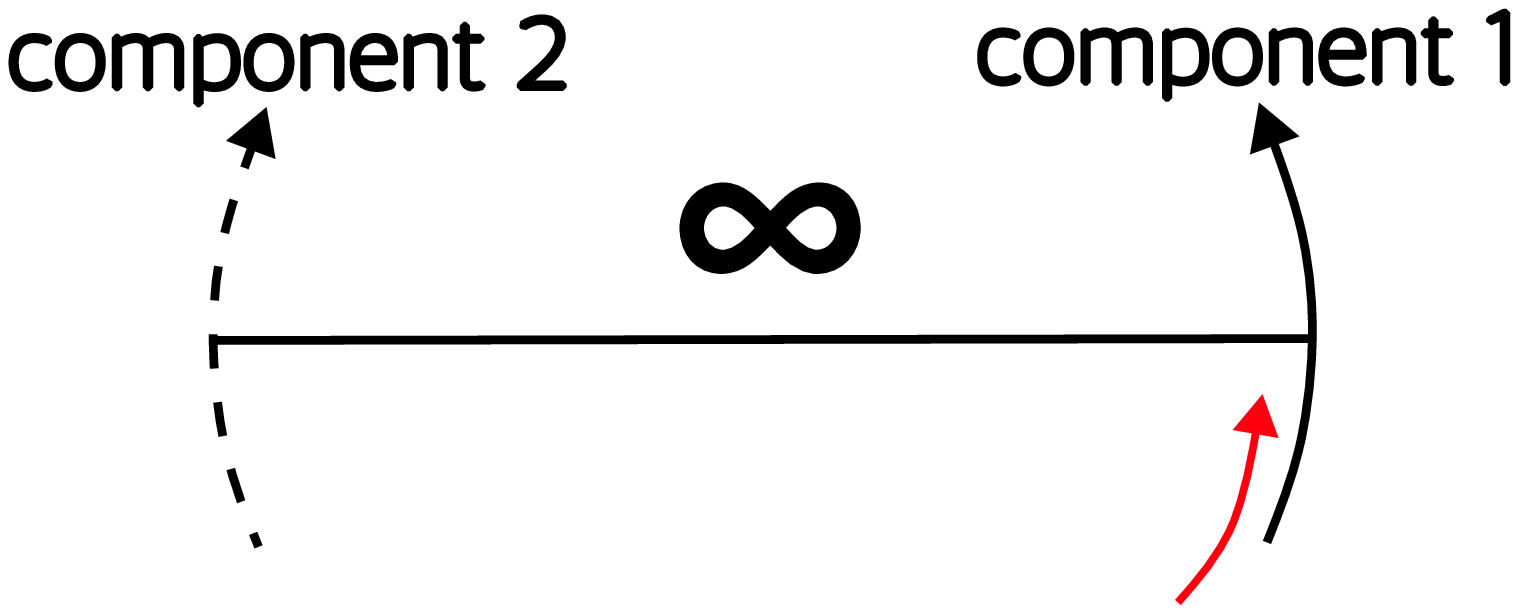}}
\end{minipage}
& \begin{minipage}[b]{0.25\columnwidth}
        \centering
		\raisebox{-.5\height}{\includegraphics[width=\linewidth]{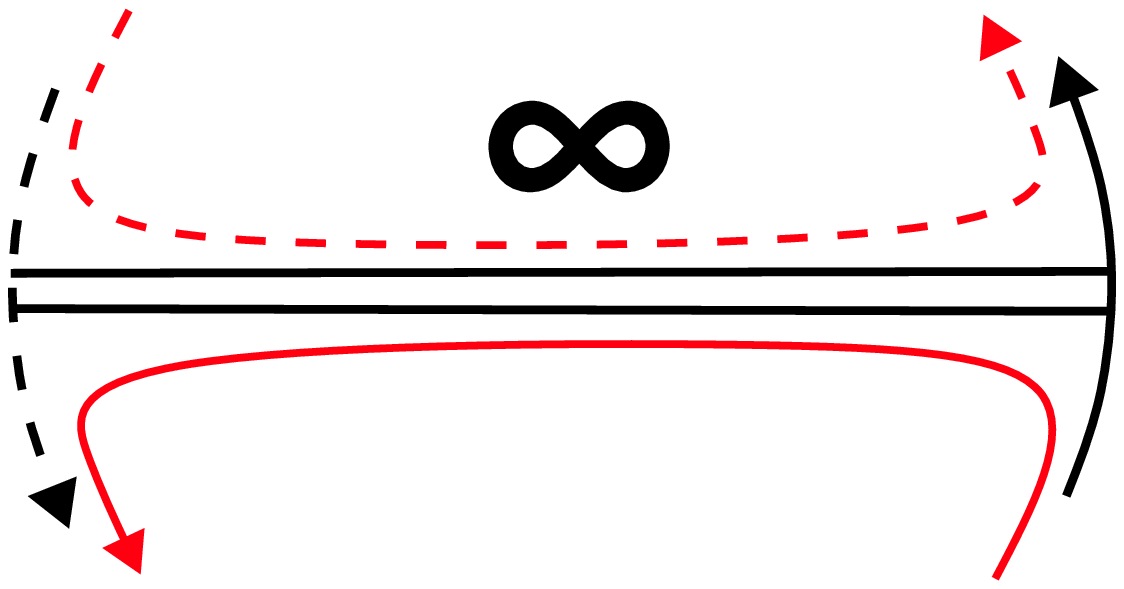}}
	\end{minipage}
&\begin{minipage}[b]{0.25\columnwidth}
    \centering
    \raisebox{-.5\height}{\includegraphics[width=\linewidth]{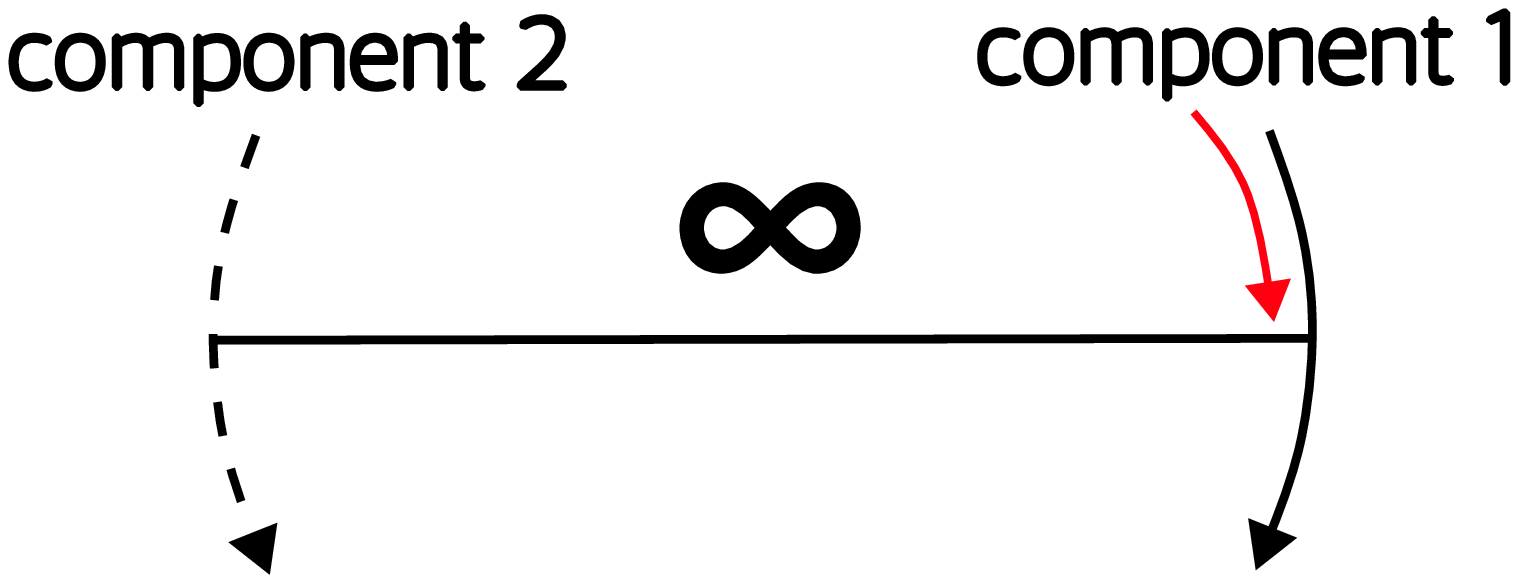}}
    \end{minipage}
& \begin{minipage}[b]{0.25\columnwidth}
		\centering
		\raisebox{-.5\height}{\includegraphics[width=\linewidth]{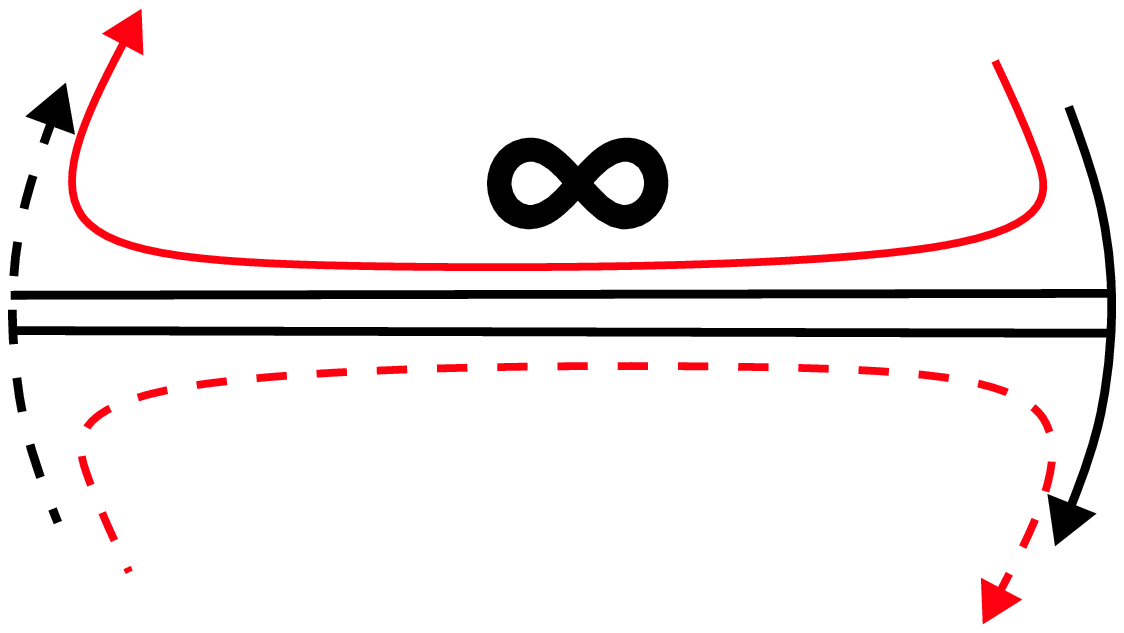}}
	\end{minipage} \\
\hline

\end{tabular}
\caption{Process}
\label{table.process}
\end{table}

\clearpage

\begin{table}[htbp]
\centering
\begin{tabular}{| c | c | c |}
\hline
& \begin{minipage}[b]{0.25\columnwidth}
		\centering
		\raisebox{-.5\height}{\includegraphics[width=\linewidth]{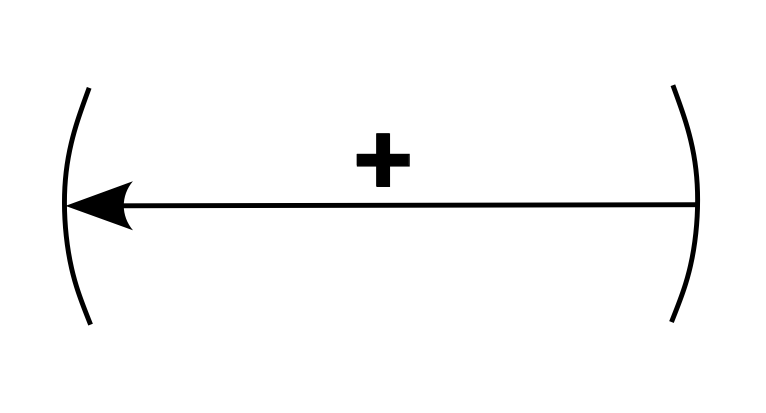}}
	\end{minipage}
& \begin{minipage}[b]{0.25\columnwidth}
		\centering
		\raisebox{-.5\height}{\includegraphics[width=\linewidth]{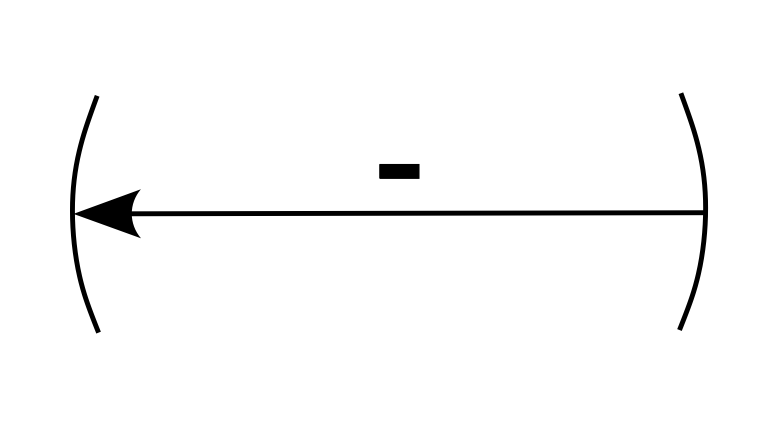}}
	\end{minipage} \\
\hline
\begin{minipage}[b]{0.25\columnwidth}
    \centering
	\raisebox{-.4\height}{\includegraphics[width=\linewidth]{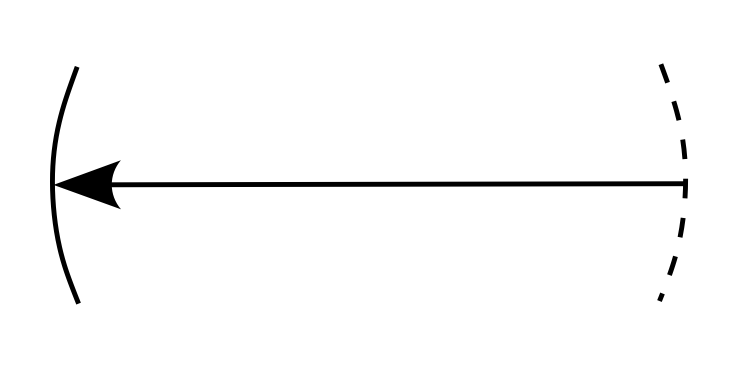}}
	\end{minipage}
 & $a^{(-1)^n-1}$ & $a^{(-1)^{n+1}+1}$ \\
\hline
\begin{minipage}[b]{0.25\columnwidth}
    \centering
	\raisebox{-.4\height}{\includegraphics[width=\linewidth]{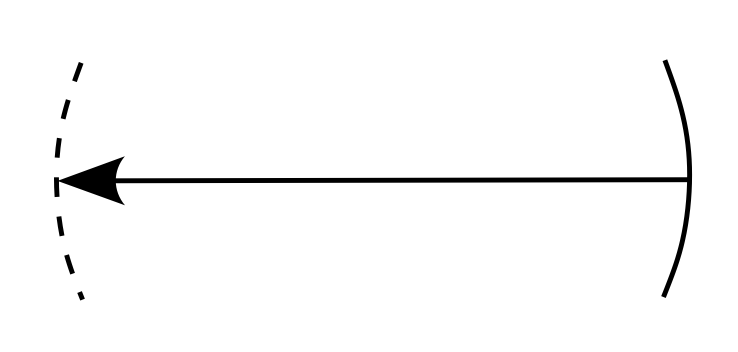}}
	\end{minipage}
 & $a^{(-1)^{n+1}-1}$ & $a^{(-1)^{n}+1}$ \\
 \hline
\begin{minipage}[b]{0.25\columnwidth}
    \centering
	\raisebox{-.4\height}{\includegraphics[width=\linewidth]{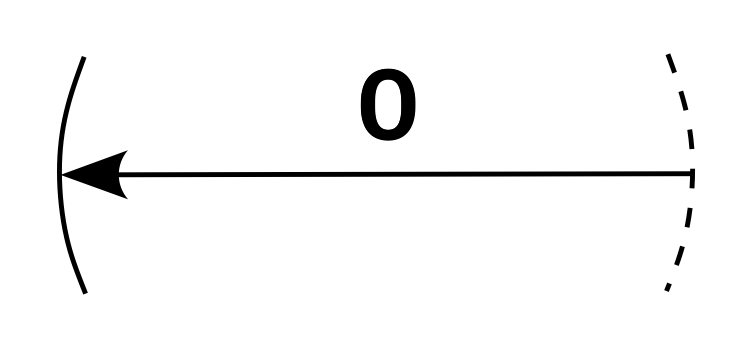}}
	\end{minipage}
 & $0$ & $0$ \\
\hline
\begin{minipage}[b]{0.25\columnwidth}
    \centering
	\raisebox{-.4\height}{\includegraphics[width=\linewidth]{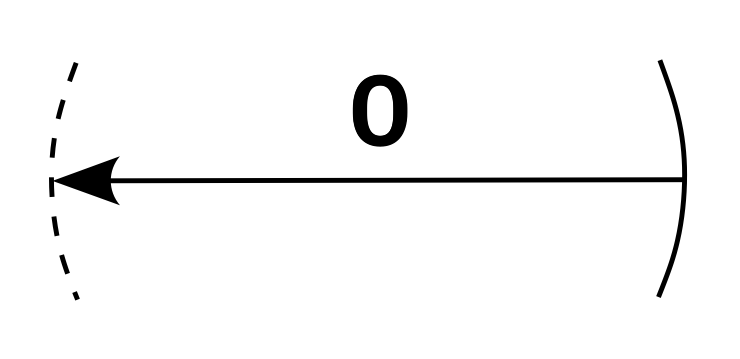}}
	\end{minipage}
 & $(-1)^{n}za^{-1}$ & $(-1)^{n+1}za$ \\
 \hline
\begin{minipage}[b]{0.25\columnwidth}
    \centering
	\raisebox{-.4\height}{\includegraphics[width=\linewidth]{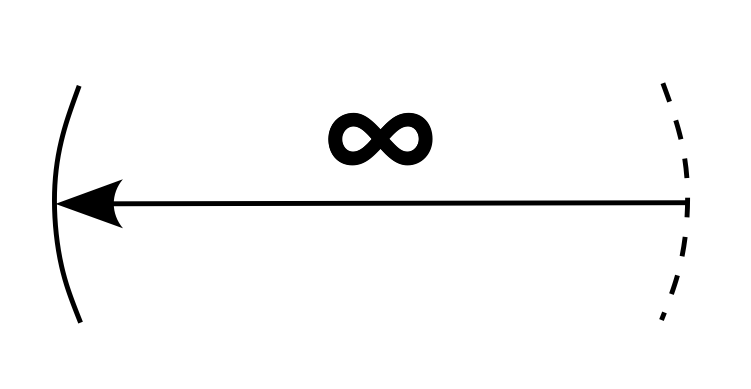}}
	\end{minipage}
 & 0 & 0 \\
\hline
\begin{minipage}[b]{0.25\columnwidth}
    \centering
	\raisebox{-.4\height}{\includegraphics[width=\linewidth]{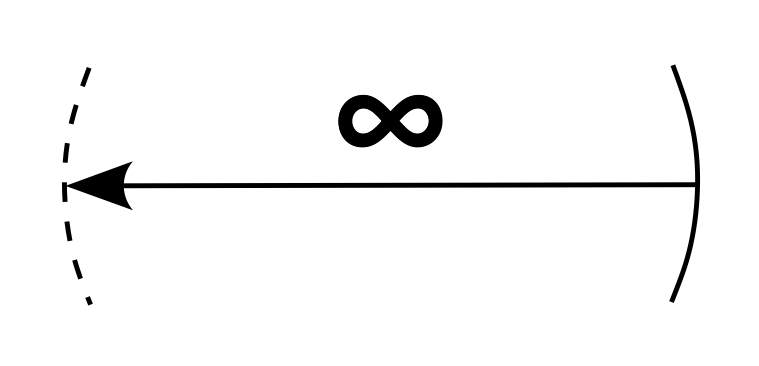}}
	\end{minipage}
 & $(-1)^{n+1}za^{-1}$ & $(-1)^{n}za$ \\
\hline
\end{tabular}
\caption{Weight}
\label{table.weight}
\end{table}

\begin{eg}[Calculation of a state]
Let us study one of the states $\sigma$ of a left handed trefoil indicated in Figure \ref{fig.ex_state}.
\begin{figure}[H]
    \centering
    \includegraphics[width = 3cm]{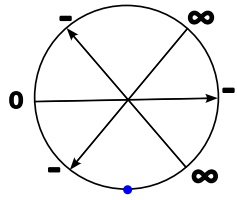}
    \caption{A state $\sigma$ of a trefoil}
    \label{fig.ex_state}
\end{figure}
Readers follow the red arrow to complete the process. The change number of each arrow is indicated in red. We adopt the convention that the arcs are ordered from the base point anti-clockwisely at the beginning of a process.

\begin{figure}[ht]
  \centering
  \begin{tikzpicture}
    \matrix (m) [matrix of nodes, nodes in empty cells,
      column sep=1.2cm, row sep=1.2cm,
      nodes={anchor=center, inner sep=0pt}]
    {
      \includegraphics[width=0.25\textwidth]{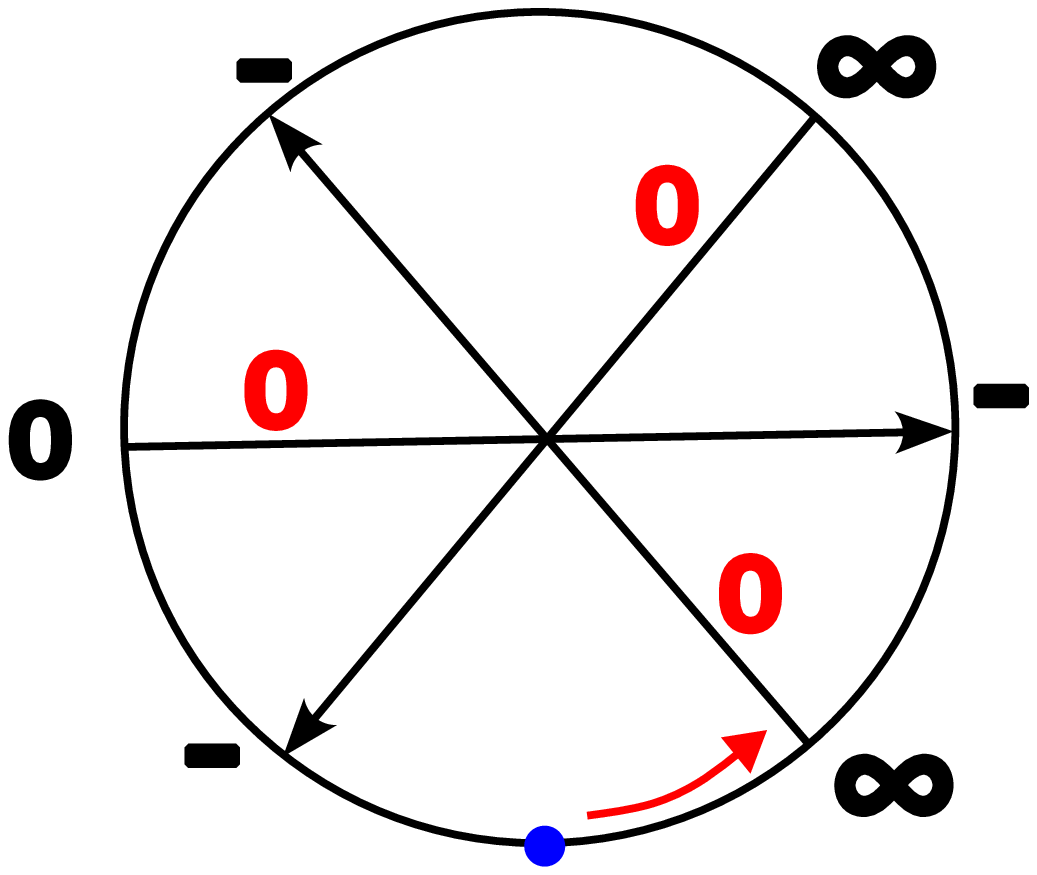} &
      \includegraphics[width=0.25\textwidth]{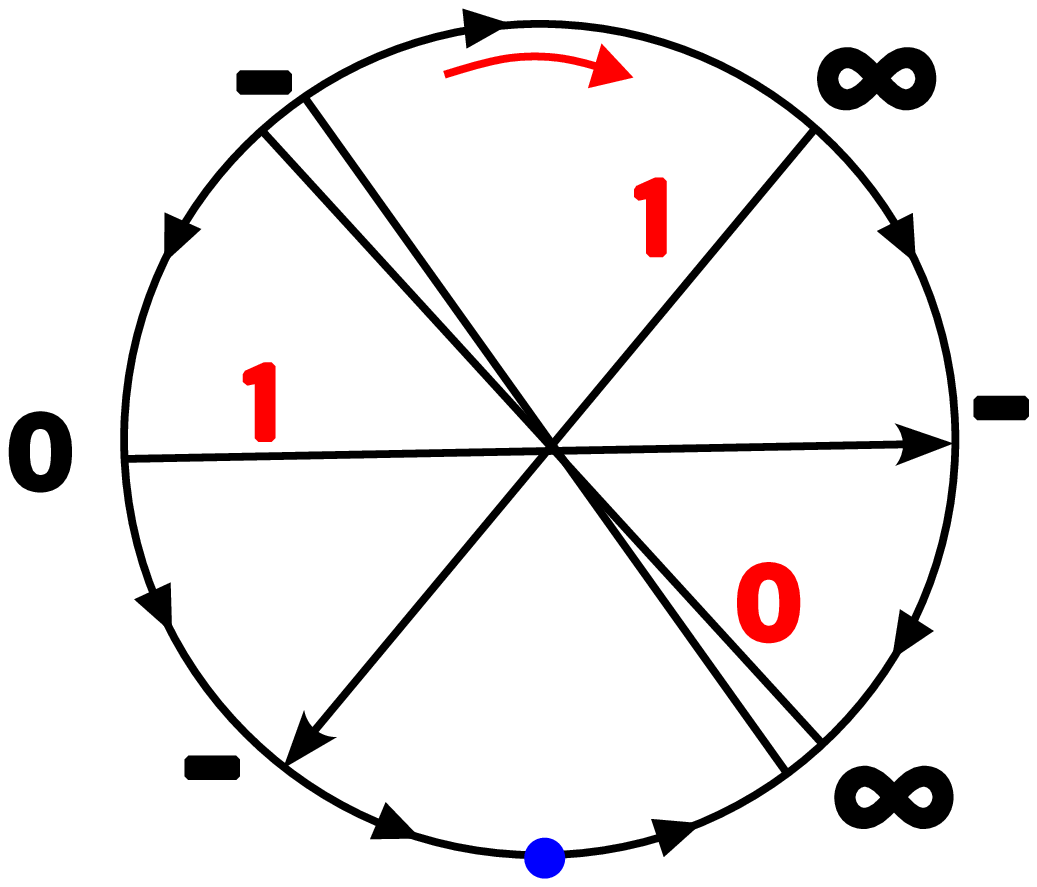} &
      \includegraphics[width=0.25\textwidth]{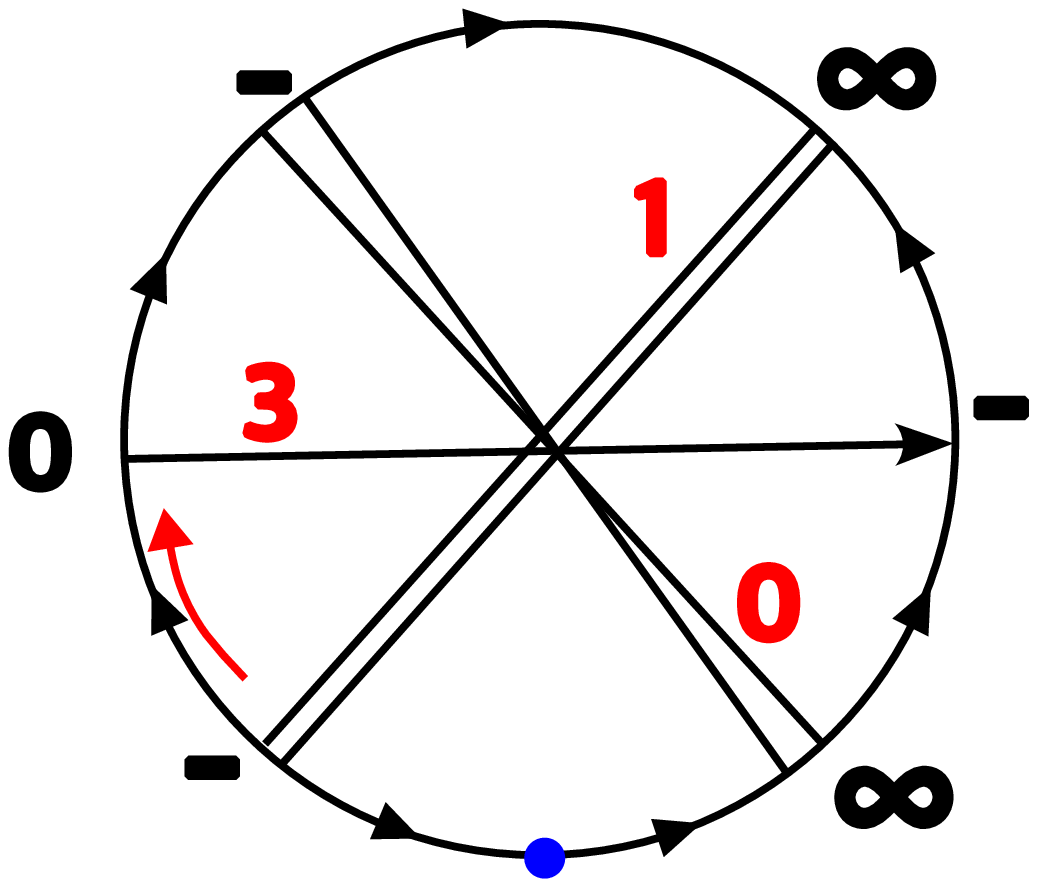} \\
      \includegraphics[width=0.25\textwidth]{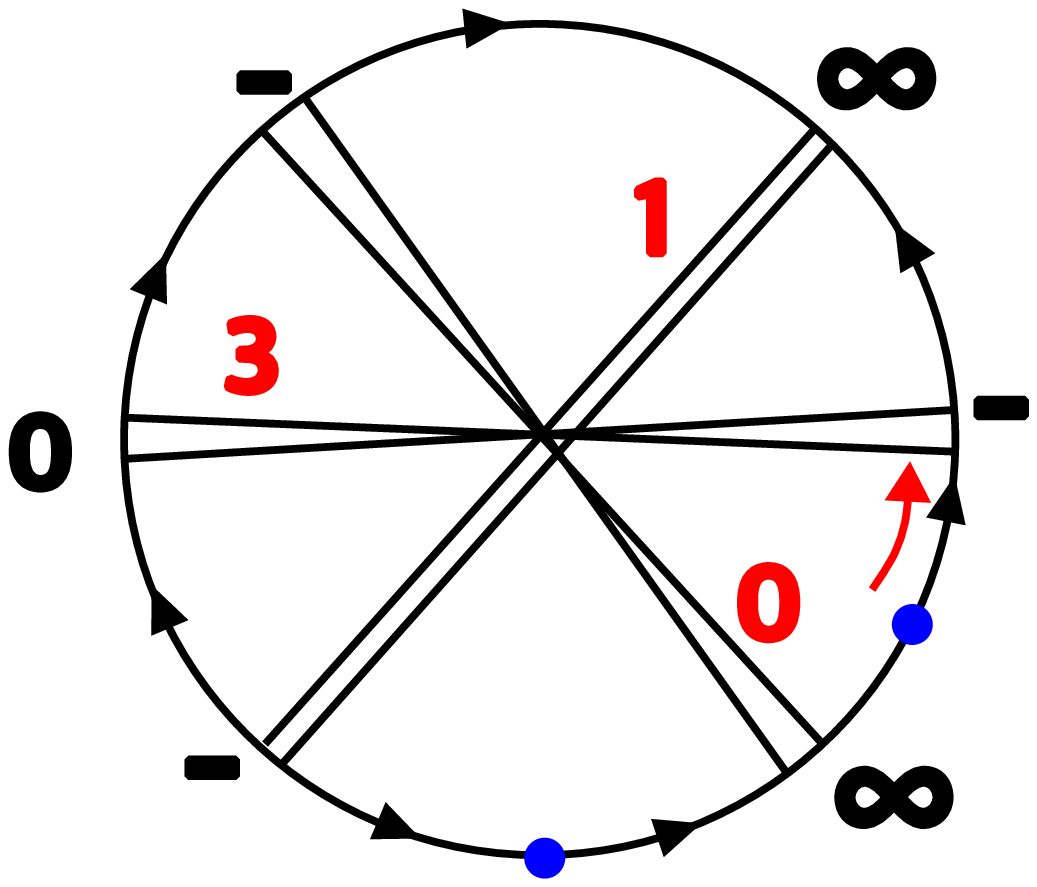} &
      \includegraphics[width=0.25\textwidth]{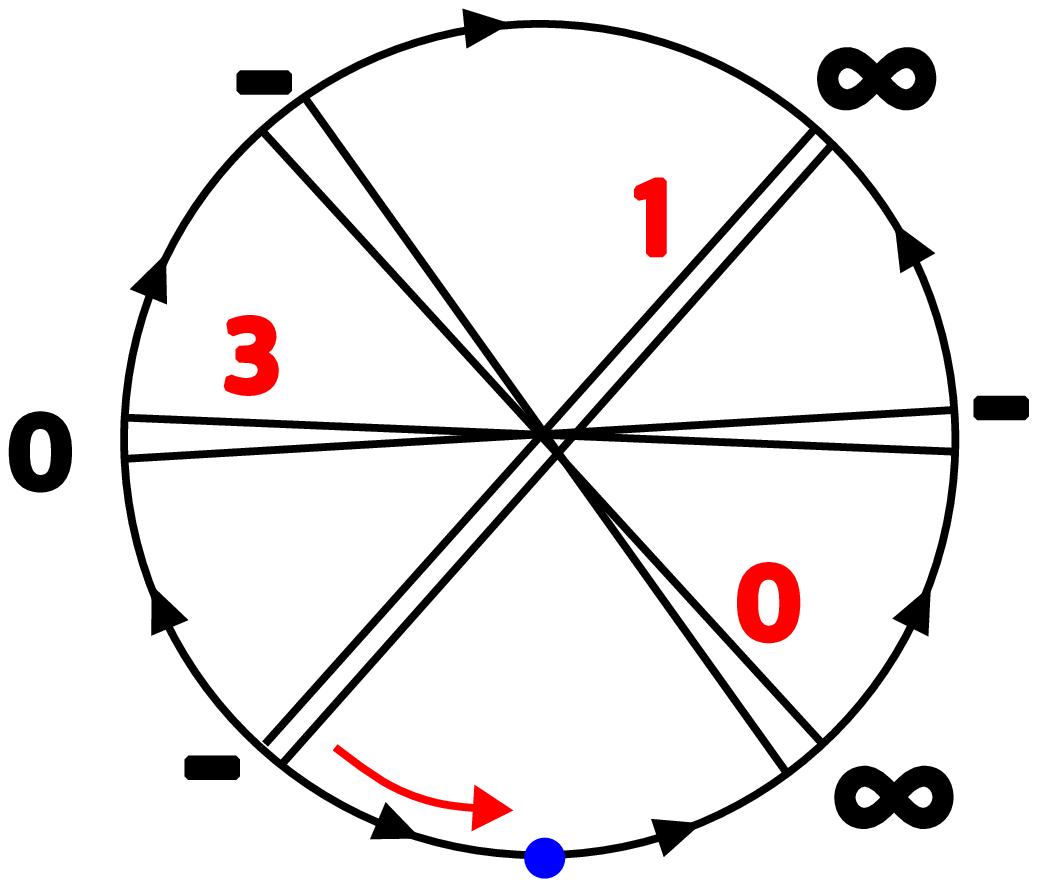} &
      \includegraphics[width=0.25\textwidth]{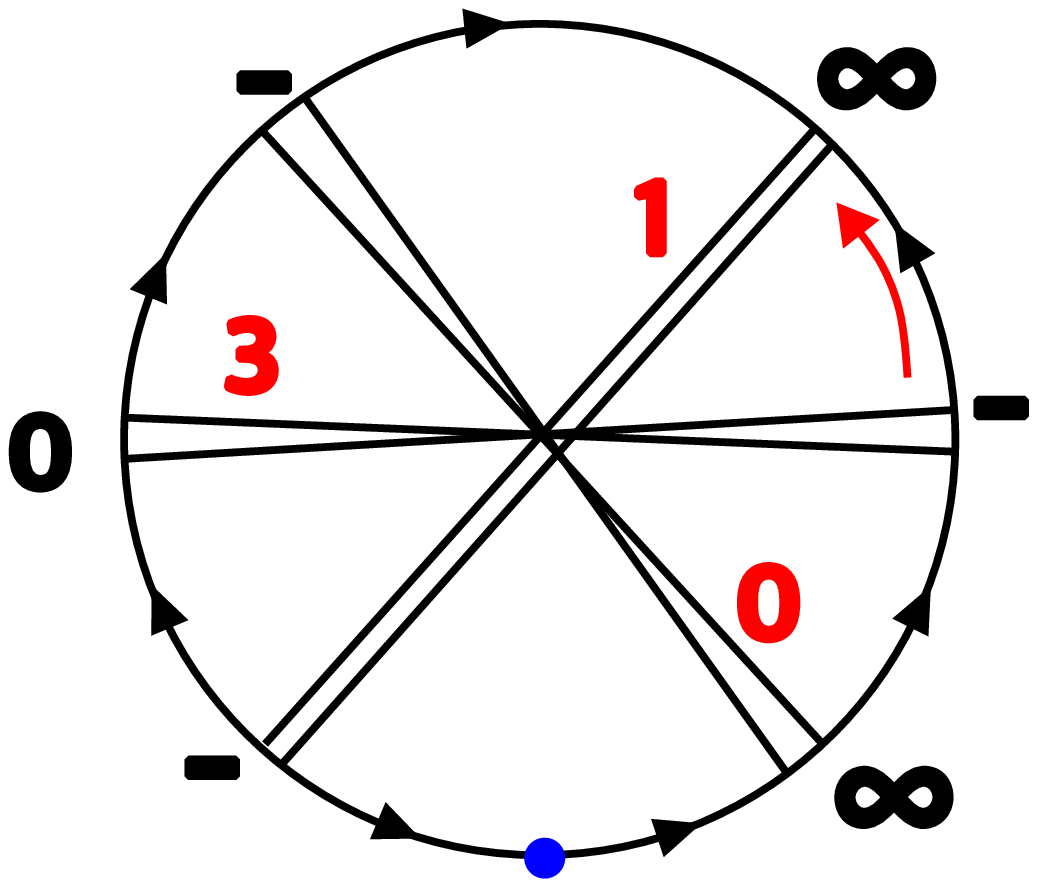} \\
      \includegraphics[width=0.25\textwidth]{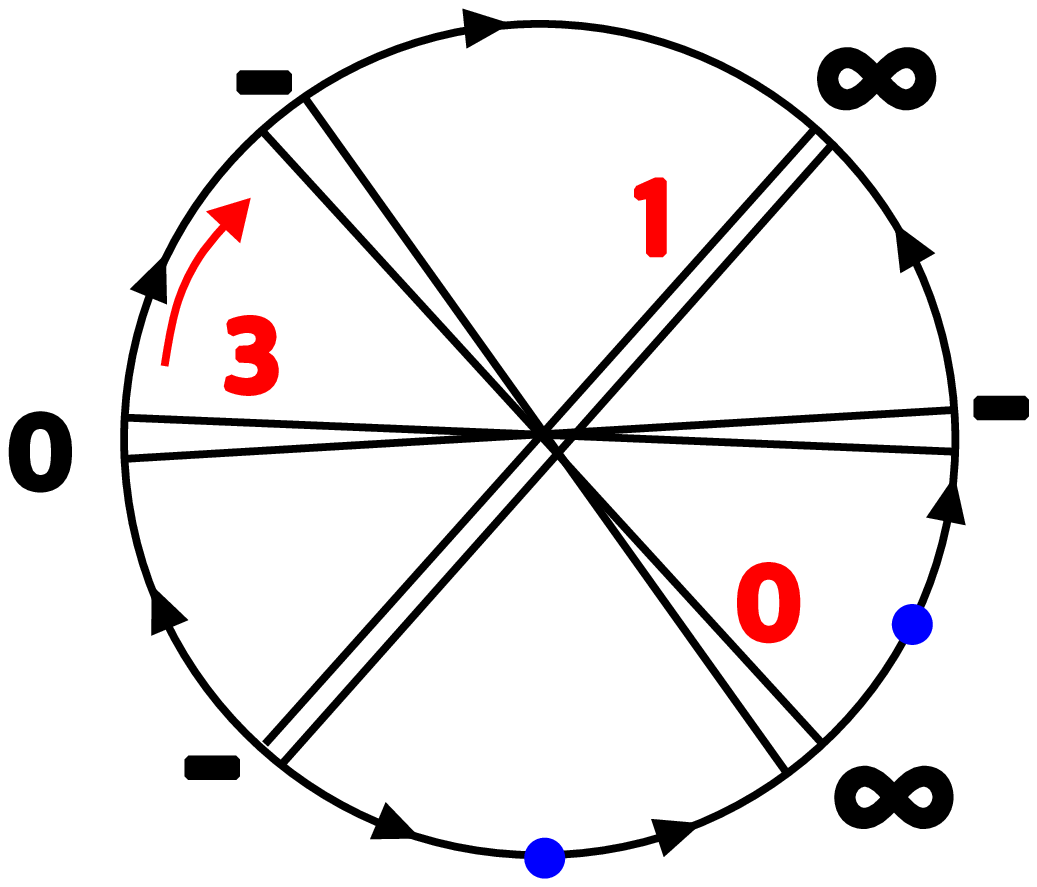} &
      \includegraphics[width=0.25\textwidth]{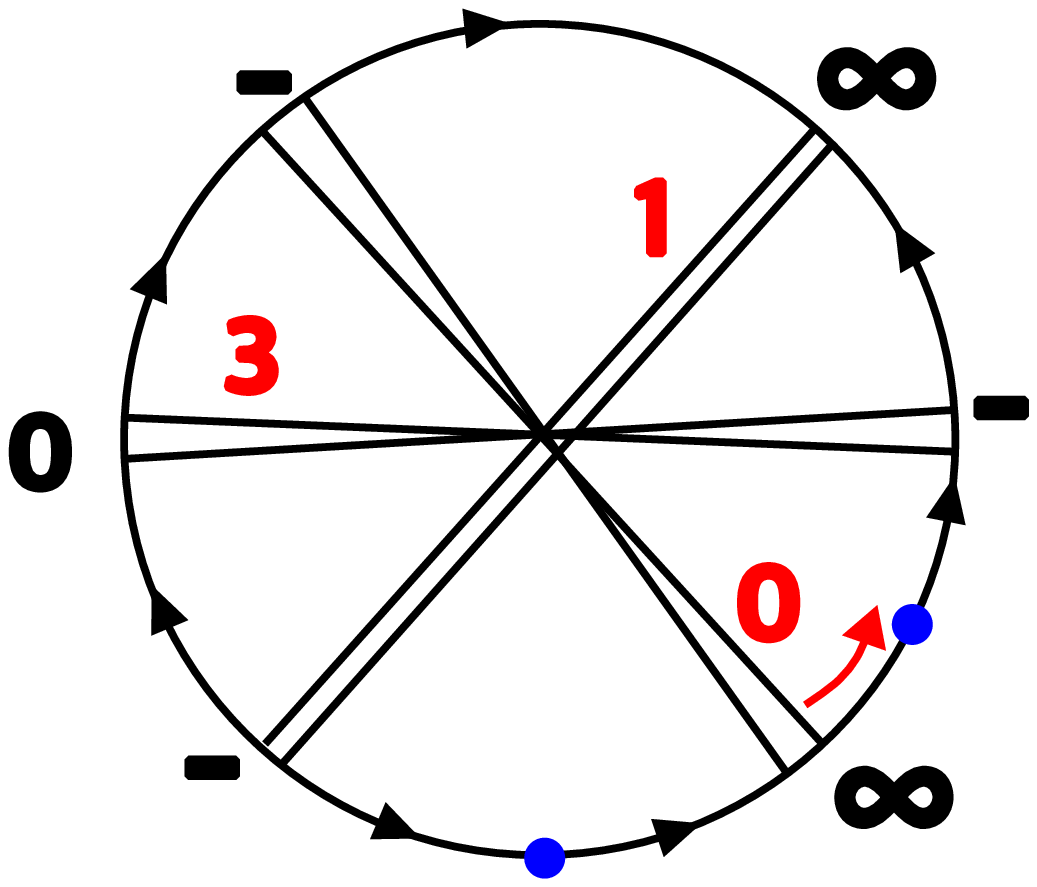} &
      \phantom{} \\
    };

    \draw[->, shorten >=1mm, shorten <=1mm] (m-1-1) -- (m-1-2);
    \draw[->, shorten >=1mm, shorten <=1mm] (m-1-2) -- (m-1-3);
    \draw[->, shorten >=1mm, shorten <=1mm] (m-1-3) -- (m-2-3);
    \draw[->, shorten >=1mm, shorten <=1mm] (m-2-3) -- (m-2-2);
    \draw[->, shorten >=1mm, shorten <=1mm] (m-2-2) -- (m-2-1);
    \draw[->, shorten >=1mm, shorten <=1mm] (m-2-1) -- (m-3-1);
    \draw[->, shorten >=1mm, shorten <=1mm] (m-3-1) -- (m-3-2);
  \end{tikzpicture}
  \caption{The process of the state}
  \label{fig.state_process}
\end{figure}

Noting the first passage of each arrow during the process we have
\begin{align*}
w(\sigma) &= (za)(-za)(za) \\
&= -a^3z^3
\end{align*}
\end{eg}

\begin{eg}[Calculation on a trefoil]
We calculate the Kauffman polynomial of a left handed trefoil (Figure \ref{fig.GD_trefoil}) based on the state model. 
\begin{figure}[H]
    \centering
    \includegraphics[width = 3cm]{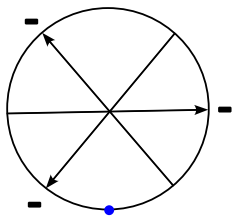}
    \caption{A left-handed trefoil}
    \label{fig.GD_trefoil}
\end{figure}

\begin{table}[htbp]
\centering
\begin{tabular}{| c | c | c | c |}
\hline
contributing state & $w(\sigma)d^{c(\sigma)-1}$ & contributing state & $w(\sigma)d^{c(\sigma)-1}$ \\

\hline
\begin{minipage}[b]{0.2\columnwidth}
    \centering
	\raisebox{-.5\height}{\includegraphics[width=\linewidth]{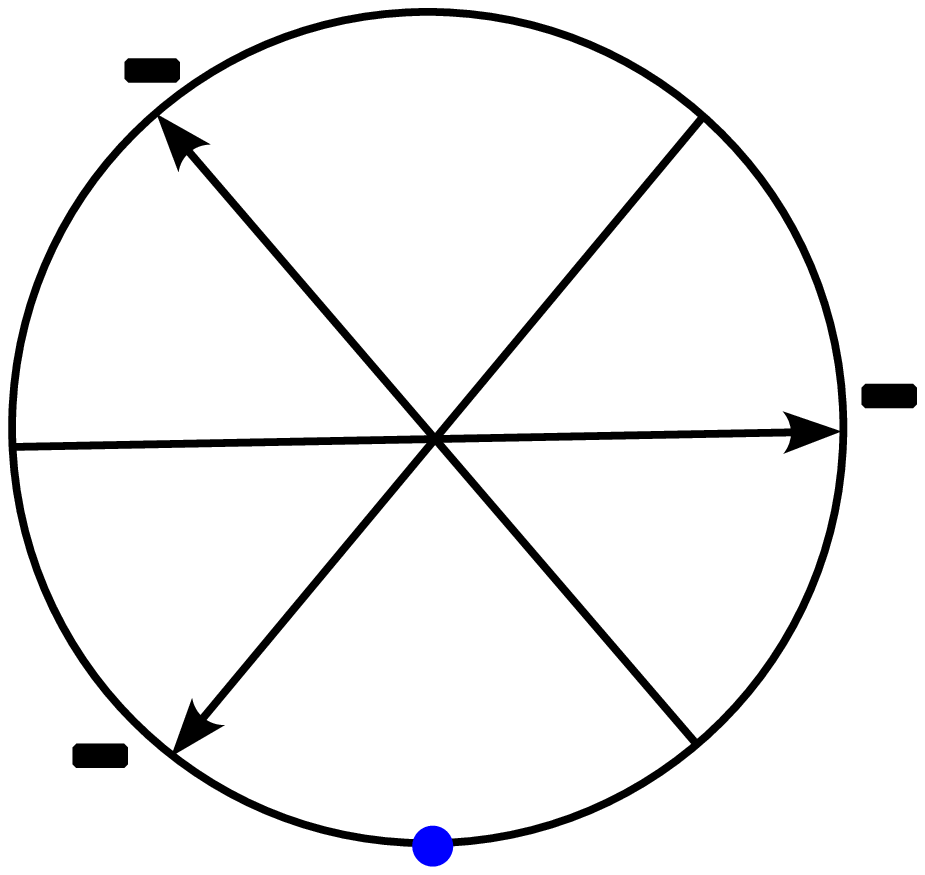}}
\end{minipage}
& $a^4$
&\begin{minipage}[b]{0.2\columnwidth}
    \centering
    \raisebox{-.5\height}{\includegraphics[width=\linewidth]{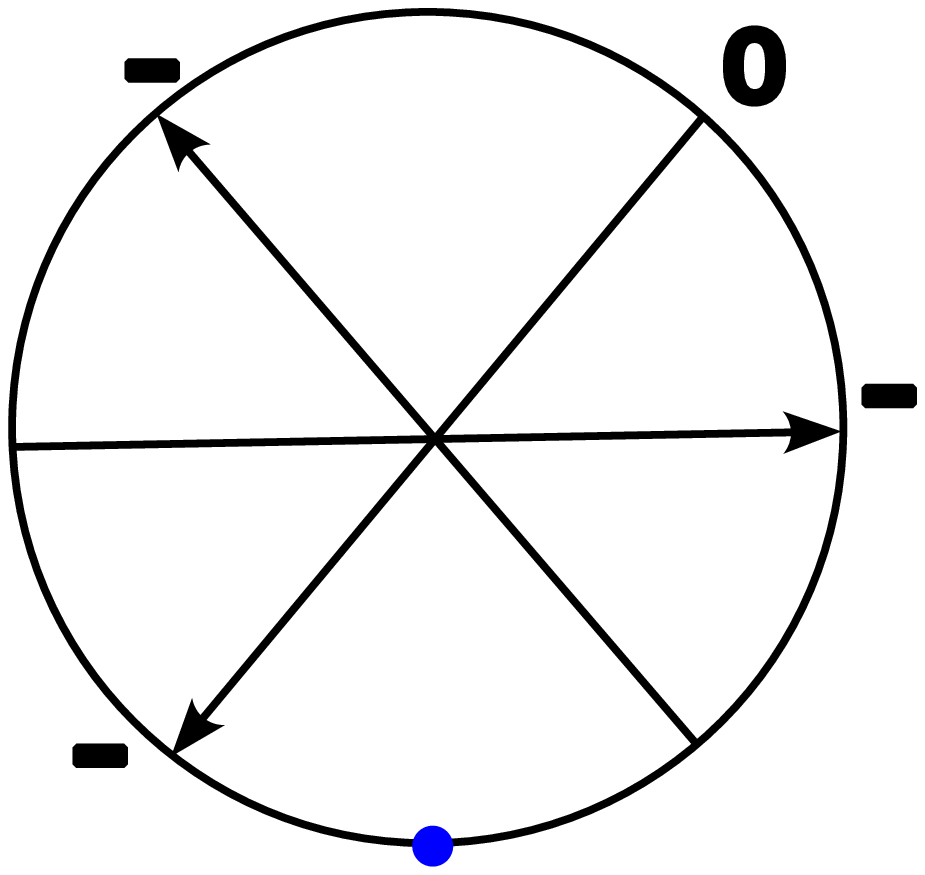}}
    \end{minipage}
& $-a^3zd$ \\
\hline

\begin{minipage}[b]{0.2\columnwidth}
    \centering
	\raisebox{-.5\height}{\includegraphics[width=\linewidth]{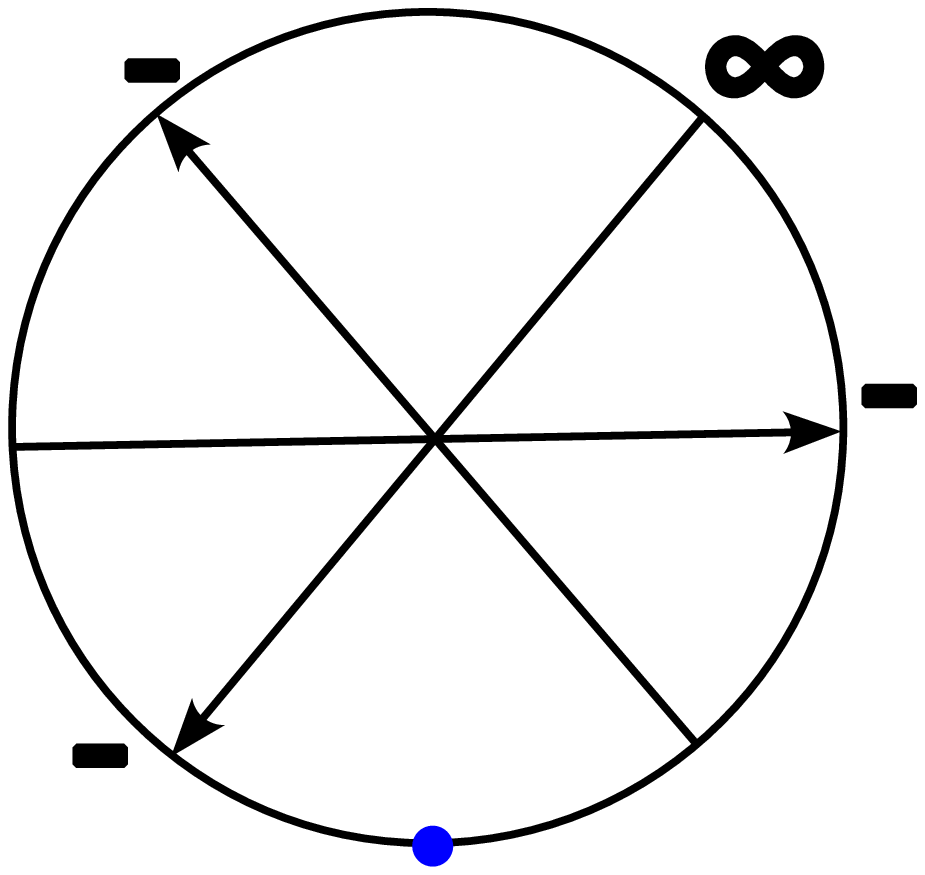}}
\end{minipage}
& $a^3z$
&\begin{minipage}[b]{0.2\columnwidth}
    \centering
    \raisebox{-.5\height}{\includegraphics[width=\linewidth]{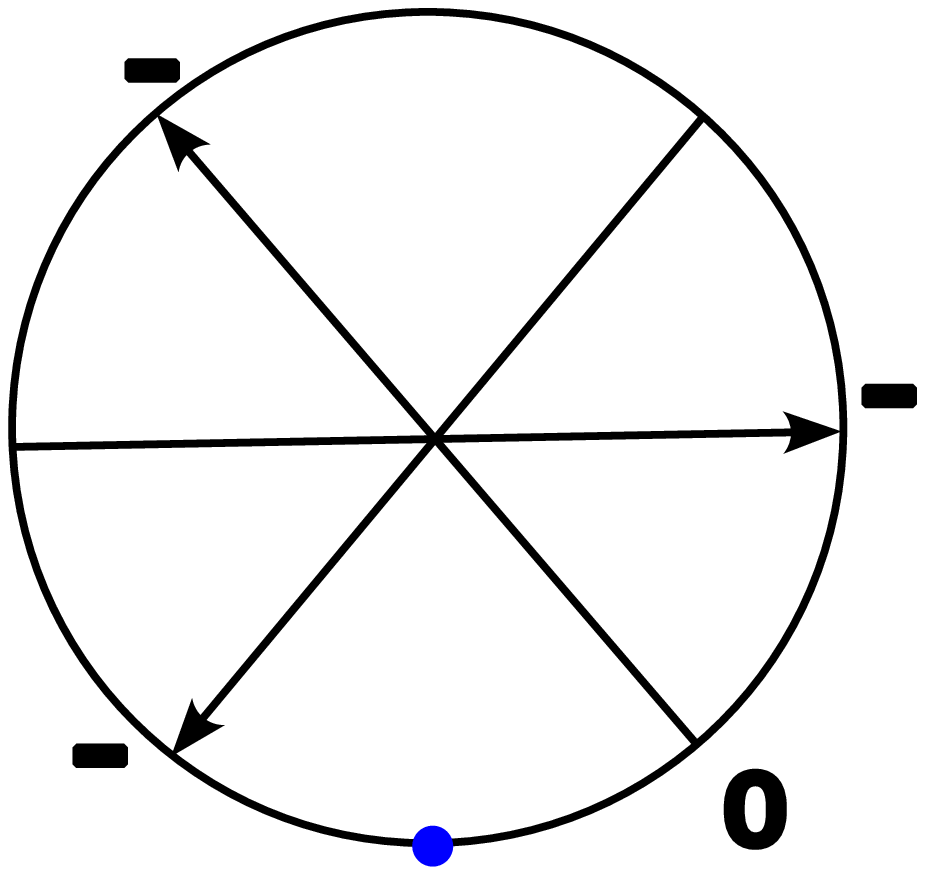}}
    \end{minipage}
& $-a^3zd$ \\
\hline

\begin{minipage}[b]{0.2\columnwidth}
    \centering
	\raisebox{-.5\height}{\includegraphics[width=\linewidth]{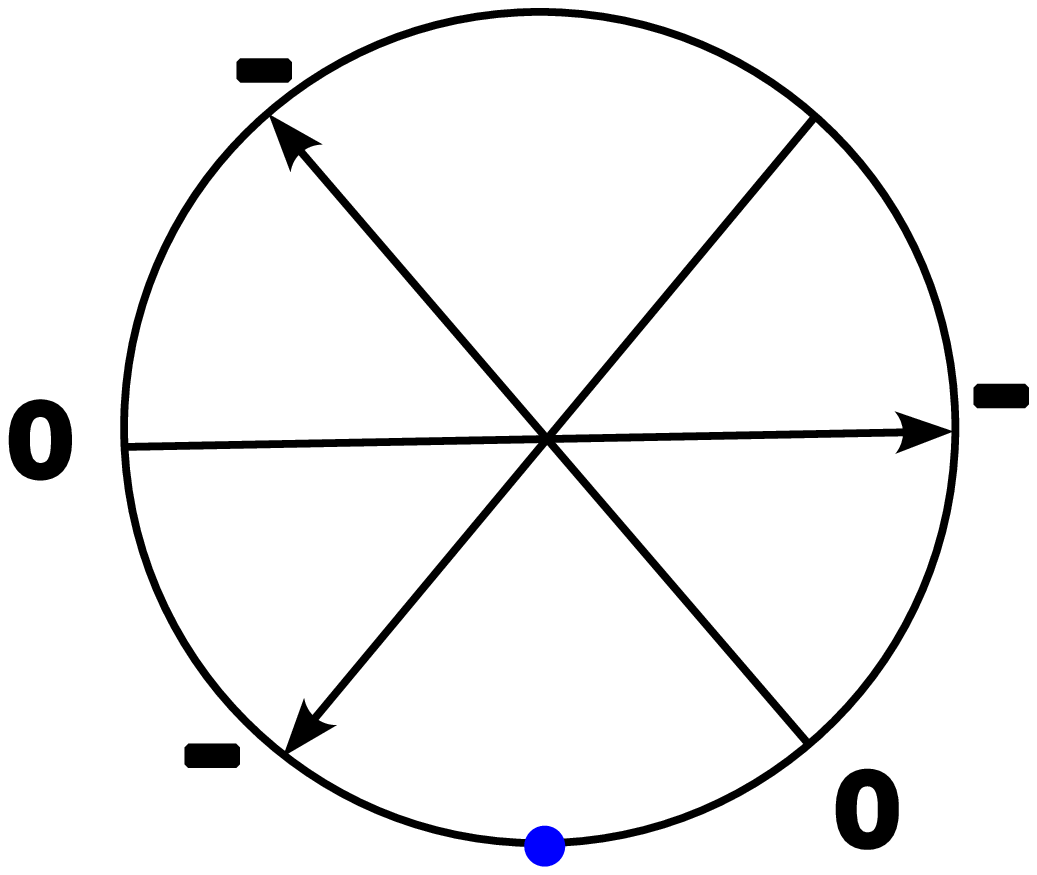}}
\end{minipage}
& $a^4z^2$
&\begin{minipage}[b]{0.2\columnwidth}
    \centering
    \raisebox{-.5\height}{\includegraphics[width=\linewidth]{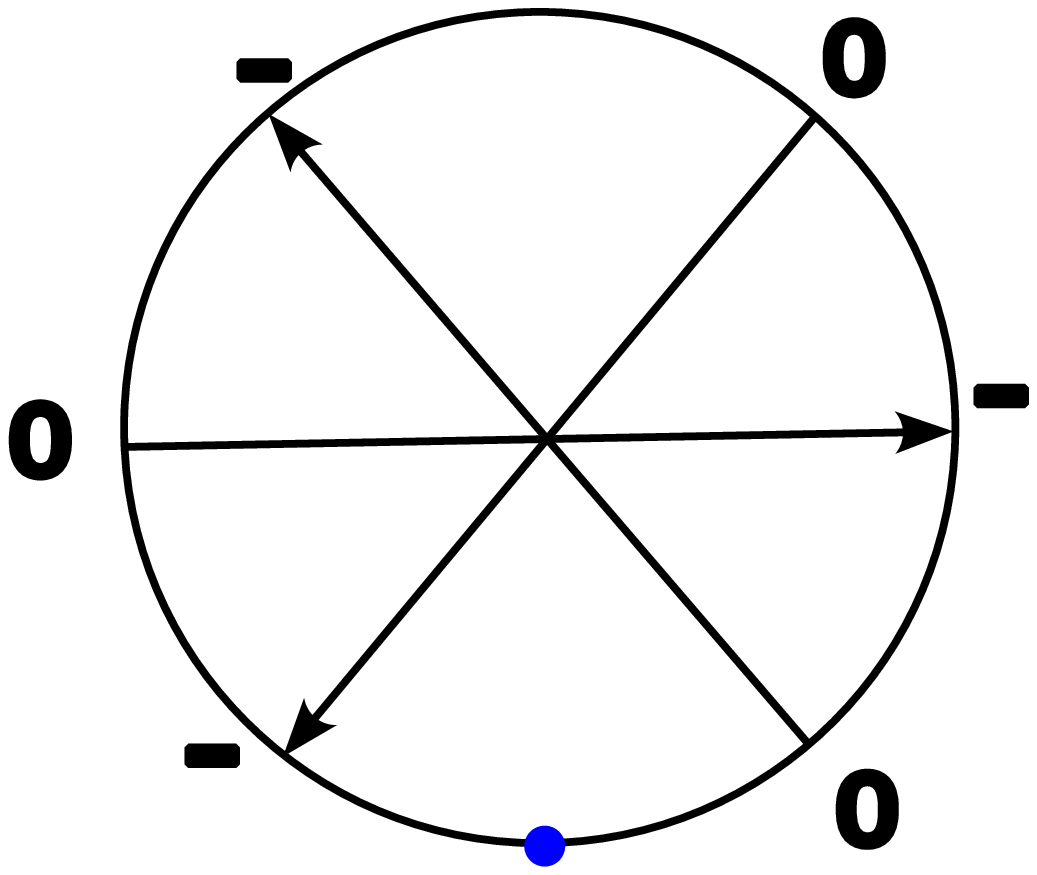}}
    \end{minipage}
& $-a^3z^3d$ \\
\hline

\begin{minipage}[b]{0.2\columnwidth}
    \centering
	\raisebox{-.5\height}{\includegraphics[width=\linewidth]{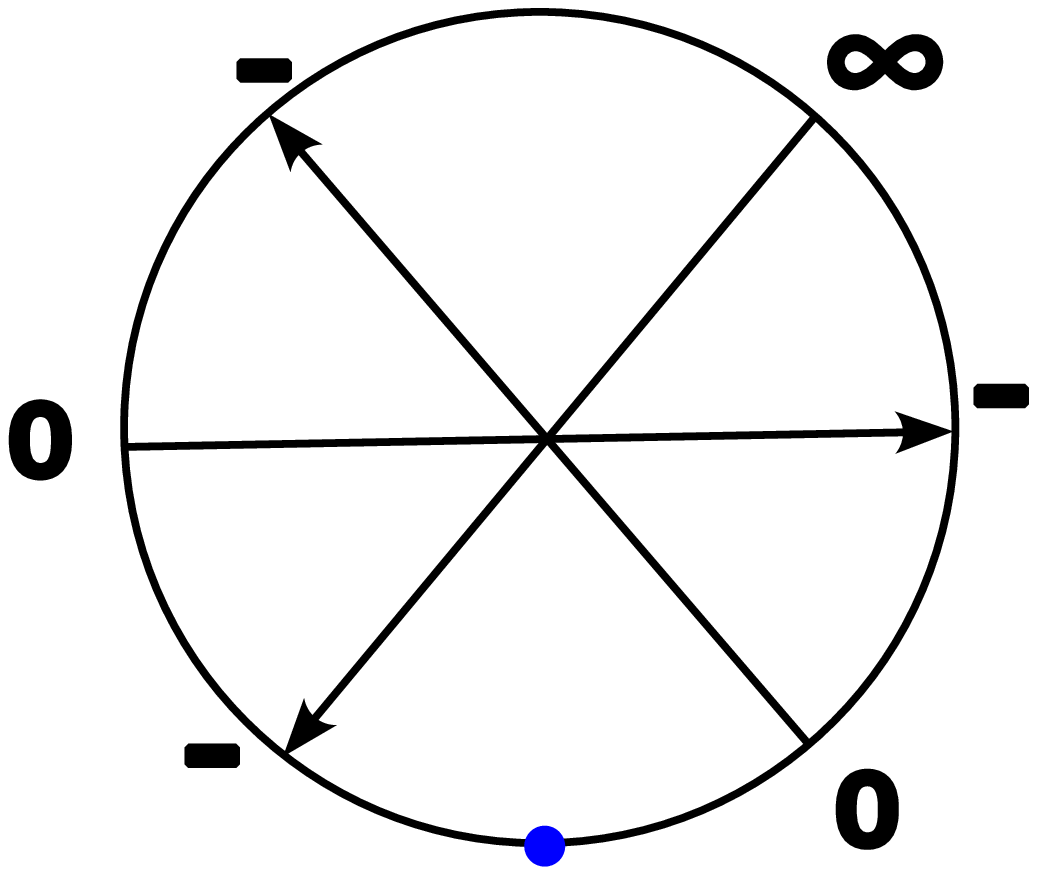}}
\end{minipage}
& $a^3z^3$
&\begin{minipage}[b]{0.2\columnwidth}
    \centering
    \raisebox{-.5\height}{\includegraphics[width=\linewidth]{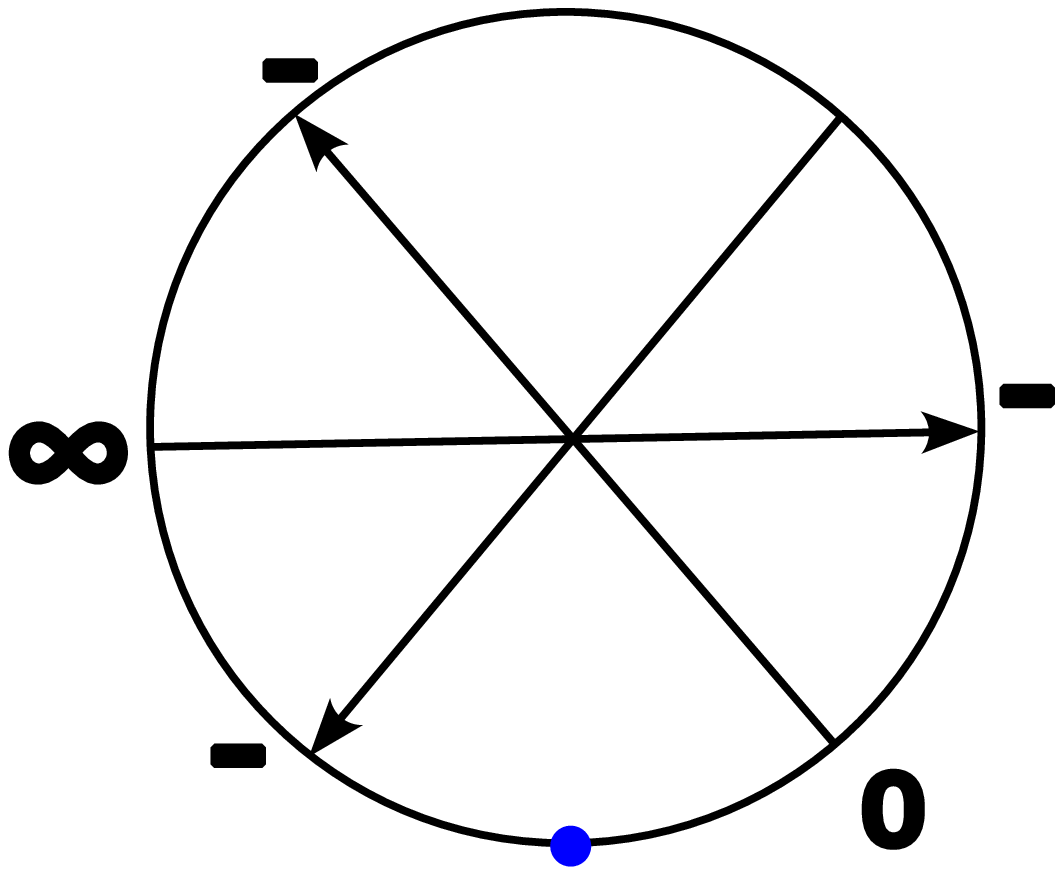}}
    \end{minipage}
& $-a^2z^2$ \\
\hline

\begin{minipage}[b]{0.2\columnwidth}
    \centering
	\raisebox{-.5\height}{\includegraphics[width=\linewidth]{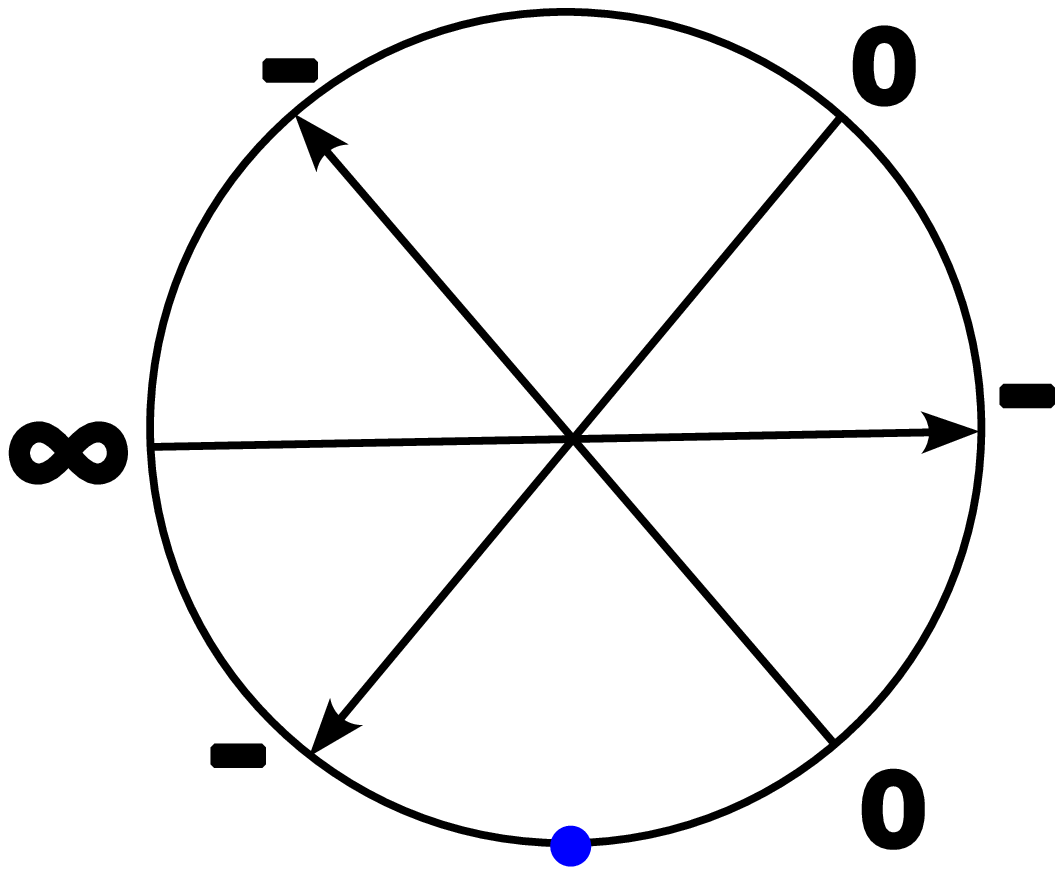}}
\end{minipage}
& $-a^3z^3d$
&\begin{minipage}[b]{0.2\columnwidth}
    \centering
    \raisebox{-.5\height}{\includegraphics[width=\linewidth]{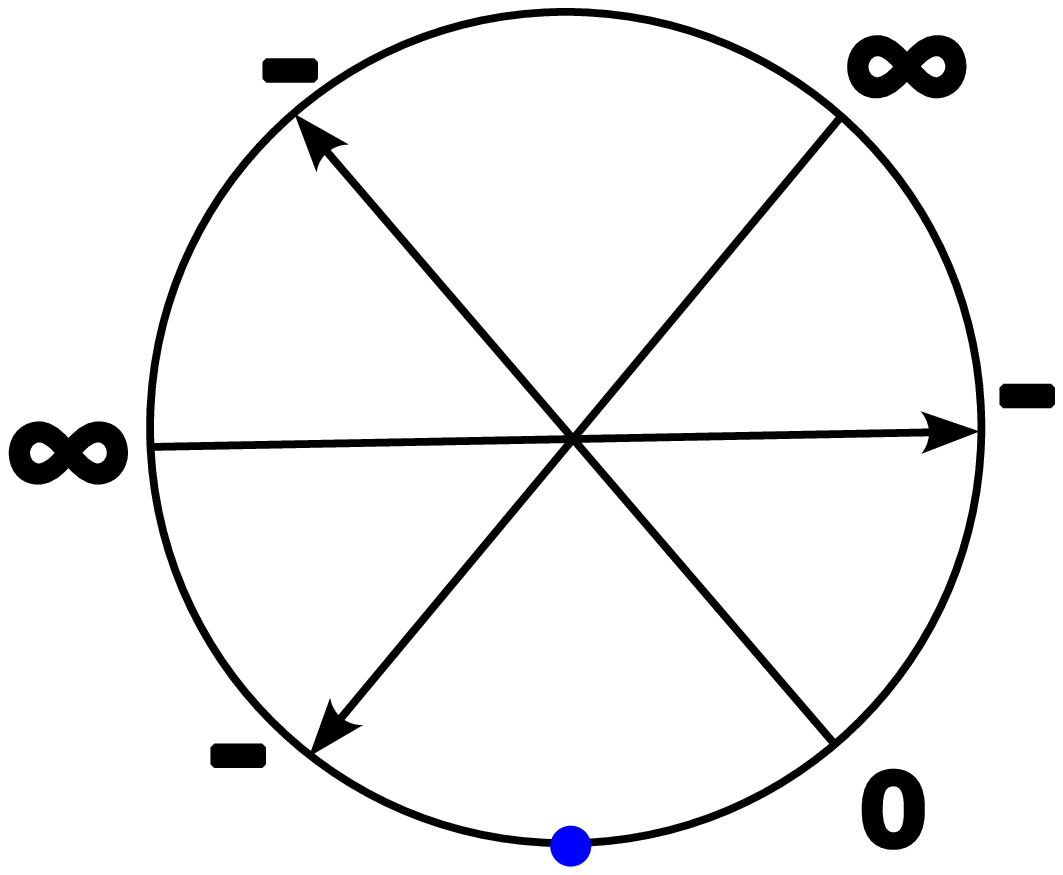}}
    \end{minipage}
& $a^3z^3$ \\
\hline

\begin{minipage}[b]{0.2\columnwidth}
    \centering
	\raisebox{-.5\height}{\includegraphics[width=\linewidth]{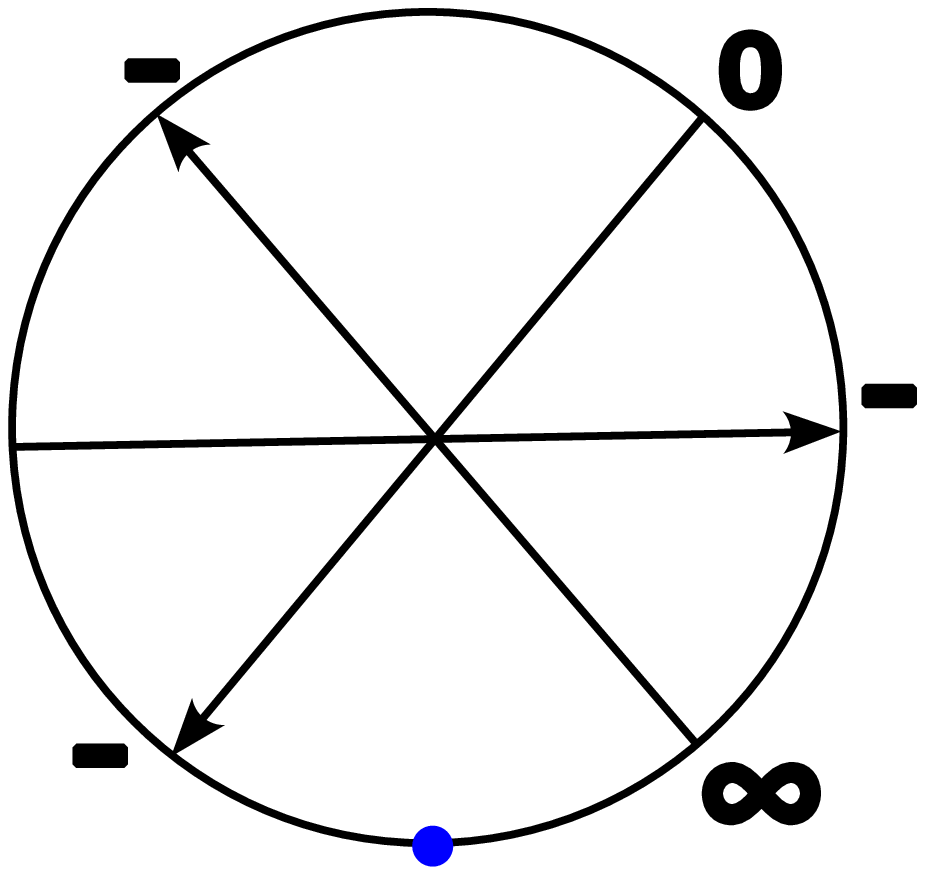}}
\end{minipage}
& $a^2z^2d$
&\begin{minipage}[b]{0.2\columnwidth}
    \centering
    \raisebox{-.5\height}{\includegraphics[width=\linewidth]{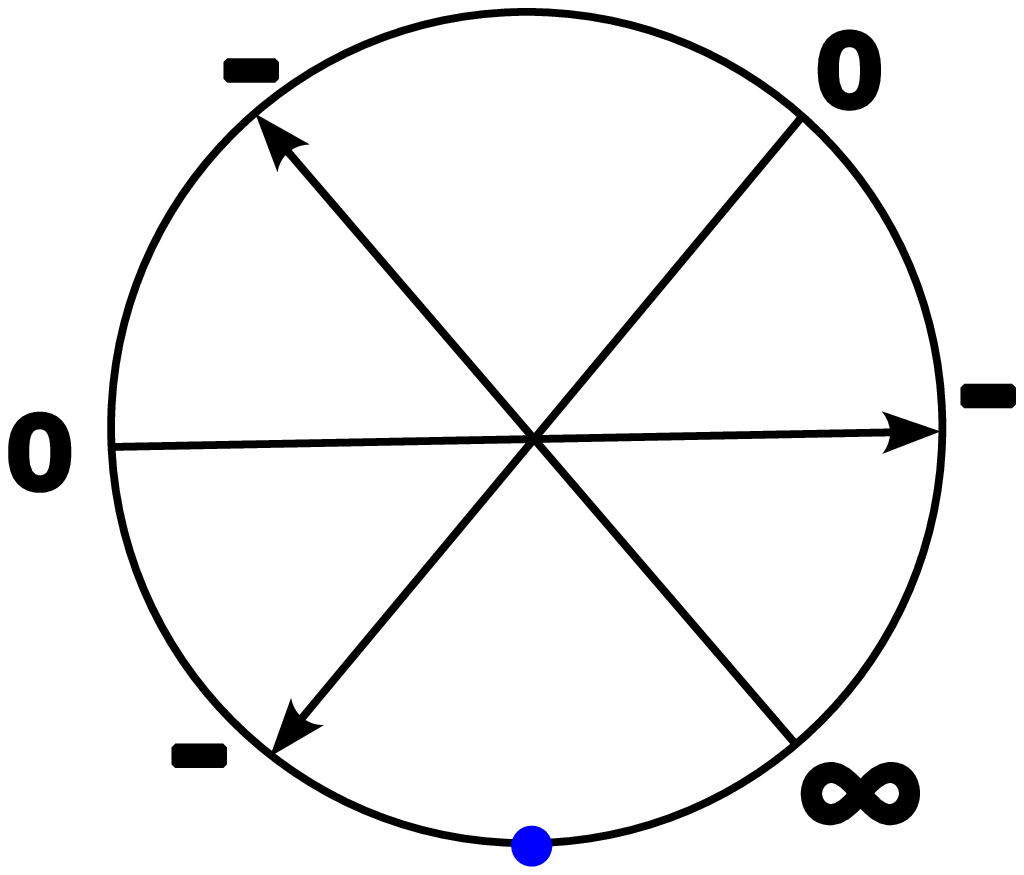}}
    \end{minipage}
& $a^3z^3d^2$ \\
\hline

\begin{minipage}[b]{0.2\columnwidth}
    \centering
	\raisebox{-.5\height}{\includegraphics[width=\linewidth]{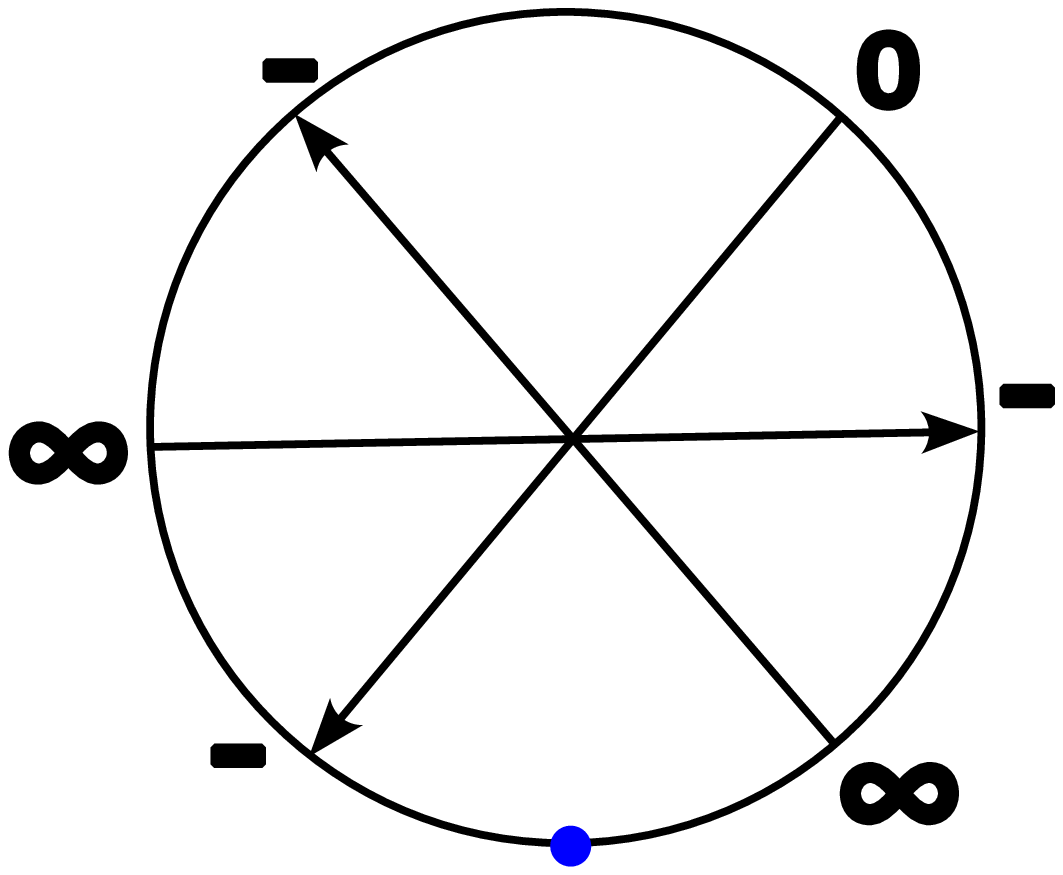}}
\end{minipage}
& $-a^3z^3d$
&\begin{minipage}[b]{0.2\columnwidth}
    \centering
    \raisebox{-.5\height}{\includegraphics[width=\linewidth]{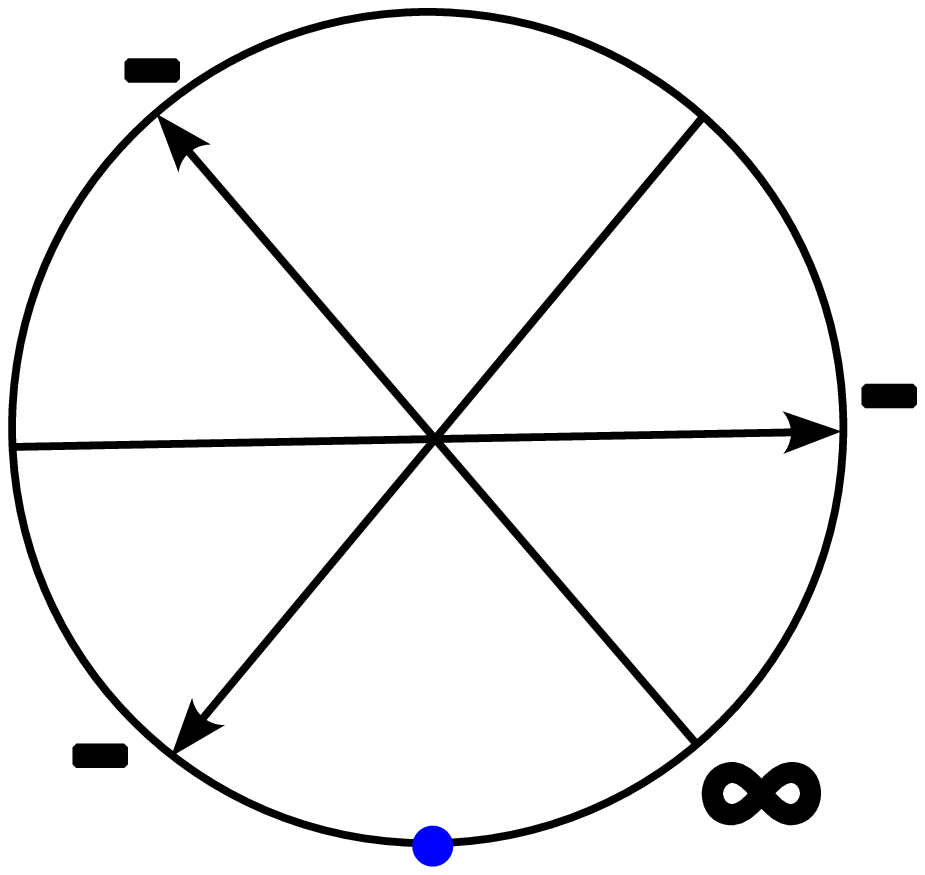}}
    \end{minipage}
& $a^3z$ \\
\hline

\begin{minipage}[b]{0.2\columnwidth}
    \centering
	\raisebox{-.5\height}{\includegraphics[width=\linewidth]{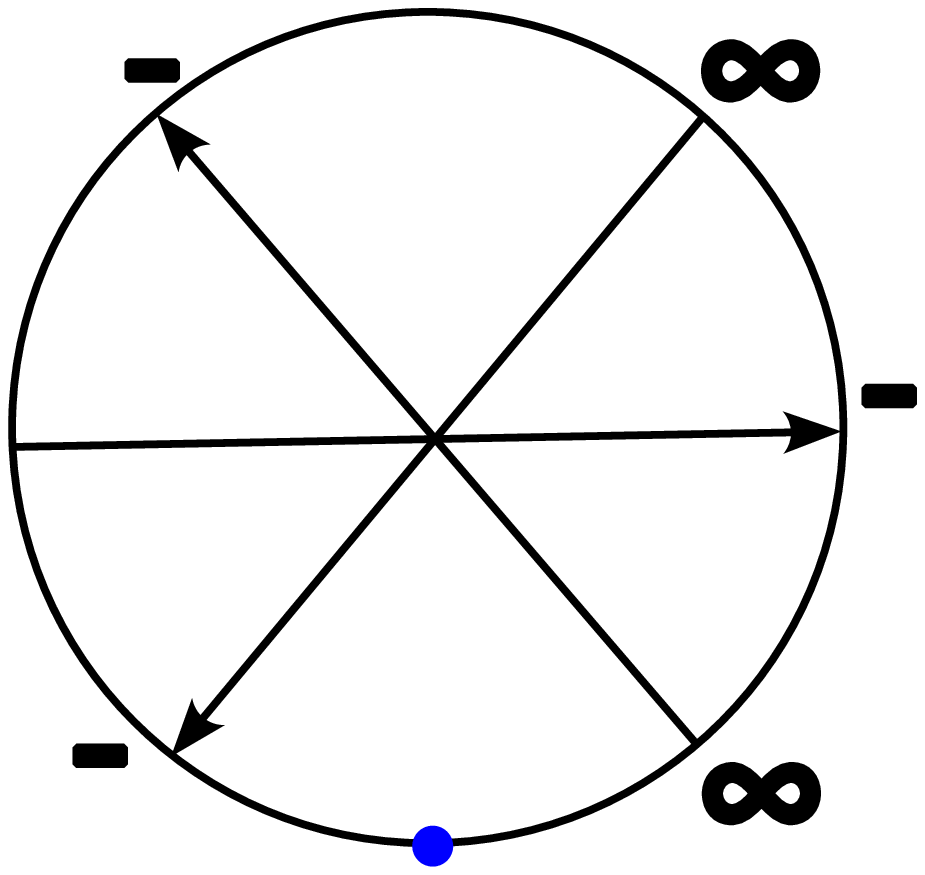}}
\end{minipage}
& $-a^2z^2$
&\begin{minipage}[b]{0.2\columnwidth}
    \centering
    \raisebox{-.5\height}{\includegraphics[width=\linewidth]{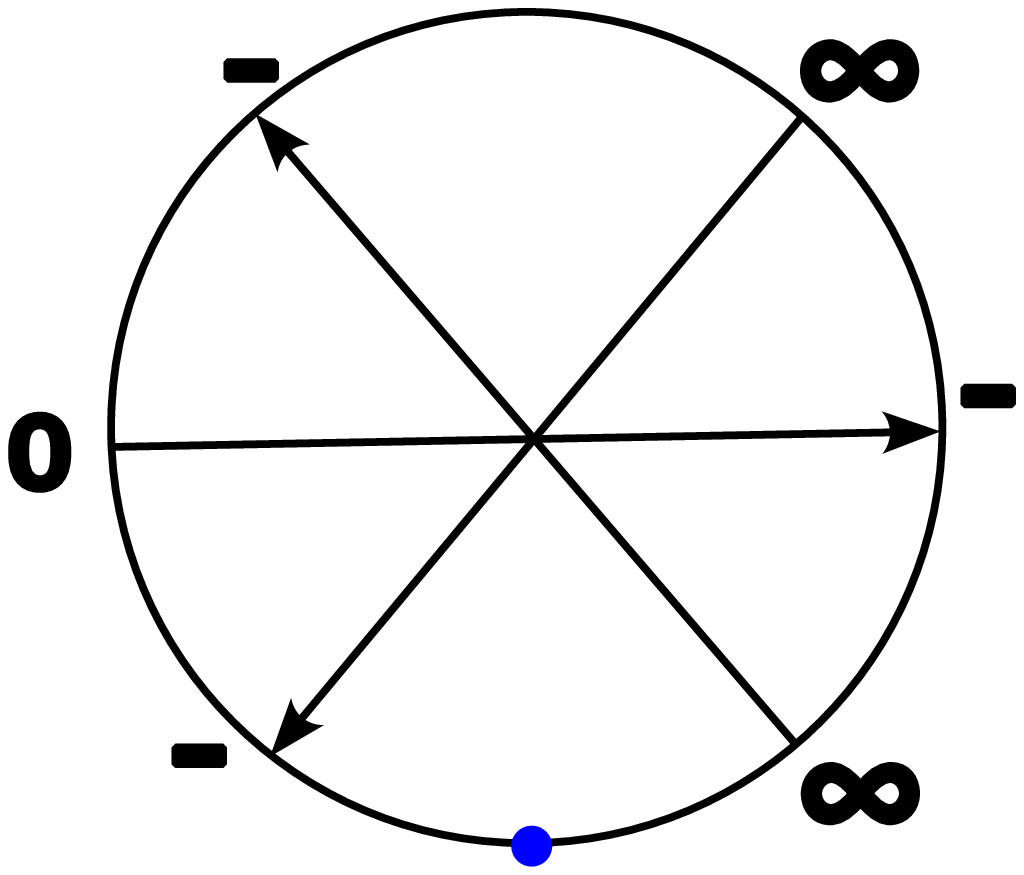}}
    \end{minipage}
& $-a^3z^3d$ \\
\hline

\begin{minipage}[b]{0.2\columnwidth}
    \centering
	\raisebox{-.5\height}{\includegraphics[width=\linewidth]{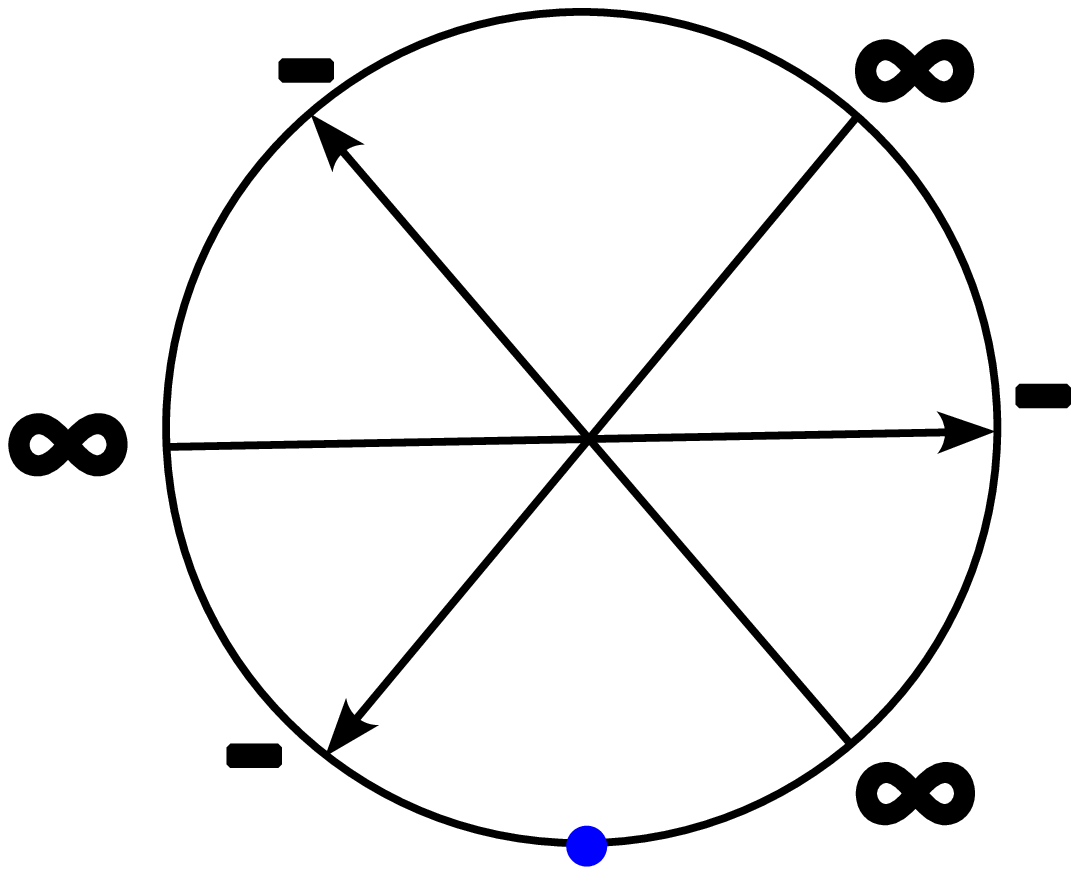}}
\end{minipage}
& $a^3z^3$
&
&  \\
\hline

\end{tabular}
\caption{Contributing states}
\label{table.Contributing states}
\end{table}
According to Table \ref{table.Contributing states}, we have $$DK= 2a^2-a^4+a^5z-a^3z+a^2z^2-a^4z^2$$
\end{eg}

\begin{prop}\label{prop.evaluation}
    For any oriented knot $K$ we have $$DK_K(1,z)=1$$
\end{prop}

\begin{proof}
    Let $G$ be the Gauss diagram of any knot $K$. We study the state contribution to $DK_K(1,z)$ based on the state model. Notice that when $a=1$, the weight of the unlabelled arrow in all the cases is always $1$ and that the component factor $d=1$. First of all, we have the state with all arrows unlabelled. This state gives the contribution $1$ to $DK$. For any other contributing state $\sigma$, we focus on the last labelled arrow $\alpha$ in the process. If $\alpha$ is labelled with $0$ (resp. $\infty$), then change this label to $\infty$ (resp. $0$) will get another contributing state $\sigma'$ and $w(G,\sigma') = -w(G,\sigma)$. Hence from the state model we can see that $DK_K(1,z) = 1$. 
\end{proof}

%%%%%%%%%%%%%%%%%%%%%%%%%%%%%
%%%%%%%%%%%%%%%%%%%%%%%%%%%%%
\section{Gauss diagram formulae}\label{sec.GDF}

After the change of variable $a = e^h$, the coefficients of the power series of Kauffman polynomial are Vassiliev invariants (Theorem 4.8 in \cite{Birman:1993aa}).
\begin{prop}
Let $$P_K(h,z)=DK_K(e^h,z)=\sum\limits_{k,l}p_{k,l}(K)h^kz^l$$
be the power series in two variables $h,z$. The coefficients $p_{k,l}$ is a Vassiliev invariant with order no more than $k+l$.
\end{prop}
\begin{proof}
Let $K_{+},K_{-},K_{0},K_{\infty}$ be four oriented links with differences at one single crossing (if one ignore the orientation). The first three have a consistent orientation and the fourth have a piece with an opposite orientation of the first three. We denote this piece $\beta$ endowed with the orientation in $K_{+}$. The writhe relation of them are:
$w(K_{+}) = w(K_{-})+2 = w(K_{0})+1 = w(K_{\infty}) + 2\lambda + 1$ where $\lambda = w(\beta,K_{+}-\beta)$ is the writhe of $\beta$ with other part of the link $K_{+}$. 
Then we can see that $$aDK_{+}-a^{-1}DK_{-} = z(DK_{0} - DK_{\infty}a^{-2\lambda})$$.
$$DK_{+}-DK_{-} = (a^{-2}-1)DK_{-} + a^{-1}zDK_{0} - a^{-(2\lambda+1)}zDK_{\infty}$$

Given any singular knot $K$ with $m$ singular crossings, 
we denote $K_{\epsilon_1\epsilon_2...\epsilon_{m} }$ the ordinary knot based on $K$ such that the singular crossings is replaced by the ordinary crossings of sign $\epsilon$. By the skein relation above, we have 

\begin{align}
&\sum\limits_{\epsilon_1,\epsilon_2,...,\epsilon_{m}} \epsilon_1\epsilon_2\cdots\epsilon_{m}P_{K_{\epsilon_1\epsilon_2...\epsilon_{m}}} \label{eq.Vassiliev_test}\\ 
&= \sum\limits_{\epsilon_1,\epsilon_2,...,\epsilon_{m-1}} \epsilon_1\epsilon_2\cdots \epsilon_{m-1} [(a^{-2}-1)P_{\epsilon_1\epsilon_2...\epsilon_{m-1}-} + a^{-1}zP_{\epsilon_1\epsilon_2...\epsilon_{m-1}0} - a^{-(2\lambda+1)}zP_{\epsilon_1\epsilon_2...\epsilon_{m-1}\infty}] \\
&= \sum\limits_{d_1,d_2,...,d_{m}}c(d_1)c(d_2)\cdots c(d_m)P_{d_1 d_2...d_{m}}
\end{align}
where $d_i \in \{-,0,\infty \}$ and 
$$c(-)=a^{-2}-1=-2h+o(1)$$
$$c(0)=a^{-1}z = z+o(1)$$
$$c(\infty) = -a^{-(2\lambda+1)}z = -z+o(1)$$
Notice that the lowest degree of $P_{d_1 d_2...d_{m}}$ is $0$ and that $$c(d_1)c(d_2)\cdots c(d_m) = (-2h+o(1))^{m_1}(z+o(1))^{m_2}(-z+o(1))^{m_3}$$ where $m_1+m_2+m_3=m$.
If $m>k+l$, we have either $m_1>k$ or $m_2+m_3>l$. Hence the coefficient of the term $h^kz^l$ in \ref{eq.Vassiliev_test}  would be 0, which proves the proposition. 
\end{proof}

Now we use the method of Chmutov and Polyak in \cite{Chmutov_Polyak:2009} to derive the Gauss diagram formula for $p_{k,l}$. We repeat the essential of their proof of Theorem \ref{thm.representation} emphasizing why the theorem still works for the model of Kauffman polynomial. 

Let $\mathscr{A}$ be the free $\mathbb{Q}$-vector space spanned by the basis of all arrow diagrams. We define an inner scalar product by \begin{equation}
 (A,B) = \left\{
\begin{array}{rcl}
1 && {\text{$A=B$}}\\
0 && {\text{otherwise}}
\end{array}
\right.
\end{equation}
And we extend this scalar product to the whole $\mathscr{A}$ by linearity. We define $$\langle A,G\rangle = (A,\sum\limits_{ C \subset G}C)$$ for any arrow diagram $G$, where the $C\subset G$ means the arrow diagram with arrows being a subset of arrows of G. In words, this product counts the number of times $A$ appears in $G$ as a subdiagram. Let $F= \sum\limits_{i}c_iA_i$ be an element of $\mathscr{A}$. If $\langle F,\cdot \rangle$ is a knot invariant, we say that $F$ is a \textit{Gauss diagram formula} (GDF). 

Given an arrow diagram $A$, we apply the previous state model procedure to it but with a different weight $w'(\cdot)$ described in Table \ref{table.weight_arrow} and we get a new polynomial $$W_A(a,z)=\sum \limits_{\sigma \in state(A)} w'(A,\sigma)(\dfrac{a-a^{-1}}{z}+1)^{c(\sigma)-1}$$

Let $W_A(e^h,z) = \sum\limits_{k.l}w_{k,l}(A)h^kz^l$ and we define $$A_{k,l} \coloneqq  \sum_{A}w_{k,l}(A)\cdot A$$. 

\begin{rem}
From Table \ref{table.weight_arrow} we can see that $A_{k,l}$ is well-defined. Indeed, if $A$ is an arrow diagram with $m$ arrows ($m > k+l$), the lowest degree of the contributing weight of any state of the arrow diagram is at least $m$. So $w_{k,l}(A) = 0$. 
\end{rem}

\begin{prop}
    For an arrow diagram $A$ with an isolated arrow $\alpha$ (i.e., the foot and the head of $\alpha$ bound piece of arc on which there is no other arrows or the base point), we have $W_A = 0$.
\end{prop}
\begin{proof}
    Notice that the change number $n(\alpha)$ must be even. So in a contributing state, $\alpha$ must be labelled. In this case, if we change the label of $\alpha$ (between $0$ and $\infty$), we will get another contributing state.If the first passage at $\alpha$ is at head, then weight is surely $0$. If the first passage is at foot and if it has positive sign, $a^{-2}-1 + za^{-1}d -za^{-1} = 0$ implies that the three contributing states cancelled out together. Similar for the negative sign case. 
\end{proof}

\begin{table}[htbp]
\centering
\begin{tabular}{| c | c | c |}
\hline
& \begin{minipage}[b]{0.25\columnwidth}
		\centering
		\raisebox{-.5\height}{\includegraphics[width=\linewidth]{image/table1_7.png}}
	\end{minipage}
& \begin{minipage}[b]{0.25\columnwidth}
		\centering
		\raisebox{-.5\height}{\includegraphics[width=\linewidth]{image/table1_8.png}}
	\end{minipage} \\
\hline
\begin{minipage}[b]{0.25\columnwidth}
    \centering
	\raisebox{-.4\height}{\includegraphics[width=\linewidth]{image/table1_1.png}}
	\end{minipage}
 & $a^{(-1)^n-1}-1$ & $a^{(-1)^{n+1}+1}-1$ \\
\hline
\begin{minipage}[b]{0.25\columnwidth}
    \centering
	\raisebox{-.4\height}{\includegraphics[width=\linewidth]{image/table1_2.png}}
	\end{minipage}
 & $a^{(-1)^{n+1}-1}-1$ & $a^{(-1)^{n}+1}-1$ \\
 \hline
\begin{minipage}[b]{0.25\columnwidth}
    \centering
	\raisebox{-.4\height}{\includegraphics[width=\linewidth]{image/table1_3.png}}
	\end{minipage}
 & $0$ & $0$ \\
\hline
\begin{minipage}[b]{0.25\columnwidth}
    \centering
	\raisebox{-.4\height}{\includegraphics[width=\linewidth]{image/table1_4.png}}
	\end{minipage}
 & $(-1)^{n}za^{-1}$ & $(-1)^{n+1}za$ \\
 \hline
\begin{minipage}[b]{0.25\columnwidth}
    \centering
	\raisebox{-.4\height}{\includegraphics[width=\linewidth]{image/table1_5.png}}
	\end{minipage}
 & 0 & 0 \\
\hline
\begin{minipage}[b]{0.25\columnwidth}
    \centering
	\raisebox{-.4\height}{\includegraphics[width=\linewidth]{image/table1_6.png}}
	\end{minipage}
 & $(-1)^{n+1}za^{-1}$ & $(-1)^{n}za$ \\
\hline
\end{tabular}
\caption{Weight for arrow diagrams}
\label{table.weight_arrow}
\end{table}
 
\begin{thm}\label{thm.representation}
    $A_{k,l}$ is a Gauss diagram formula. In fact, given a Gauss diagram $G$ of an oriented knot $K$, we have $$p_{k,l}(K) = \langle A_{k,l} ,G\rangle$$.
\end{thm}

\begin{proof}
Given a state $\sigma: A(G) \rightarrow \{\phi, 0, \infty \}$, we denote $A_{\phi}(G)$ and $A_{L}(G)$ the sets of arrows unlabelled and labelled respectively. 

\begin{align*}
    w(G,\sigma) &= \prod\limits_{\alpha \in A(G)}w(G,\sigma,\alpha)\\
    &= \prod\limits_{\alpha \in A_{L}(G)}w(G,\sigma,\alpha) \cdot \prod\limits_{\alpha \in A_{\phi}(G)}w(G,\sigma,\alpha)
\end{align*}
Using the algebraic identity $t_1t_2\cdots t_n = 1+[(t_1-1)+(t_2-1)+\cdots (t_n-1)] + [(t_1-1)(t_2-1)+\cdots (t_{n-1}-1)(t_n-1)]+\cdots +[(t_1-1)(t_2-1)\cdots (t_n-1)]$ on $\prod\limits_{\alpha \in A_{\phi}(G)}w(G,\sigma,\alpha)$, we have
\begin{align*}
    DK_K(a,z) &= \sum\limits_{\sigma} d^{c(G,\sigma)-1}w(G,\sigma) \\
    &= \sum\limits_{\sigma} \prod\limits_{\alpha \in A_{L}(G)}w(G,\sigma,\alpha)(\sum\limits_{A_L(G) \subset A(B) \subset A(G)}\prod\limits_{\alpha \in A(B)-A_L(G)}(w(G,\sigma,\alpha)-1))d^{c(G,\sigma)-1} \\
    &=\sum\limits_{B\subset G}\sum\limits_{\sigma : A_{L}(G)\subset A(B)} \prod\limits_{\alpha \in A_{L}(G)}w(G,\sigma,\alpha) \prod\limits_{\alpha \in A(B)-A_{L}(G)}(w(G,\sigma,\alpha)-1))d^{c(G,\sigma)-1}
\end{align*}
Meanwhile we have
\begin{align*}
    \langle A_{k,l}, G \rangle &= \sum\limits_{B \subset G}(B,\sum\limits_{D \in \mathscr{A}}w'_{k,l}(D)D)\\
    &=\sum\limits_{B \subset G}w'_{k,l}(B)\\
    &=\sum\limits_{B \subset G} [ \sum\limits_{\sigma_B} \prod\limits_{\alpha \in A_L(B)}w'(B,\sigma_B,\alpha) \prod\limits_{\alpha \in A_{\phi}(B)}w'(B,\sigma_B,\alpha)d^{c(B,\sigma_B)-1} ]_{k,l}\\
\end{align*}
where $[\cdots]_{k,l}$ means taking the coefficient of term $h^kz^l$. 
For a subdiagram $B \subset G$, there is a bijection between the set $\{\sigma : A(G) \rightarrow \{\phi, 0, \infty \} | A_L(G) \subset A(B)\}$ and the set $\{\sigma_B : A(B) \rightarrow \{\phi, 0, \infty \} \}$. For the corresponding $\sigma$ and $\sigma_B$, we have $A_{L}(G) = A_{L}(B)$ and $A(B)-A_{L}(G) = A_{\phi}(B)$. According to the state model of Kauffman polynomial, we see that $c(G,\sigma) = c(B,\sigma_B)$ because the components is determined by the labelled arrows. If $\alpha \in A_{\phi}(B)$, we have $w'(A,\sigma_B,\alpha) = w(A,\sigma_B,\alpha)-1 = w(G,\sigma, \alpha)-1$ because the arrows in $G-B$ will not influence the weight of $\alpha$. If $\alpha \in A_{L}(B)$, we have $w'(A,\sigma_B,\alpha) = w(A,\sigma_B,\alpha) = w(G,\sigma, \alpha)$ because of the same reason. Then we can easily see that $p_{k,l} = \langle A_{k,l} , \cdot \rangle$. 
\end{proof}

For an arrow diagram $A_{\epsilon_1\epsilon_{2}\cdots \epsilon_{m}}$ with $m$ arrows, where $\epsilon_{i}$ denotes the sign of the $i$-th arrow, consider $w_{k,l}(A_{\epsilon_1\epsilon_{2}\cdots \epsilon_{m}})$ with $k+l = m$. Notice from Table \ref{table.weight_arrow} that the lowest degree of the weight of each arrow is $1$ and that the lowest degree of $d$ is 0. We can see that the multiplication contribution to $w_{k,l}$ must come from the lowest degree. Notice that this contribution from the lowest degree will change to the opposite sign if we change the sign of any one of the $m$ arrows. So we introduce the notion of unsigned arrow diagrams as the algebraic sum of signed arrow diagrams with the same underlying diagram $$A \coloneqq \sum\limits_{\epsilon_1,\epsilon_2,...,\epsilon_{m}}\epsilon_1\epsilon_{2}\cdots \epsilon_{m}A_{\epsilon_1\epsilon_{2}\cdots \epsilon_{m}}$$

\begin{eg}[GDFs from Kauffman polynomial of order 2]
$$A_{2,0} = -4\ \raisebox{-.4\height}{\includegraphics[width=1.75cm]{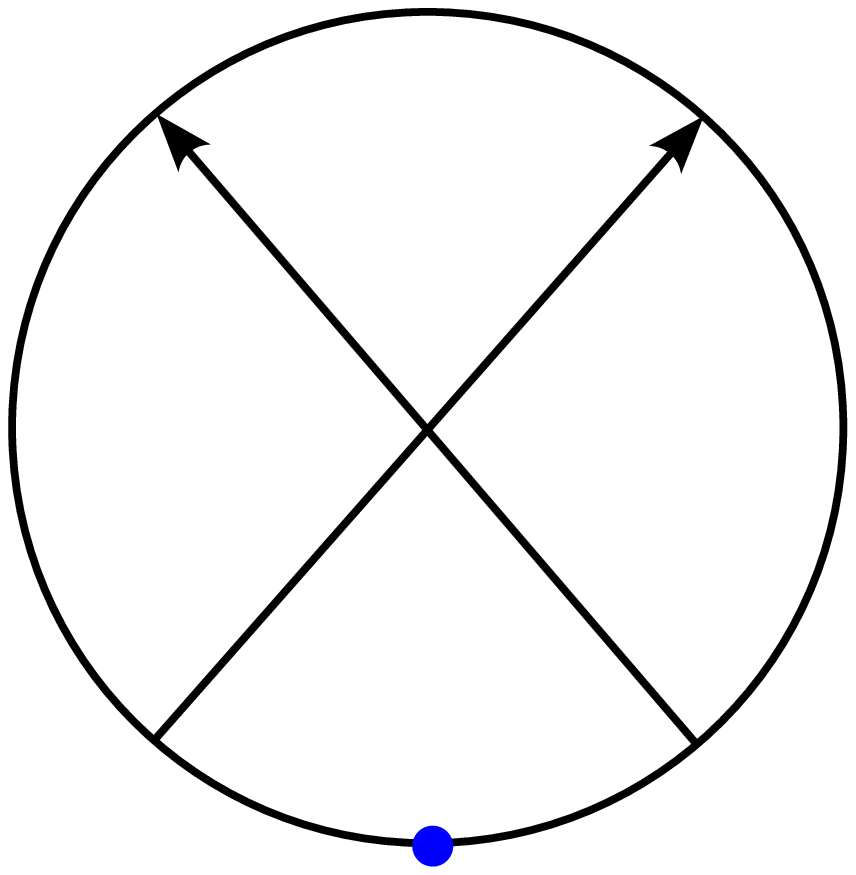}} = -4\ \raisebox{-.4\height}{\includegraphics[width=1.75cm]{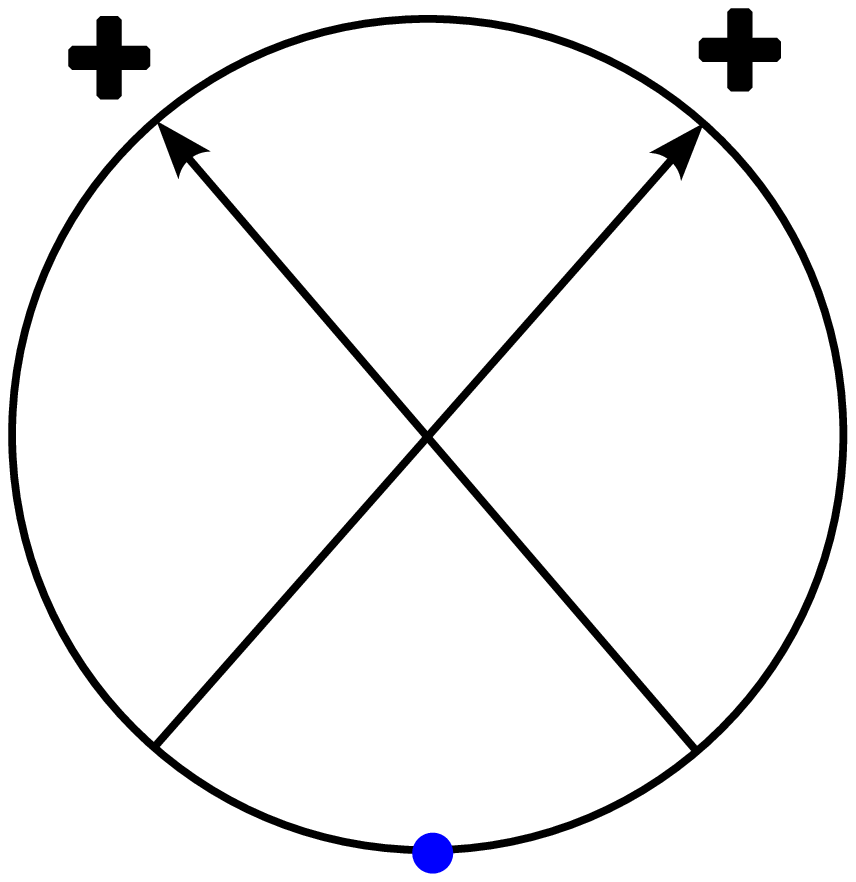}} +4\  \raisebox{-.4\height}{\includegraphics[width=1.75cm]{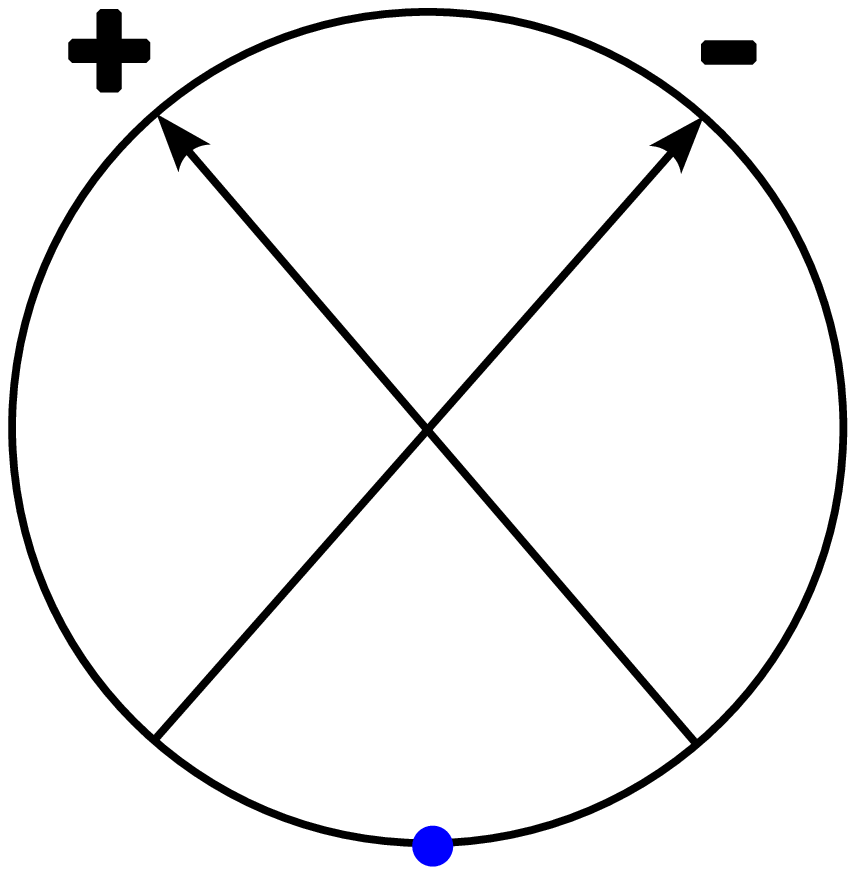}} +4\  \raisebox{-.4\height}{\includegraphics[width=1.75cm]{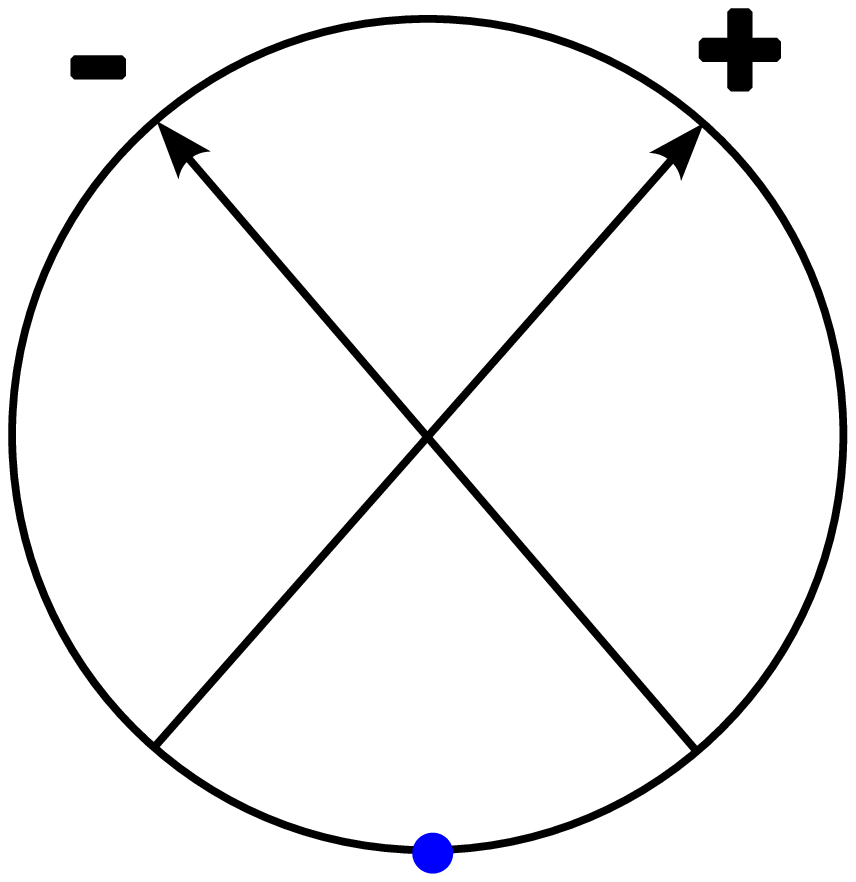}} -4\  \raisebox{-.4\height}{\includegraphics[width=1.75cm]{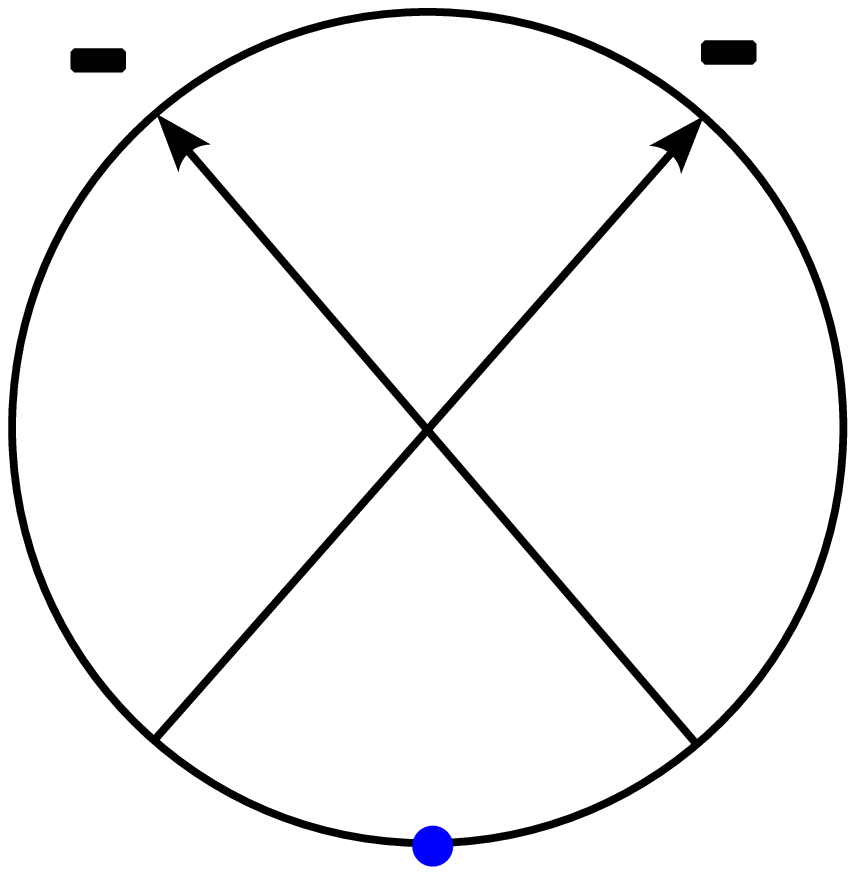}}$$
$$A_{1,1} = 2\ \raisebox{-.4\height}{\includegraphics[width=1.75cm]{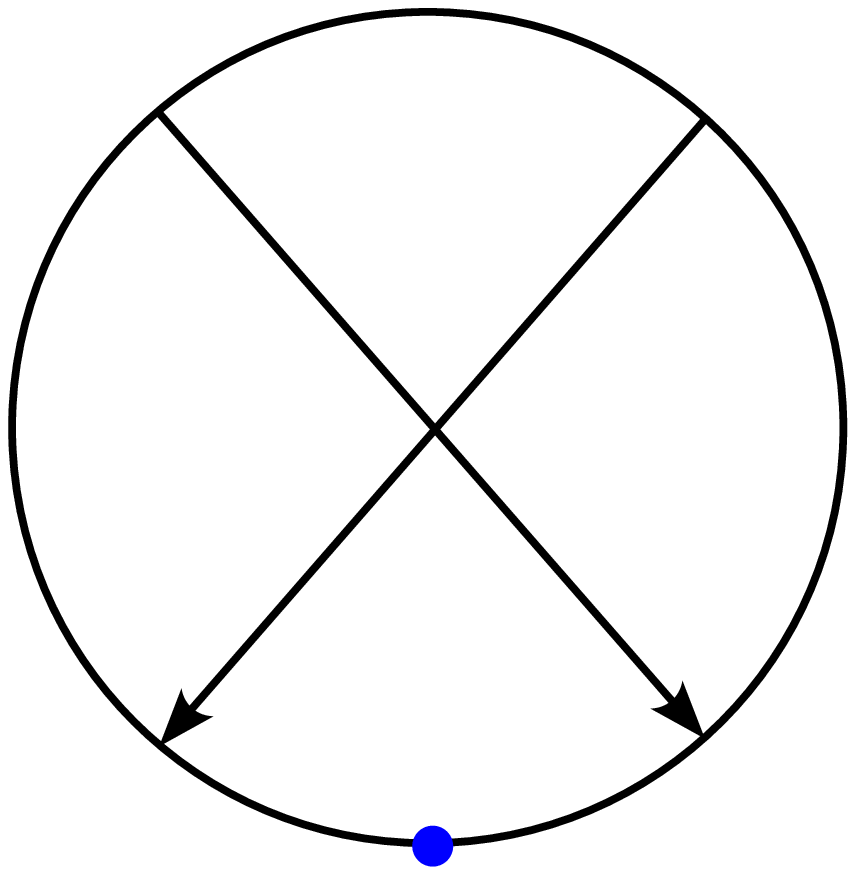}}$$
\end{eg}

\begin{eg}[GDFs from Kauffman polynomial of order 3]
\begin{align*}
    \dfrac{1}{2}A_{1,2} &= \raisebox{-.4\height}{\includegraphics[width=1.75cm]{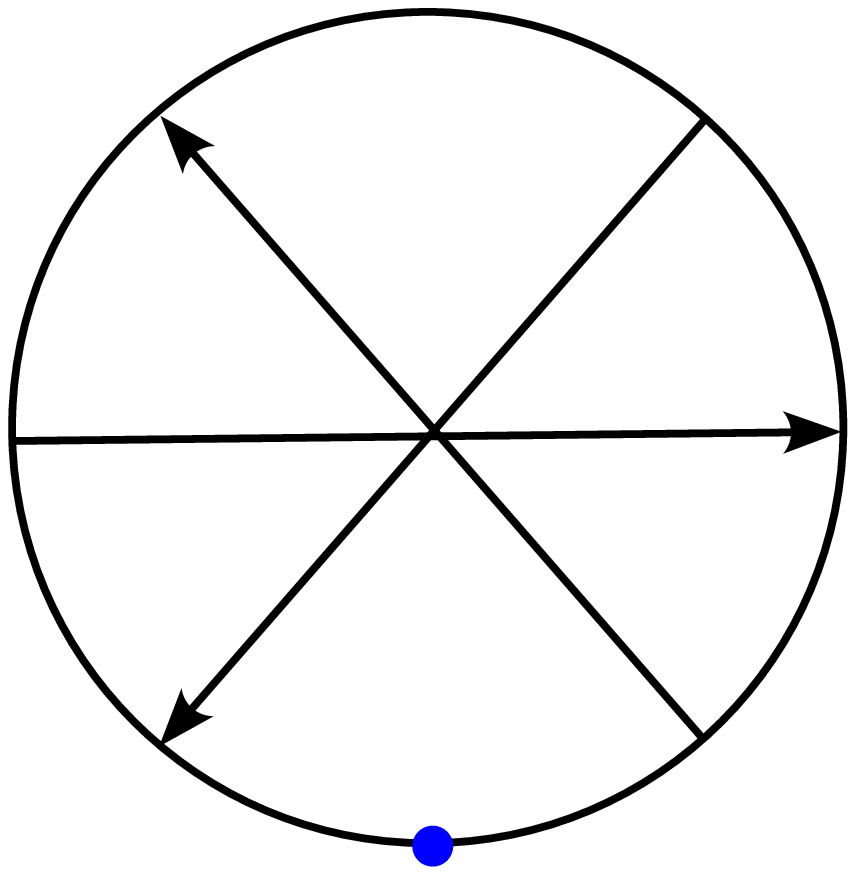}} + \raisebox{-.4\height}{\includegraphics[width=1.75cm]{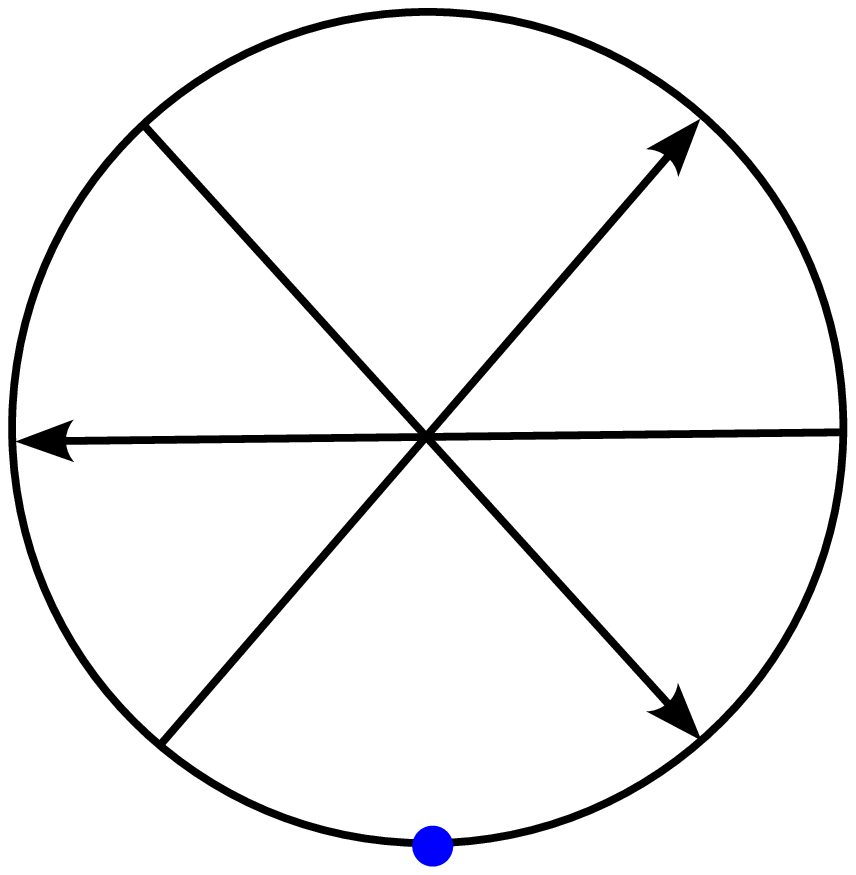}} - \raisebox{-.4\height}{\includegraphics[width=1.75cm]{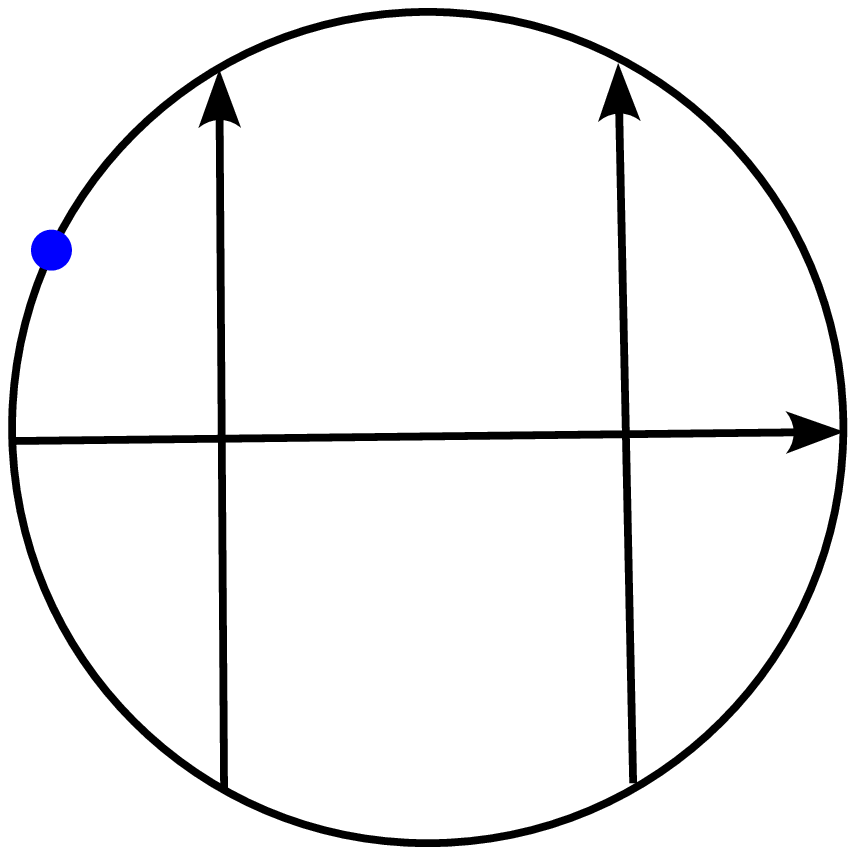}} + \raisebox{-.4\height}{\includegraphics[width=1.75cm]{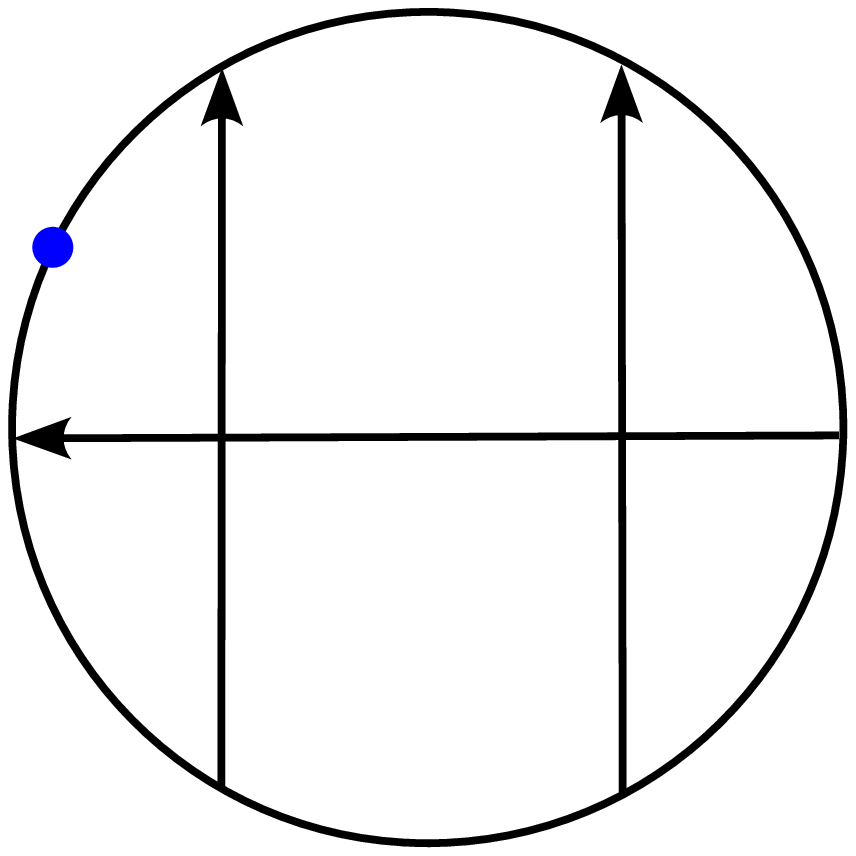}} + \raisebox{-.4\height}{\includegraphics[width=1.75cm]{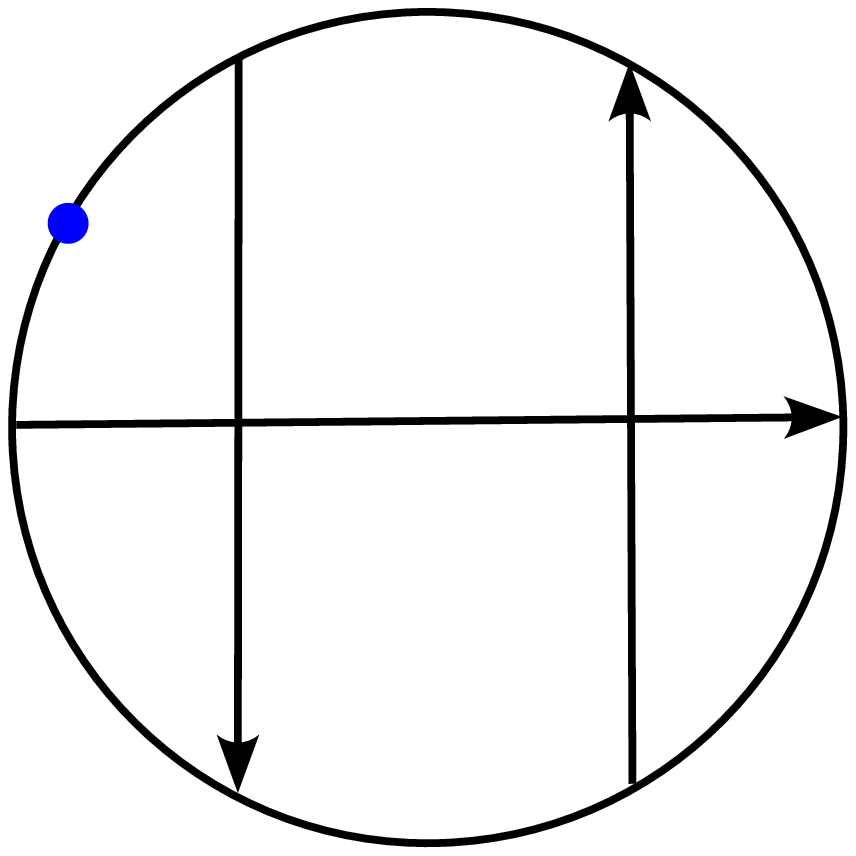}}\\
    &- \raisebox{-.4\height}{\includegraphics[width=1.75cm]{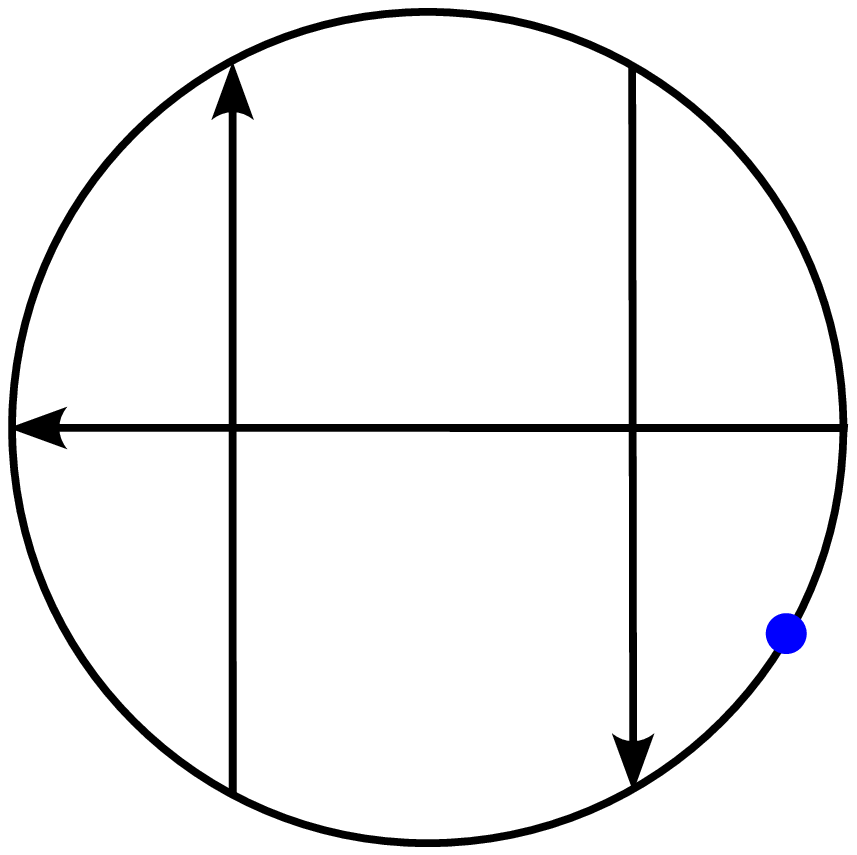}} + \raisebox{-.4\height}{\includegraphics[width=1.75cm]{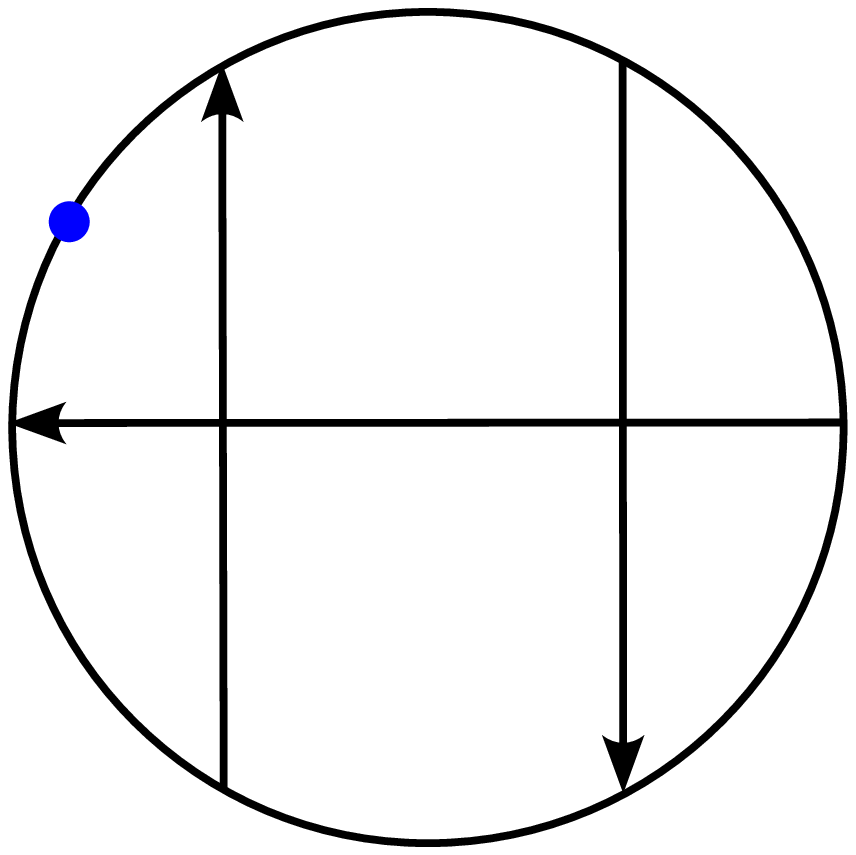}} + \raisebox{-.4\height}{\includegraphics[width=1.75cm]{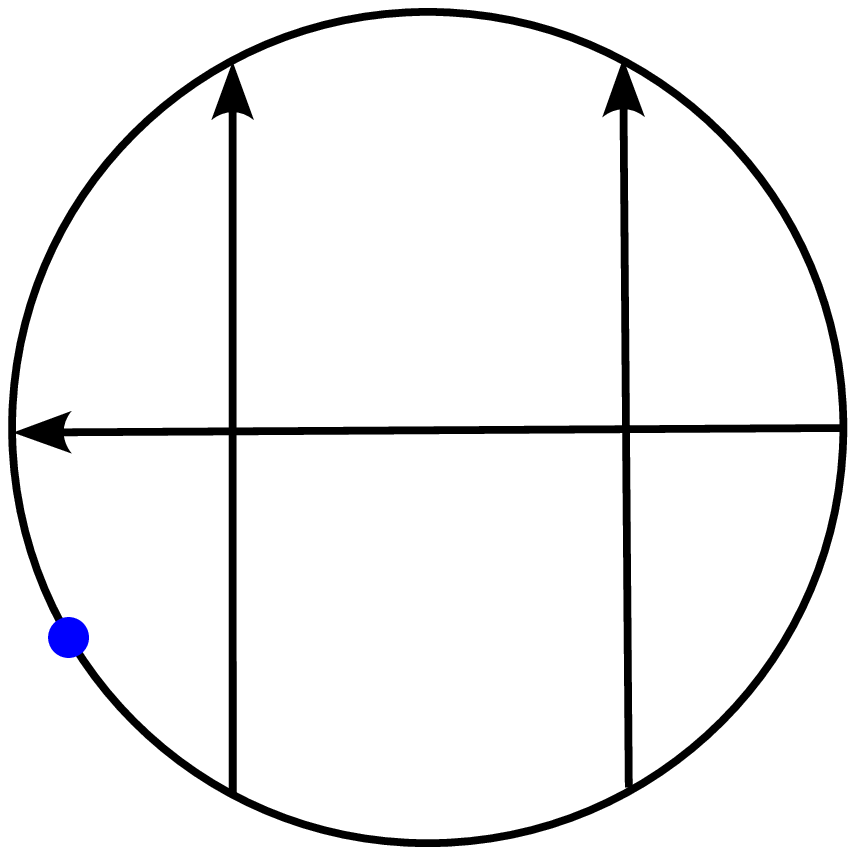}} + \raisebox{-.4\height}{\includegraphics[width=1.75cm]{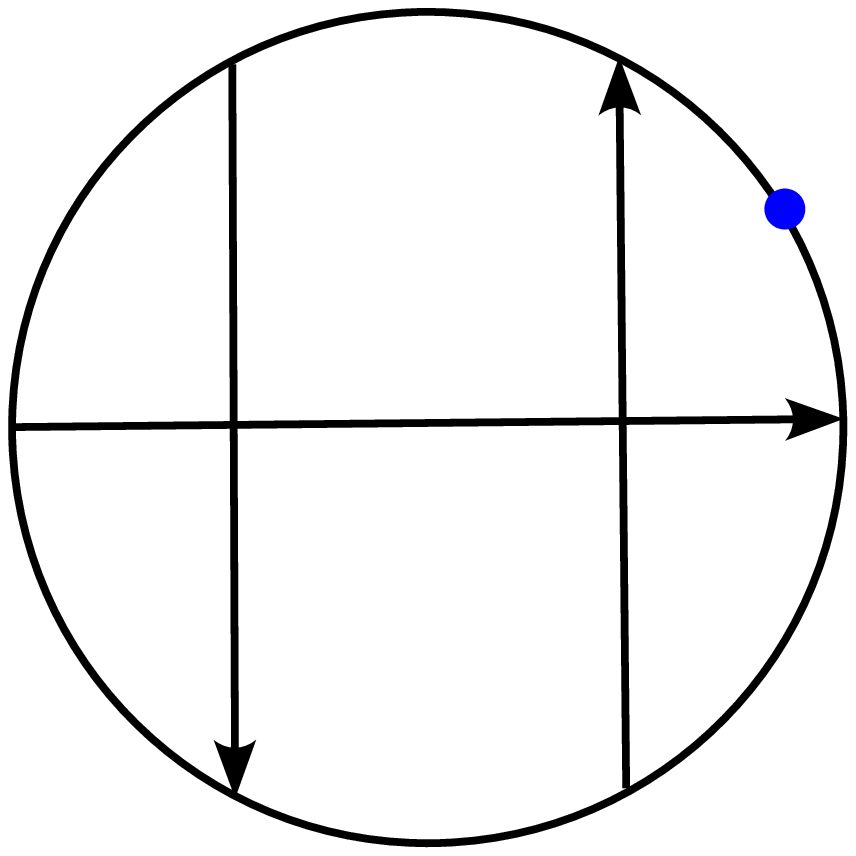}} - \raisebox{-.4\height}{\includegraphics[width=1.75cm]{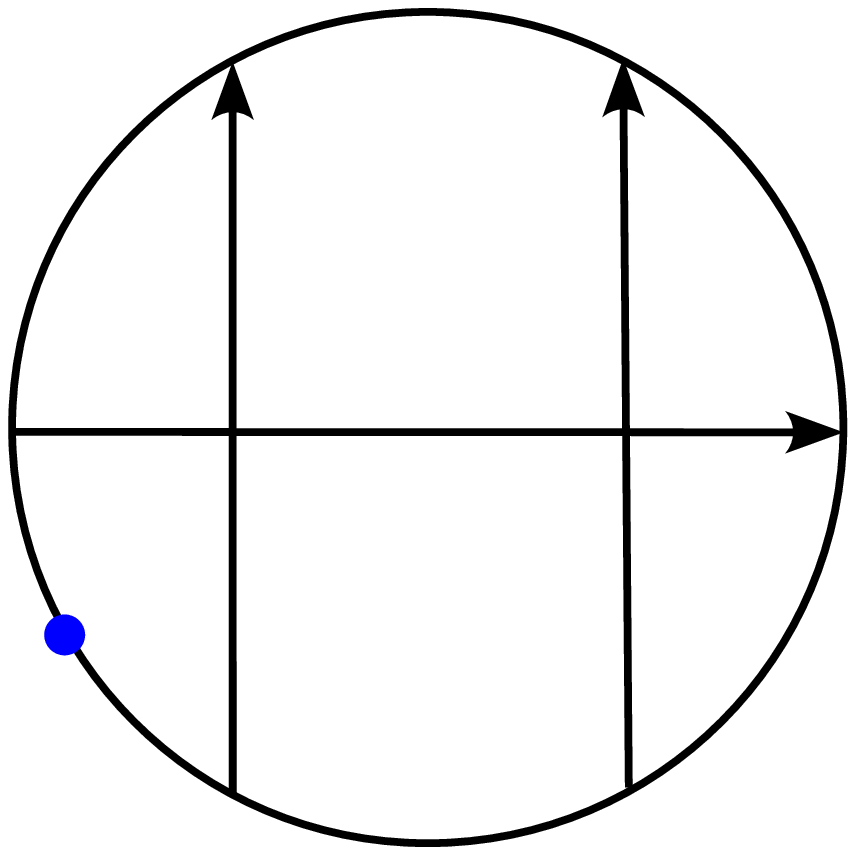}} \\
    &+ \raisebox{-.4\height}{\includegraphics[width=1.75cm]{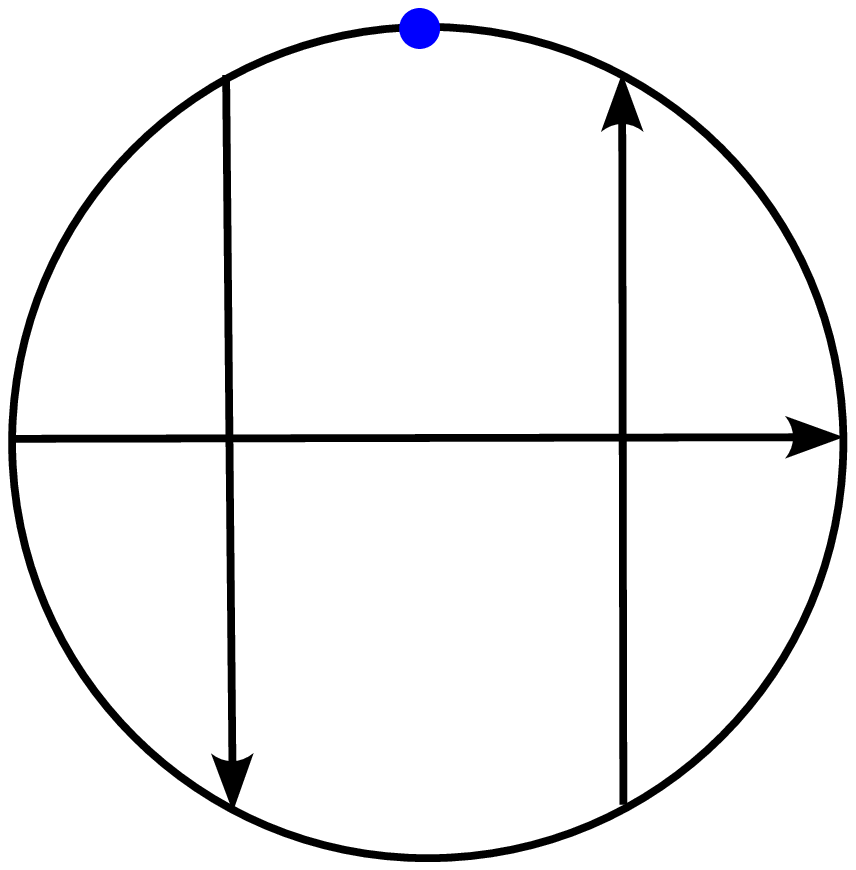}} - \raisebox{-.4\height}{\includegraphics[width=1.75cm]{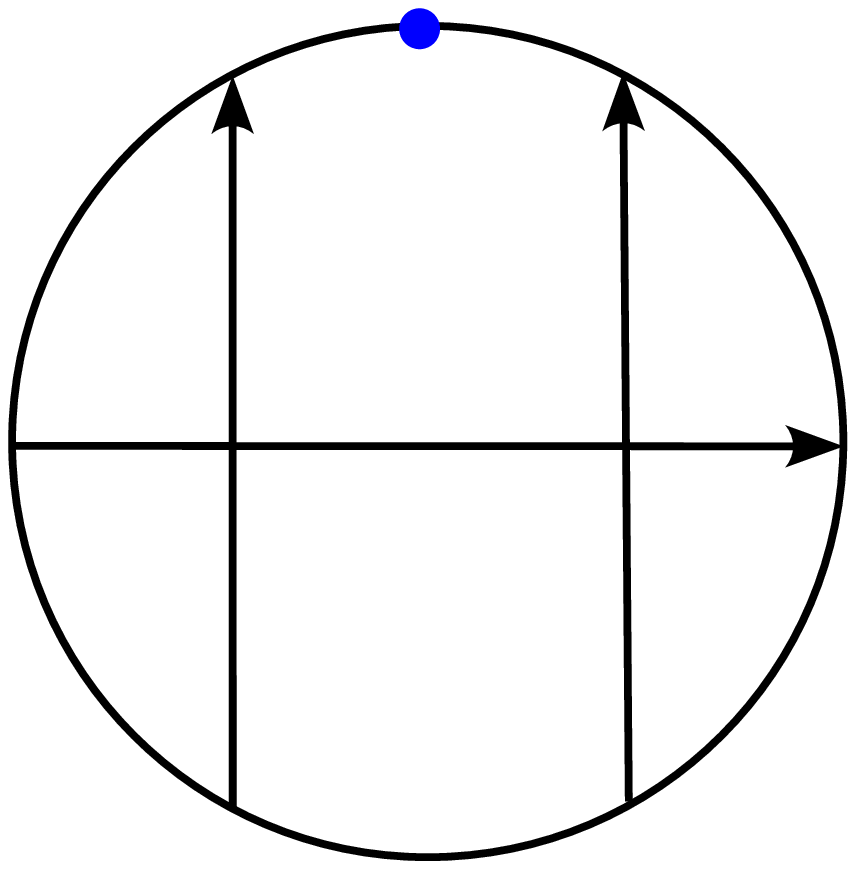}} + \raisebox{-.4\height}{\includegraphics[width=1.75cm]{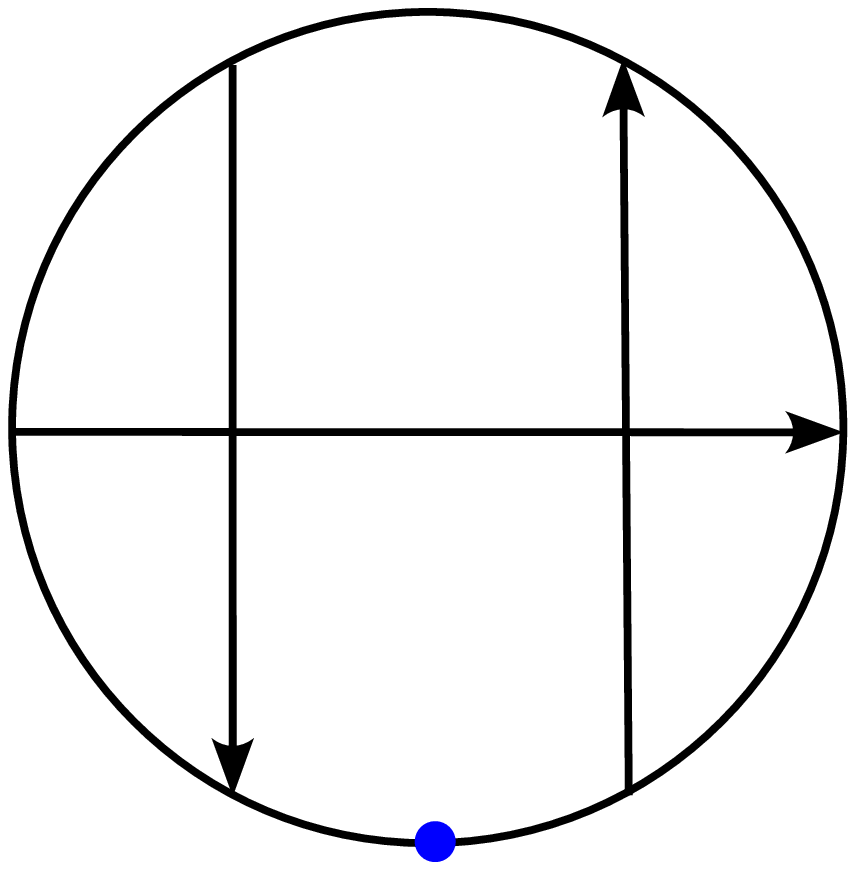}} - \raisebox{-.4\height}{\includegraphics[width=1.75cm]{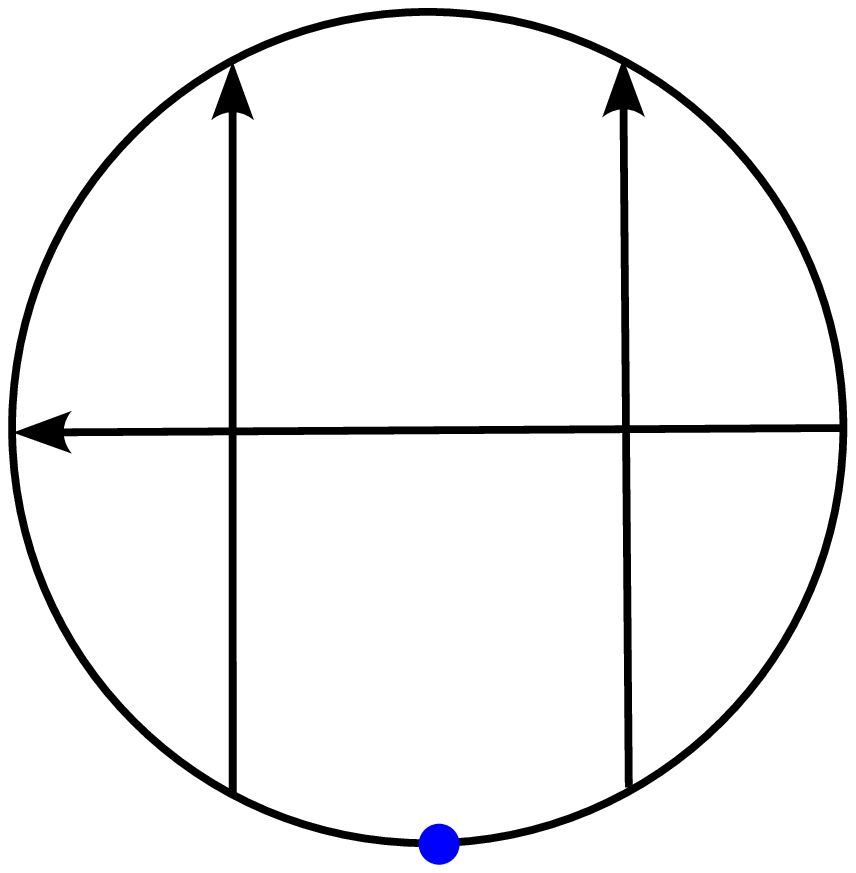}} + \raisebox{-.4\height}{\includegraphics[width=1.75cm]{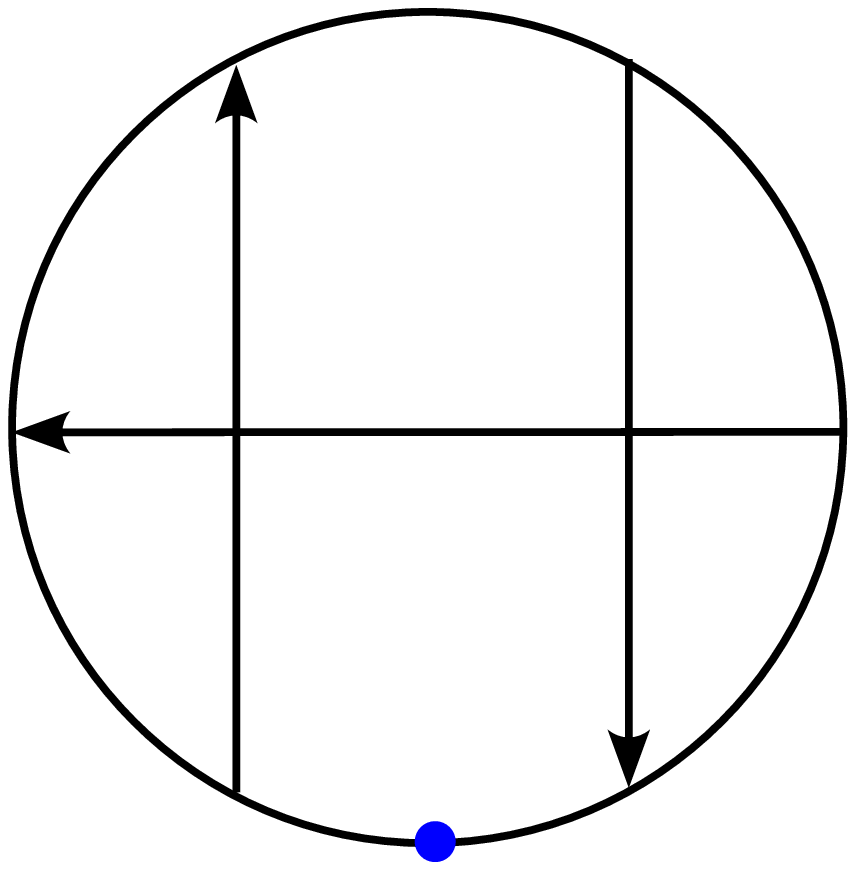}}
\end{align*}
\begin{align*}
    \dfrac{1}{8}A_{3,0} &= - \raisebox{-.4\height}{\includegraphics[width=1.75cm]{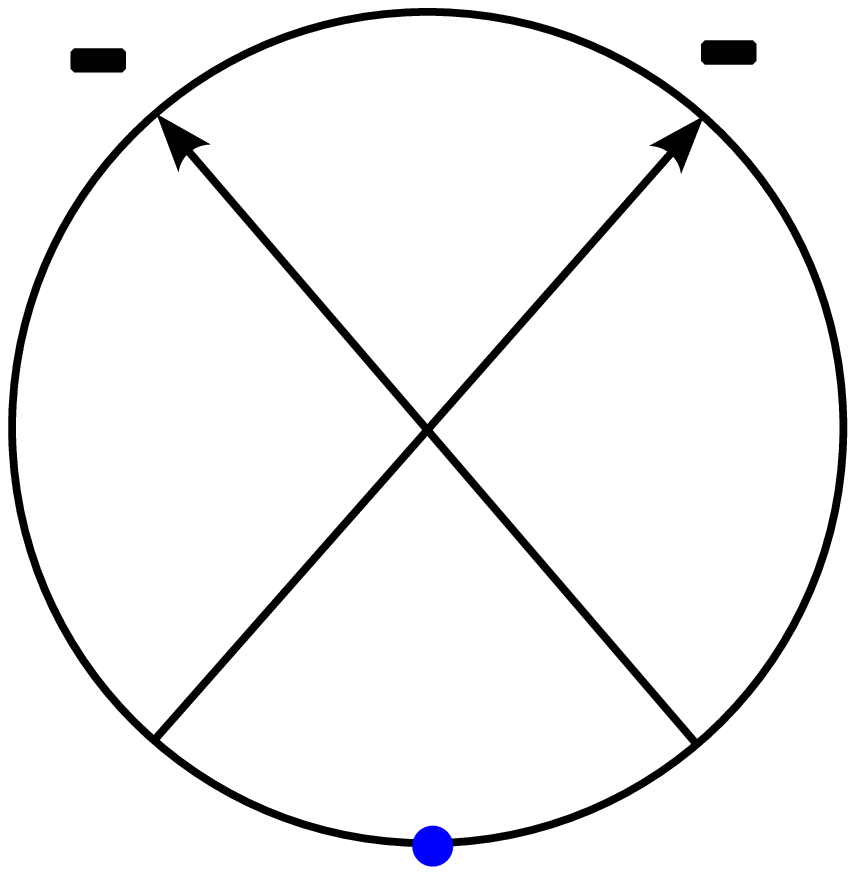}} + \raisebox{-.4\height}{\includegraphics[width=1.75cm]{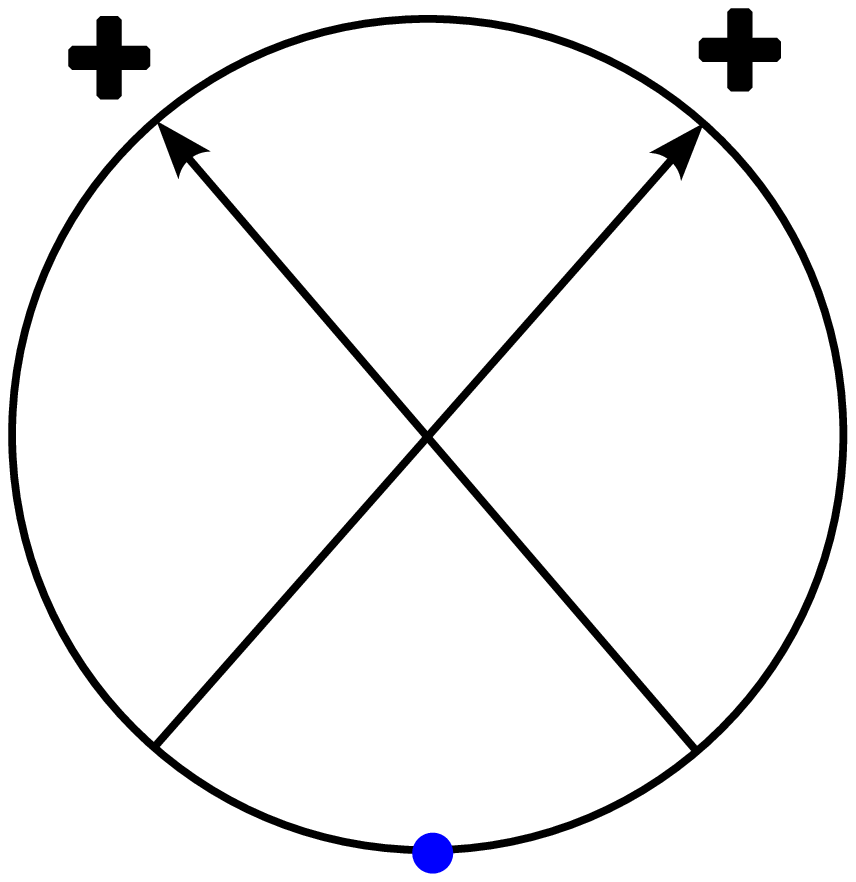}} + \raisebox{-.4\height}{\includegraphics[width=1.75cm]{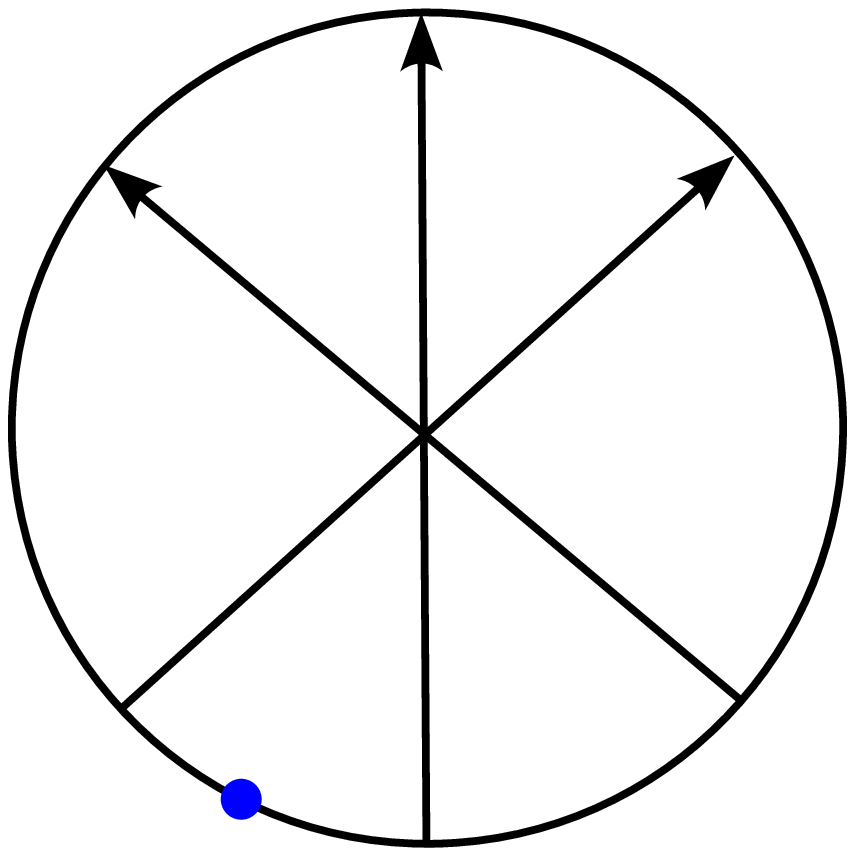}} + \raisebox{-.4\height}{\includegraphics[width=1.75cm]{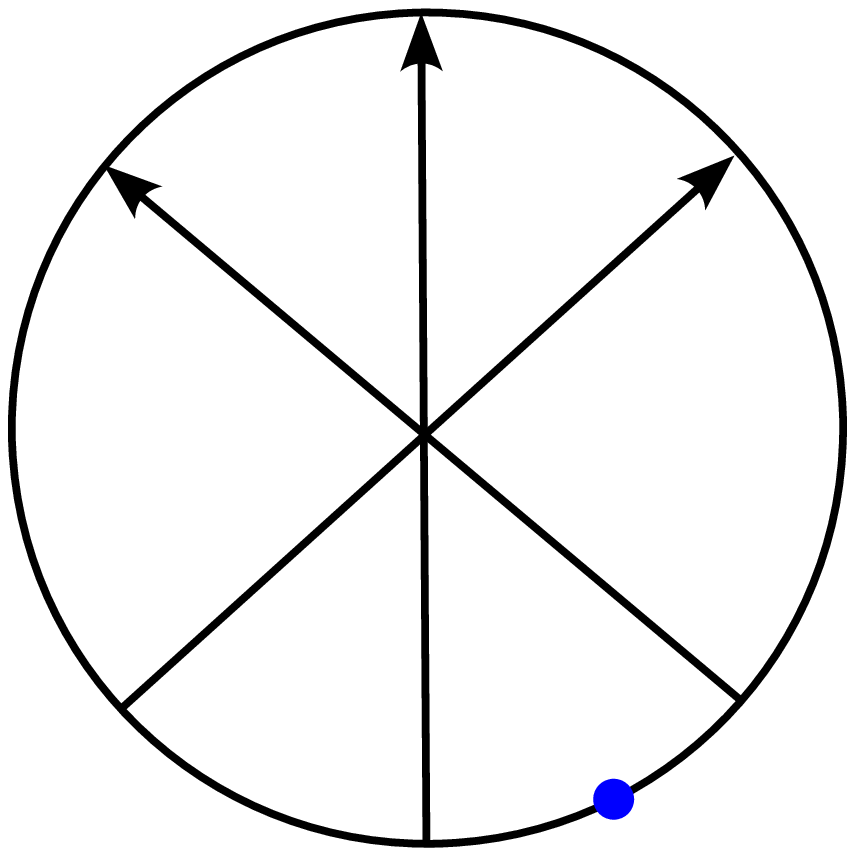}} + \raisebox{-.4\height}{\includegraphics[width=1.75cm]{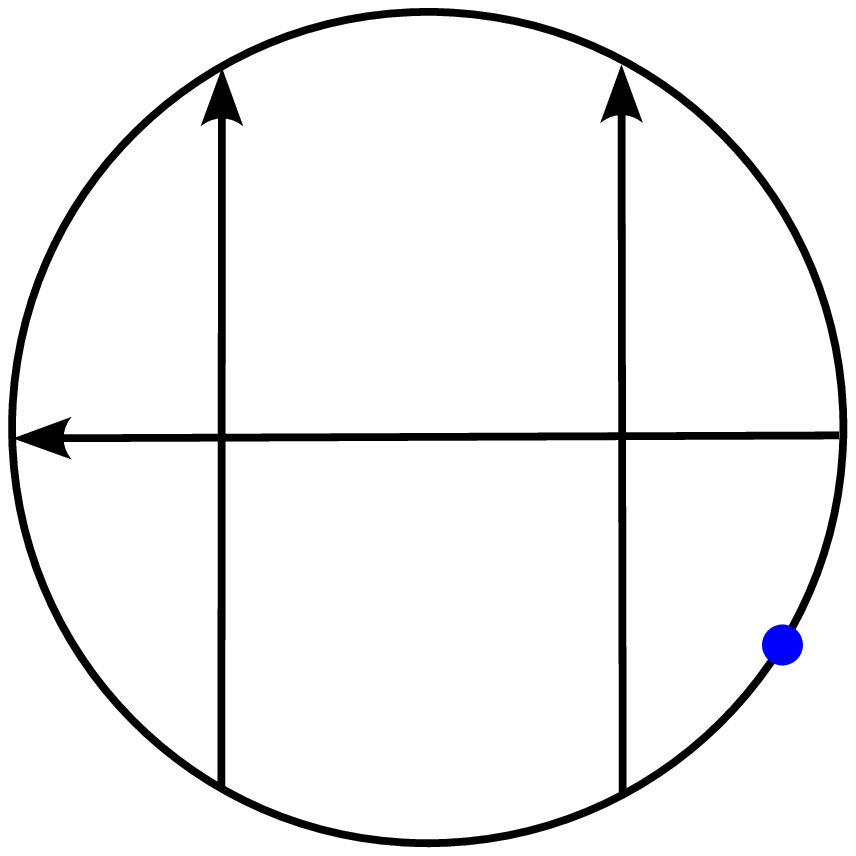}} \\
    &+ \raisebox{-.4\height}{\includegraphics[width=1.75cm]{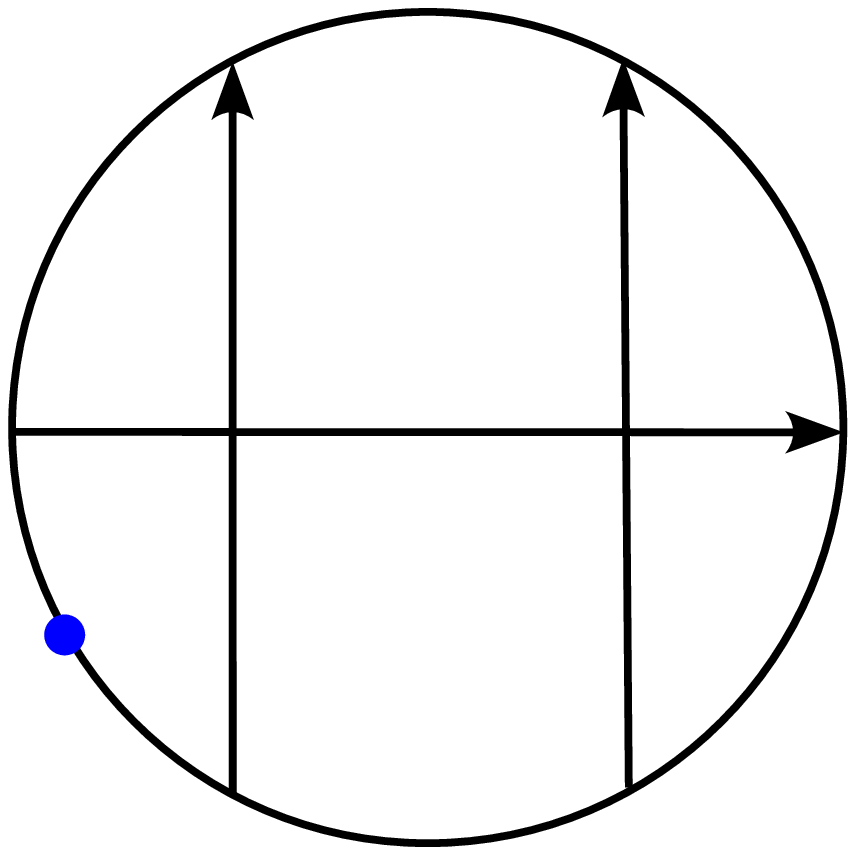}} - \raisebox{-.4\height}{\includegraphics[width=1.75cm]{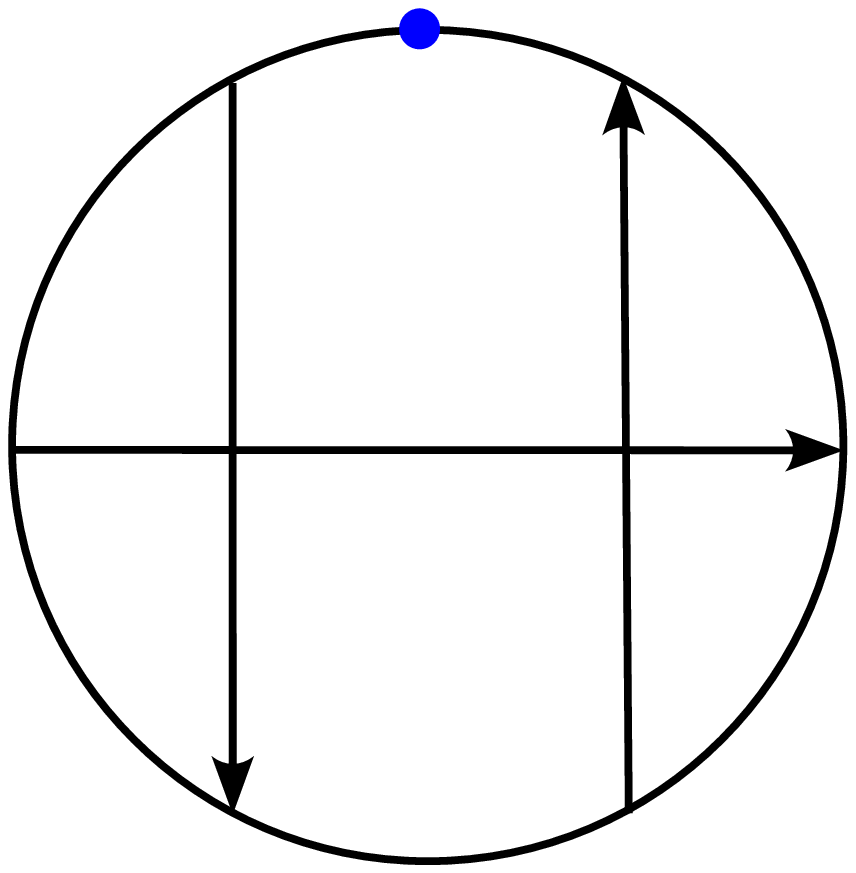}} + \raisebox{-.4\height}{\includegraphics[width=1.75cm]{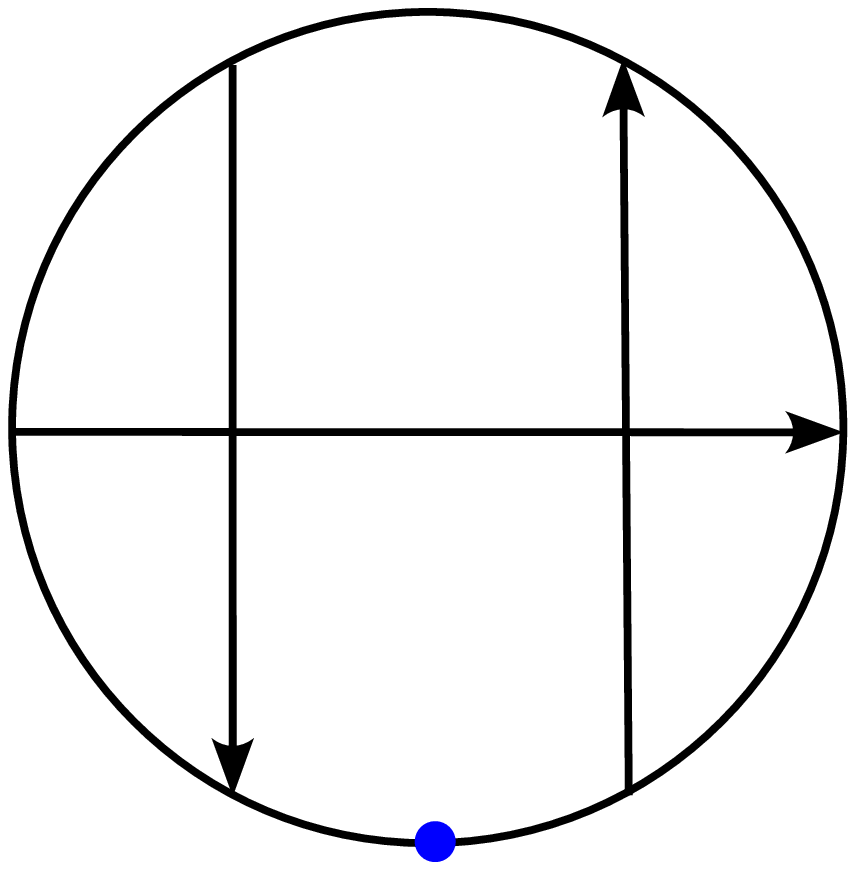}} + \raisebox{-.4\height}{\includegraphics[width=1.75cm]{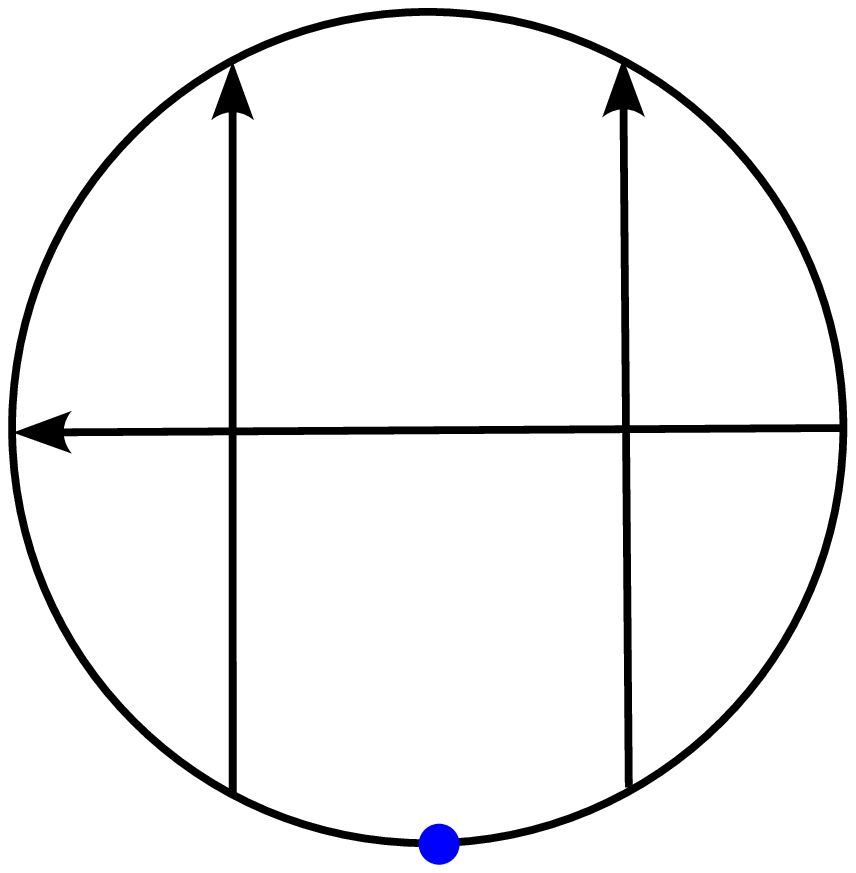}}
\end{align*}

\begin{align*}
    -\dfrac{1}{4}A_{2,1} &= -\raisebox{-.4\height}{\includegraphics[width=1.75cm]{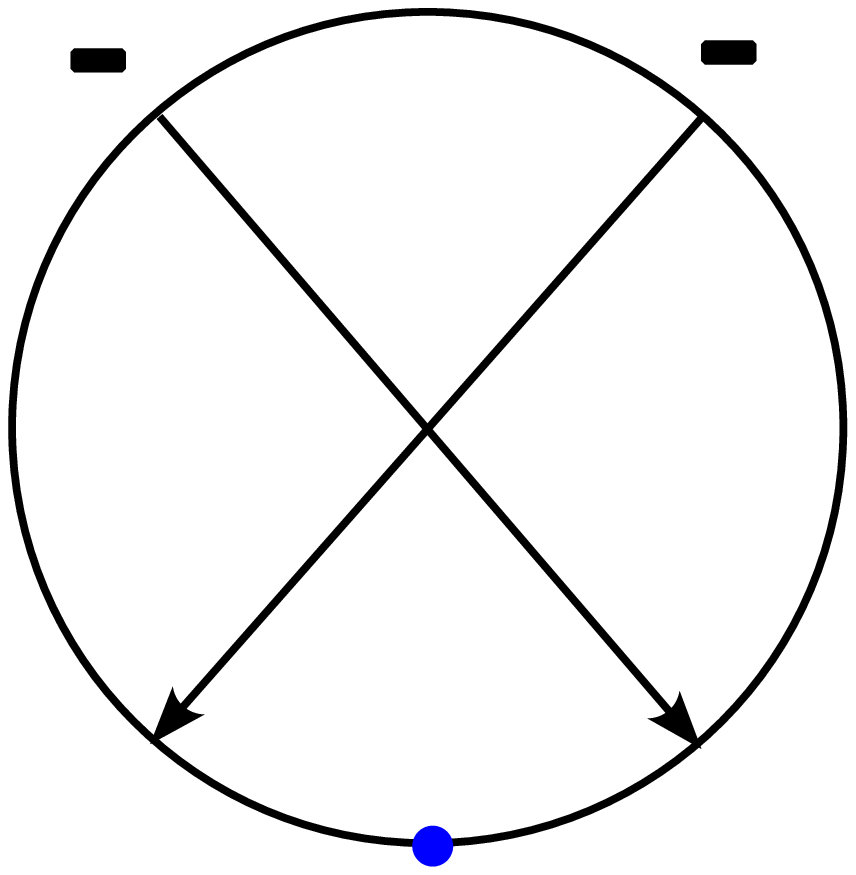}} + \raisebox{-.4\height}{\includegraphics[width=1.75cm]{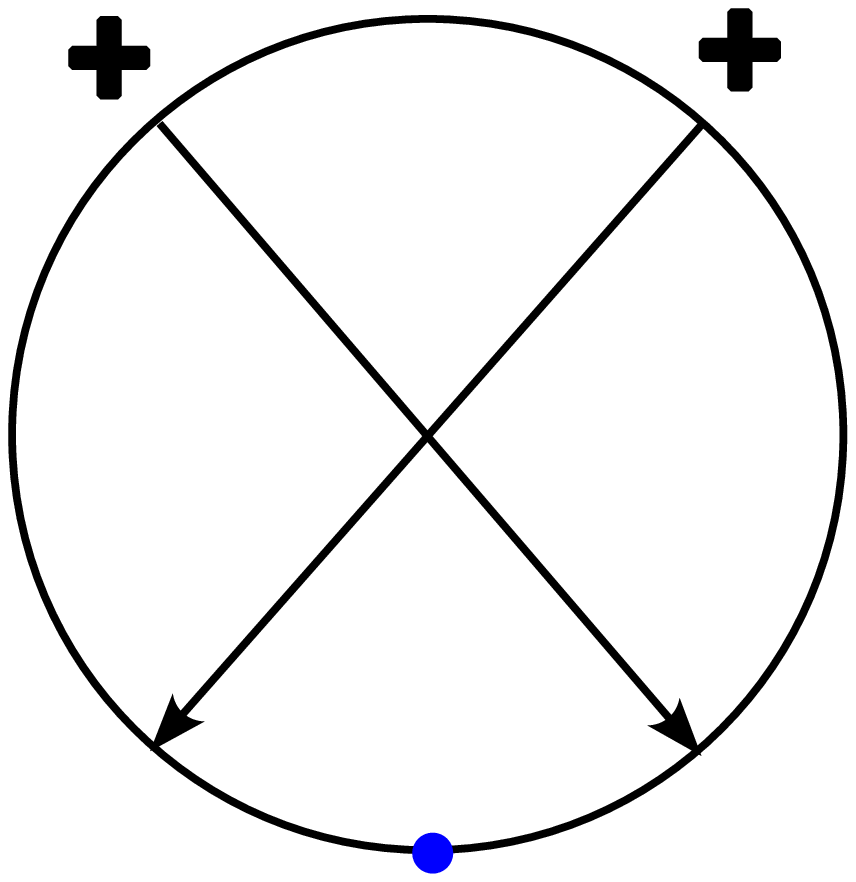}} + \raisebox{-.4\height}{\includegraphics[width=1.75cm]{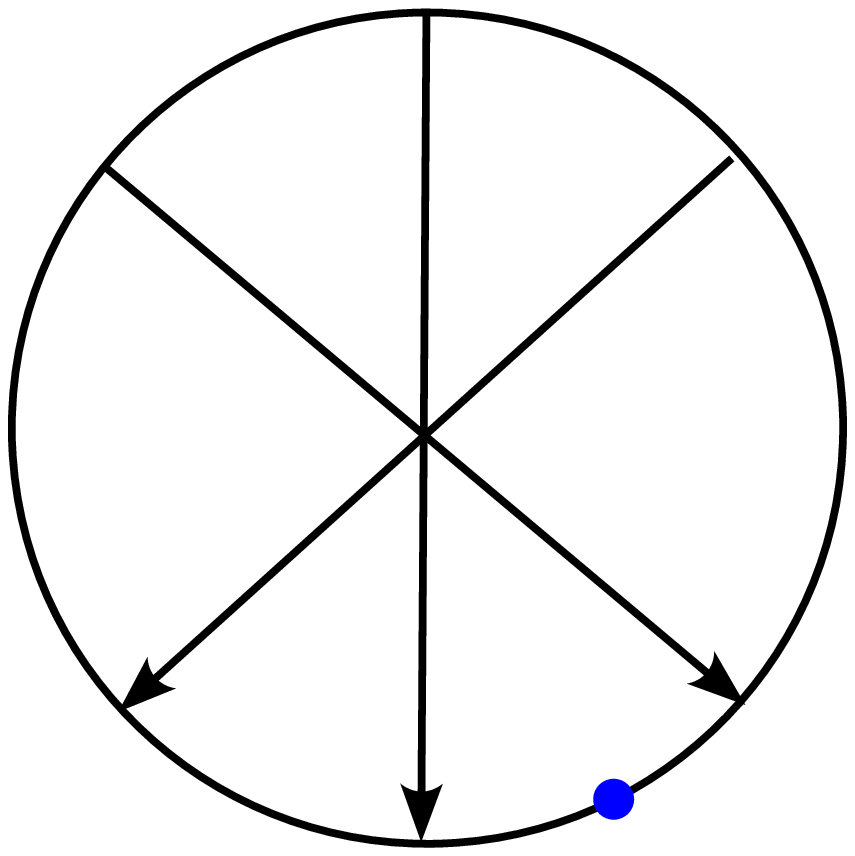}} + \raisebox{-.4\height}{\includegraphics[width=1.75cm]{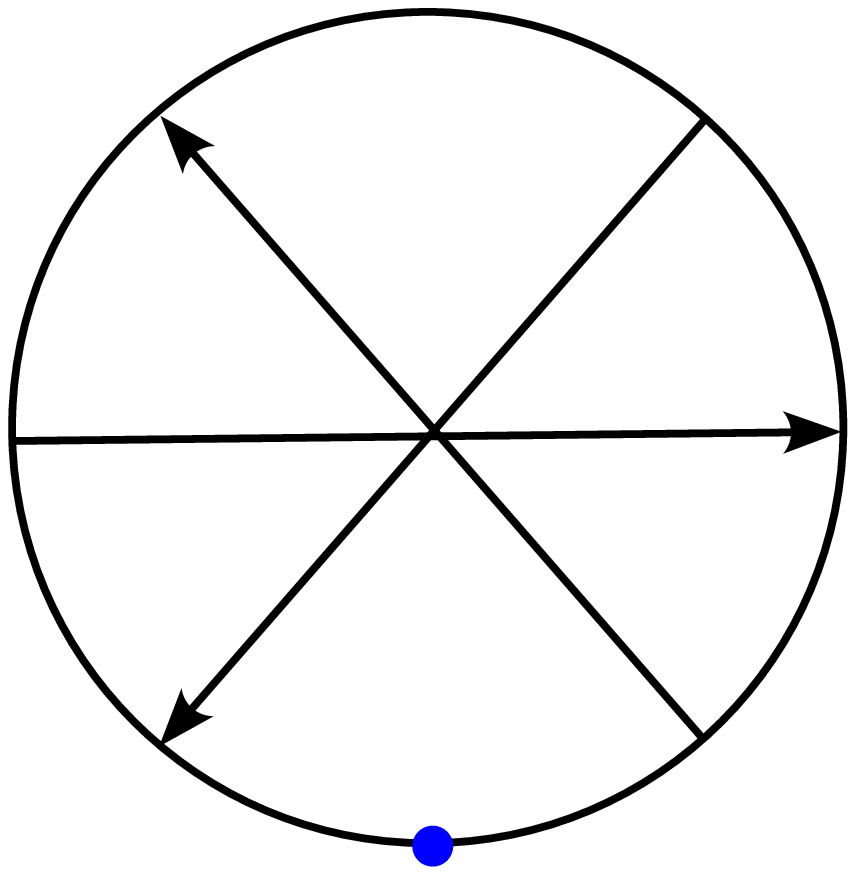}} + \raisebox{-.4\height}{\includegraphics[width=1.75cm]{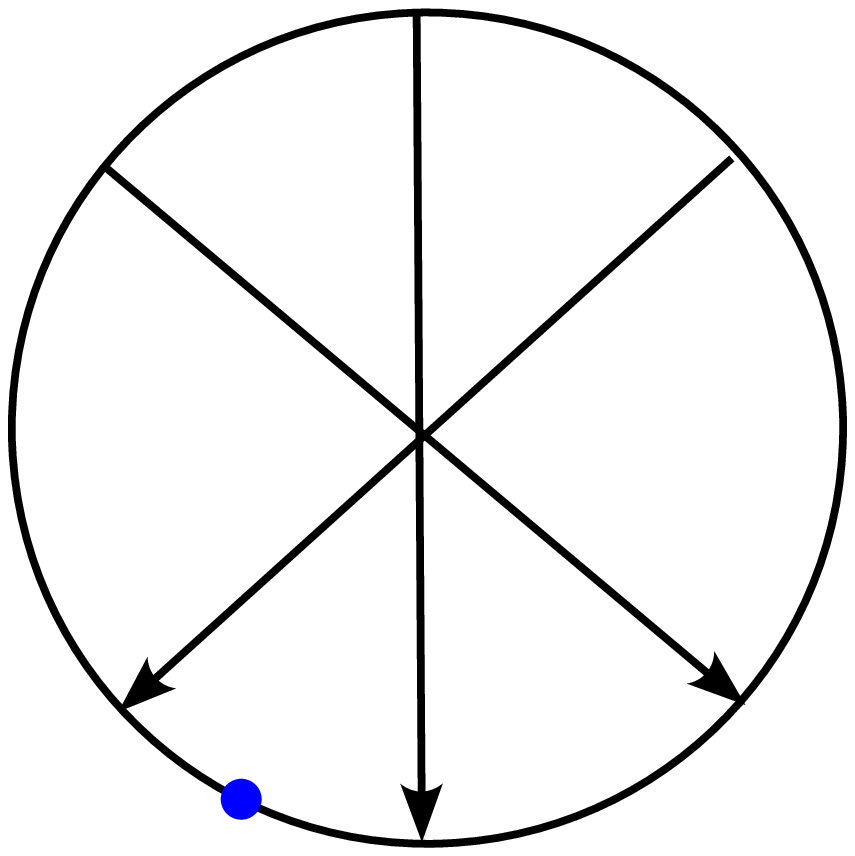}} \\
    &+ \raisebox{-.4\height}{\includegraphics[width=1.75cm]{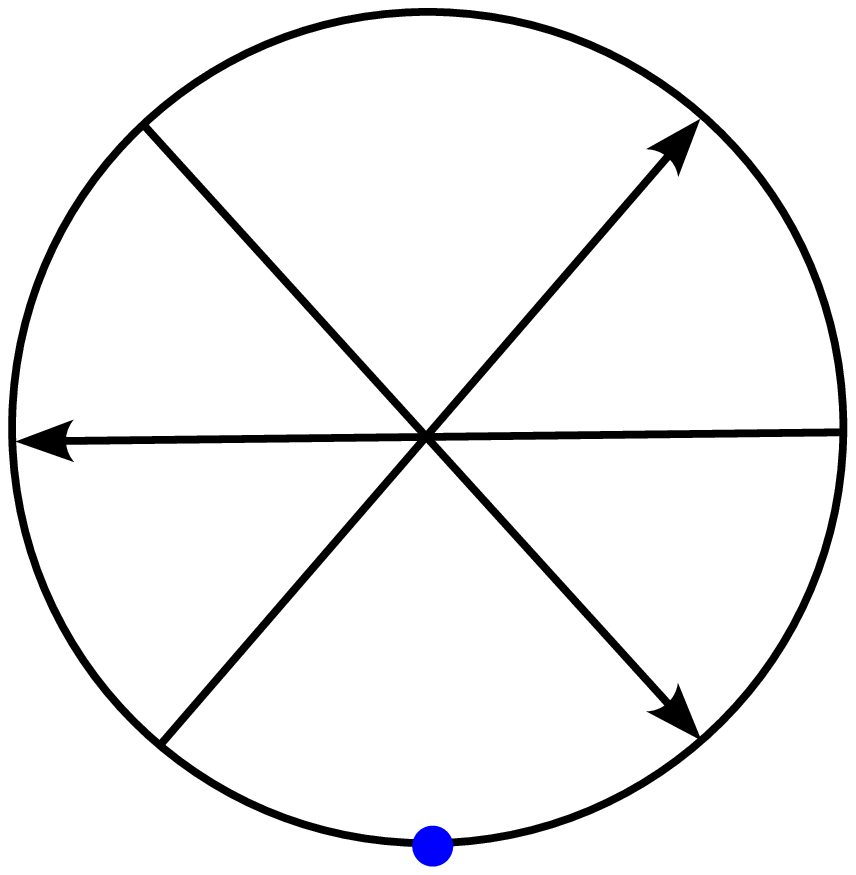}} + \raisebox{-.4\height}{\includegraphics[width=1.75cm]{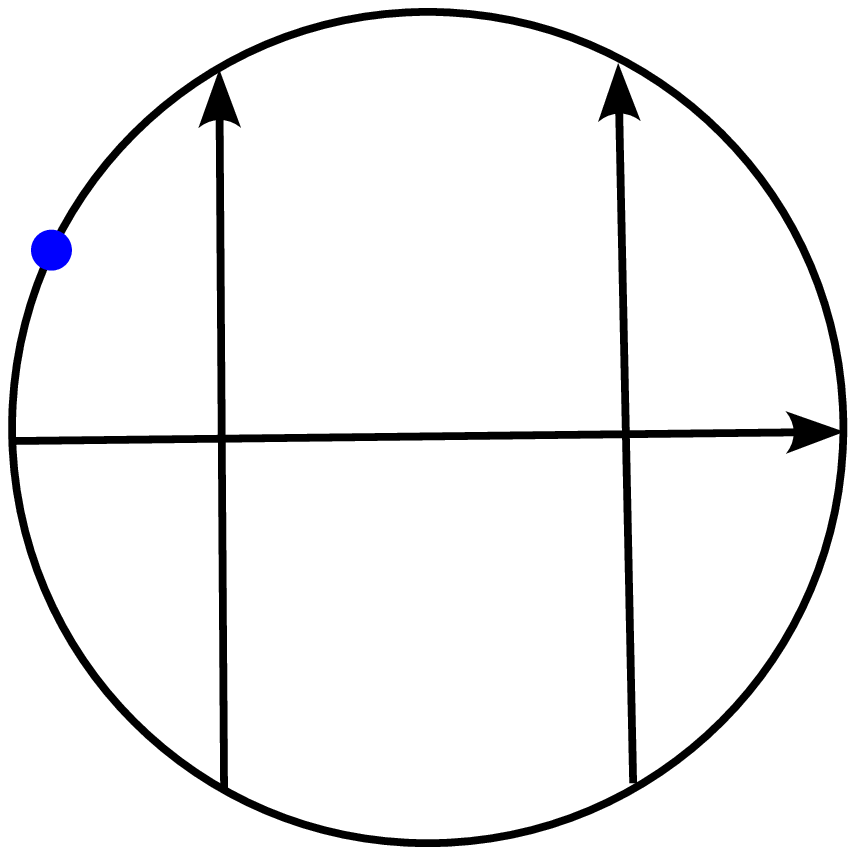}} + \raisebox{-.4\height}{\includegraphics[width=1.75cm]{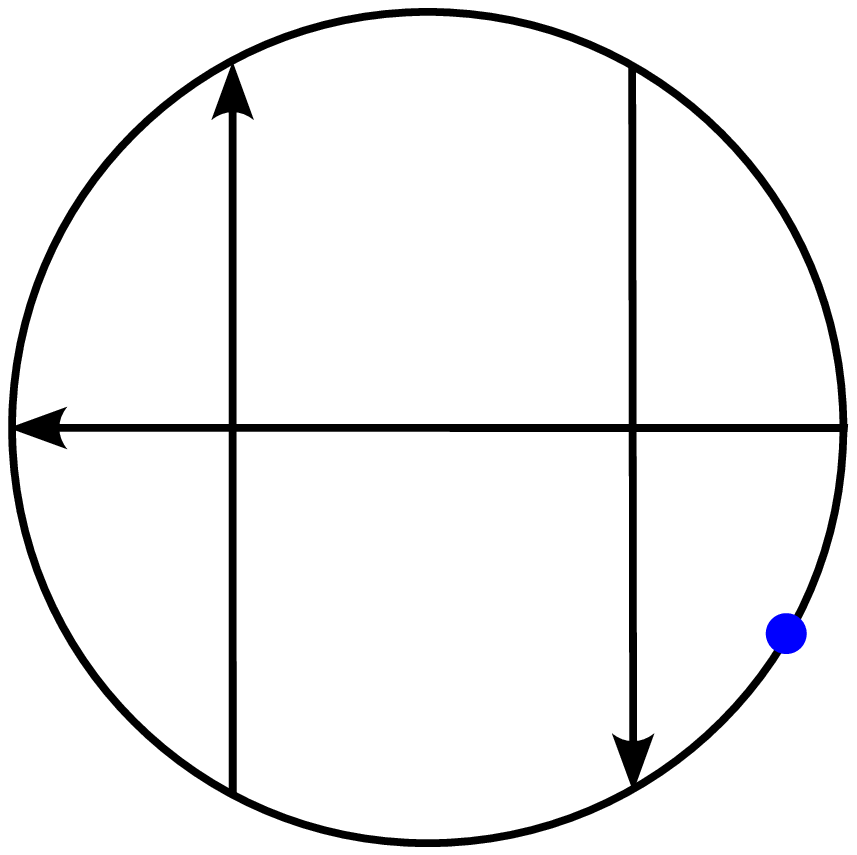}} - \raisebox{-.4\height}{\includegraphics[width=1.75cm]{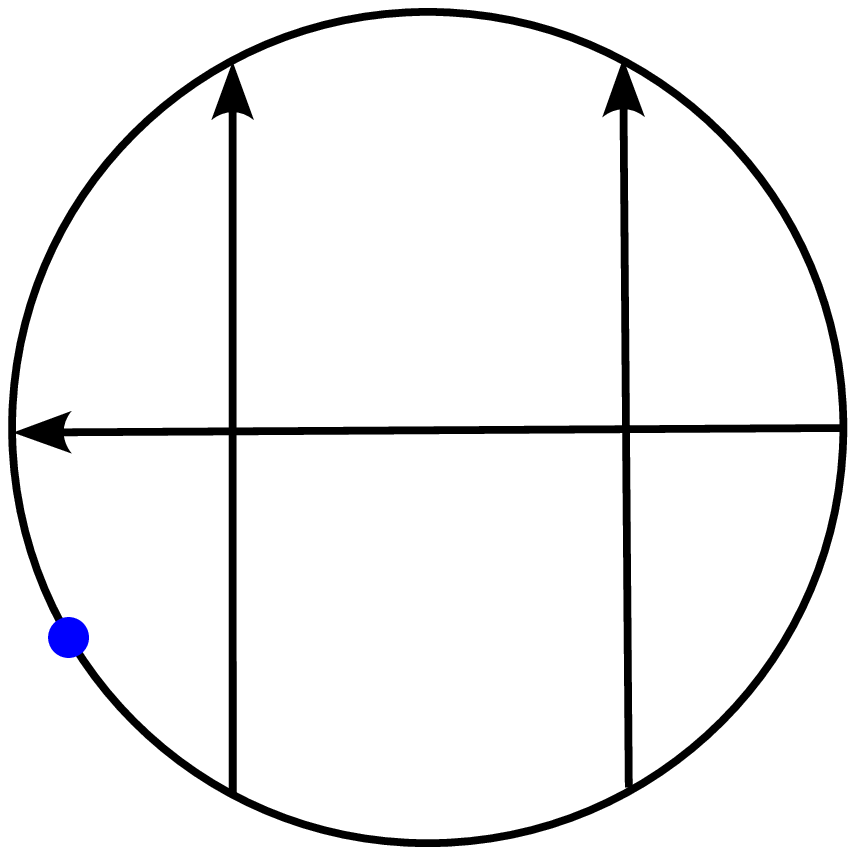}} + \raisebox{-.4\height}{\includegraphics[width=1.75cm]{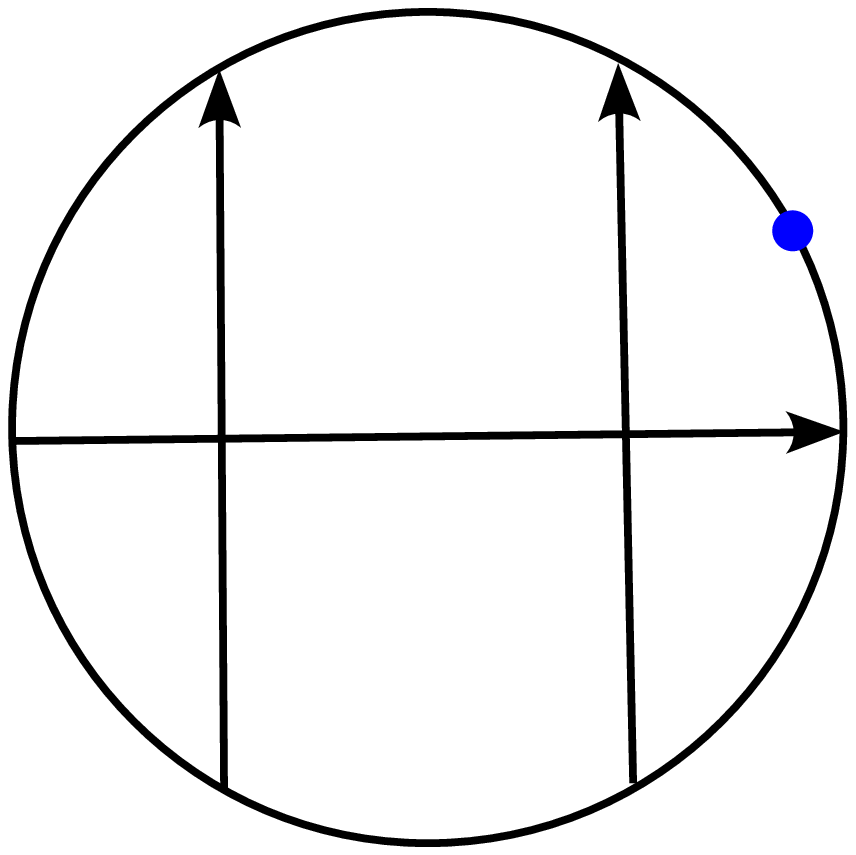}}\\ 
    &+ \raisebox{-.4\height}{\includegraphics[width=1.75cm]{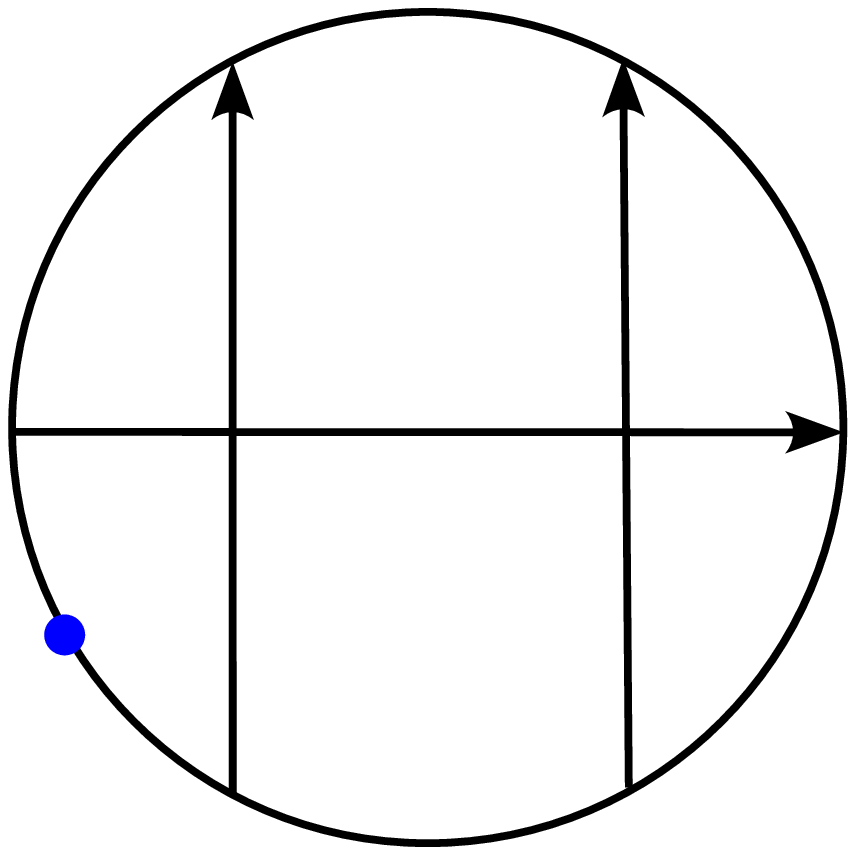}} + \raisebox{-.4\height}{\includegraphics[width=1.75cm]{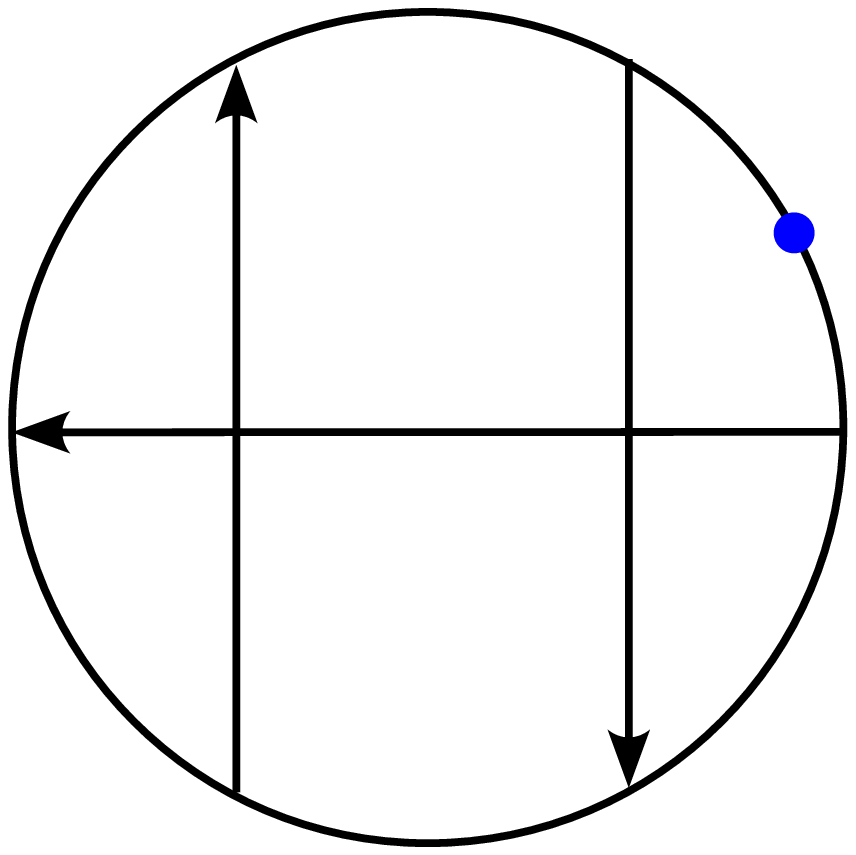}} + \raisebox{-.4\height}{\includegraphics[width=1.75cm]{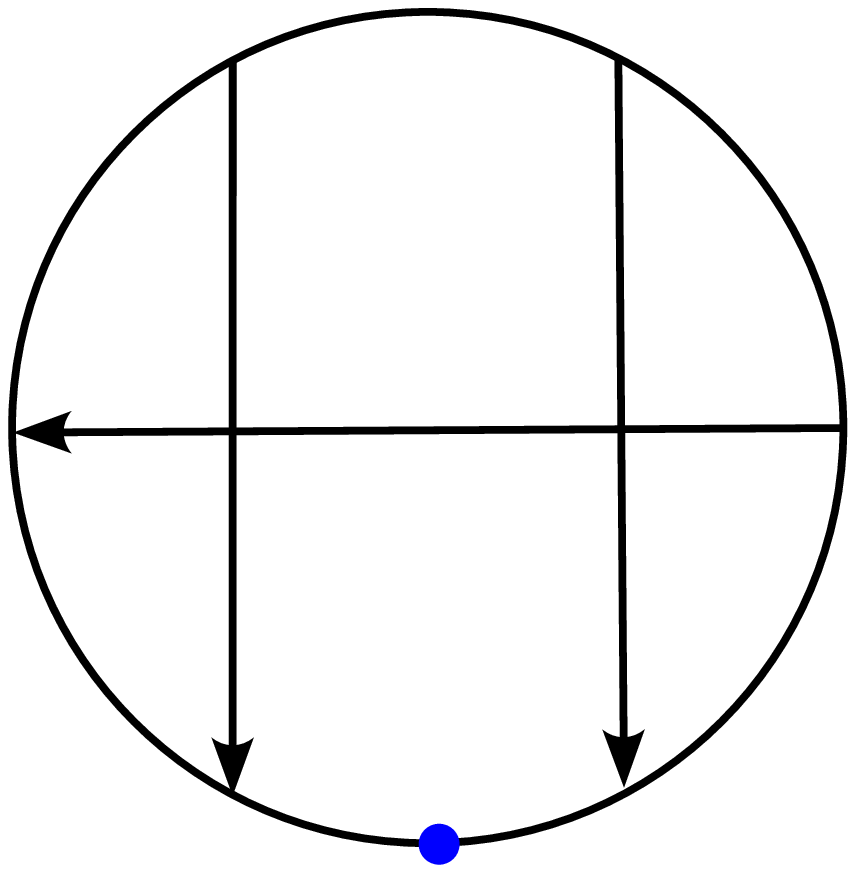}} + \raisebox{-.4\height}{\includegraphics[width=1.75cm]{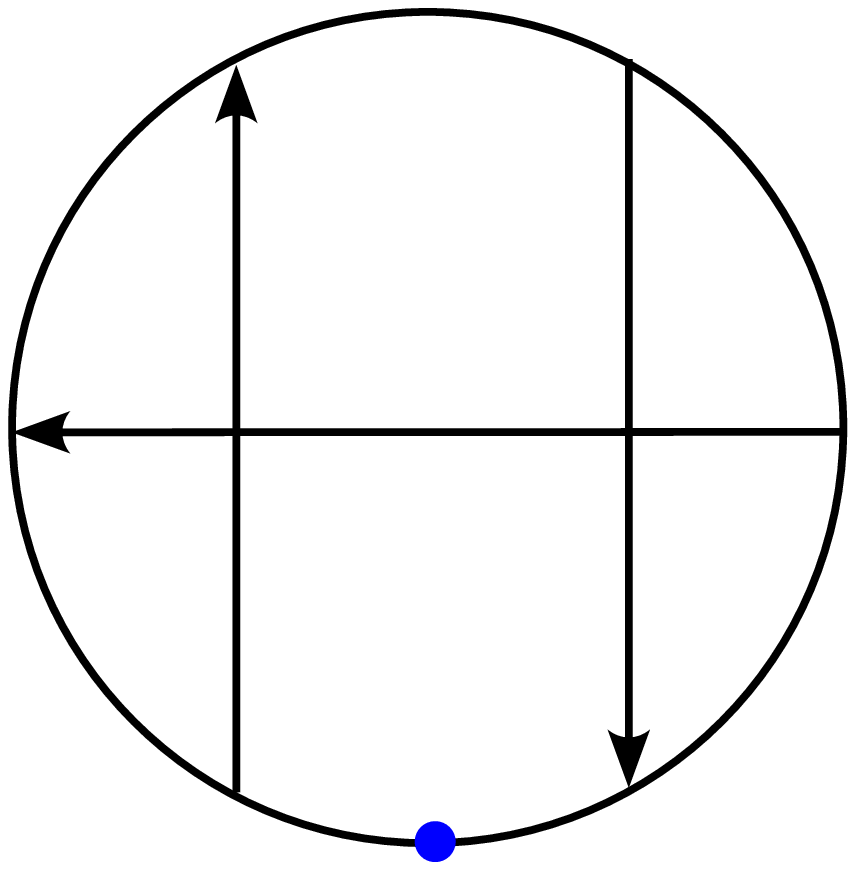}}
\end{align*}
\end{eg}

\begin{prop}
    We have $A_{0,n} = 0$ and $p_{0,n}(\cdot)=0$ for $n \ge 1$.
\end{prop}

This can be seen as a corollary of Proposition \ref{prop.evaluation} or be proved using the same method.
%%%%%%%%%%%%%%%%%%%%%%%%%%%%%
%%%%%%%%%%%%%%%%%%%%%%%%%%%%%
\section{Identities and simplification}\label{sec.identities}
Östlund gave a proof of several Gauss diagram identities in \cite{Ostlund:2004} using the method of counting linking number  in two ways, namely \raisebox{-.4\height}{\includegraphics[width=1.5cm]{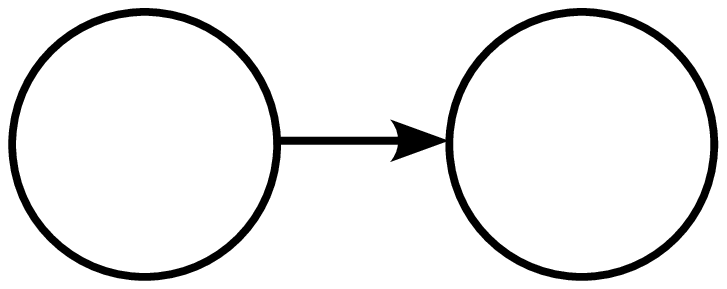}} $=$ \raisebox{-.4\height}{\includegraphics[width=1.5cm]{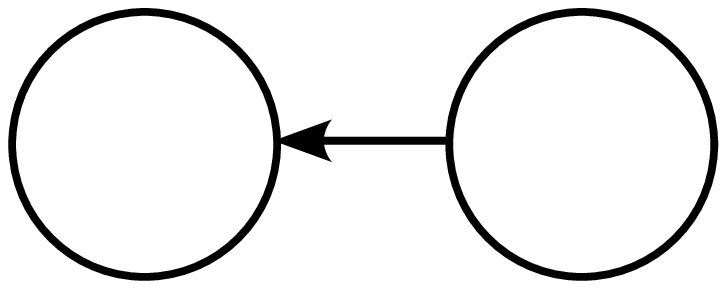}}. 
\begin{prop}\label{eq.id_deg2}
    As functions on Gauss diagrams of knots, we have 
    $$\raisebox{-.4\height}{\includegraphics[width=1.75cm]{image/A_20.eps}} = \raisebox{-.4\height}{\includegraphics[width=1.75cm]{image/A_11.eps}}$$
\end{prop}
\begin{proof}
    For each arrow $\alpha$ \raisebox{-.4\height}{\includegraphics[width=1.5cm]{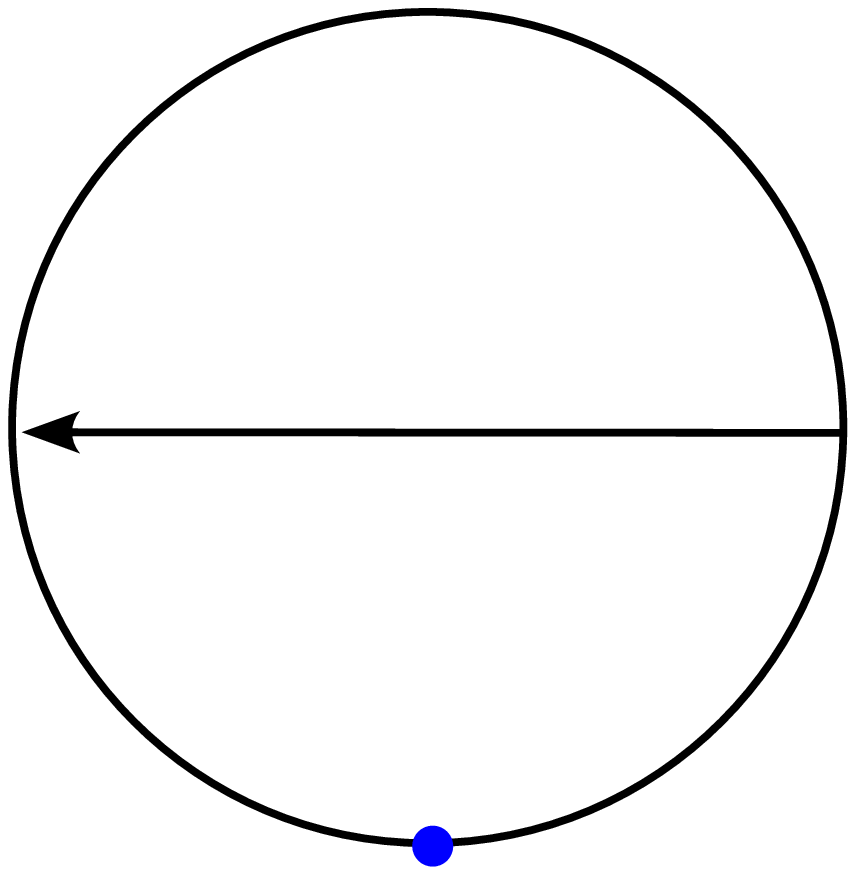}} in a Gauss diagram $G$ of a knot $K$, we can smooth it splitting the knot into two components \raisebox{-.4\height}{\includegraphics[width=1.5cm]{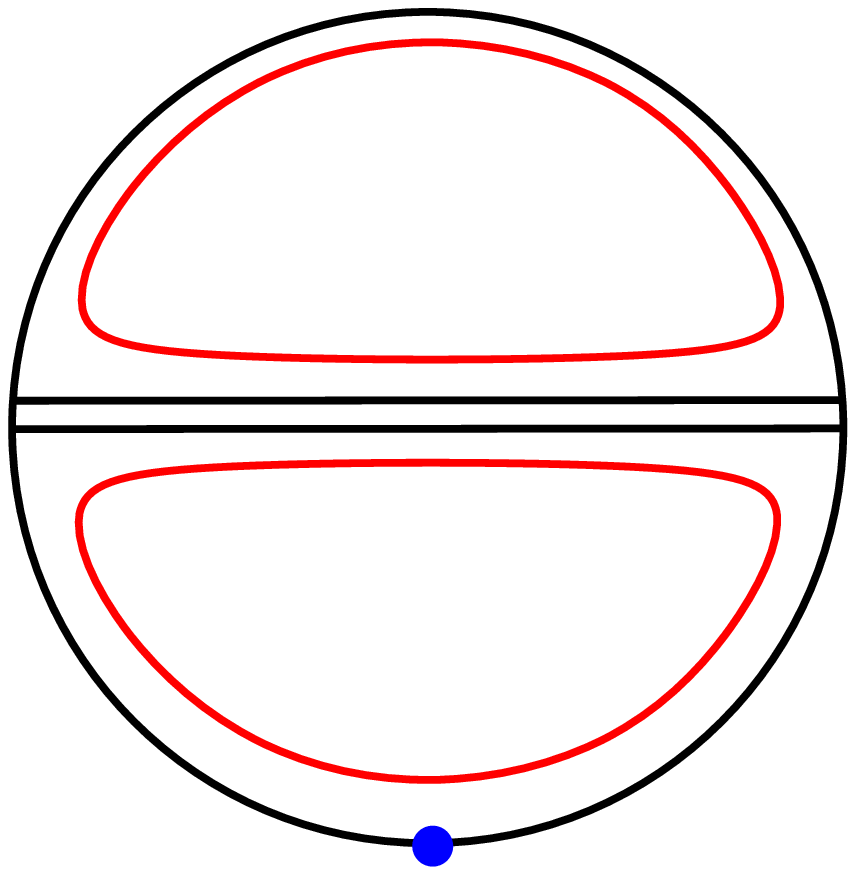}}. There are two ways to count the linking number of the two components.
    \begin{align*}
        \intertext{One is }
        \raisebox{-.4\height}{\includegraphics[width=1.75cm]{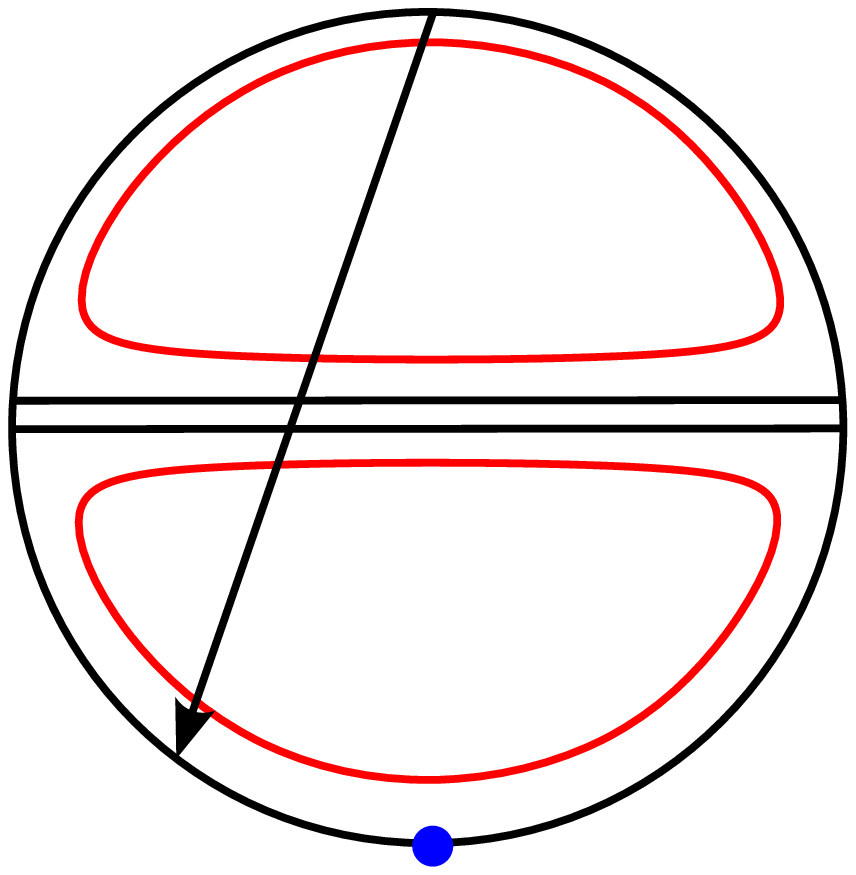}} + \raisebox{-.4\height}{\includegraphics[width=1.75cm]{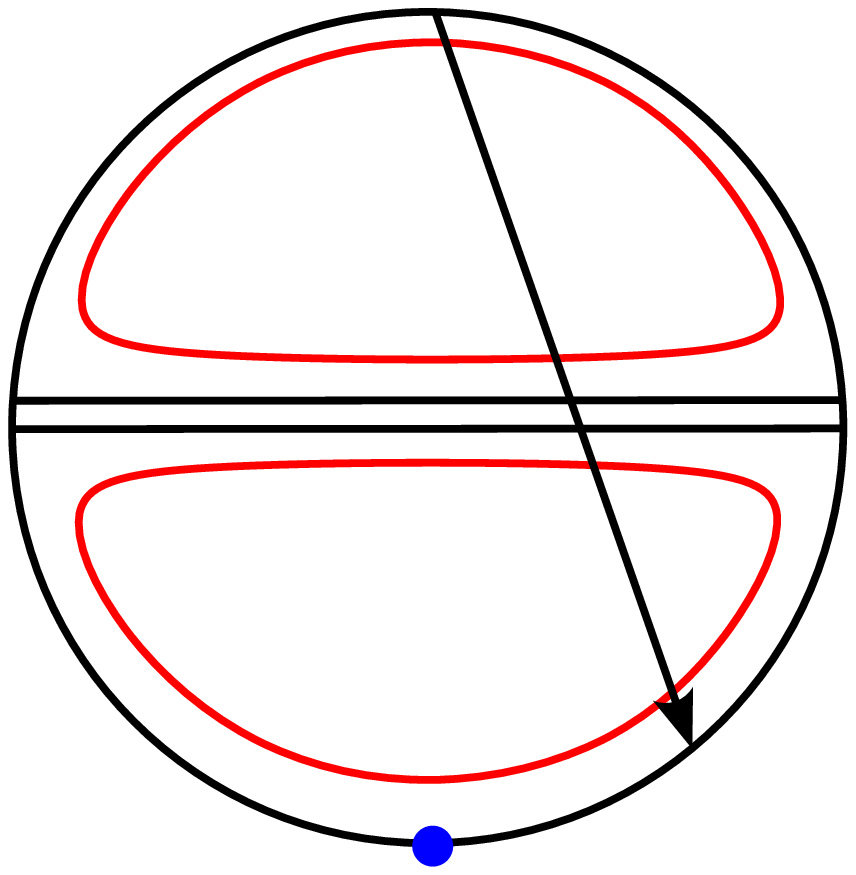}}\\
        \intertext{Another is }
        \raisebox{-.4\height}{\includegraphics[width=1.75cm]{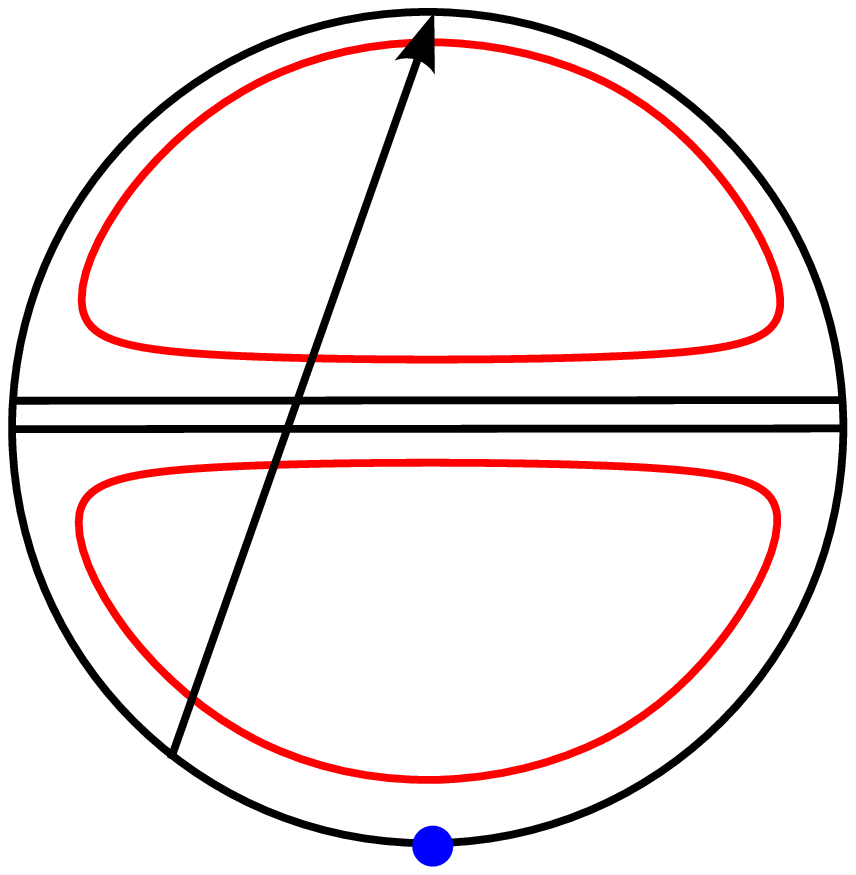}} + \raisebox{-.4\height}{\includegraphics[width=1.75cm]{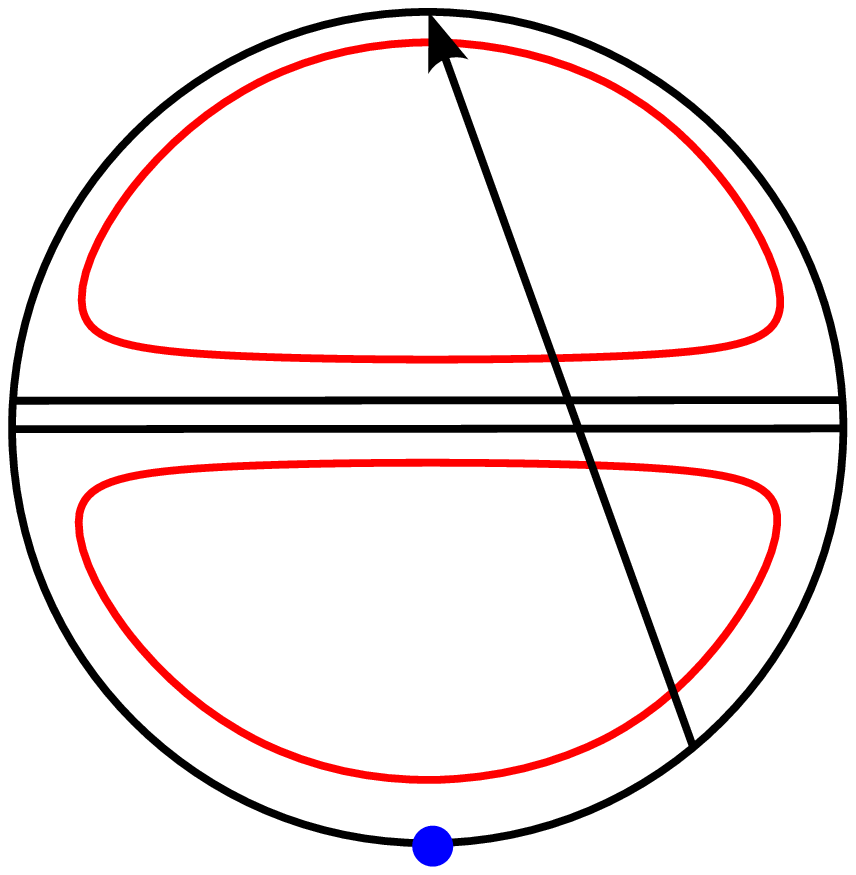}}. 
    \end{align*}
    We add all linking numbers of all such two components in $G$ regarding the sign $\alpha$. Then we have
    \begin{align*}
        \raisebox{-.4\height}{\includegraphics[width=1.75cm]{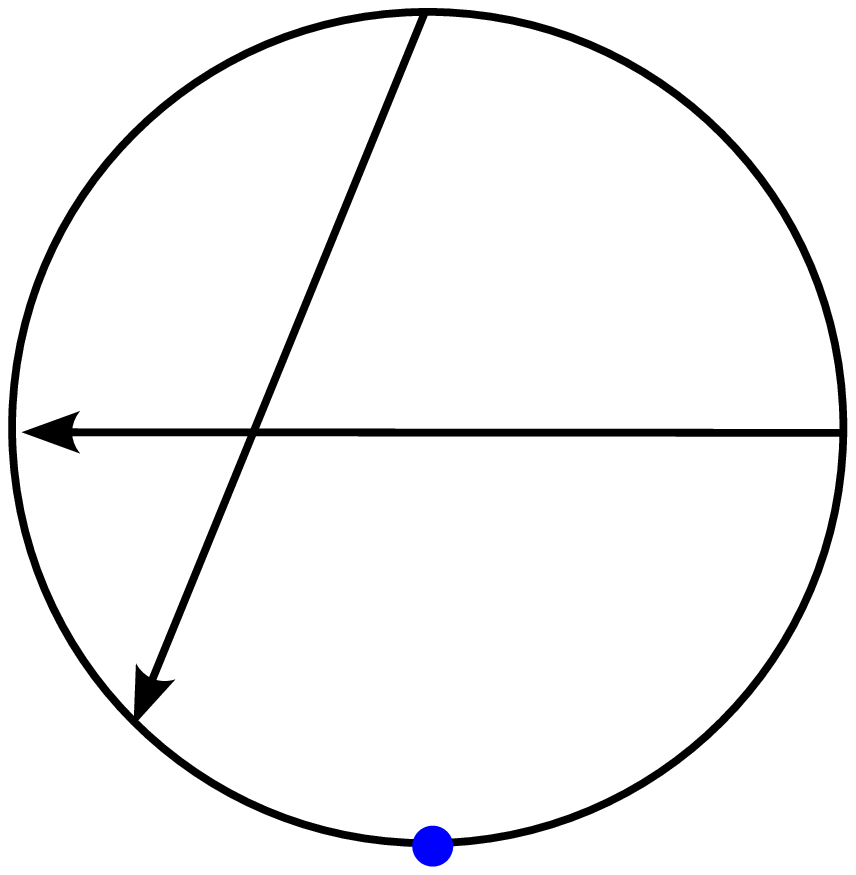}} + \raisebox{-.4\height}{\includegraphics[width=1.75cm]{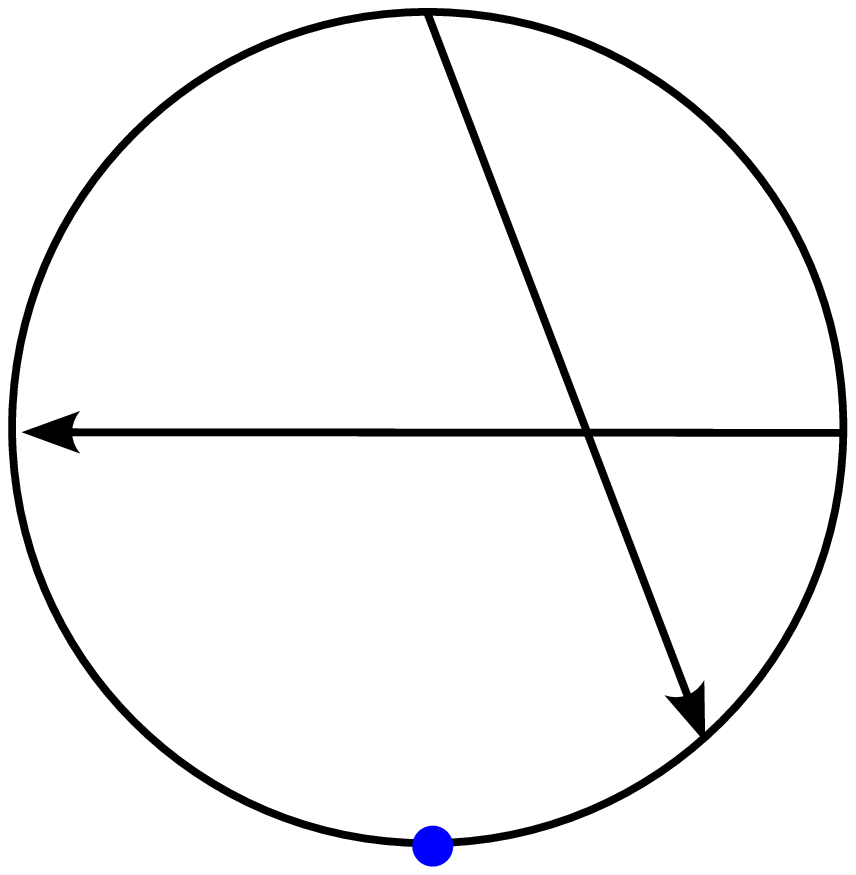}} = \raisebox{-.4\height}{\includegraphics[width=1.75cm]{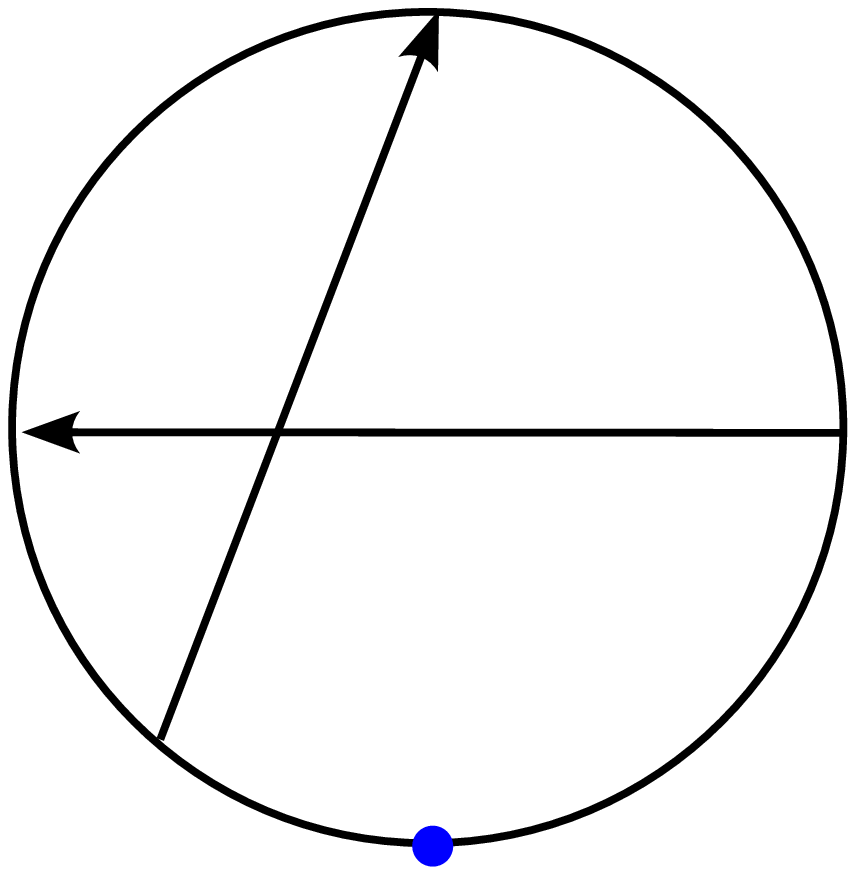}} + \raisebox{-.4\height}{\includegraphics[width=1.75cm]{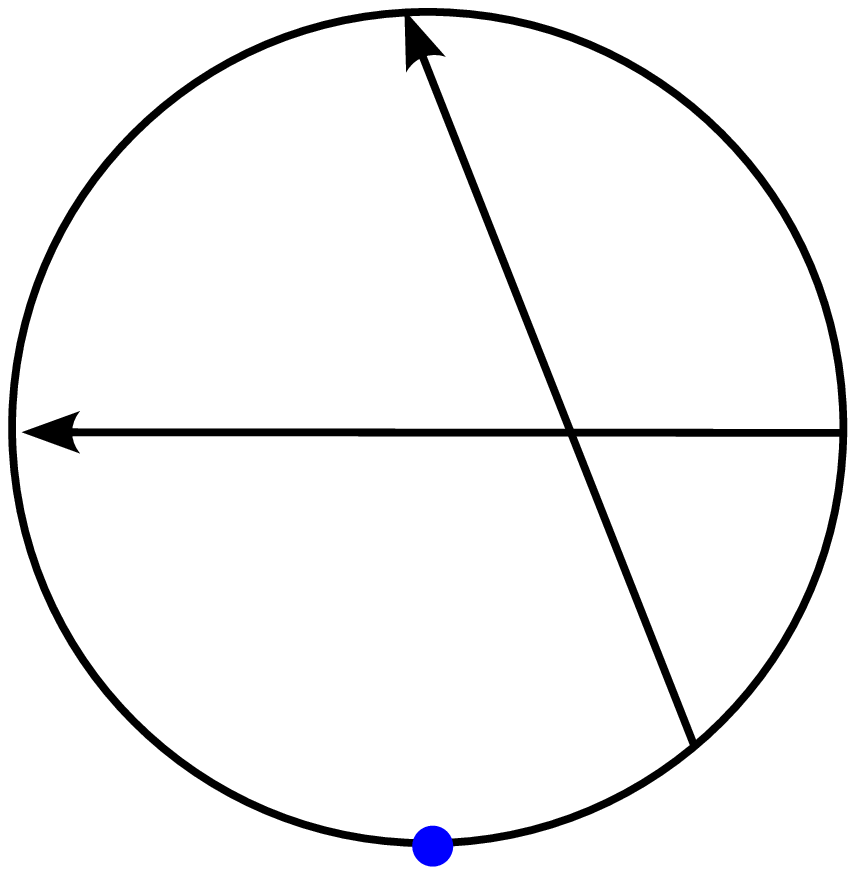}}
    \end{align*}
\end{proof}
Using the same method one can get the following identities.
\begin{prop}
    \begin{align*}
        \raisebox{-.4\height}{\includegraphics[width=1.75cm]{image/A_30_7.eps}} = \raisebox{-.4\height}{\includegraphics[width=1.75cm]{image/A_30_8.eps}}\ \ 
        \raisebox{-.4\height}{\includegraphics[width=1.75cm] {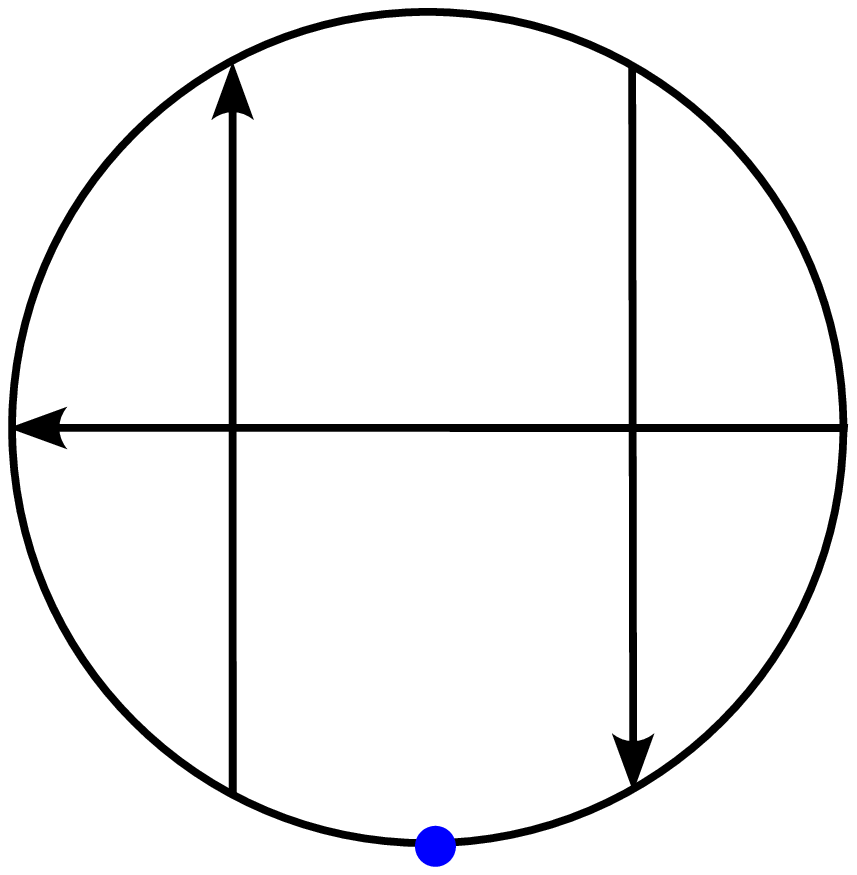}} &= \raisebox{-.4\height}{\includegraphics[width=1.75cm]{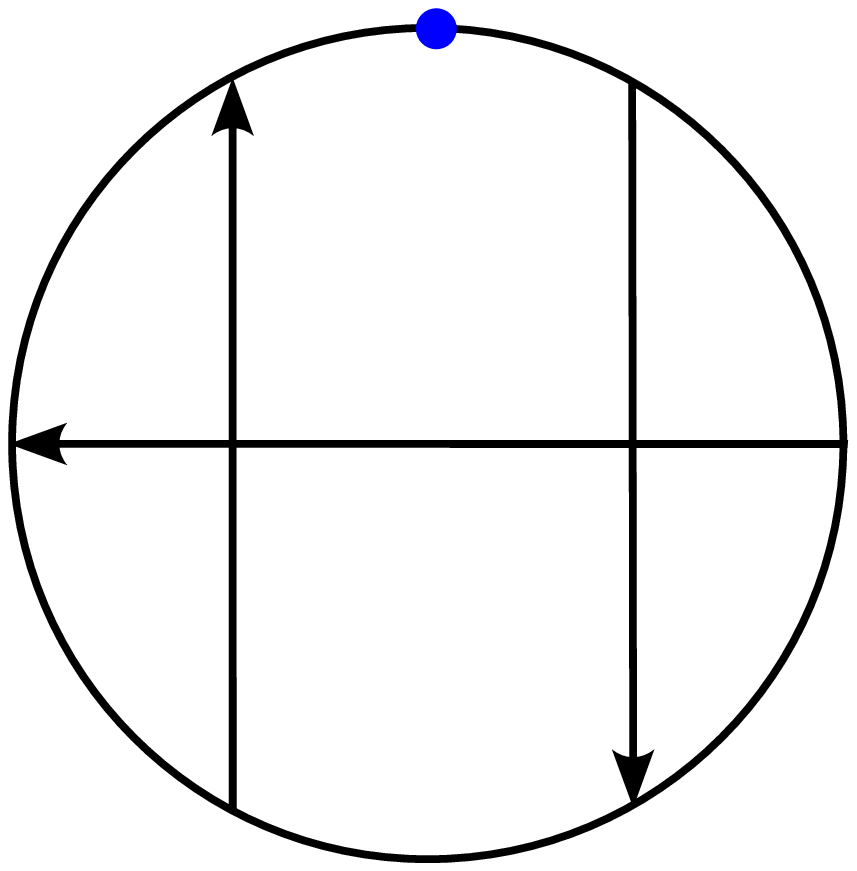}} \ \ \raisebox{-.4\height}{\includegraphics[width=1.75cm] {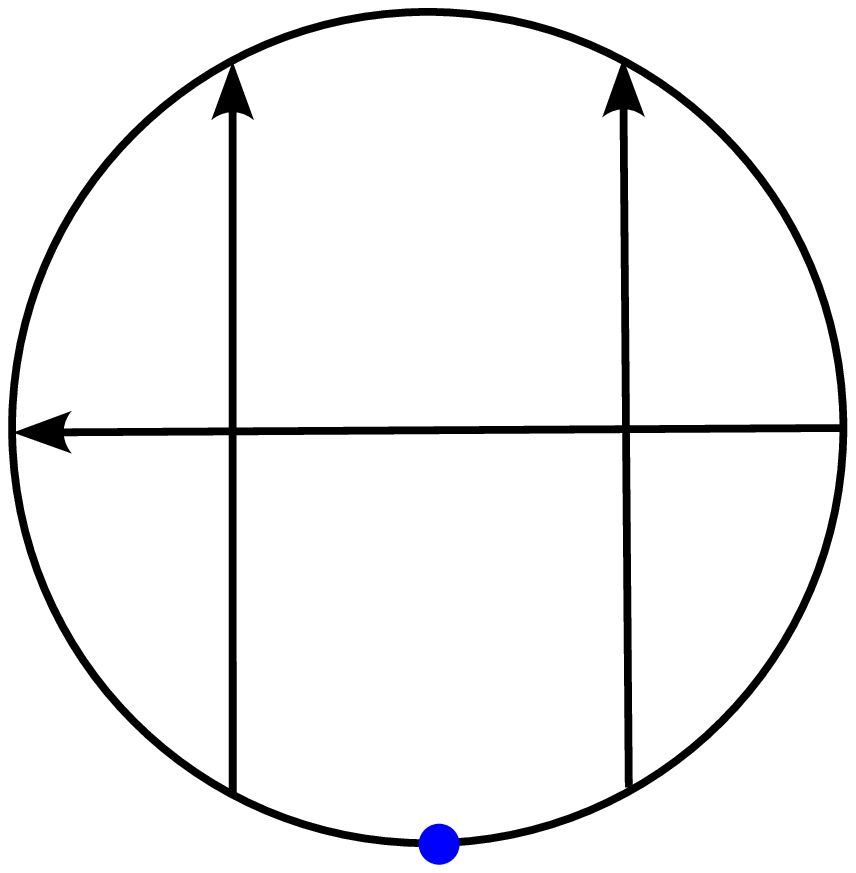}} = \raisebox{-.4\height}{\includegraphics[width=1.75cm]{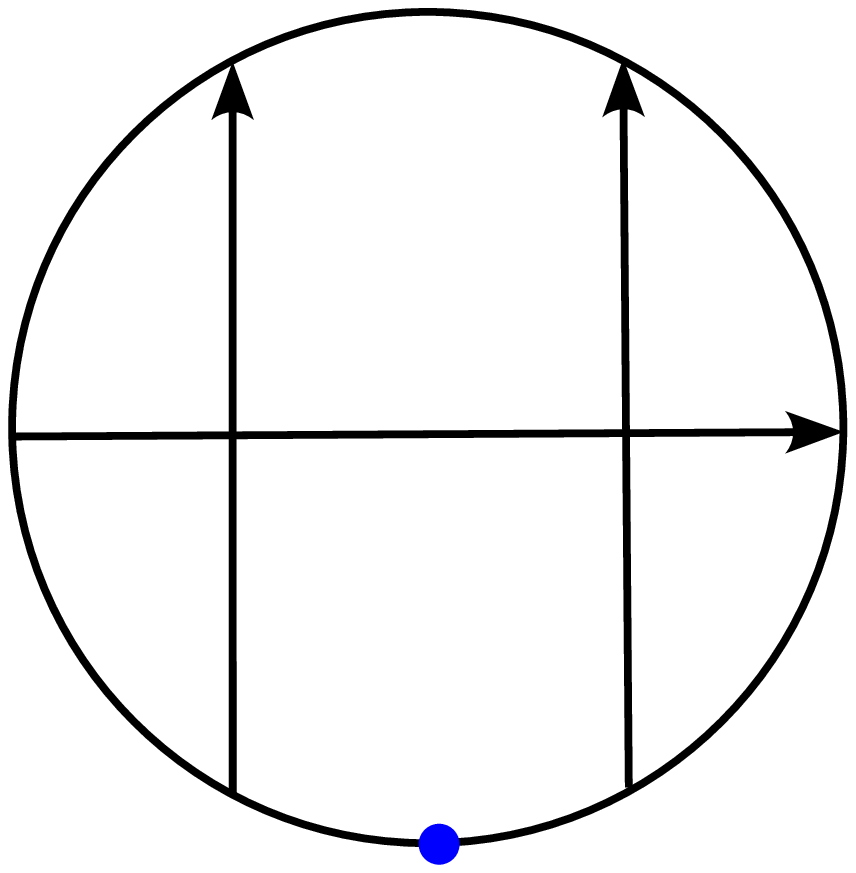}}
        \\
        \raisebox{-.4\height}{\includegraphics[width=1.75cm]{image/A_12_4.eps}} + \raisebox{-.4\height}{\includegraphics[width=1.75cm]{image/A_12_8.eps}} &= \raisebox{-.4\height}{\includegraphics[width=1.75cm]{image/A_12_3.eps}} + \raisebox{-.4\height}{\includegraphics[width=1.75cm]{image/A_12_10.eps}}
    \end{align*}
\end{prop}

Proposition \ref{eq.id_deg2} can also be proved in another point of view by Polyak and Viro in \cite{Polyak_Viro:2001}. Notice that rotating a long knot $\pi$ around its axis is an isotopy and that the only difference is that the top strand and bottom strand exchange at every crossing. We have 
\begin{prop}
    If $F$ is a Gauss diagram formula, we have $$\langle F, \cdot \rangle = \langle F', \cdot \rangle$$ where $F'$ is the result that we change the direction of all arrows in $F$. 
\end{prop}

Besides the operation "smooth" on a crossing, we also have the operation "singularize". This allows us to derive more identities using the linking number. 
\begin{prop}
    \begin{align*}
        \raisebox{-.4\height}{\includegraphics[width=1.75cm]{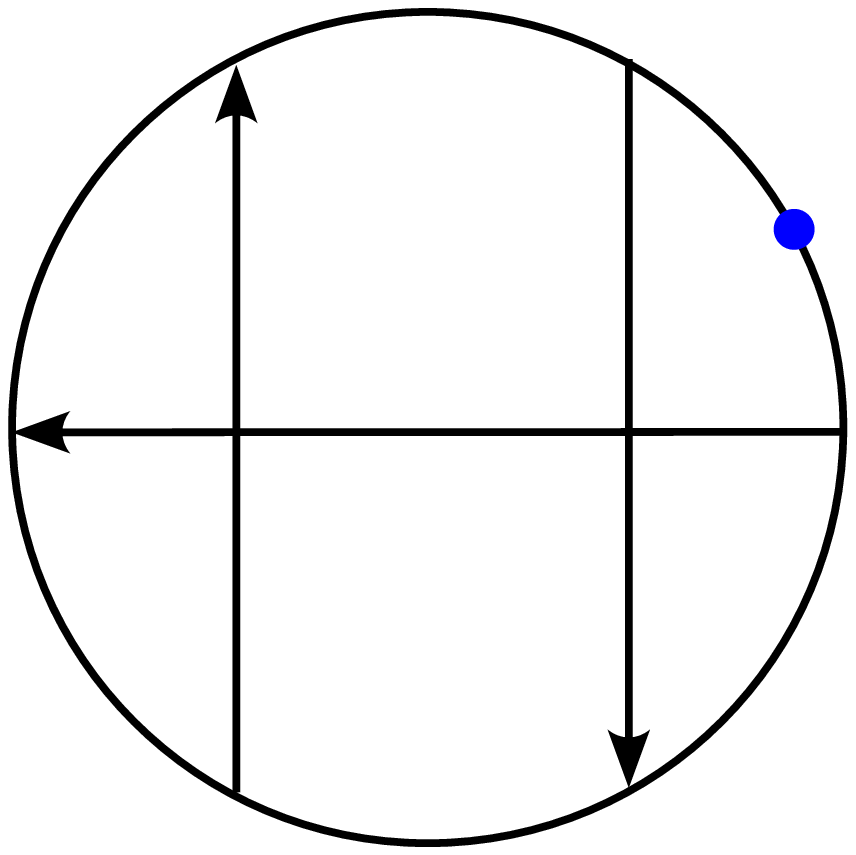}} + \raisebox{-.4\height}{\includegraphics[width=1.75cm]{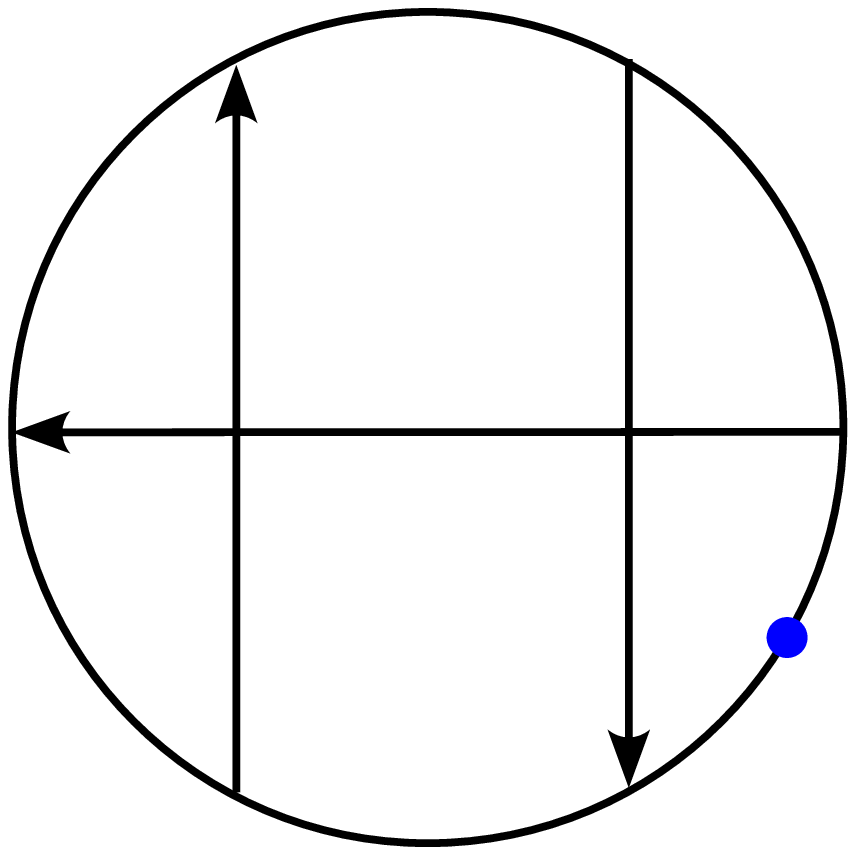}} + \raisebox{-.4\height}{\includegraphics[width=1.75cm]{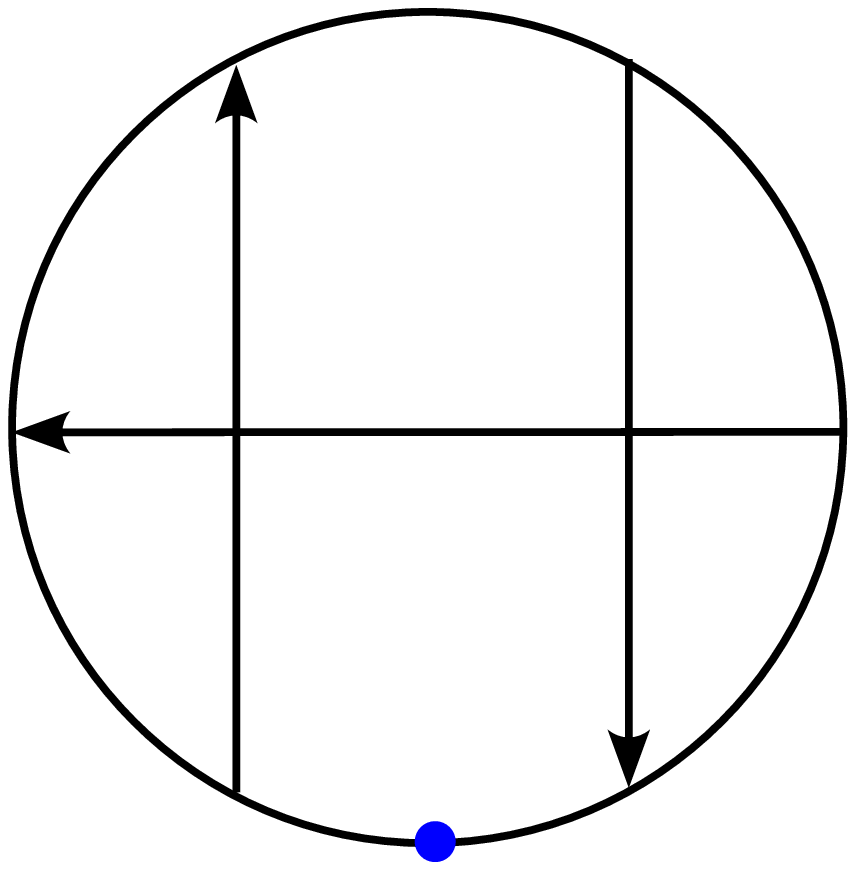}} &= \raisebox{-.4\height}{\includegraphics[width=1.75cm]{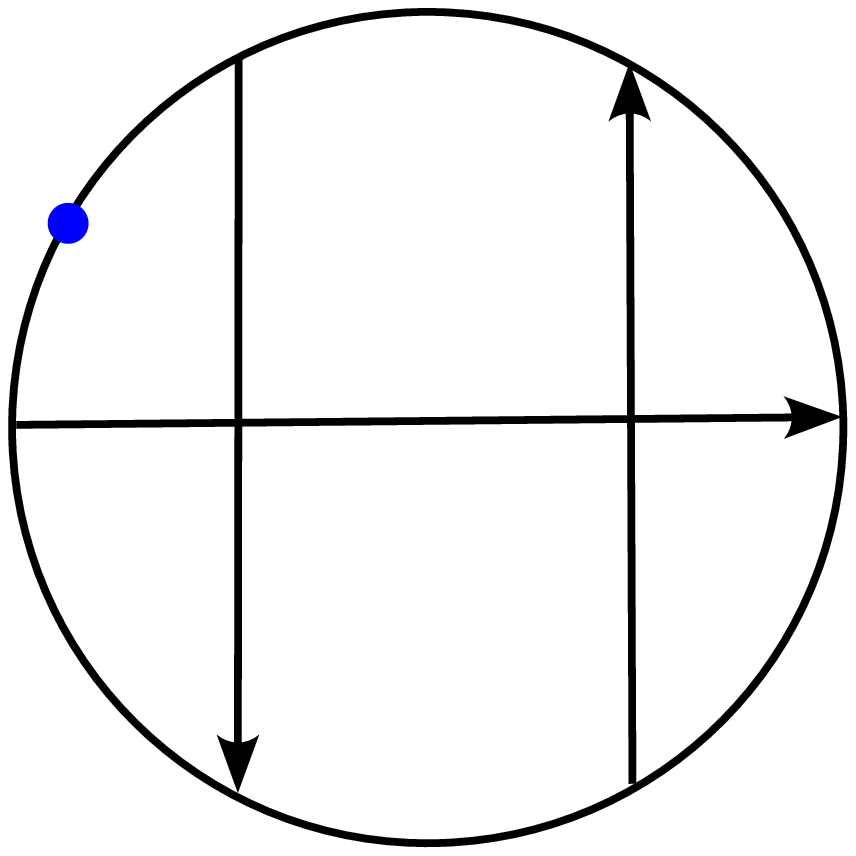}} + \raisebox{-.4\height}{\includegraphics[width=1.75cm]{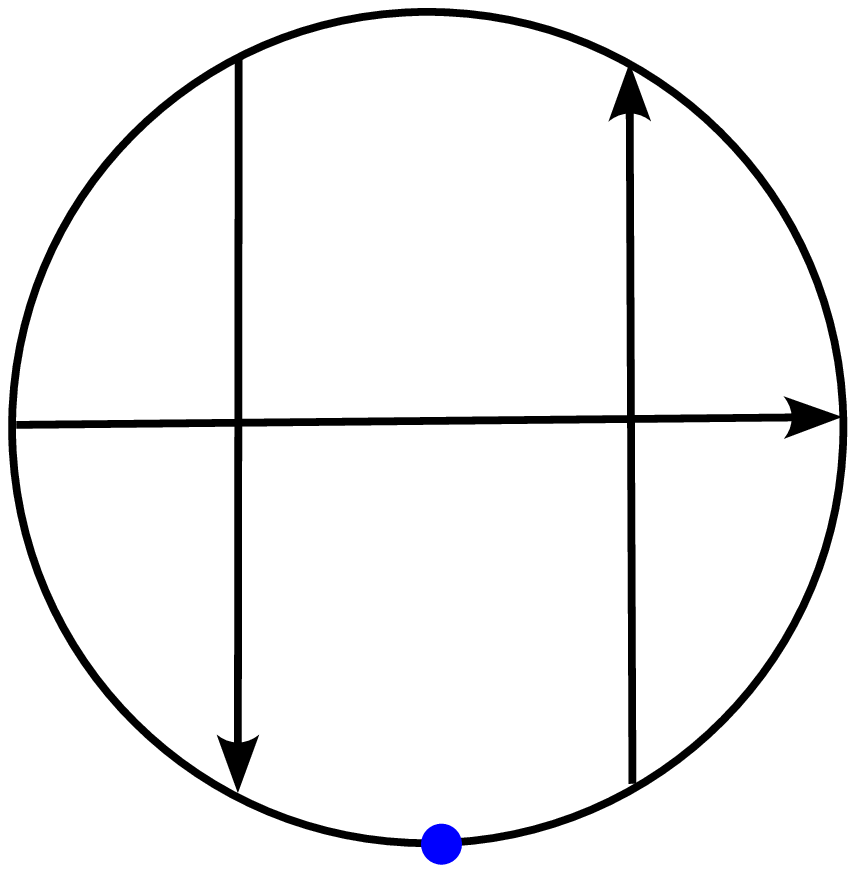}} + \raisebox{-.4\height}{\includegraphics[width=1.75cm]{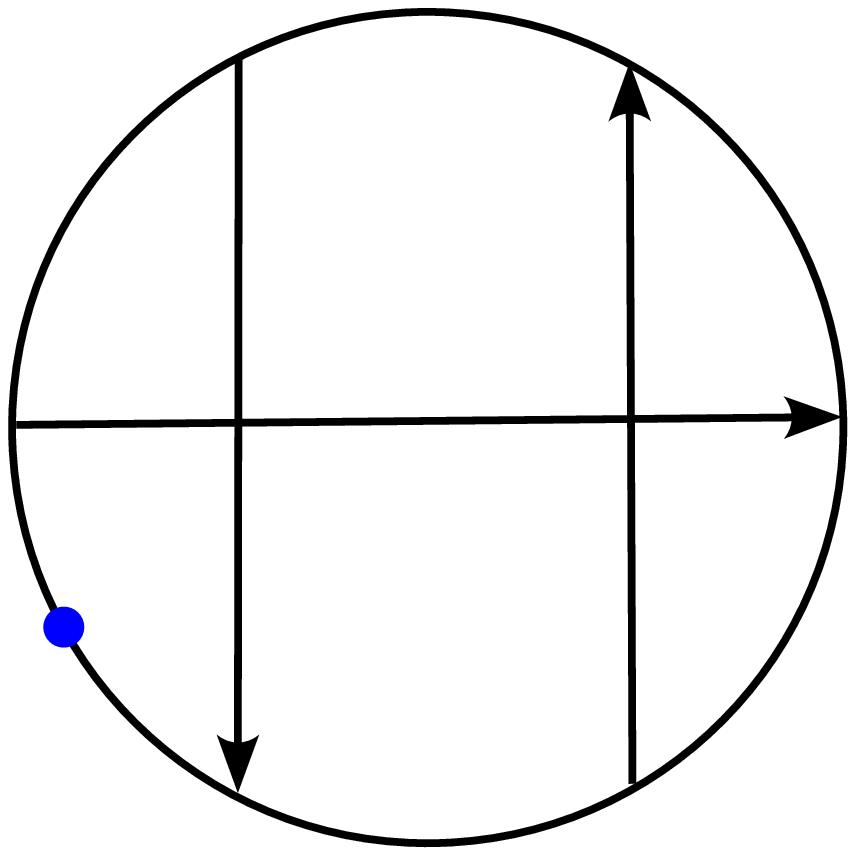}}\\
        \raisebox{-.4\height}{\includegraphics[width=1.75cm]{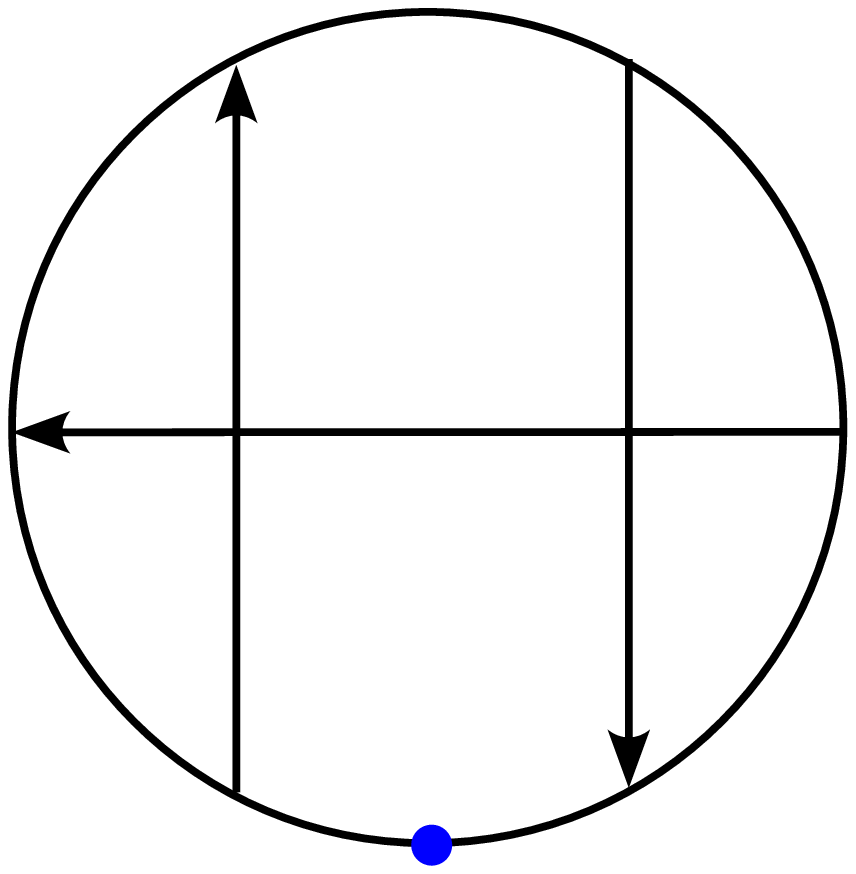}} + \raisebox{-.4\height}{\includegraphics[width=1.75cm]{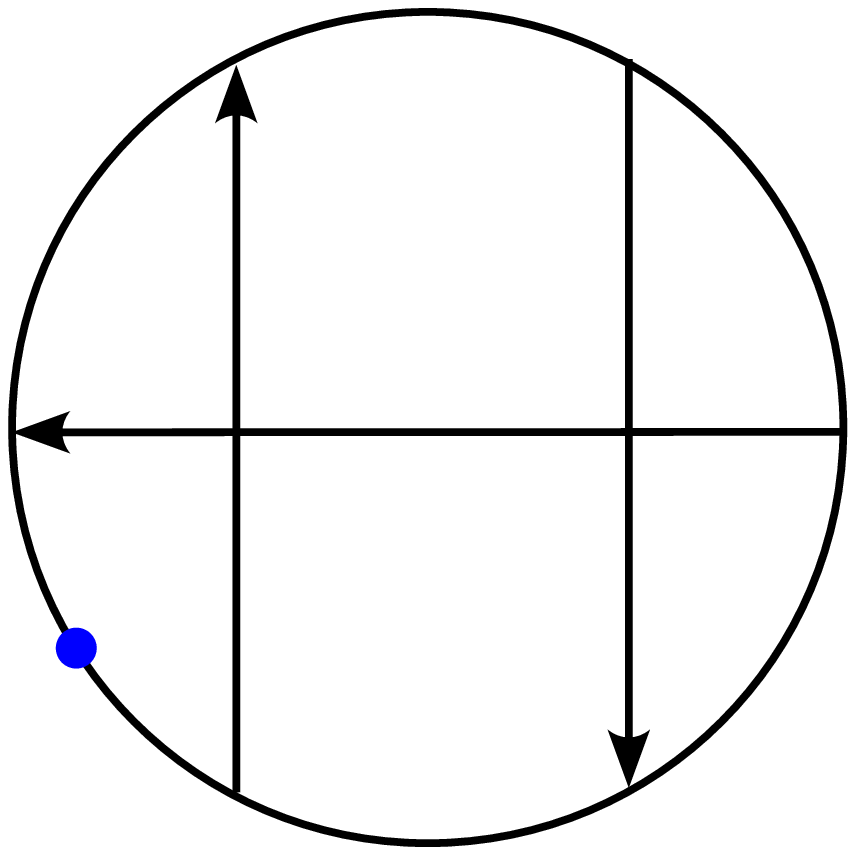}} + \raisebox{-.4\height}{\includegraphics[width=1.75cm]{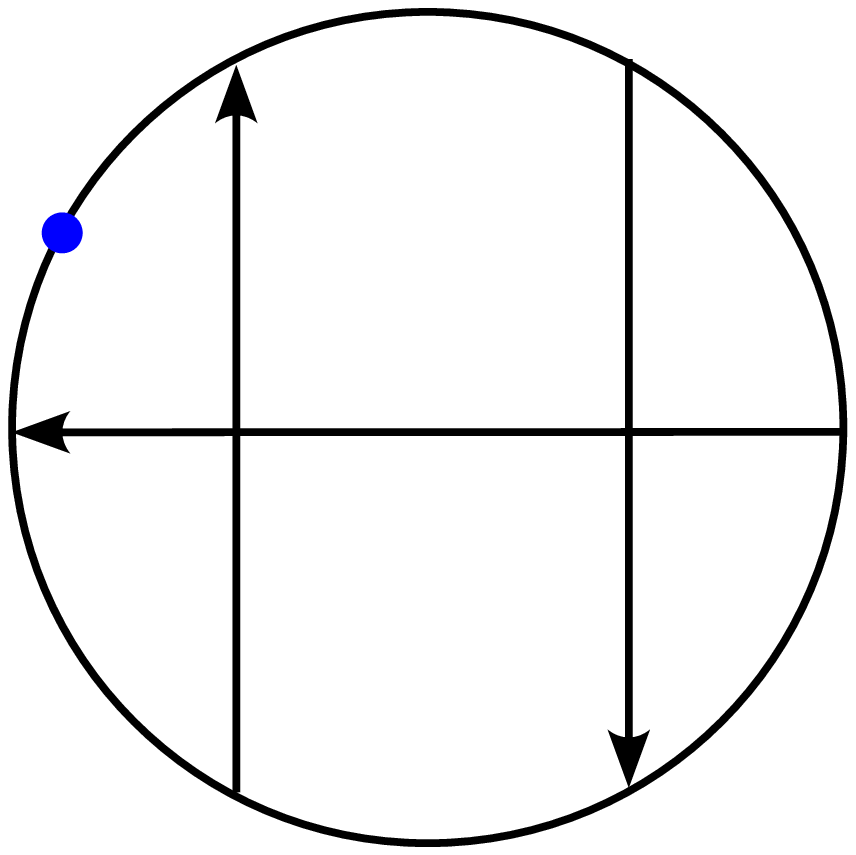}} &= \raisebox{-.4\height}{\includegraphics[width=1.75cm]{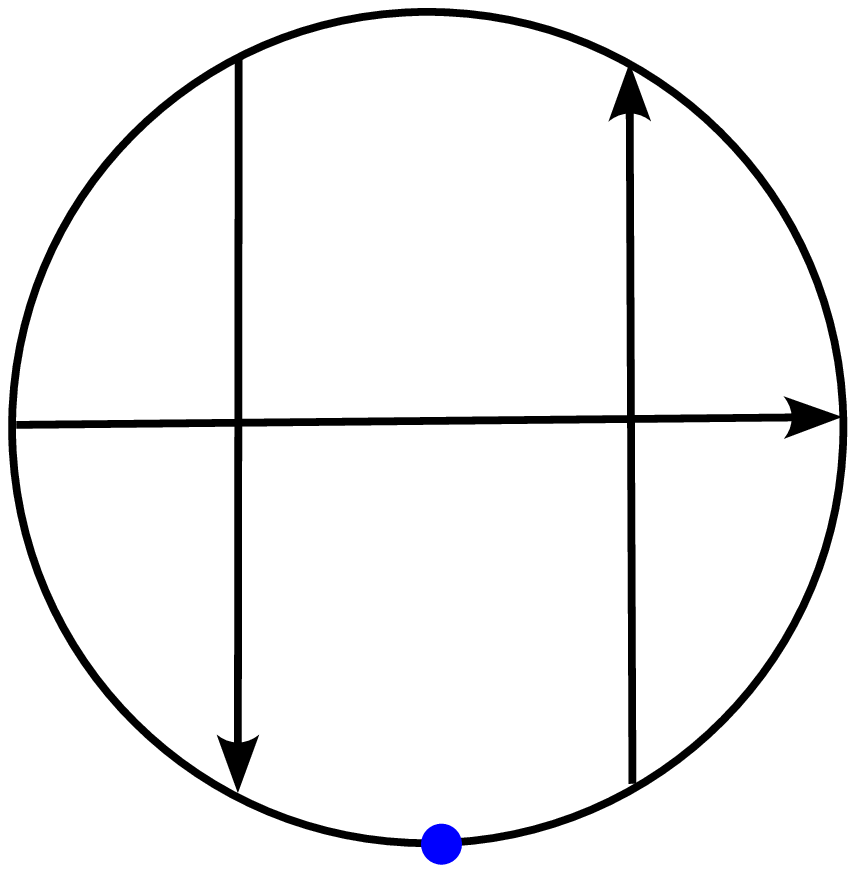}} + \raisebox{-.4\height}{\includegraphics[width=1.75cm]{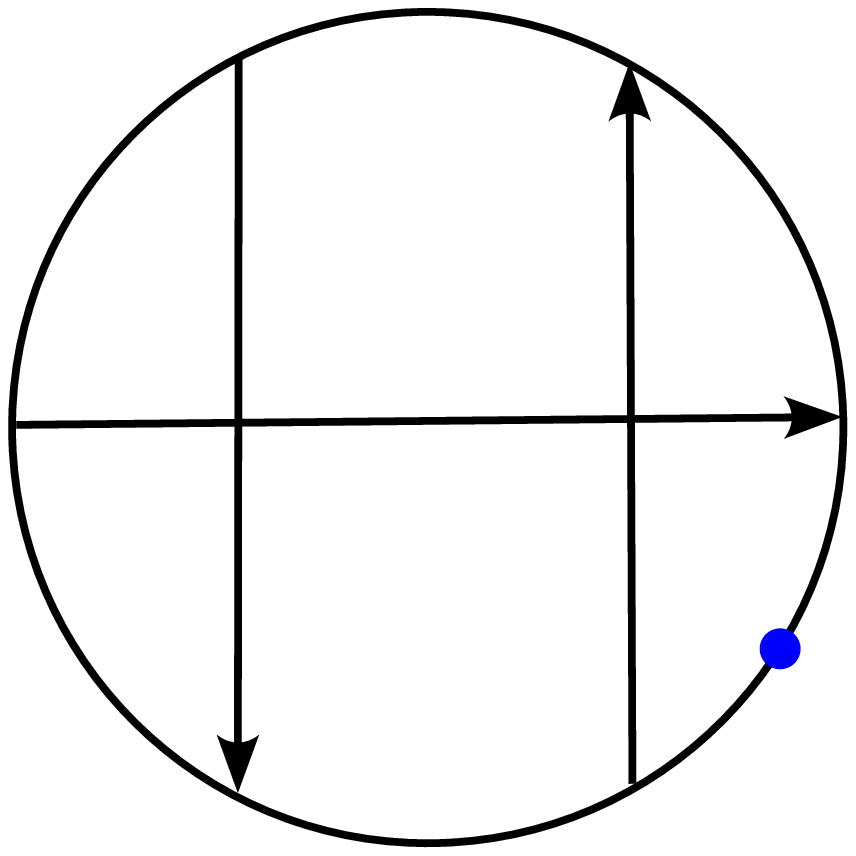}} + \raisebox{-.4\height}{\includegraphics[width=1.75cm]{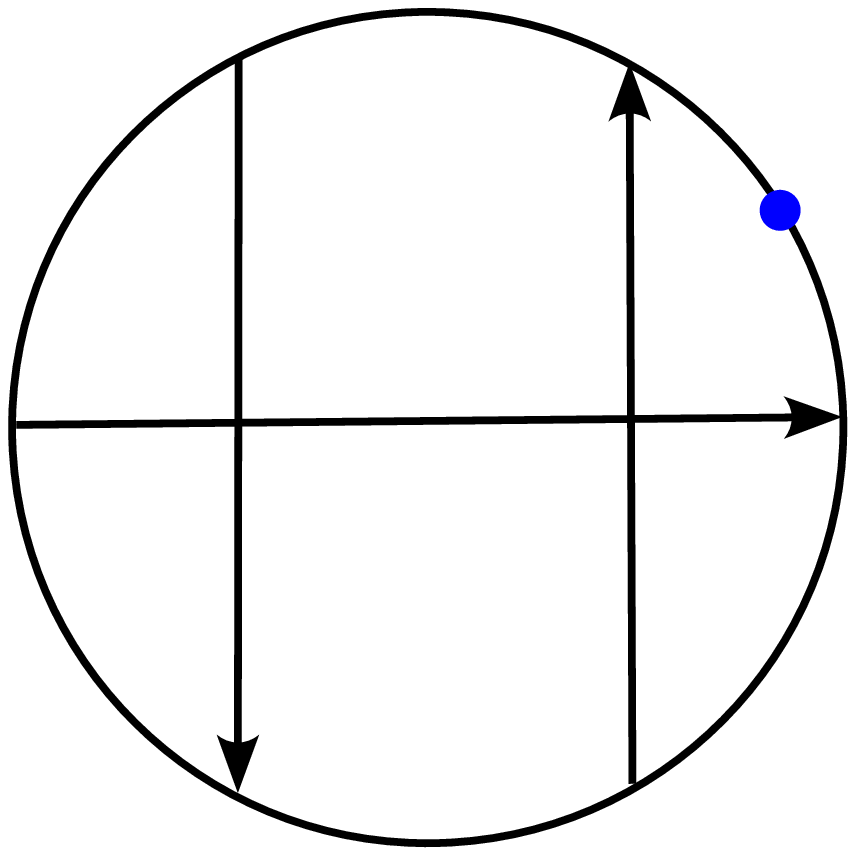}}\\
        \raisebox{-.4\height}{\includegraphics[width=1.75cm]{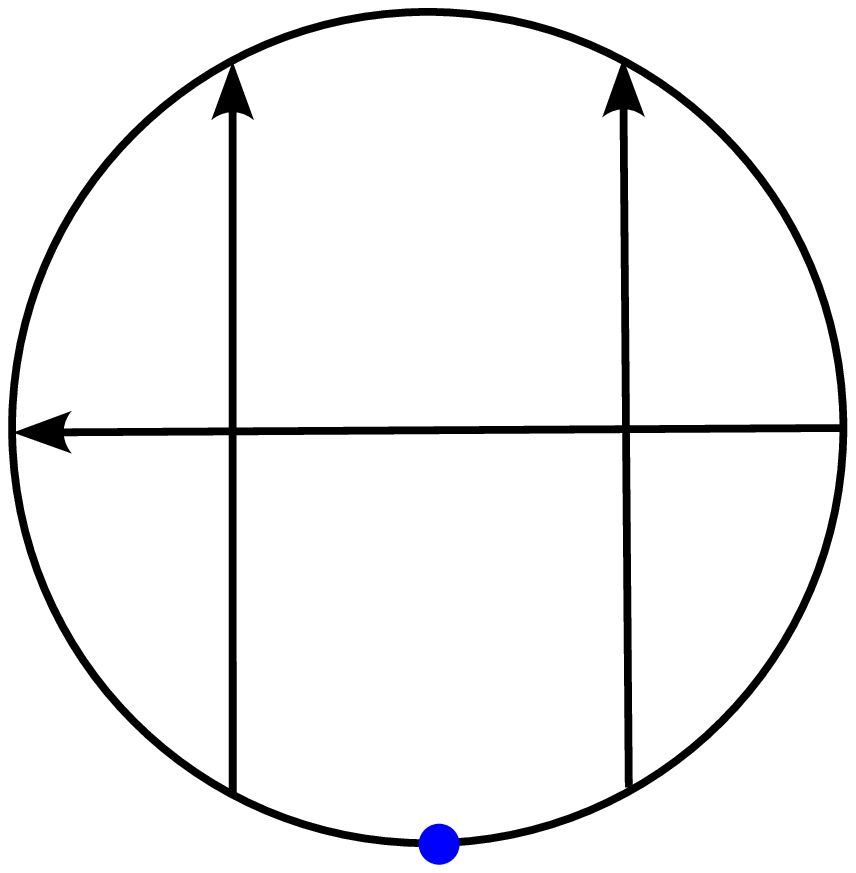}} + \raisebox{-.4\height}{\includegraphics[width=1.75cm]{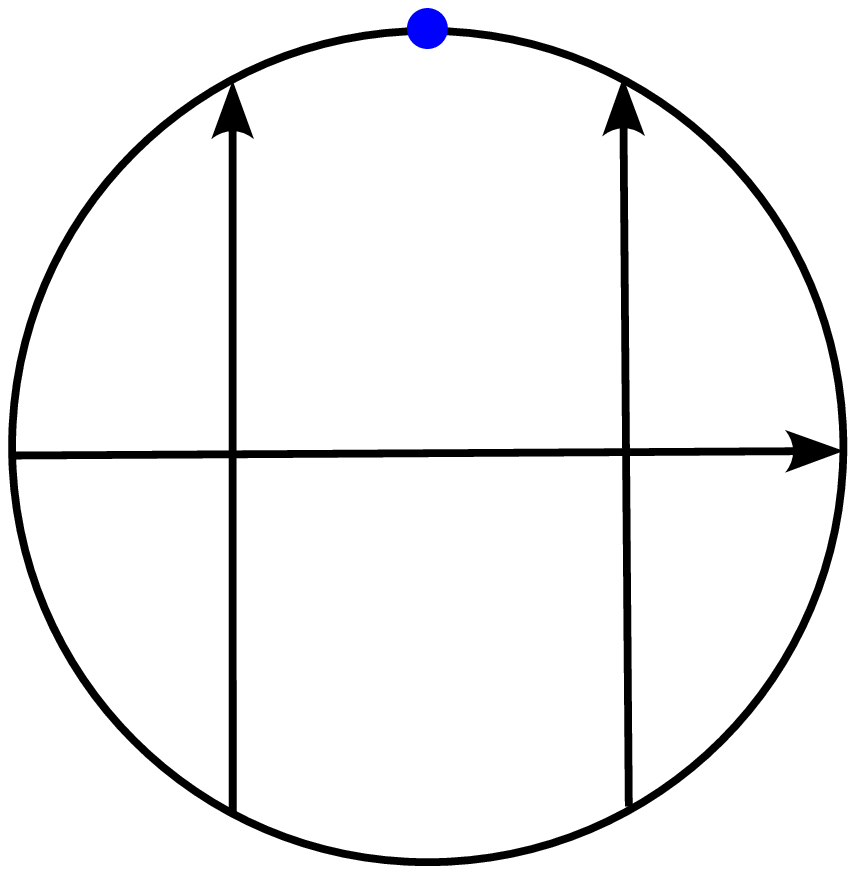}} + \raisebox{-.4\height}{\includegraphics[width=1.75cm]{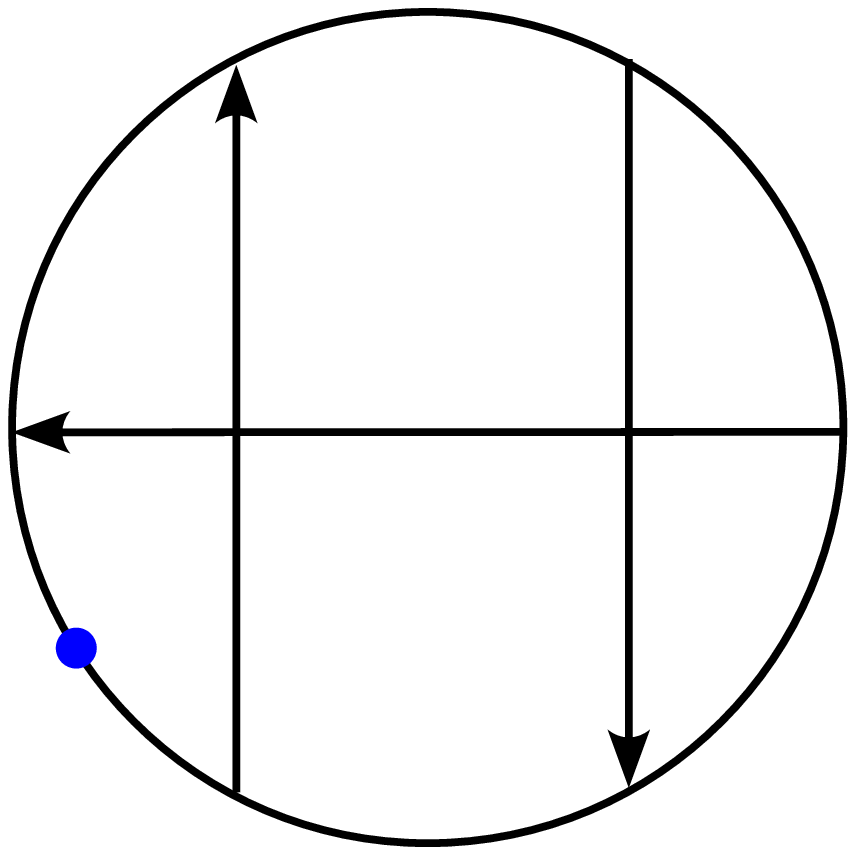}} + \raisebox{-.4\height}{\includegraphics[width=1.75cm]{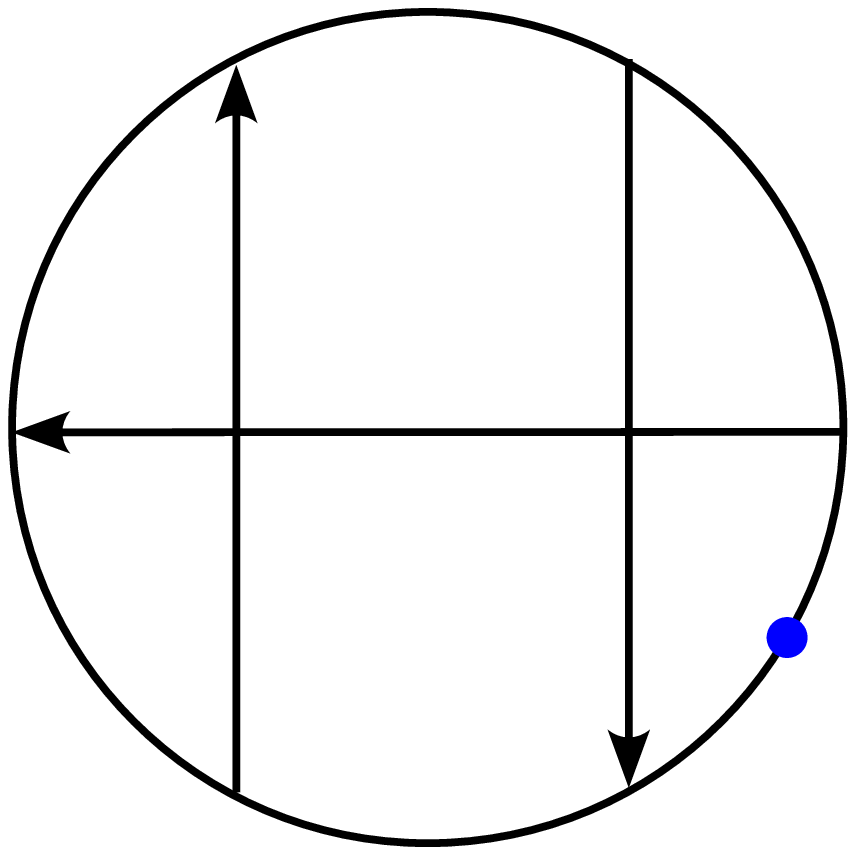}} &= 2\raisebox{-.4\height}{\includegraphics[width=1.75cm]{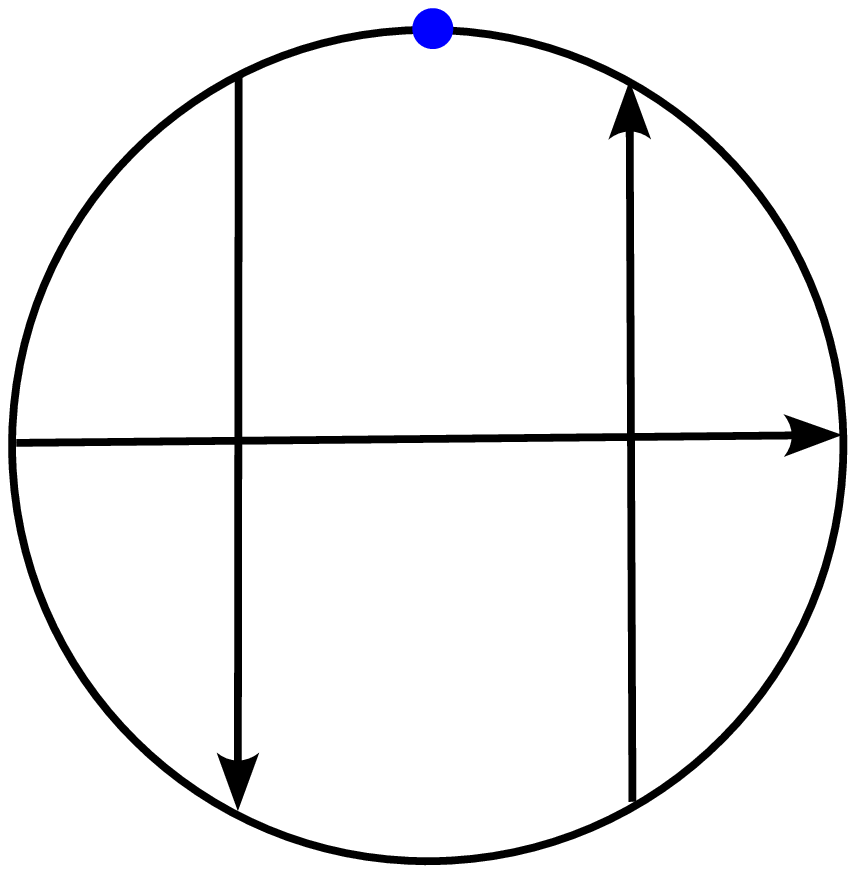}} + \raisebox{-.4\height}{\includegraphics[width=1.75cm]{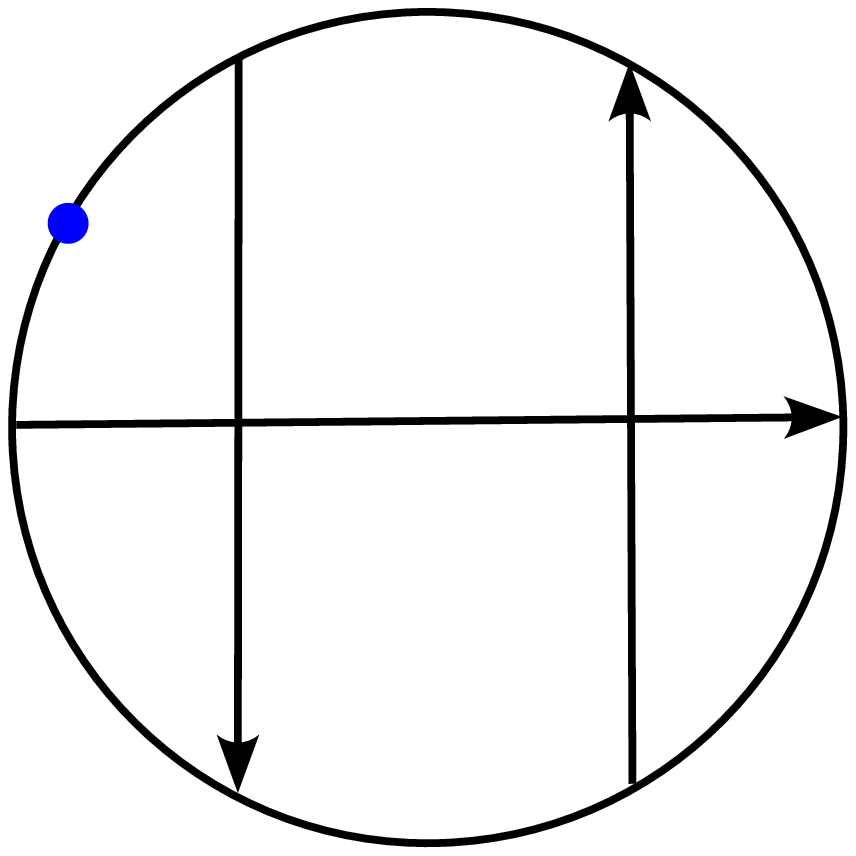}} + \raisebox{-.4\height}{\includegraphics[width=1.75cm]{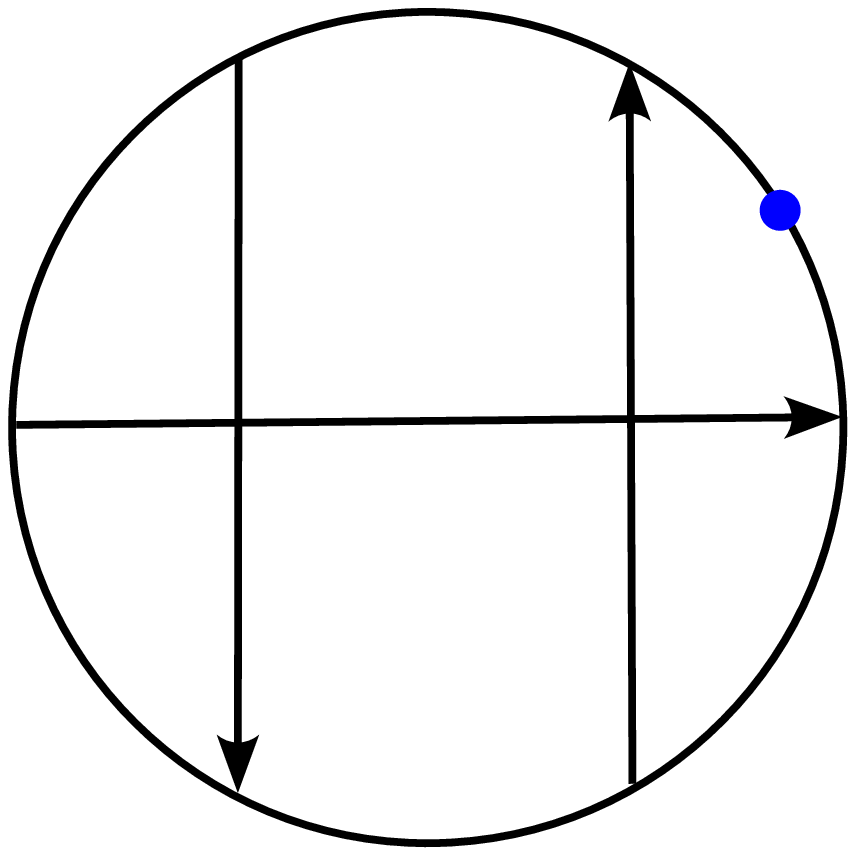}}\\
        \raisebox{-.4\height}{\includegraphics[width=1.75cm]{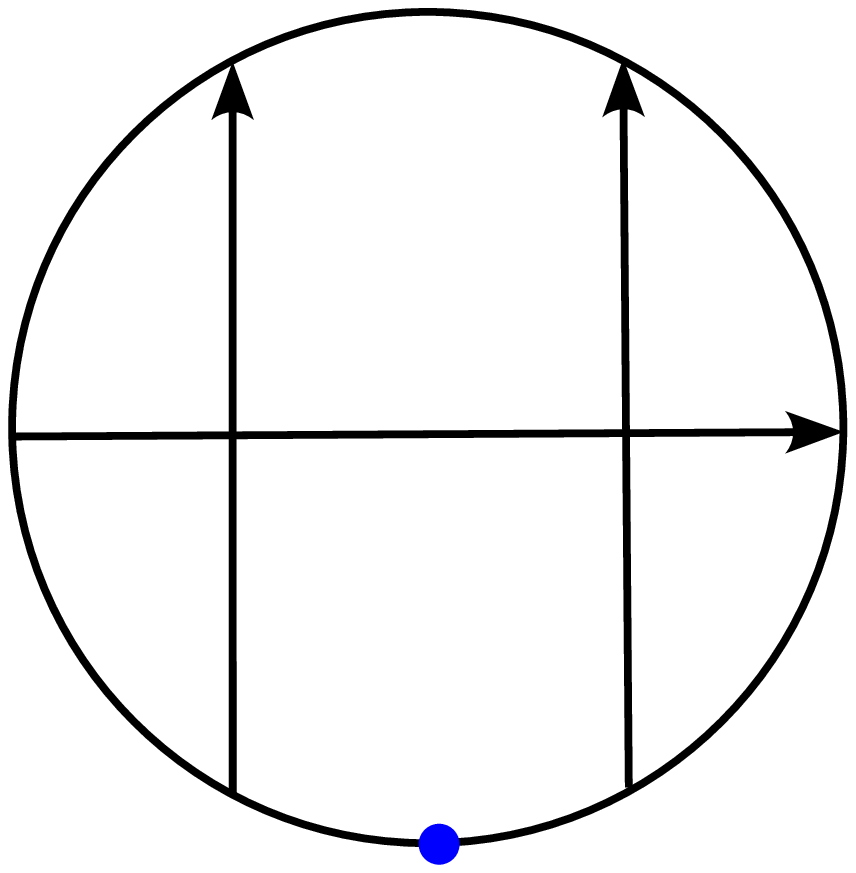}} + \raisebox{-.4\height}{\includegraphics[width=1.75cm]{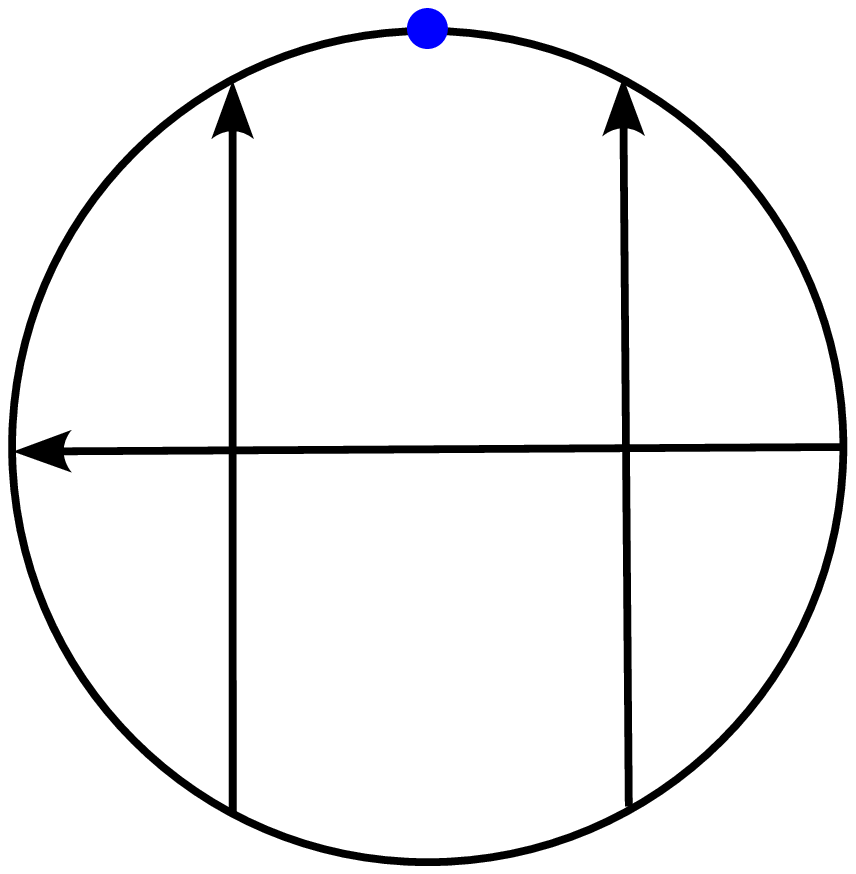}} + \raisebox{-.4\height}{\includegraphics[width=1.75cm]{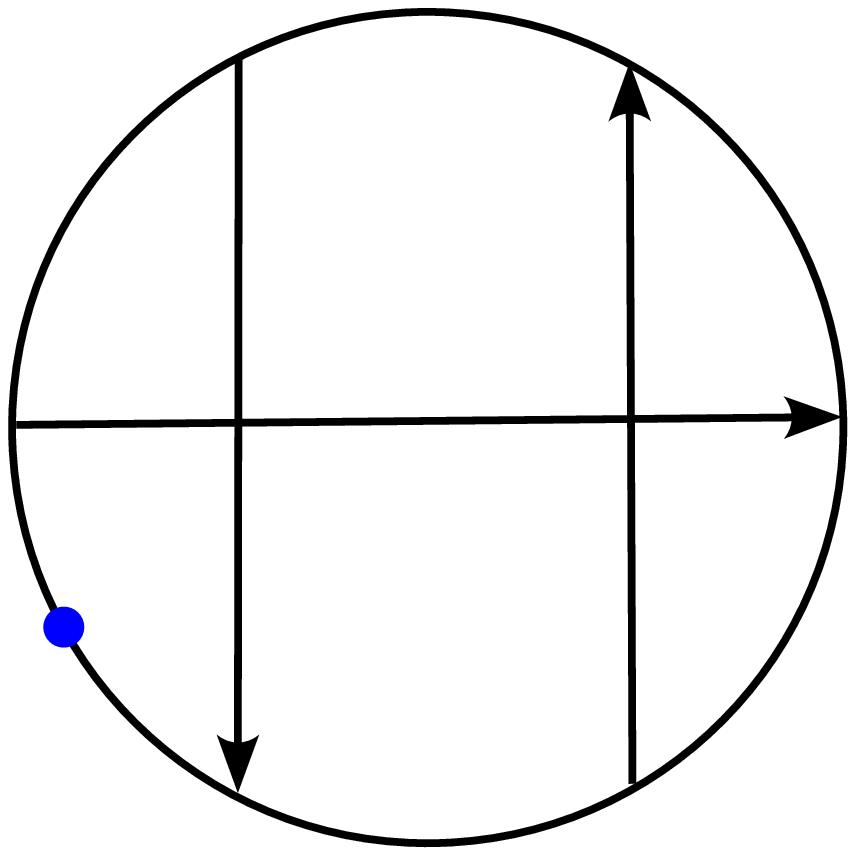}} + \raisebox{-.4\height}{\includegraphics[width=1.75cm]{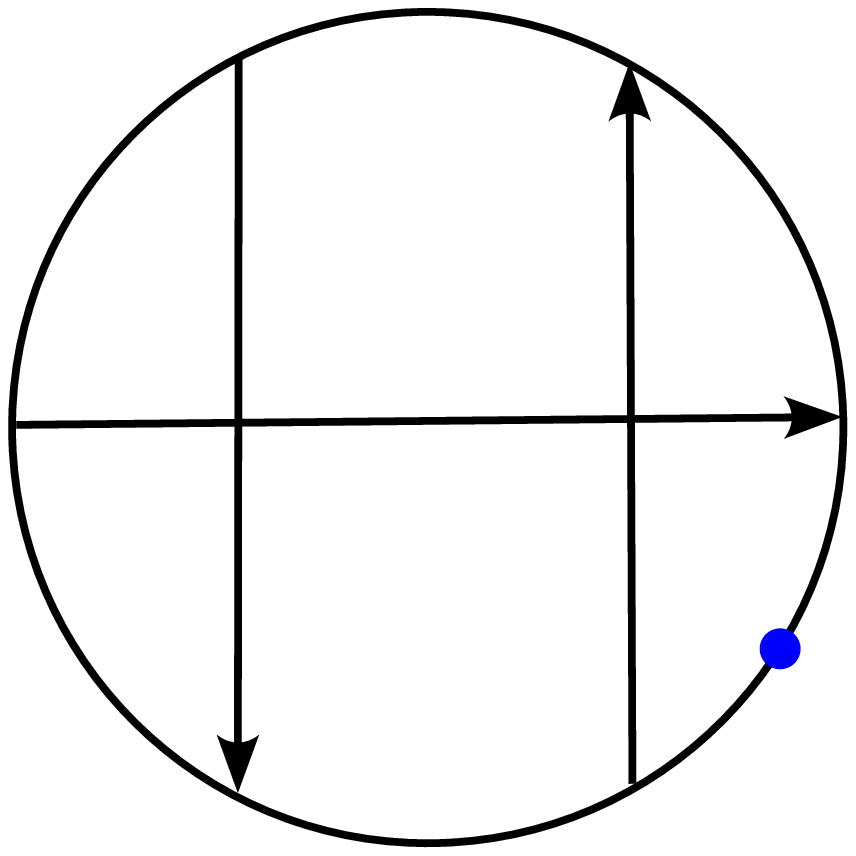}} &= 2\raisebox{-.4\height}{\includegraphics[width=1.75cm]{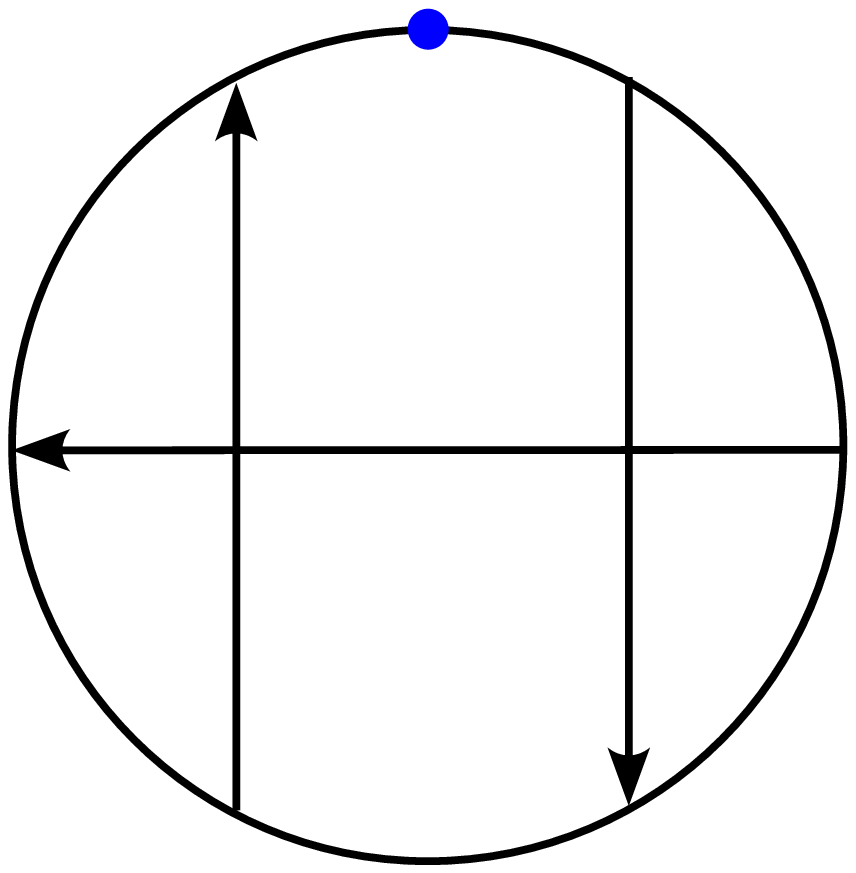}} + \raisebox{-.4\height}{\includegraphics[width=1.75cm]{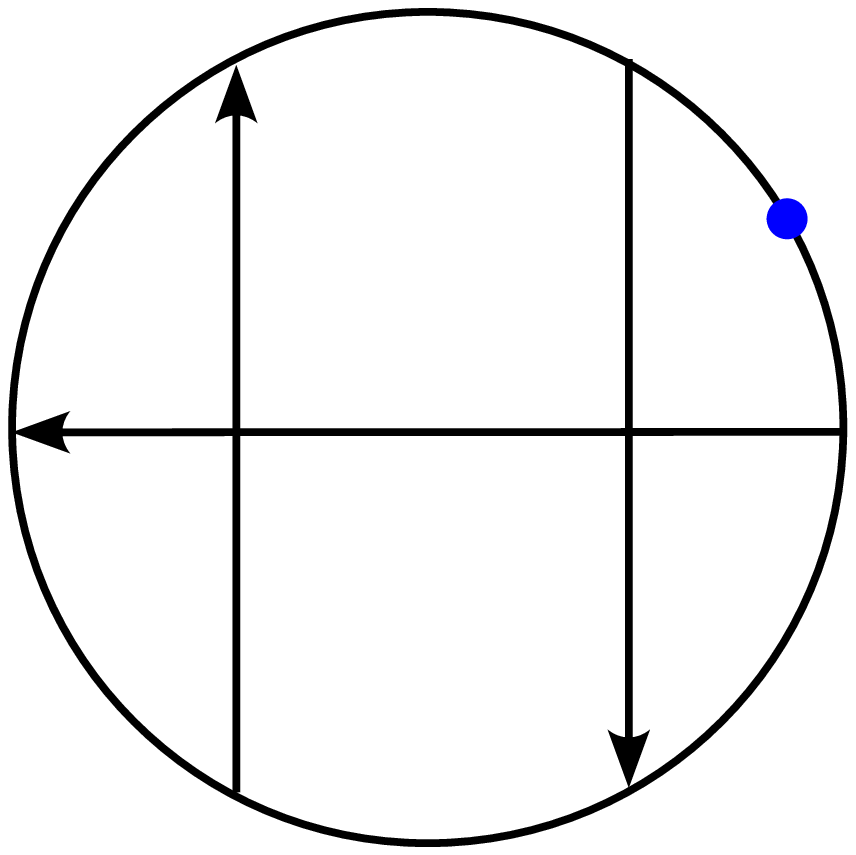}} + \raisebox{-.4\height}{\includegraphics[width=1.75cm]{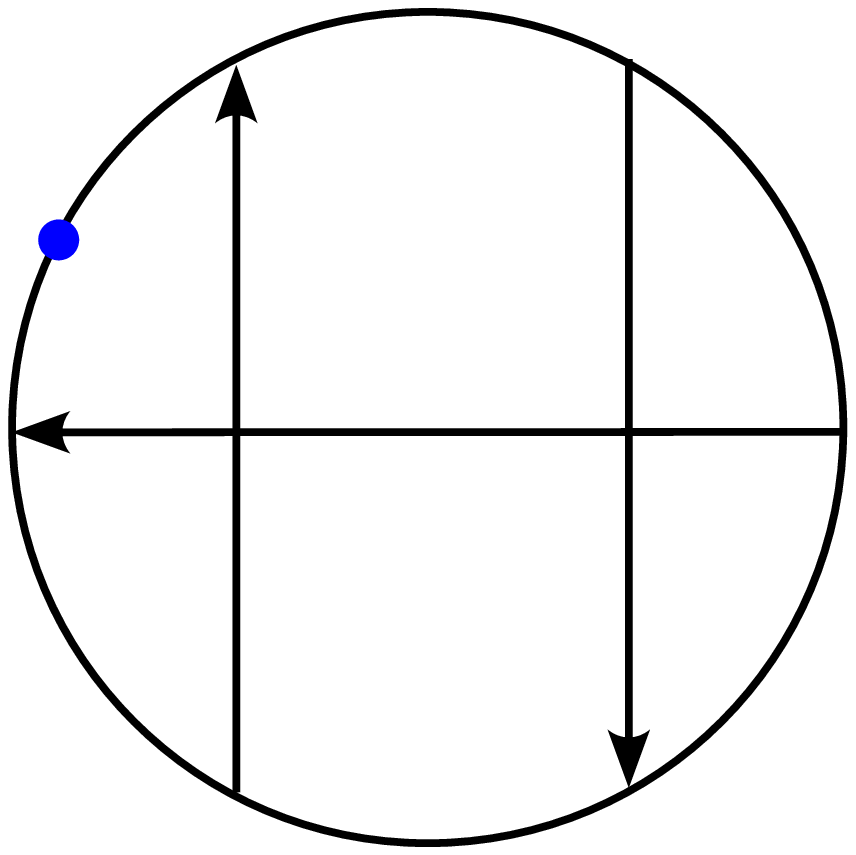}}
    \end{align*}
\end{prop}
\begin{proof}
    We prove the first identity.
    For each pattern \raisebox{-.4\height}{\includegraphics[width=1.5cm]{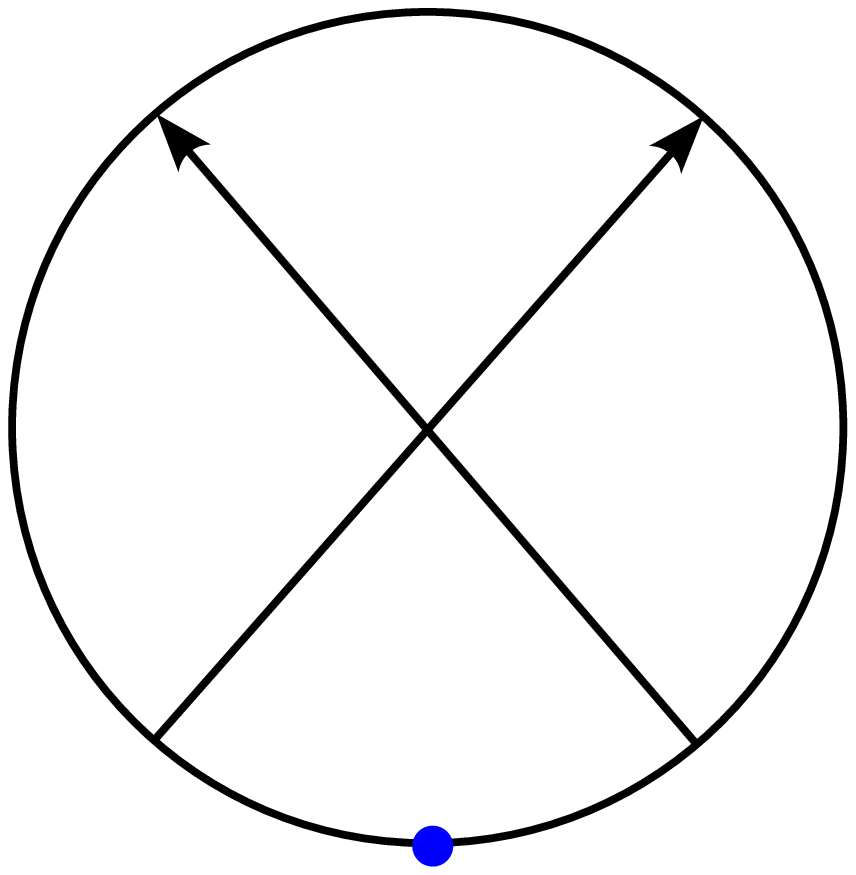}} in a Gauss diagram $G$ of a knot $K$, we singularize and smooth the two arrows as \raisebox{-.4\height}{\includegraphics[width=2cm]{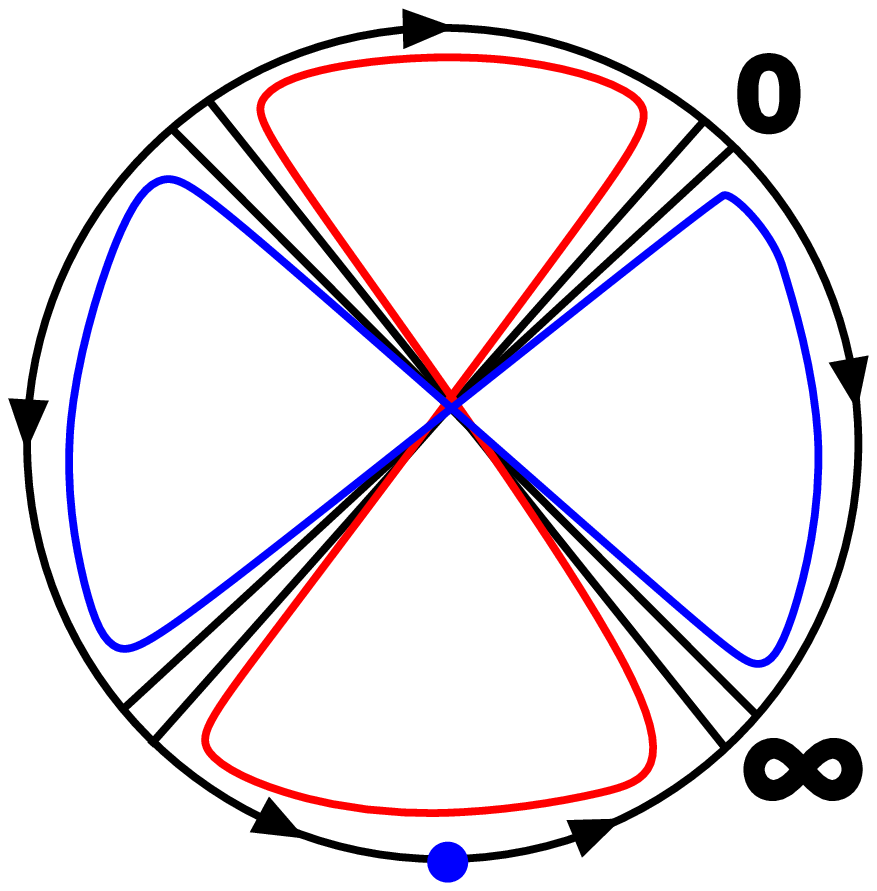}}. There are two ways to count the linking number of two components as in the proof of Proposition \ref{eq.id_deg2}. We add all the linking numbers of the two components coming from the pattern in $G$ regarding the sign of the arrows in the pattern. We have 
    \begin{align*}
        - \raisebox{-.4\height}{\includegraphics[width=1.75cm]{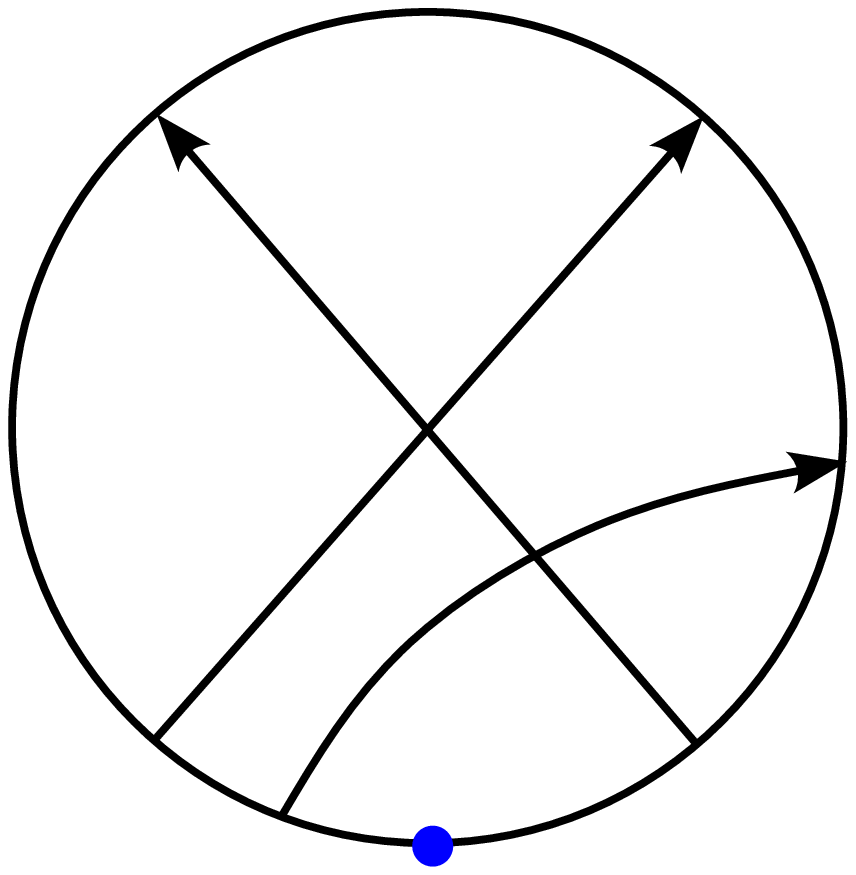}} + \raisebox{-.4\height}{\includegraphics[width=1.75cm]{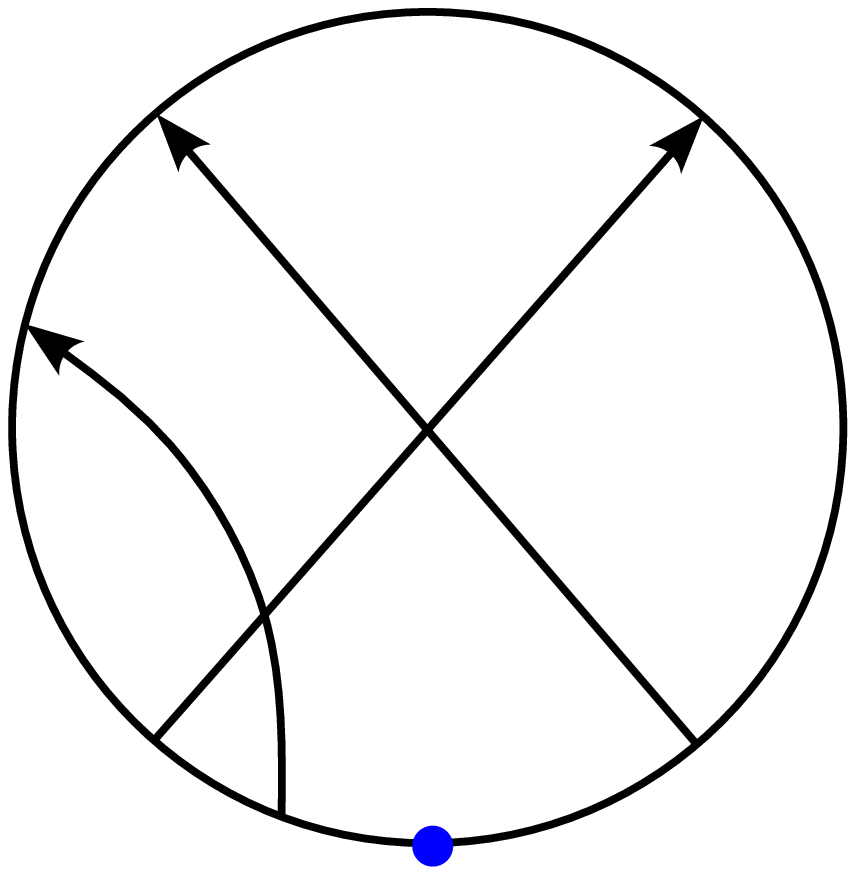}} + \raisebox{-.4\height}{\includegraphics[width=1.75cm]{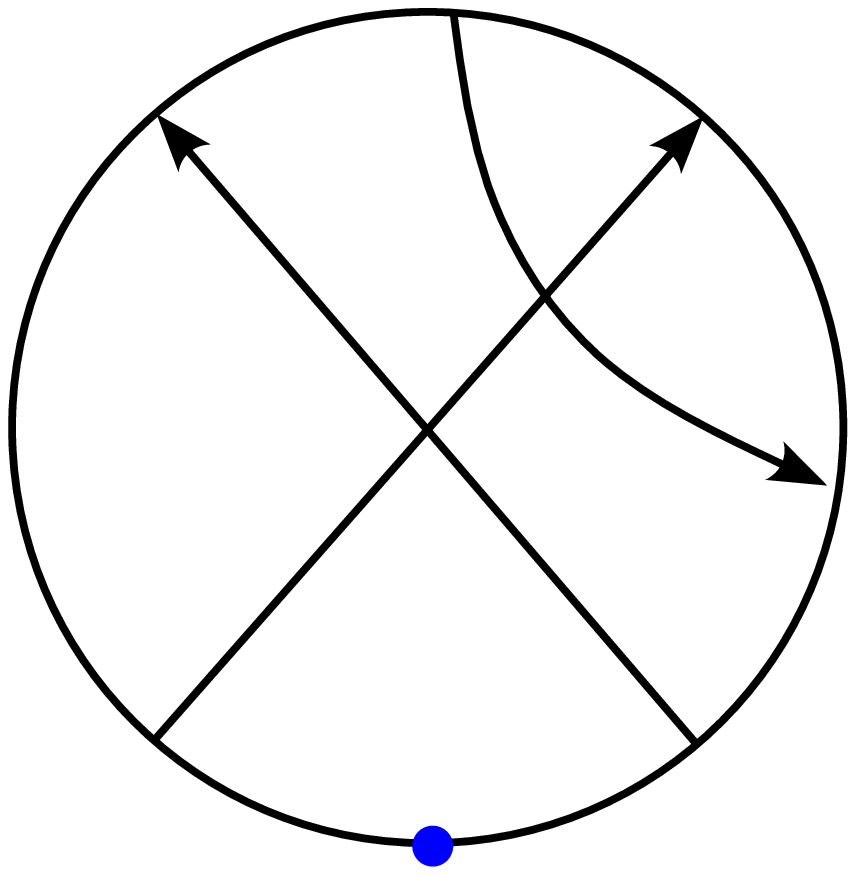}} - \raisebox{-.4\height}{\includegraphics[width=1.75cm]{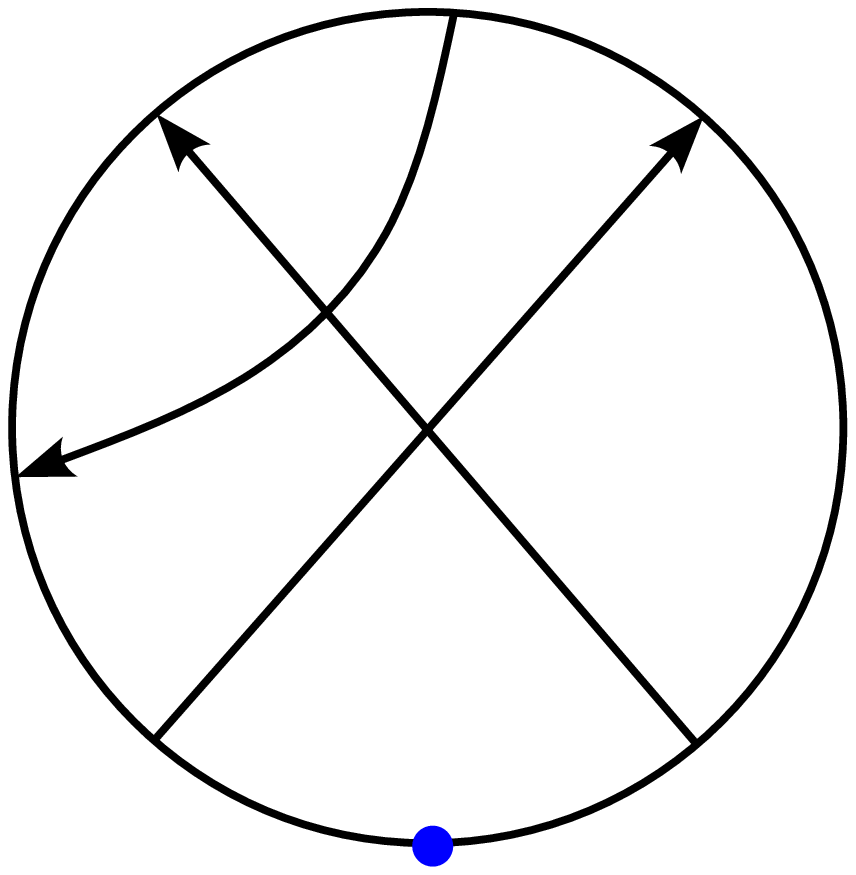}} - \raisebox{-.4\height}{\includegraphics[width=1.75cm]{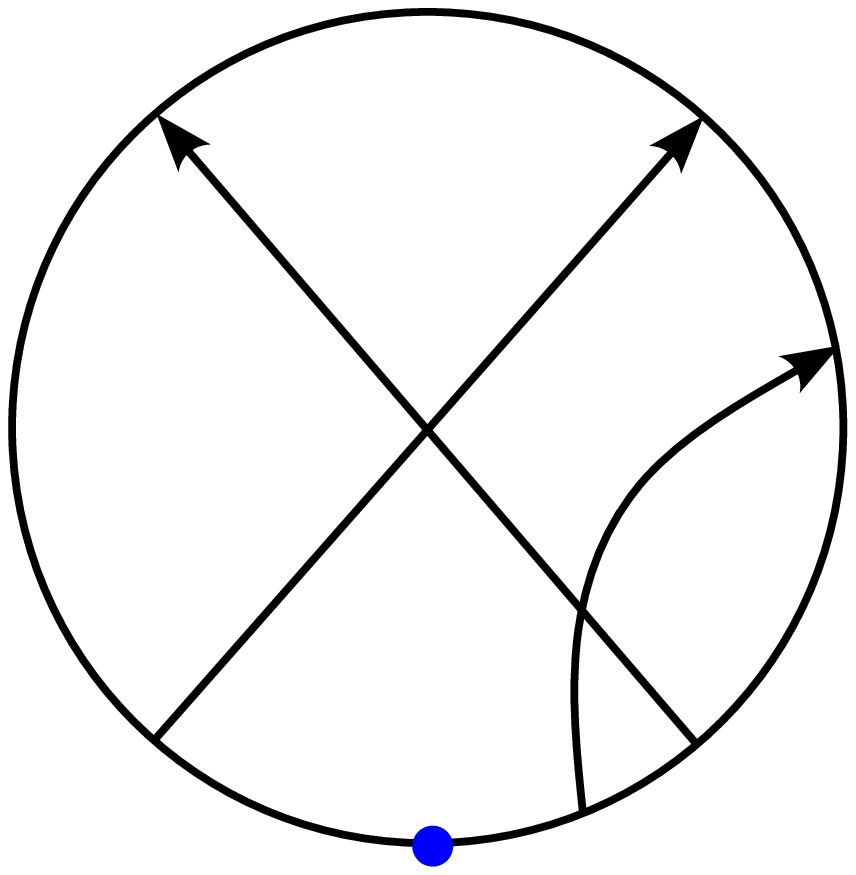}} + \raisebox{-.4\height}{\includegraphics[width=1.75cm]{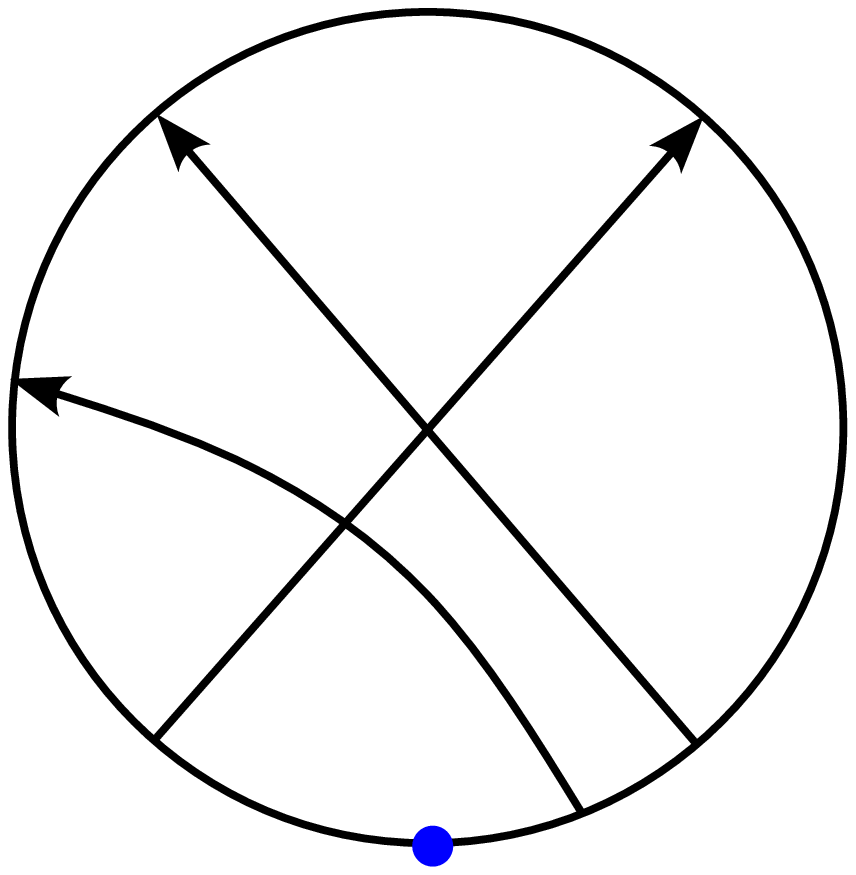}}\\
        = - \raisebox{-.4\height}{\includegraphics[width=1.75cm]{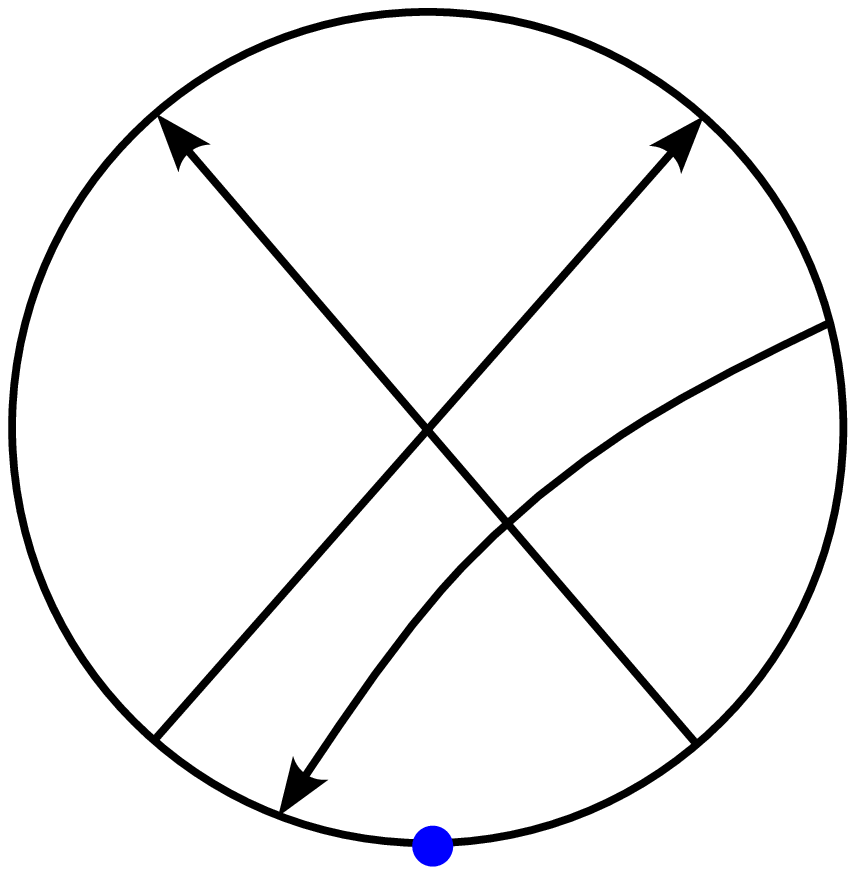}} + \raisebox{-.4\height}{\includegraphics[width=1.75cm]{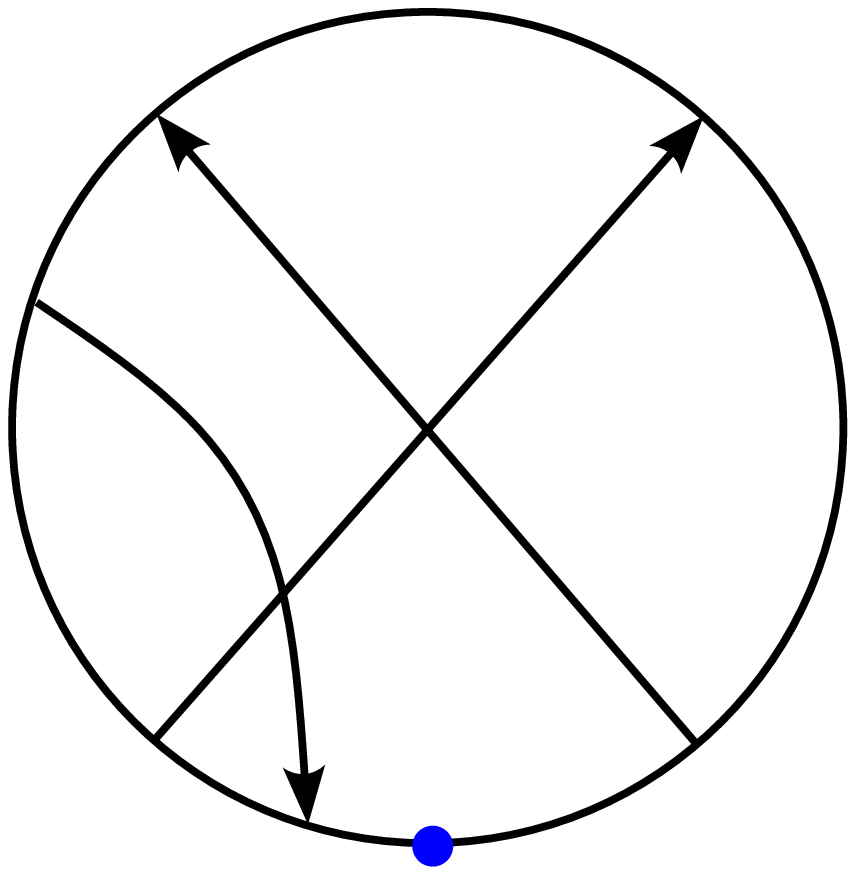}} + \raisebox{-.4\height}{\includegraphics[width=1.75cm]{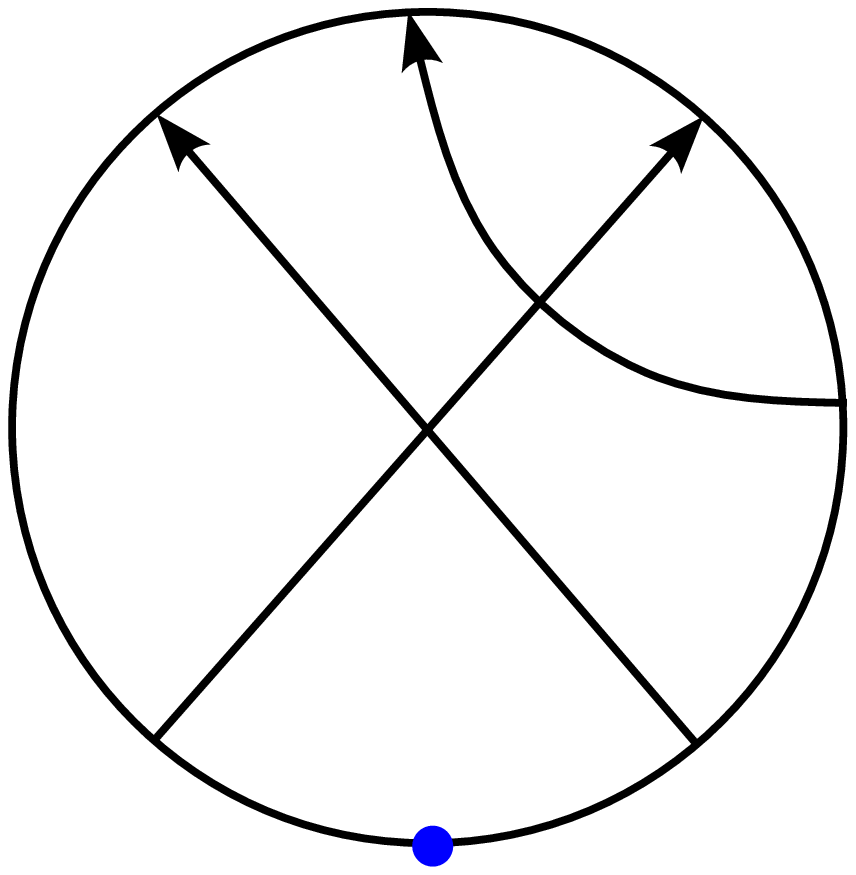}} - \raisebox{-.4\height}{\includegraphics[width=1.75cm]{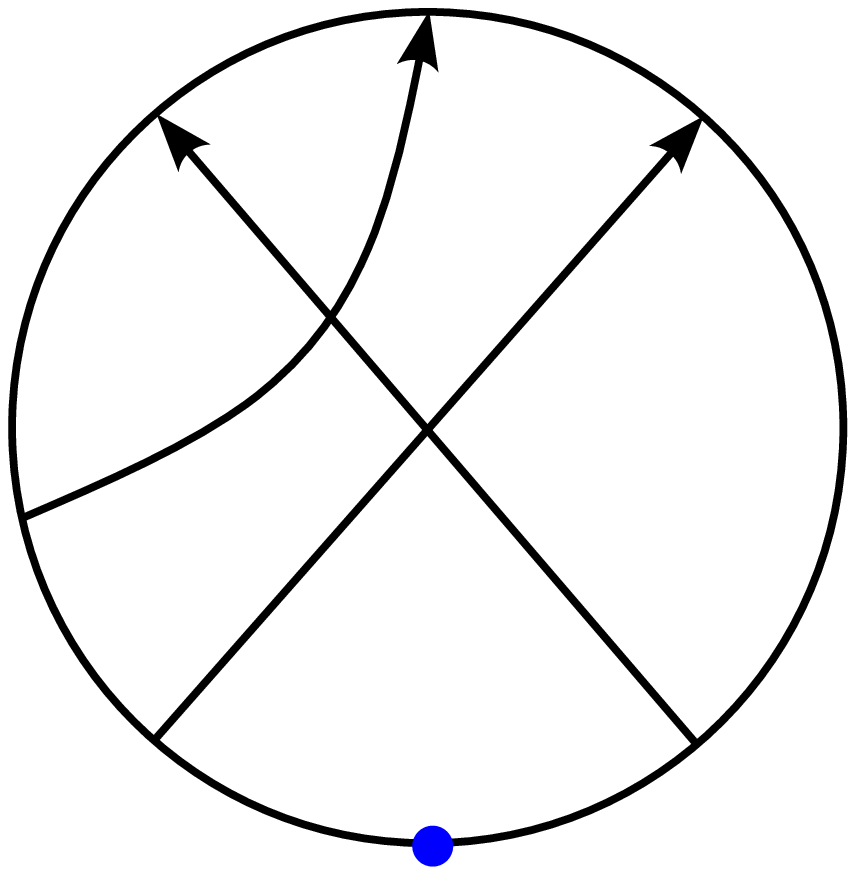}} - \raisebox{-.4\height}{\includegraphics[width=1.75cm]{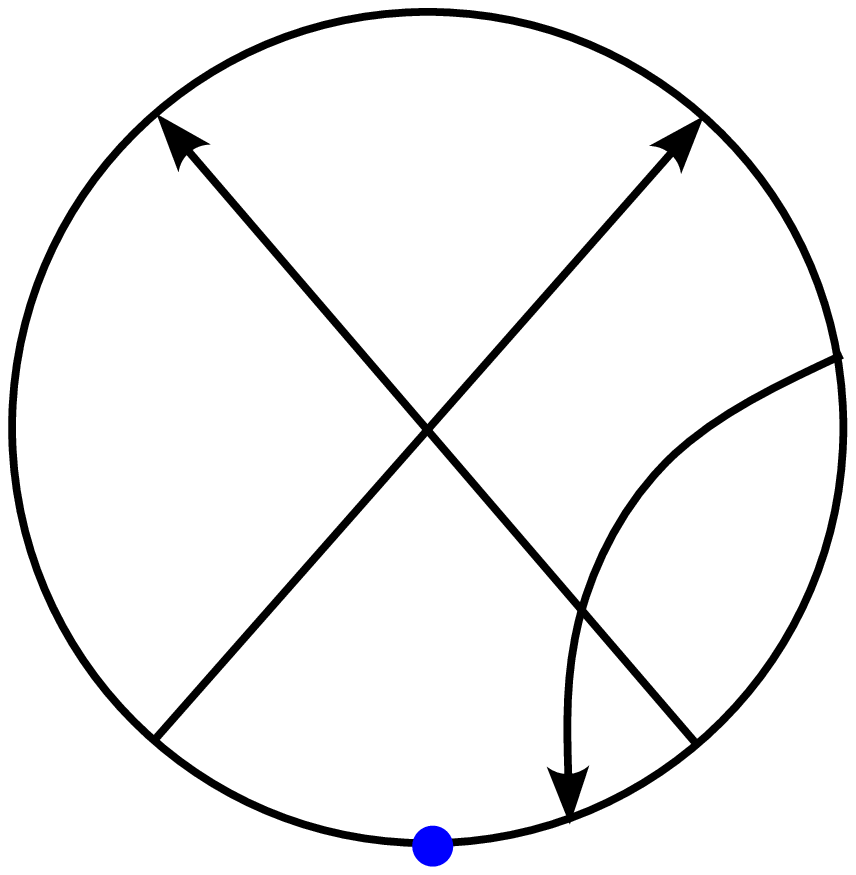}} + \raisebox{-.4\height}{\includegraphics[width=1.75cm]{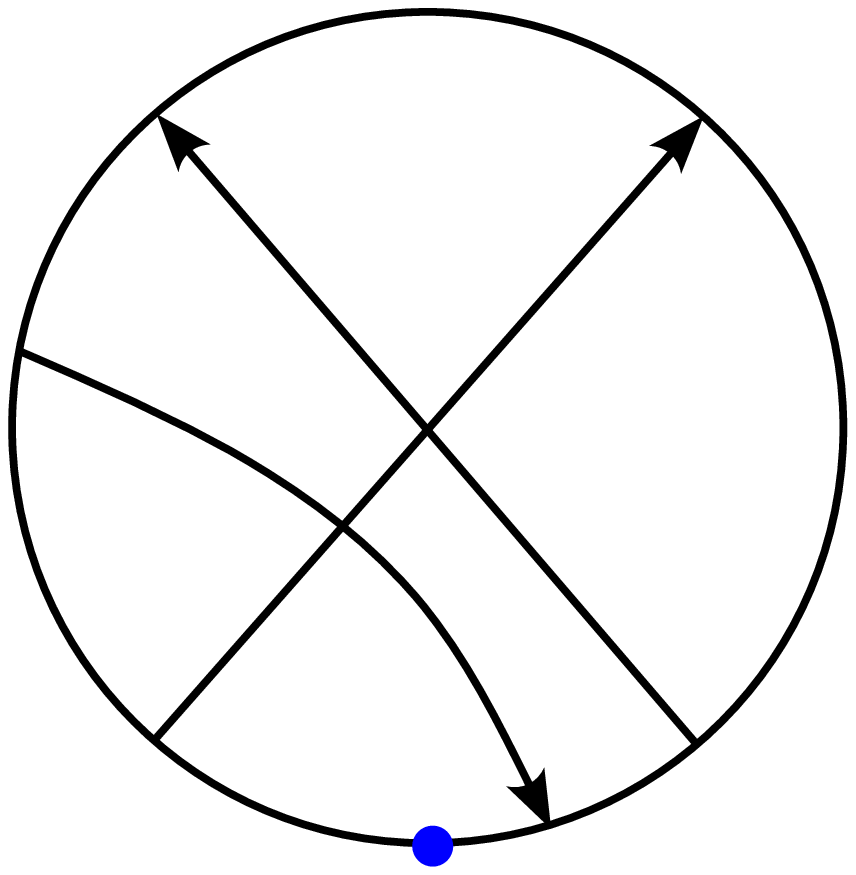}}
    \end{align*}
    Notice that the singularize operation would change the sign of the arrows on the arc that changes the orientation. After deleting the same arrow diagrams on the two sides, we get the identity. 
\end{proof}

\begin{cor}
    \begin{align*}
        \dfrac{1}{2}A_{1,2} &= \raisebox{-.4\height}{\includegraphics[width=1.75cm]{image/A_12_1.eps}} + \raisebox{-.4\height}{\includegraphics[width=1.75cm]{image/A_12_2.eps}} + \raisebox{-.4\height}{\includegraphics[width=1.75cm]{image/id4_1.eps}} + \raisebox{-.4\height}{\includegraphics[width=1.75cm]{image/id4_2.eps}} + \raisebox{-.4\height}{\includegraphics[width=1.75cm]{image/id4_3.eps}}
    \end{align*}
    $$\dfrac{1}{2}A_{1,2} + \dfrac{1}{4}A_{2,1} + \dfrac{1}{8}A_{3,0} = 0$$
    $$\dfrac{1}{2}p_{1,2} + \dfrac{1}{4}p_{2,1} + \dfrac{1}{8}p_{3,0} = 0$$
\end{cor}
\begin{rem}
    $A_{3,0}$ is the GDF given by Chmutov and Polyak in \cite{Chmutov_Polyak:2009}. If we set
    \begin{align*}
        \raisebox{-.4\height}{\includegraphics[width=1.75cm]{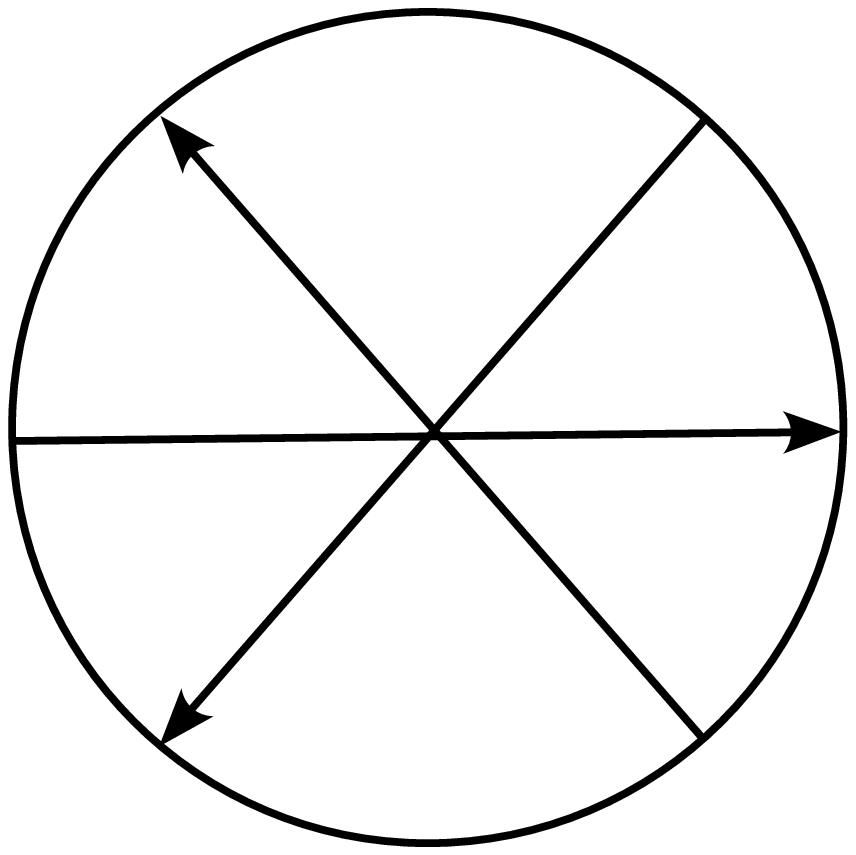}} &= 3\raisebox{-.4\height}{\includegraphics[width=1.75cm]{image/A_12_1.eps}} + 3\raisebox{-.4\height}{\includegraphics[width=1.75cm]{image/A_12_2.eps}} \\
        \raisebox{-.4\height}{\includegraphics[width=1.75cm]{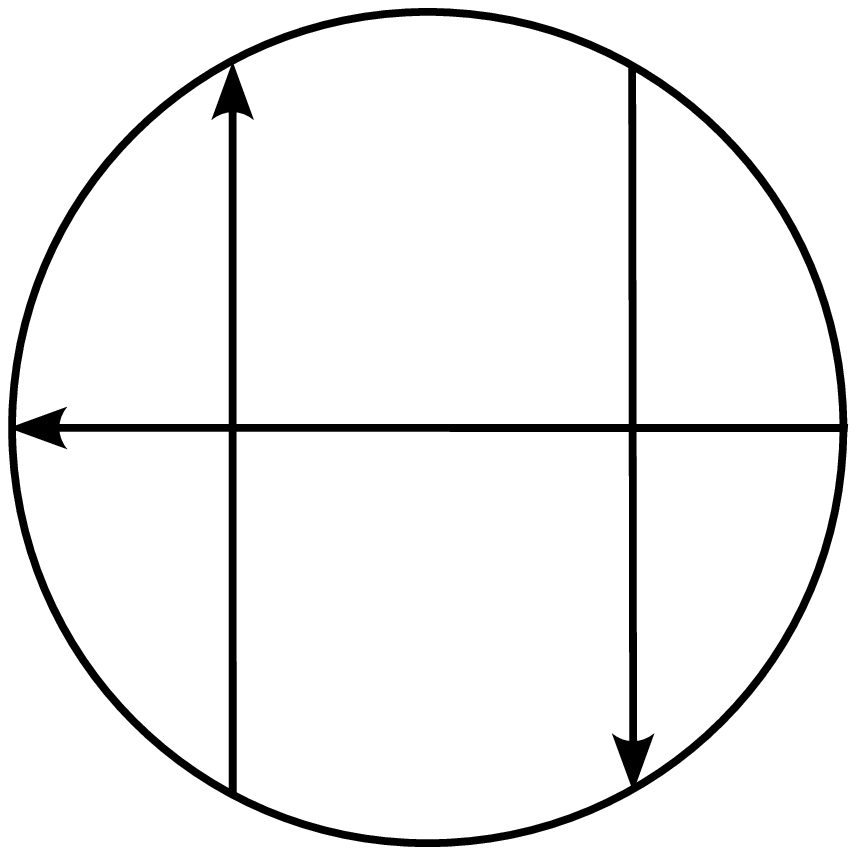}} &= \raisebox{-.4\height}{\includegraphics[width=1.75cm]{image/id3_1.eps}} + \raisebox{-.4\height}{\includegraphics[width=1.75cm]{image/id3_2.eps}} + \raisebox{-.4\height}{\includegraphics[width=1.75cm]{image/id3_3.eps}} + \raisebox{-.4\height}{\includegraphics[width=1.75cm]{image/id4_2.eps}} + \raisebox{-.4\height}{\includegraphics[width=1.75cm]{image/id4_3.eps}} + \raisebox{-.4\height}{\includegraphics[width=1.75cm]{image/id6_5.eps}}
    \end{align*}
    we will get 
    \begin{align*}
        \dfrac{1}{2}A_{1,2} = \dfrac{1}{3} \raisebox{-.4\height}{\includegraphics[width=1.75cm]{image/PV_deg3_1.eps}} + \dfrac{1}{2} \raisebox{-.4\height}{\includegraphics[width=1.75cm]{image/PV_deg3_2.eps}}\\
        \intertext{and}
        \raisebox{-.4\height}{\includegraphics[width=1.75cm]{image/PV_deg3_2.eps}} =  \raisebox{-.4\height}{\includegraphics[width=1.75cm]{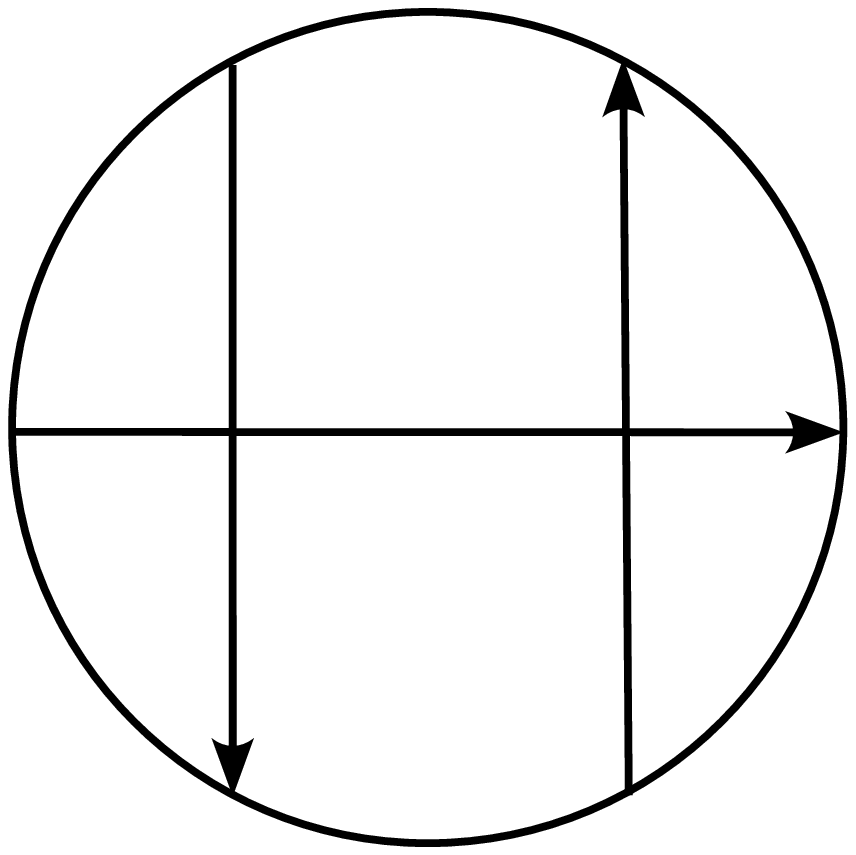}}
    \end{align*}
    which are the formulae given by Polyak and Viro in \cite{Polyak_Viro:1994}. One can find other proofs of the two formulae in \cite{Ostlund:2004} and \cite{Ito:2022}. Notice that in \cite{Ito:2022} there are a list of GDFs, some of which coincide with the formulae we give. 
\end{rem}
\section{Gauss diagram formulae from Jones Polynomial}\label{sec.GDF_Jones}
It is well-known that HOMFLY-PT polynomial and Kauffman polynomial are both generalization of Jones polynomial. We shall compare the GDFs from Jones polynomial based the 2 state models of HOMFLY-PT polynomial and Kauffman polynomial. 

Recall that Jones polynomial $V_K(t)$ satisfies the skein relation $$t^{-1}V_{+}-tV_{-} = (t^{1/2}-t^{-1/2})V_0 $$
and HOMFLY-PT polynomial $H_{K}(a,z)$ satisfies the skein relation 
$$aH_{+}-a^{-1}H_{-} = zH_{0}$$ With a change of variables we have $$V(t) = H(t^{-1},t^{1/2}-t^{-1/2})$$ and $$V(t) = F(-t^{-3/4},t^{-1/4}+t^{1/4})$$
For the technical reason we study $J(s) \coloneqq V(s^4) = H(s^{-4},s^2-s^{-2}) = F(-s^{-3},s^{-1}+s) = DK(-is^{-3},-i(s^{-1}+s))$ and find the GDFs for the coefficients of $J(ie^h) = \sum c_kh^k$. The process based on the state model of HOMFLY-PT polynomial is obtained by ignoring $\infty$ label (i.e., we do not use singularization operation on crossings.) and the corresponding weight is given by the following table (see \cite{Chmutov_Polyak:2009} and compare it to Table \ref{table.weight_arrow} for DK). 
\begin{table}[htbp]
\centering
\begin{tabular}{| c | c | c |}
\hline
& \begin{minipage}[b]{0.25\columnwidth}
		\centering
		\raisebox{-.5\height}{\includegraphics[width=\linewidth]{image/table1_7.png}}
	\end{minipage}
& \begin{minipage}[b]{0.25\columnwidth}
		\centering
		\raisebox{-.5\height}{\includegraphics[width=\linewidth]{image/table1_8.png}}
	\end{minipage} \\
\hline
\begin{minipage}[b]{0.25\columnwidth}
    \centering
	\raisebox{-.4\height}{\includegraphics[width=\linewidth]{image/table1_1.png}}
	\end{minipage}
 & $0$ & $0$ \\
\hline
\begin{minipage}[b]{0.25\columnwidth}
    \centering
	\raisebox{-.4\height}{\includegraphics[width=\linewidth]{image/table1_2.png}}
	\end{minipage}
 & $a^{-2}-1$ & $a^{2}-1$ \\
 \hline
\begin{minipage}[b]{0.25\columnwidth}
    \centering
	\raisebox{-.4\height}{\includegraphics[width=\linewidth]{image/table1_3.png}}
	\end{minipage}
 & $0$ & $0$ \\
\hline
\begin{minipage}[b]{0.25\columnwidth}
    \centering
	\raisebox{-.4\height}{\includegraphics[width=\linewidth]{image/table1_4.png}}
	\end{minipage}
 & $za^{-1}$ & $-za$ \\
 \hline
\end{tabular}
\caption{Weight $w^{H}$ for arrow diagrams}
\label{table.weight_arrow_HOMFLY}
\end{table}
We define $$W_A^{H}(a,z)=\sum \limits_{\sigma \in state(A)} w^{H}(A,\sigma)(\dfrac{a-a^{-1}}{z}+1)^{c(\sigma)-1}$$

Let $W_A^{H}((ie^h)^{-4},(ie^h)^{2}-(ie^h)^{-2}) = \sum\limits_{k}w_{k}^{H}(A)h^k$ and we define $$A_{k}^{H} \coloneqq  \sum_{A}w_{k}^{H}(A)\cdot A$$ For $DK$ let $W_A(-i(ie^h)^{-3}, -i((ie^h)^{-1}+(ie^h))) = \sum\limits_{k}w_{k}^{DK}(A)h^k$ and we define $$A_{k}^{DK} \coloneqq  \sum_{A}w_{k}^{DK}(A)\cdot A$$

$A_{k}^{H}$ and $A_{k}^{DK}$ are both well defined when $A$ runs over all the arrow diagrams and they both represent the GDF for the coefficient $c_k$ in $J(ie^h) = \sum c_kh^k$. 

\begin{eg}[GDFs for Jones Polynomial of order 2 and 3]
$$A_{2}^{H} = -48\ \raisebox{-.4\height}{\includegraphics[width=1.75cm]{image/A_20.eps}}$$
$$A_{2}^{DK} = -12\ \raisebox{-.4\height}{\includegraphics[width=1.75cm]{image/A_11.eps}} - 36\ \raisebox{-.4\height}{\includegraphics[width=1.75cm]{image/A_20.eps}} $$

\begin{align*}
    -\dfrac{1}{384}A_{3}^{H} = \dfrac{1}{8}A_{3,0} &= - \raisebox{-.4\height}{\includegraphics[width=1.75cm]{image/A_30_1.eps}} + \raisebox{-.4\height}{\includegraphics[width=1.75cm]{image/A_30_2.eps}} + \raisebox{-.4\height}{\includegraphics[width=1.75cm]{image/A_30_3.eps}} + \raisebox{-.4\height}{\includegraphics[width=1.75cm]{image/A_30_4.eps}} + \raisebox{-.4\height}{\includegraphics[width=1.75cm]{image/A_30_5.eps}} \\
    &+ \raisebox{-.4\height}{\includegraphics[width=1.75cm]{image/A_30_6.eps}} - \raisebox{-.4\height}{\includegraphics[width=1.75cm]{image/A_30_7.eps}} + \raisebox{-.4\height}{\includegraphics[width=1.75cm]{image/A_30_8.eps}} + \raisebox{-.4\height}{\includegraphics[width=1.75cm]{image/A_30_9.eps}}
\end{align*}
\begin{align*}
    A_{3}^{DK} &= 72\ \raisebox{-.4\height}{\includegraphics[width=1.75cm]{image/A_21_1.eps}} -72\ \raisebox{-.4\height}{\includegraphics[width=1.75cm]{image/A_21_2.eps}} - 72\ \raisebox{-.4\height}{\includegraphics[width=1.75cm]{image/A_21_3.eps}} - 72\ \raisebox{-.4\height}{\includegraphics[width=1.75cm]{image/A_21_5.eps}} - 72\ \raisebox{-.4\height}{\includegraphics[width=1.75cm]{image/A_21_10.eps}} \\
    &- 48\ \raisebox{-.4\height}{\includegraphics[width=1.75cm]{image/A_21_13.eps}} - 24\ \raisebox{-.4\height}{\includegraphics[width=1.75cm]{image/A_12_4.eps}} \\
    &+ 216\ \raisebox{-.4\height}{\includegraphics[width=1.75cm]{image/A_30_1.eps}} - 216\ \raisebox{-.4\height}{\includegraphics[width=1.75cm]{image/A_30_2.eps}} - 216\ \raisebox{-.4\height}{\includegraphics[width=1.75cm]{image/A_30_3.eps}} - 216\ \raisebox{-.4\height}{\includegraphics[width=1.75cm]{image/A_30_4.eps}} - 216\ \raisebox{-.4\height}{\includegraphics[width=1.75cm]{image/A_30_5.eps}} \\
    &- 264\ \raisebox{-.4\height}{\includegraphics[width=1.75cm]{image/A_30_6.eps}} - 192\ \raisebox{-.4\height}{\includegraphics[width=1.75cm]{image/A_30_9.eps}} \\
    & -96\ \raisebox{-.4\height}{\includegraphics[width=1.75cm]{image/A_12_1.eps}} -96\ \raisebox{-.4\height}{\includegraphics[width=1.75cm]{image/A_12_2.eps}} -48\ \raisebox{-.4\height}{\includegraphics[width=1.75cm]{image/A_12_3.eps}} -24\ \raisebox{-.4\height}{\includegraphics[width=1.75cm]{image/A_12_5.eps}} -48\ \raisebox{-.4\height}{\includegraphics[width=1.75cm]{image/A_12_6.eps}}\\
    &- 24\ \raisebox{-.4\height}{\includegraphics[width=1.75cm]{image/A_12_7.eps}} + 48\ \raisebox{-.4\height}{\includegraphics[width=1.75cm]{image/A_12_8.eps}} -24\ \raisebox{-.4\height}{\includegraphics[width=1.75cm]{image/A_12_9.eps}} -72\ \raisebox{-.4\height}{\includegraphics[width=1.75cm]{image/A_21_12.eps}}+ 192\ \raisebox{-.4\height}{\includegraphics[width=1.75cm]{image/A_12_11.eps}}\\
    &-240\ \raisebox{-.4\height}{\includegraphics[width=1.75cm]{image/A_12_13.eps}} -96\ \raisebox{-.4\height}{\includegraphics[width=1.75cm]{image/A_12_15.eps}}
\end{align*}
\end{eg}
By the identities, we have $A_{3}^{DK} = \dfrac{72+216}{384}A_3^{H}-\dfrac{96}{2}A_{1,2}$. Since $A_{3}^{DK} = A_{3}^{H}$, we have $\frac{1}{2}A_{1,2} = \frac{1}{8}A_{3,0}$. When using different models, we will get different GDF expressions for Jones Polynomial. 
%%%%%%%%%%%%%%%%%%%%%%%%%%%%%
%%%%%%%%%%%%%%%%%%%%%%%%%%%%%

\bibliographystyle{plain}
\bibliography{GDF_Kauffman}
\end{document}